\documentclass[reqno,a4paper]{article}

\usepackage{amssymb,amsmath,epsfig,graphics,psfrag,graphicx,color,cite,fancyhdr}

\definecolor{lightblue}{rgb}{0.22,0.45,0.70}
\usepackage[colorlinks=true,breaklinks=true,linkcolor=lightblue,citecolor=lightblue]{hyperref}
\usepackage{float}

\usepackage[T1]{fontenc}
\usepackage[utf8]{inputenc}

\usepackage[greek,francais,english]{babel}
\usepackage{color}
\usepackage{float}
\usepackage{textcomp}
\usepackage{amsthm}
\usepackage{amsmath}
\usepackage{enumerate}
\usepackage{relsize}
\usepackage{amsmath}
\usepackage{graphicx}
\usepackage{amssymb}
\usepackage{multicol,listings}
\usepackage{placeins}
\usepackage{subfigure}
\usepackage{algorithm,algorithmic}
\usepackage{makeidx}
\usepackage{amsfonts}
\usepackage{xcolor}
\usepackage{mathrsfs}

\makeatletter

\floatname{algorithm}{Algorithme}

\let\mylistof\listof
\renewcommand\listof[2]{\mylistof{algorithm}{Liste des algorithmes}}

\theoremstyle{plain}

\usepackage{hyperref}
\usepackage{fancyhdr}
\newtheorem{theo}{Theorem}[section]

\newtheorem{coro}[theo]{Corollary}

\newtheorem{rmk}[theo]{Remark}
\newtheorem{defi}[theo]{Definition}

\newcommand{\R}{\mathbb{R}}
\newcommand{\N}{\mathbb{N}}
\numberwithin{equation}{section}
\usepackage{graphicx}
\usepackage{graphics}
\usepackage{placeins}
\usepackage{epsfig}
\usepackage{enumerate}
\usepackage{psfrag}
\usepackage{subfigure}
  \usepackage{floatflt}
  \usepackage{float}

\makeatother

\usepackage{babel}
\addto\extrasfrench{\providecommand{\fg}{\ifdim\lastskip>\z@\unskip\fi~\frqq}}

\DeclareUnicodeCharacter{2212}{-}
\newtheorem*{rem*}{Remarque}

\numberwithin{equation}{section}
\newcommand\bB{\boldsymbol{B}}

\newcommand\bF{\boldsymbol{F}}
\newcommand\bG{\boldsymbol{G}}
\newcommand\bL{\boldsymbol{L}}

\newcommand\bV{\boldsymbol{V}}
\newcommand\bW{\boldsymbol{W}}
\newcommand\bX{\boldsymbol{X}}
\newcommand\bZ{\boldsymbol{Z}}

\newcommand\ff{\boldsymbol{f}}
\newcommand\bg{\boldsymbol{g}}

\newcommand\bu{\boldsymbol{u}}
\newcommand\bv{\boldsymbol{v}}
\newcommand\bx{\boldsymbol{x}}
\newcommand\by{\boldsymbol{y}}
\newcommand\bz{\boldsymbol{z}}
\newcommand\bc{\boldsymbol{c}}
\newcommand\bdd{\boldsymbol{d}}

\newcommand\bw{\boldsymbol{w}}

\newcommand{\seqna}{\begin{eqnarray}}
\newcommand{\eeqna}{\end{eqnarray}}

\renewcommand{\bar}{\overline}

\newcommand{\norm}[1]{\left\|#1\right\|}

\renewcommand\O{\Omega}

\renewcommand\L{\mathrm{L}}

\newcommand\vdiv{\mathop{\mathrm{div}}\nolimits}

\newcommand\ba{\boldsymbol{a}}
\newcommand\bb{\boldsymbol{b}}
\newcommand\bn{\boldsymbol{n}}

\newcommand\bH{\boldsymbol{H}}
\newcommand\bK{\boldsymbol{K}}

\newtheorem{theorem}{Theorem}[section]

\newtheorem{lemma}[theorem]{Lemma}
\newtheorem{proposition}[theorem]{Proposition}
\newtheorem{corollary}[theorem]{Corollary}

\newcommand\curl{\mathop{\mathbf{curl}}\nolimits}

\newcommand{\normT}[1]{{\left\vert\kern-0.25ex\left\vert\kern-0.25ex\left\vert #1 
    \right\vert\kern-0.25ex\right\vert\kern-0.25ex\right\vert}}

\newcommand{\abs}[1]{| #1 |}

\def\P{\mathbb P}

\begin{document}

%***********************************************************************************
\title{Regularity results for a model in magnetohydrodynamics with imposed pressure}
%***********************************************************************************

\author{{\sc Julien Poirier}\thanks{Normandie Univ., UNICAEN, CNRS, Laboratoire de Mathématiques Nicolas Oresme,  UMR CNRS 6139, 14000 Caen, France. E-mail:  
{\tt julien.poirier.prof@gmail.com}},   \ 
{\sc Nour Seloula}\thanks{Normandie Univ., UNICAEN, CNRS, Laboratoire de Mathématiques Nicolas Oresme,  UMR CNRS 6139, 14000 Caen, France. E-mail: {\tt nour-elhouda.seloula@unicaen.fr}}
}

\date{}
%\date{\today}
\maketitle

%***********************************************************************************
\begin{abstract}
\noindent 
The magnetohydrodynamics (MHD) problem is most often studied in a framework where
Dirichlet type boundary conditions on the velocity field is imposed. In this Note, we study the (MHD) system with pressure boundary condition, together with zero tangential trace for the velocity and the magnetic field. In a three-dimensional bounded possibly multiply connected domain, we first
prove the existence of weak solutions in the Hilbert case, and later, the regularity in $\bW^{1,p}(\O)$ for $p\geq 2$ and in $\bW^{2,p}(\O)$ for $p\geq 6/5$ using the regularity results for some Stokes and elliptic problems with this type of boundary conditions. Furthermore, under the condition of small data, we obtain the existence and uniqueness of solutions in $\bW^{1,p}(\O)$ for $3/2<p<2$ by using a fixed-point technique over a linearized (MHD) problem. 
\end{abstract}



\section{Introduction}
Let $\O$ be an open bounded set of $\R^3$ of class $\mathcal{C}^{1,1}$. In this work, we consider the following incompressible stationary magnetohydrodynamics (MHD) system:
find the velocity field $\bu$, the pressure $P$, the magnetic field $\bb$ and the constant vector $\boldsymbol{\alpha}=(\alpha_1,\ldots,\alpha_I)$ such that for $1\leq i \leq I$:
\begin{equation}\label{MHD with constants}
 %\tag{{\it MHD}}% reference a equation with any symbol
	\left\lbrace
		\begin{aligned}
		&   -\nu \, \Delta \bu + (\curl \bu) \times \bu + \nabla P - \kappa (\curl \bb) \times \bb = \ff \quad \mathrm{and}\quad  \vdiv \bu = h  \quad \mathrm{in} \,\, \Omega, \\
        & \kappa \mu \, \curl \curl \bb - \kappa \curl (\bu \times \bb) = \bg  \quad \mathrm{and}\quad \vdiv \bb = 0 \quad \text{in} \,\, \Omega, \\
       & \bu \times \bn = \textbf{0} \quad \mathrm{and}\quad \bb \times \bn = \textbf{0}\quad \mathrm{on}\,\,\Gamma,\\
        & P = P_0 \quad \text{on} \, \,\Gamma_0\quad  \mathrm{and}\quad 
         P = P_0 + \alpha_i \,\,\,\text{on} \,\, \Gamma_i,\\
         & \langle\bu \cdot \bn, 1\rangle_{\Gamma_i} = 0, \quad \mathrm{and}\quad \langle\bb \cdot \bn, 1\rangle_{\Gamma_i} = 0, \, \, \forall 1 \leqslant i \leqslant I
		\end{aligned}\right.
		\end{equation}
		where $\Gamma$ is the boundary of $\O$ which is not necessary connected. Here $\Gamma=\bigcup_{i=0}^{I}{\Gamma_{i}}$ where $\Gamma_i$ are the connected components of $\Gamma$ with $\Gamma_0$ the exterior boundary which contains $\O$ and all the other boundaries. We denote by $\bn$ the unit vector normal to $\Gamma$. The constants $\nu$, $\mu $ and $\kappa$  are constant kinematic, magnetic viscosity and a coupling number respectively. We refer
to \cite{Davidson,Gerbeau-Bris} for further discussion of typical values for these parameters. The vector $\ff$ , $\bg$ and the scalar $h$ and $P_0$ are given. In this work, we assume that $\nu=\mu=\kappa=1$ for convenience. 
			Using the identity $\bu\cdot\nabla\bu=(\curl\bu)\times\bu+\frac{1}{2}\nabla\vert\bu\vert^2$, the classical nonlinear term $\bu\cdot\nabla\bu$ in the Navier-Stokes equations is replaced by $(\curl\bu)\times\bu$. The pressure $P=p+\frac{1}{2}\vert\bu\vert^2$ is then the Bernoulli (or dynamic) pressure, where $p$ is the kinematic pressure. The boundary conditions involving the pressure are used in various physical applications. For example, in hydraulic networks, as oil ducts, microfluidic channels or the blood circulatory system. Pressure driven flows occur also in the modeling of the cerebral venous network from three-dimensional angiographic images obtained by magnetic resonance. We note that the MHD system \eqref{MHD with constants} has been extensively studied by many authors. We note that most of the contributions are often given where Dirichlet type boundary conditions on the velocity field are imposed. At a continuous level, we can refer, for example to \cite{Zeng-Zhang, Am-Saliha} for the existence and the regularity of the solutions of \eqref{MHD with constants}, to \cite{Alekseev} for the global solvability of \eqref{MHD with constants} under mixed boundary conditions for the magnetic field. For the discretization approaches of \eqref{MHD with constants}, a few related contributions include mixed finite elements \cite{Hiptmair, Schozau, Schozau2}, discontinuous Galerkin finite elements \cite{Qiu-Shi} or iterative penalty finite element methods \cite{Deng} and so on. The boundary condition under the form $P=P_0+\alpha_i$ on $\Gamma_i$, $i=1,\ldots,I$ was first introduced in \cite{Pironneau, Conca-Pironeau} for the Stokes and the Navier-Stokes systems in steady hilbertian case. The authors studied  the differences $\alpha_i-\alpha_0$, $i=1\ldots I$ which represent the unknown pressure drop on inflow and outflow sections $\Gamma_i$ in a network of pipes. This work is extended to $L^p$-theory for $1<p<\infty$ in \cite{AS_DCDS}. 
			 In our work, we study the MHD system \eqref{MHD with constants} with pressure boundary condition, together with no tangential flow and no tangential magnetic field on the boundary. Up to our knowledge, with these type of booundary conditions, this work is the first one to give a complete $L^p$-theory for the MHD system \eqref{MHD with constants} not only for large values of $p\geq 2$ but also for small values $3/2<p<2$ in $\O\subset\R^{3}$ domain with a boundary $\Gamma$ not necessary connected. \medskip
			 
			 The work is organized as follows. We start with presenting the main results of our work in section 2. In section 3, we introduce
the necessary notations and some useful results. Section 4 is devoted to the study of the linearized MHD system in Hilbert space. Using Lax-Milgram theorem, we prove the existence and uniqueness of weak solution in $\bH^1(\O)\times \bH^1(\O)\times L^2(\O)$. Later, we study the $L^p$-theory for the linearized MHD system in section 5. In particular, the proof of the regularity $\bW^{2,\,p}(\O)$ with $1< p < \frac{6}{5}$ for a non-zero divergence condition is presented in the Appendix (see Section 7). Finally, the nonlinear MHD system is discussed in section 6. The proof of the existence of weak solution in the Hilbertian case is  based on the Leray-Schauder fixed point theorem. Then, we prove the regularity of the weak solution in $\bW^{1,\,p}(\O)$ with $p >2$, and $\bW^{2,\,p}(\O)$ with $p \geq \frac{6}{5}$. For this, we use the regularity results for the Stokes and some elliptic equations combining them with a bootstrap argument. The existence of a weak solution in $\bW^{1,p}(\O)$ with $\frac{3}{2}< p< 2$ is proved by applying Banach’s fixed-point theorem over the
linearized problem.\medskip

Some results of this work are announced in \cite{Julien-Nour}
			 %These results were announced in a CRAS note in Comptes Rendus Mathématique \cite{Julien-Nour}.\\ 

\section{Main results}
In this section, we briefly discuss the main results, for which the following notations are needed:\\
For $p\in[1, \infty)$, $p'$ denotes the conjugate exponent of $p$, i.e. $\frac{1}{p'}=1-\frac{1}{p}$. We introduce the following space 
\begin{equation}\label{espace Hr,p main result} {\textbf{\textit{H}}}^{\,r,p}(\mathbf{curl},\Omega):=\left\lbrace \textbf{\textit{v}}\in\textbf{\textit{L}}^{r}(\Omega);\,\mathbf{curl}\,\, \textbf{\textit{v}}\in\textbf{\textit{L}}^{p}(\Omega)\right\rbrace,\quad \mathrm{with}\quad \frac{1}{r}=\frac{1}{p}+\frac{1}{3}
\end{equation}
equipped with the norm  $$\Vert\bv\Vert_{{\textbf{\textit{H}}}^{\,r,p}(\mathbf{curl},\Omega)}=\Vert\bv\Vert_{\bL^r(\O)}+\Vert\curl\bv\Vert_{\bL^p(\O)}.$$ 
The closure of $\boldsymbol{\mathcal{D}}(\O)$ in $ {\textbf{\textit{H}}}^{\,r,p}(\mathbf{curl},\Omega)$ is denoted by $ {\textbf{\textit{H}}}^{\,r,p}_{0}(\mathbf{curl},\Omega)$ with
$${\textbf{\textit{H}}}^{\,r,p}_{0}(\mathbf{curl},\Omega):=\left\lbrace \textbf{\textit{v}}\in {\textbf{\textit{H}}}^{\,r,p}(\mathbf{curl},\Omega);\,\,\, \bv\times\bn=\textbf{0}\,\,\,\mathrm{on}\,\,\Gamma  \right\rbrace.$$
The dual space of ${\textbf{\textit{H}}}^{\,r,p}_{0}(\mathbf{curl},\Omega)$ is denoted by $[{\textbf{\textit{H}}}^{\,r,p}_{0}(\mathbf{curl},\Omega)]'$ and its characterization is given in Proposition \ref{charac dual H^r,p}. We introduce also the kernel $$\bK^p_N(\O)=\{\bv\in\bL^p(\O);\,\,\vdiv\,\bv=0,\,\,\curl\bv=\textbf{0},\,\,\bv\times\bn=\textbf{0}\,\,\,\mathrm{on}\,\,\Gamma\},$$
which is spanned by the functions $\nabla q_i^N\in\bW^{1,q}(\O)$ for any $1< q < \infty$ \cite[Corollary 4.2]{AS_M3AS} and $q_i^N$ is the unique solution of the problem
\noindent\begin{eqnarray}\label{Definition des q_i^N}
\begin{cases}
 -\Delta q_{i}^{N}=0\quad\quad \mathrm{in}\,\,\Omega,\quad 
q_{i}^{N}\vert_{\Gamma_{0}}=0\quad\mathrm{and}\quad q_{i}^N\vert_{\Gamma_{k}}=\mathrm{constant},\,\,\,1\leq k \leq I\\
\left\langle \partial_{n}\,q_{i}^{N},\,1\right\rangle_{\Gamma_{k}}=\delta_{i\,k},\,\, 1\leq k \leq I,\quad\mathrm{and}\quad \left\langle \partial_{n}\,q_{i}^{N},\,1\right\rangle_{\Gamma_{0}}=-1.
\end{cases}
\end{eqnarray}
We will use the symbol $\sigma$ to represent a set of divergence free functions. For exemple the space $\bL^p_{\sigma}(\O)$ is the space of functions in $\bL^p(\O)$ with divergence free. We will denote by 
$C$ an unspecified positive constant which may depend on $\O$ and the dependence on other parameters will be specified if necessary.

The first theorem is concerned with the existence of weak solutions in the case of Hilbert spaces for the following MHD problem:
\begin{equation}\label{eqn:MHD}
 \tag{{\it MHD}}% reference a equation with any symbol
	\left\lbrace
		\begin{aligned}
		&   - \, \Delta \bu + (\curl \bu) \times \bu + \nabla P -  (\curl \bb) \times \bb = \ff \quad \mathrm{and}\quad  \vdiv \bu = h  \quad \mathrm{in} \,\, \Omega, \\
        & \, \curl \curl \bb -  \curl (\bu \times \bb) = \bg  \quad \mathrm{and}\quad \vdiv \bb = 0 \quad \text{in} \,\, \Omega, \\
       & \bu \times \bn = \textbf{0} \quad \mathrm{and}\quad \bb \times \bn = \textbf{0}\quad \mathrm{on}\,\,\Gamma,\\
        & P = P_0 \quad \text{on} \, \,\Gamma_0\quad  \mathrm{and}\quad 
         P = P_0 + \alpha_i \,\,\,\text{on} \,\, \Gamma_i,\\
         & \langle\bu \cdot \bn, 1\rangle_{\Gamma_i} = 0 \quad \mathrm{and}\quad \langle\bb \cdot \bn, 1\rangle_{\Gamma_i} = 0, \, \, \forall 1 \leqslant i \leqslant I
		\end{aligned}\right.
		\end{equation}
The proof is given in Subsection \ref{subsection L2 theory nonlinear} (see Theorem \ref{thm:existence_weak_sol}). We note that in the case when $\partial\O$ is not connected, to ensure the solvability of problem \eqref{eqn:MHD}, we need to impose the conditions for $\bu$ and $\bb$ on the connected components $\Gamma_i$: $ \langle\bu \cdot \bn, 1\rangle_{\Gamma_i} = 0 \quad \mathrm{and}\quad \langle\bb \cdot \bn, 1\rangle_{\Gamma_i} = 0$ for $1\leq i \leq I$.
(See \cite{AS_M3AS} and \cite{AS_DCDS} for an equivalent form of these conditions). Of course, if $\partial \O$ is connected, the above conditions are no longer necessary.

\begin{theo}\emph{(Weak solutions of the \eqref{eqn:MHD} system in $\bH^{1}(\O)$)}.\label{Main results theo result sol H^1}
Let $ \ff, \bg \in [\bH_0^{6,2}(\curl, \Omega)]'$, $h=0$ and $P_0 \in H^{-\frac{1}{2}}(\Gamma)$ with the compatibility conditions 
\begin{eqnarray} 
\forall \, \bv \in \bK_N^2(\Omega),\quad\langle \bg, \bv \rangle_{\O_{_{6,2}}} = 0,  \label{Condition compatibilite K_N hilbert: main results} \\
\vdiv \bg = 0 \quad \mathrm{in}  \,\, \Omega,\label{Condition div nulle: main results1}
 \end{eqnarray}
 where $\langle\cdot,\cdot\rangle_{\O_{_{r,p}}}$ denotes the duality product between $[\bH_0^{r,p}(\curl, \Omega)]'$ and $\bH_0^{r,p}(\curl, \Omega)$. 
  Then the \eqref{eqn:MHD} problem has at least one weak solution $
(\bu, \bb, P, \boldsymbol{\alpha}) \in \bH^1(\Omega) \times \bH^1(\Omega) \times L^2(\Omega) \times \R^{I}$ such that
\begin{equation*}\label{estim weak H1 solution MHD}
 \norm{\bu}_{\bH^1(\O)}+\norm{\bb}_{\bH^1(\O)}+\norm{P}_{L^2(\O)}\leq M,
\end{equation*}
where $M=C(\norm{\ff}_{[\bH_0^{6,2}(\curl, \Omega)]'} + \norm{\bg}_{[\bH_0^{6,2}(\curl, \Omega)]'} + \norm{P_0}_{H^{-1/2}(\Gamma)}) $ and $ \boldsymbol{\alpha}=(\alpha_1,\ldots,\alpha_I)$ defined by
\begin{equation}\label{def alpha_i}
\alpha_i=\langle \ff, \nabla q_i^N \rangle_{\O_{_{6,2}}} - \langle P_0, \nabla q_i^N \cdot \bn\rangle_{\Gamma} +  \int_\Omega (\curl \bb) \times \bb \cdot \nabla q_i^N \, dx- \int_\Omega (\curl \bu) \times \bu \cdot \nabla q_i^N \, dx, 
 \end{equation}
where $\langle\cdot,\cdot\rangle_{\Gamma}$ denotes the duality product between $H^{-1/2}(\Gamma)$ and $H^{1/2}(\Gamma)$. 

In addition, suppose that $\ff$, $\bg$ and $P_0$ are small in the sense that 
\begin{eqnarray}\label{condition uniqueness sol H^1 nonlinear}
C_1C_2^2 M \leq \frac{2}{3 C_{\mathcal{P}} ^2},
\end{eqnarray}
where ${C_{\mathcal{P}}}$ is the constant in \eqref{Poincaré inequality} and $C_1$, $C_2$ are the constants defined in \eqref{def C_1 and C_2}. Then the weak solution $(\bu,\bb,P)$ of \eqref{eqn:MHD} is unique.
\end{theo}

The next two theorems are concerned with generalized solutions in $\bW^{1,p}(\O)$ for $p>2$ and strong solutions in $\bW^{2,p}(\O)$ for $p\geq 6/5$. The existence of
weak solution in $\bW^{1,p}(\O)$ for $\frac{3}{2}< p <2$ is not trivial. We will precise this case later. 

\begin{theo}(Weak solutions in $\bW^{1,p}(\O)$ with $p>2$ for the \eqref{eqn:MHD} system)\label{Main results thm:regularity_weak_sol W^1,p p>2}. Let $p>2$. Suppose that $\ff, \bg\in [\bH_0^{r',p'}(\curl,\Omega)]'$, $h=0$ and $P_0 \in W^{1-\frac{1}{r},r}(\Gamma)$ with the compatibility condition \eqref{Condition div nulle: main results1} and 
\begin{eqnarray} 
\forall \, \bv \in \bK_{N}^{p'}(\Omega),\quad\langle \bg, \bv \rangle_{{\O}_{r',p'}} = 0.  \label{Condition compatibilite K_N Lp: main results} 
%\vdiv \bg = 0 \quad \mathrm{in}  \,\, \Omega,\label{Condition div nulle: main results2}
\end{eqnarray}
 Then the weak solution for the \eqref{eqn:MHD} system given by Theorem \ref{Main results theo result sol H^1} satisfies $$(\bu, \bb, P) \in \bW^{1,p}(\Omega) \times \bW^{1,p}(\Omega) \times W^{1,r}(\Omega).$$ Moreover, we have the following estimate:
\begin{eqnarray*}\label{estim weak p>2 MHD}
\begin{aligned}
\!\norm{\bu}_{\bW^{1,p}(\Omega)} \!+\! \norm{\bb}_{\bW^{1,p}(\Omega)}\!+\! \norm{P}_{W^{1,r}(\Omega)} \!\leqslant\! C (\norm{\ff}_{(\bH_0^{r',p'}(\curl, \Omega))'} \!+\! \norm{\bg}_{(\bH_0^{r',p'}(\curl, \Omega))'} \!+\! \norm{P_0}_{W^{1/r',r}(\Gamma)}) 
\end{aligned}
\end{eqnarray*}
\end{theo}

\begin{theo}(Strong solutions in $\bW^{2,p}(\O)$ with $p\geq\frac{6}{5}$ for the \eqref{eqn:MHD} system).\label{thm:regularity_Strong_sol_W2,p p>6/5: main results}
Let us suppose that $\O$ is of class $\mathcal{C}^{2,1}$ and $p\geq \frac{6}{5}$. Let $\ff$, $\bg$ and $P_0$ with the compatibility conditions \eqref{Condition div nulle: main results1} and \eqref{Condition compatibilite K_N Lp: main results} and  
$$\ff \in \bL^p(\Omega), \quad \bg \in \bL^p(\Omega), \quad h=0\quad \mathrm{and}\quad P_0 \in W^{1-\frac{1}{p},p}(\Gamma).$$
Then the weak solution  $(\bu, \bb, P)$ for the \eqref{eqn:MHD} system given by Theorem \ref{Main results theo result sol H^1} belongs to \\ $ \bW^{2,p}(\Omega) \times \bW^{2,p}(\Omega) \times W^{\,1,p}(\Omega)$ and satisfies the following estimate:
\begin{eqnarray*}\label{estim strong p>6/5 MHD}
\begin{aligned}
&\norm{\bu}_{\bW^{2,p}(\Omega)} + \norm{\bb}_{\bW^{2,p}(\Omega)}+ \norm{P}_{W^{1,p}(\Omega)} \leqslant C(\norm{\ff}_{\bL^p(\Omega)} + \norm{\bg}_{\bL^p(\Omega)} + \norm{P_0}_{W^{1-\frac{1}{p},p}(\Gamma)}) 
\end{aligned}
\end{eqnarray*}
\end{theo}

We refer to Theorem \ref{thm:regularity_weak_sol1} and Theorem \ref{thm:regularity_weak_sol1_W2p} for the proof of the above result, where we use the estimates obtained in the Hilbert case and a bootstrap argument using regularity results of some Stokes and elliptic problems in \cite{AS_M3AS} and \cite{AS_DCDS}. \\
To deal with the regularity of the solutions of the \eqref{eqn:MHD} system in $\bW^{1,p}(\O)$ with $\frac{3}{2}< p <2$, we need to study the following linearized MHD system: Find ($\bu$, $\bb$,  $\mathrm{P}$, $\bc$) with $\bc=(c_1,\ldots,c_I)$ such that for $1\leq i \leq I$:
\begin{equation}\label{linearized MHD-pressure: main results}
	\left\lbrace
		\begin{split}
		   - \, \Delta \bu + (\curl \bw) \times \bu + \nabla \mathrm{P} -  (\curl \bb) \times \bdd = \ff \quad &\mathrm{and}\quad  \vdiv \bu = h   \quad \mathrm{in} \,\, \Omega, \\
           \, \curl \curl \bb -  \curl (\bu \times \bdd) = \bg  \quad & \mathrm{and}\quad \vdiv \bb = 0 \quad \text{in} \,\, \Omega, \\
        \bu \times \bn = \textbf{0} \quad &\mathrm{and}\quad \bb \times \bn = \textbf{0}\quad \mathrm{on}\,\,\Gamma,\\
        \mathrm{P} = P_0 \quad \text{on} \, \,\Gamma_0\quad & \mathrm{and}\quad 
         \mathrm{P} = P_0 + c_i \,\,\,\text{on} \,\, \Gamma_i,\\
         \langle\bu \cdot \bn, 1\rangle_{\Gamma_i} = 0, \quad &\mathrm{and}\quad \langle\bb \cdot \bn, 1\rangle_{\Gamma_i} = 0.
		\end{split}\right.
\end{equation}

The next theorem gives existence of weak and strong solutions for the linearized problem \eqref{linearized MHD-pressure: main results}

\begin{theo}(Existence of weak and strong solutions of the linearized MHD problem)\label{thm weak ans strong solutions linearized system: main results}. Suppose that $$\ff, \bg\in [\bH_0^{r',p'}(\curl,\Omega)]',\quad P_0 \in W^{1-\frac{1}{r},r}(\Gamma),\quad h \in W^{1,r}(\Omega)$$ with the compatibility conditions \eqref{Condition div nulle: main results1} and \eqref{Condition compatibilite K_N Lp: main results}.\\
 \textbf{\emph{(1).}} For any $p\geq 2$, if  $\curl\bw\in\bL^{s}(\O)$, $\bdd\in \bW^{1,s}_{\sigma}(\O)$ where $s$ is given by
\begin{equation*}\label{def of s for weak p>2}
s= \frac{3}{2} \quad \mathrm{if} \,\,2<p < 3,\quad s>\frac{3}{2}\quad \mathrm{if}\,\,p=3\quad \mathrm{and}\quad s=r\quad \mathrm{if}\,\,p>3,
\end{equation*} then the linearized system \eqref{linearized MHD-pressure: main results}
has a unique solution $(\bu, \bb, P, \bc) \in \bW^{1,p}(\Omega) \times \bW^{1,p}(\Omega) \times W^{1,r}(\Omega) \times \R^I$ with $\bc=(c_1,\ldots,c_I)$. Moreover, we have the estimate: 
\small
\begin{eqnarray*}\label{estim W^,1,p p>2 div u= phi_u}
\begin{aligned}
&\norm{\bu}_{\bW^{1,p}(\Omega)} + \norm{\bb}_{\bW^{1,p}(\Omega)}+ \norm{P}_{W^{1,r}(\Omega)} \!\leq\!  C(1+\norm{\curl\bw}_{\bL^s(\Omega)} +\norm{\bdd}_{\bW^{1,s}(\O)}\!)\big(\norm{\ff}_{[\bH_0^{r',p'}(\curl,\Omega)]'} \\&+ \! \norm{P_0}_{W^{1-\frac{1}{r},r}(\Gamma)} +\!\norm{\bg}_{[\bH_0^{r',p'}(\curl, \Omega)]}\!+\!(1 \!+\!  \norm{\curl\bw}_{\bL^s(\Omega)}\! +\! \norm{\bdd}_{\bW^{1,s}(\O)})\norm{h}_{W^{1,r}(\O)}\big) 
\end{aligned}
\end{eqnarray*}
\textbf{\emph{(2).}} Let $\frac{3}{2}<p < 2$. If  $\curl\bw\in\bL^{3/2}(\O)$ and $\bdd\in\bW^{1,3/2}_{\sigma}(\O)$, then the linearized problem \eqref{linearized MHD-pressure: main results} has a unique solution $(\bu, \bb, P, \bc) \in \bW^{1,p}(\Omega) \times \bW^{1,p}(\Omega) \times W^{1,r}(\Omega) \times \R^I$. Moreover, we have the following estimates: 
\small
 \begin{eqnarray*}\label{estim u,b W^1,p for p<2}
 \begin{aligned} 
&\norm{\bu}_{\bW^{1,\,p}(\Omega)} + \norm{\bb}_{\bW^{1,\,p}(\Omega)}\leq  C(1 + \norm{\curl\bw}_{\bL^{3/2}(\Omega)} + \norm{\bdd}_{\bW^{1,3/2}(\O)})\Big(\norm{\ff}_{[\bH_0^{r',p'}(\curl,\Omega)]'}+  \norm{P_0}_{W^{1-\frac{1}{r},r}(\Gamma)}\\& +\norm{\bg}_{[\bH_0^{r',p'}(\curl, \Omega)]}+(1 +  \norm{\curl\bw}_{\bL^{3/2}(\Omega)} + \norm{\bdd}_{\bW^{1,3/2}(\O)})\norm{h}_{W^{1,r}(\O)}\Big) 
\end{aligned}
\end{eqnarray*}
and
\small
 \begin{eqnarray*}\label{estim P W^1,r for p<2}
 \begin{aligned} 
 &\norm{P}_{W^{1,r}(\Omega)} \leq C(1 + \norm{\curl\bw}_{\bL^{3/2}(\Omega)} + \norm{\bdd}_{\bW^{1,3/2}(\O)})^2\times\Big(\norm{\ff}_{[\bH_0^{r',p'}(\curl,\Omega)]'} +\norm{P_0}_{W^{1-\frac{1}{r},r}(\Gamma)}\\&+ \norm{\bg}_{[\bH_0^{r',p'}(\curl, \Omega)]}+(1+\norm{\curl\bw}_{\bL^{3/2}(\Omega)} + \norm{\bdd}_{\bW^{1,3/2}(\O)})\norm{h}_{W^{1,r}(\O)}\Big).
\end{aligned}
\end{eqnarray*} 
\textbf{\emph{(3).}} Also for any $p\in(1,\,\infty)$, if $\O$ is of class $\mathcal{C}^{2,1}$, $h=0$ in $\O$ and 
$$\ff\in \bL^p(\Omega),\,\ \bg \in \bL^p(\Omega),\,\,\curl\bw\in \bL^{3/2}(\O),\,\,\bdd\in\bW^{\,1,3/2}(\O),\,\,\,\mathrm{and} \,\,\,P_0 \in W^{1-\frac{1}{p}, p}(\Gamma)$$ with the compatibility conditions \eqref{Condition div nulle: main results1} and \eqref{Condition compatibilite K_N Lp: main results}, then $(\bu, \bb, P, \bc)$ belongs to $\bW^{2,p}(\Omega) \times \bW^{2,p}(\Omega) \times W^{1,p}(\Omega) \times \R^I$ and satisfies the estimate:
\begin{eqnarray*}
\begin{aligned}\label{estim u,b,P strong p>6/5}
\norm{\bu}_{\bW^{2,p}(\Omega)} + \norm{\bb}_{\bW^{2,p}(\Omega)}+ \norm{P}_{W^{1,p}(\Omega)} &\leq C(1 + \norm{\curl\bw}_{\bL^\frac{3}{2}(\Omega)} + \norm{\bdd}_{\bW^{1,3/2}(\Omega)}) \\
&\times \big(\norm{\ff}_{\bL^p(\Omega)} + \norm{\bg}_{\bL^p(\Omega)} + \norm{P_0}_{W^{1-\frac{1}{p},p}(\Gamma)}\big) 
\end{aligned}
\end{eqnarray*}
where $C=C(\O,p)$ if $p\geq 6/5$ and $C=C(\O,p)(1 + \norm{\curl\bw}_{\bL^\frac{3}{2}(\Omega)} + \norm{\bdd}_{\bW^{1,3/2}(\Omega)}) $ if $1<p<6/5$.
\end{theo}
We refer to Theorem \ref{thm weak sol W^1,p p>2 with phi_u and chi}, Theorem \ref{thm: weak W1,p for p<2}, Theorem \ref{thm strong solution p>6/5} and Theorem \ref{thm: strong W2,p for 1<p<6/5} for the proof of the above results. Note that in the above theorem, to prove the existence of
weak solutions in $\bW^{1,p}(\O)$ with $3/2< p < 2$, we use a duality argument.\\
We also note that we proved more general existence results in Corollary \ref{Corollary improvement pressure} where the regularity of the pressure is improved by supposing a data $P_0$ less regular. 

Finally, the next result shows the existence and uniqueness of weak solutions with $3/2 < p < 2$ for the nonlinear \eqref{eqn:MHD} problem (see
Theorem \ref{thm:regularity_weak_ MHD p<2}). The proof is essentially based on the estimates obtained above for the linearized problem \eqref{linearized MHD-pressure: main results}. 

\begin{theo}(Regularity $\bW^{1,p}(\O)$ with $\frac{3}{2}< p<2$ for the \eqref{eqn:MHD} system).\label{thm:regularity_weak_ MHD p<2: main results} Assume that $\frac{3}{2} < p < 2$ and $r$ with $\frac{1}{r}=\frac{1}{p}+\frac{1}{3}$. 
Let us consider $\ff, \bg \in [\bH_0^{r',p'}(\curl, \Omega)]'$, $P_0 \in W^{1-\frac{1}{r},r}(\Gamma)$ and $h\in W^{1,r}(\O)$ with the compatibility conditions \eqref{Condition div nulle: main results1} and \eqref{Condition compatibilite K_N Lp: main results}. 

\emph{\textbf{(i)}} There exists a constant $\delta_1$ such that, if 
\begin{eqnarray*}
\norm{\ff}_{[\bH_0^{r',p'}(\curl, \Omega)]'} + \norm{\bg}_{[\bH_0^{r',p'}(\curl,\Omega)]'} + \norm{P_0}_{W^{1-\frac{1}{r},r}(\Gamma)}+\norm{h}_{W^{1,r}(\O)} \leq \delta_1
\end{eqnarray*}
Then, the \eqref{eqn:MHD} problem has at least one solution $(\bu, \bb, P, \boldsymbol{\alpha}) \in \bW^{1,p}(\Omega) \times \bW^{1,p}(\Omega) \times W^{1,r}(\O) \times \R^I$. Moreover, we have the following estimates:
\begin{eqnarray} \label{estim u,b non linear p<2}
\begin{aligned}
\!\!\norm{\bu}_{\bW^{1,p}(\Omega)}\! +\! \norm{\bb}_{\bW^{1,p}(\Omega)}\! 
\leqslant \!C_1\! &\big( \norm{\ff}_{[\bH_0^{r',p'}(\curl, \Omega)]'} \!+ \!\norm{\bg}_{[\bH_0^{r',p'}(\curl,\Omega)]'}\! +\! \norm{P_0}_{W^{1-\frac{1}{r},r}(\Gamma)} \\
&\!+ \!\norm{h}_{W^{1,r}(\O)} \big)
\end{aligned}
\end{eqnarray}
\begin{eqnarray} \label{estim P non linear p<2}
\begin{aligned}
\norm{P}_{W^{1,r}(\O)}\!\!\leqslant \!C_1 (1+C^*\eta)&\big( \!\norm{\ff}_{[\bH_0^{r',p'}(\curl, \Omega)]'} \!+\! \norm{\bg}_{[\bH_0^{r',p'}(\curl,\Omega)]'} \!+\! \norm{P_0}_{W^{1-\frac{1}{r},r}(\Gamma)}\! \\
&+\!\norm{h}_{W^{1,r}(\O)}\!\! \big),
\end{aligned}
\end{eqnarray}
where $\delta_1=(2C^2C^*)^{-1}$, $C_1=C(1+C^*\eta)^2$ with $C>0$, $C^*>0$ are the constants given in \eqref{definition C^* preuve W1p p<2 non lin} and $\eta$ defined by \eqref{definition eta preuve W1p p<2 non lin}. Furthermore, we have for all $1 \leqslant i \leqslant I$
\begin{eqnarray*}
\begin{aligned}
&\alpha_i = \langle \ff, \nabla q_i^N \rangle_{\O}- \int_{\O} (\curl \bw) \times \bu\cdot \nabla q_i^N \,d \bx + \int_{\O}(\curl \bb) \times \bdd \cdot \nabla q_i^N \,d\bx+ \int_{\Gamma}(h -P_0) \nabla q^{N}_{i}\cdot\bn\,\,d\sigma  
\end{aligned}
\end{eqnarray*}
\emph{\textbf{(ii)}} Moreover, if the data satisfy that 
\begin{eqnarray*}
\norm{\ff}_{[\bH_0^{r',p'}(\curl, \Omega)]'} + \norm{\bg}_{[\bH_0^{r',p'}(\curl,\Omega)]'} + \norm{P_0}_{W^{1-\frac{1}{r},r}(\Gamma)}  +\norm{h}_{W^{1,r}(\O)} \leq \delta_2,
\end{eqnarray*}
for some $\delta_2\in]0,\,\delta_1]$, then the weak solution of \eqref{eqn:MHD} problem is unique.
\end{theo}

\section{Notations and preliminary results}
Before studying the MHD problem \eqref{eqn:MHD}, we introduce some basic notations
and specific functional framewor. If we do not state otherwise, $\O$ will be considered as an open bounded domain of $\R^3$, which is not necessary connected, of class at least $\mathcal{C}^{1,1}$ and sometimes of class  $\mathcal{C}^{2,1}$. We denote by $\Gamma_{i}$, $0\leq i \leq I$, the connected components of $\Gamma$, $\Gamma_{0}$ being the boundary of the only unbounded connected component of $\mathbb{R}^{3}\backslash\overline{\Omega}$. 
% We fix a smooth open set $\mathcal{O}$ with a connected boundary, such that $\overline{\Omega}$ is contained in $\mathcal{O}$, and we denote by $\Omega_{i}$, $1\leq i\leq I$, the connected components of $\mathcal{O}\backslash\overline{\Omega}$ with boundary $\Gamma_{i}\,\,(\Gamma_{0}\cup\partial\mathcal{O} \,\,\mathrm{for}\,\, i=0)$. We do not assume that $\Omega$ is simply-connected, but we suppose that there exists $J$ connected, oriented and open surfaces $\Sigma_{j}$, $1\leq j\leq J$, called \textquoteleft cuts\textquoteright, contained in $\Omega$, such that each surface $\Sigma_{j}$ is an open subset of a smooth manifold $\mathcal{M}_{j}$, the boundary of $\Sigma_{j}$ is contained in $\Gamma$ for $1\leq j\leq J$, the intersection $\overline{\Sigma_{i}}\cap\overline{\Sigma_{j}}$ is empty for $i\neq j$, and the open set $\Omega^{\circ}=\Omega\backslash\bigcup_{j=1}^{J}\Sigma_{j}$ is simply-connected. For $J=1$ with $I=3$, see for example Figure \ref{Fig:geometry}.
% \begin{figure}[H]
% \begin{center}
% \includegraphics[width=9cm]{image.jpg}
% %\begin{center}
% \caption{}
% \label{Fig:geometry}
% %\end{center}
% \end{center}
% \end{figure}\noindent

The vector fields and matrix fields as well as the corresponding spaces are denoted by bold font. We will use $C$ to denote a generic positive constant which may depend on $\O$ and the dependence on other parameters will be specified if necessary. 
For $1<p<\infty$, $\bL^p(\O)$ denotes the usuel vector-valued $\bL^p$-space over $\O$. As usual, we denote  by $\bW^{\,m,p}(\O)$ the Sobolev space of functions
in $\bL^{p}(\O)$ whose weak derivatives of order less than or equal to $m$ are also in
$\bL^p(\O)$. In the case $p = 2$, we shall write $\bH^{m}(\O)$ instead to $\bW^{m,2}(\O)$. If $ p\in[1, \infty)$, $p'$ denotes the conjugate exponent of $p$, i.e. $\frac{1}{p'}=1-\frac{1}{p}$. We define the spaces  $$\bX^{p}(\O)=\{\bv\in\bL^{p}(\O);\,\vdiv\,\bv\in L^{p}(\O),\,\curl\bv\in\bL^{p}(\O)\},$$
which is equipped with the norm:
\begin{equation*}
 \norm{\bv}_{\bX^p(\O)}=\norm{\bv}_{\bL^p(\O)}+\norm{\curl \bv}_{\bL^p(\O)}+\norm{\vdiv \bv}_{L^p(\O)}.
\end{equation*}
The subspaces $\bX_{N}^{p}(\O)$ and $\bV_N^p(\O)$ are defined by
$$\bX_{N}^{p}(\O)=\{\bv\in\bL^{p}(\O);\,\vdiv\,\bv\in L^{p}(\O),\,\curl\bv\in\bL^{p}(\O),\,\bv\times\bn=\textbf{0}\,\,\,\mathrm{on}\,\Gamma\},$$
$$\bV_{N}^{p}(\O)=\{\bv\in\bX^{p}_N(\O);\,\,\vdiv\,\bv=0\,\,\,\mathrm{in}\,\O\}.
$$
When $p=2$, we will use the notation $\bX_N(\O)$ instead to $\bX_N^{2}(\O)$. We denote by $\boldsymbol{\mathcal{D}}(\O)$ the set of smooth functions (infinitely differentiable) with compact support in $\O$. 
For $p,r\in[1, \infty)$, we introduce the following space 
\begin{equation}\label{espace Hr,p} {\textbf{\textit{H}}}^{\,r,p}(\mathbf{curl},\Omega):=\left\lbrace \textbf{\textit{v}}\in\textbf{\textit{L}}^{r}(\Omega);\,\mathbf{curl}\,\, \textbf{\textit{v}}\in\textbf{\textit{L}}^{p}(\Omega)\right\rbrace,\quad \mathrm{with}\quad \frac{1}{r}=\frac{1}{p}+\frac{1}{3}
\end{equation}
equipped with the norm  $$\Vert\bv\Vert_{{\textbf{\textit{H}}}^{\,r,p}(\mathbf{curl},\Omega)}=\Vert\bv\Vert_{\bL^r(\O)}+\Vert\curl\bv\Vert_{\bL^p(\O)}.$$ 
It can be shown that $\boldsymbol{\mathcal{D}}(\overline{\O})$ is dense in ${\textbf{\textit{H}}}^{\,r,p}(\mathbf{curl},\Omega)$ (cf. \cite[Proposition 1.0.2]{Sel-thesis} for the case $r=p$).
The closure of $\boldsymbol{\mathcal{D}}(\O)$ in $ {\textbf{\textit{H}}}^{\,r,p}(\mathbf{curl},\Omega)$ is denoted by $ {\textbf{\textit{H}}}^{\,r,p}_{0}(\mathbf{curl},\Omega)$ with
$${\textbf{\textit{H}}}^{\,r,p}_{0}(\mathbf{curl},\Omega):=\left\lbrace \textbf{\textit{v}}\in {\textbf{\textit{H}}}^{\,r,p}(\mathbf{curl},\Omega);\,\,\, \bv\times\bn=\textbf{0}\,\,\,\mathrm{on}\,\,\Gamma  \right\rbrace.$$
 $\boldsymbol{\mathcal{D}}(\O)$ is dense in ${\textbf{\textit{H}}}^{\,r,p}_{0}(\mathbf{curl},\Omega)$ and its dual space denoted by $[ {\textbf{\textit{H}}}^{\,r,p}_{0}(\mathbf{curl},\Omega)]'$ can be
characterized as follows (cf. \cite[Lemma 2.4]{AS_DCDS}, \cite[Lemma 2.5]{AS_DCDS} and \cite[Proposition 1.0.6]{Sel-thesis} for the case $r=p$):
\begin{proposition}\label{charac dual H^r,p}
A distribution $\ff$ belongs to  $[{\textbf{\textit{H}}}^{\,r,p}_{0}(\mathbf{curl},\Omega)]'$  iff there exists $\bF\in\bL^{r'}(\O)$ and $\boldsymbol{\psi}\in\bL^{p'}(\O)$ such that $\ff=\bF+\curl\boldsymbol{\psi}$. Moreover, we have the estimate :
$$\Vert\ff\Vert_{[ {\textbf{\textit{H}}}^{\,r,p}_{0}(\mathbf{curl},\Omega)]'}\leq \inf_{\ff=\bF+\curl\boldsymbol{\psi}}\max \{\Vert\bF\Vert_{\bL^{r'}(\O)},\Vert\boldsymbol{\psi}\Vert_{\bL^{p'}(\O)}\}.$$
\end{proposition}
Next we introduce the kernel $$\bK^p_N(\O)=\{\bv\in\bL^p(\O);\,\,\vdiv\,\bv=0,\,\,\curl\bv=\textbf{0},\,\,\bv\times\bn=\textbf{0}\,\,\,\mathrm{on}\,\,\Gamma\}.$$
			Thanks to \cite[Corollary 4.2]{AS_M3AS}, we know that this kernel is of finite dimension and spanned by the functions $\nabla q_i^N$, $1\leq i\leq I$, where $q_i^N$ is the unique solution of the problem
\noindent\begin{eqnarray}\label{Definition des q_i^N}
\begin{cases}
 -\Delta q_{i}^{N}=0\quad\quad \mathrm{in}\,\,\Omega,\quad 
q_{i}^{N}\vert_{\Gamma_{0}}=0\quad\mathrm{and}\quad q_{i}^N\vert_{\Gamma_{k}}=\mathrm{constant},\,\,\,1\leq k \leq I\\
\left\langle \partial_{n}\,q_{i}^{N},\,1\right\rangle_{\Gamma_{k}}=\delta_{i\,k},\,\, 1\leq k \leq I,\quad\mathrm{and}\quad \left\langle \partial_{n}\,q_{i}^{N},\,1\right\rangle_{\Gamma_{0}}=-1.
\end{cases}
\end{eqnarray}
	Moreover, the functions $\nabla q_i^N$, $1\leq i\leq I$, belong to $\bW^{1,q}(\O)$ for any $1< q < \infty$. 		
We will use also the symbol $\sigma$ to represent a set of divergence free functions. In other words, if $\bX$ is a Banach space, then $\bX_{\sigma}=\{\bv\in\bX;\,\,\,\vdiv\,\bv=0\,\,\mathrm{in}\,\O\}$.\\
We recall some useful results that play an important role in the proof of the regularity of solutions in this work. We begin with the following result (see \cite[Theorem 3.2.]{AS_M3AS})

\begin{theo}\label{injection continue X_N}
The space $\textbf{\textit{X}}^{\,p}_{N}(\Omega)$ is continuously embedded in $\textbf{\textit{W}}^{\,1,p}(\Omega)$ and there exists a constant $C$, such that for any $\textbf{\textit{v}}$ in $\textbf{\textit{X}}^{p}_{N}(\Omega)$: 
\begin{equation}\label{inegalite injection continue X_N}
 \Vert\textbf{\textit{v}}\Vert_{\textbf{\textit{W}}^{\,1,p}(\Omega)}\leq C\big( \Vert\textbf{\textit{v}}\Vert_{\textbf{\textit{L}}^{p}(\Omega)}+\Vert\vdiv\,\textbf{\textit{v}}\Vert_{L^{p}(\Omega)}+ \Vert\mathbf{curl}\,\textbf{\textit{v}} \Vert_{\textbf{\textit{L}}^{p}(\Omega)}+ \sum_{i=1}^{I}\vert\langle\textbf{\textit{v}}\cdot\textbf{\textit{n}},\,1\rangle_{\Gamma_{i}} \vert\big).
\end{equation}
 \end{theo}
And more generally (see \cite[Corollary 5.3]{AS_M3AS})
\begin{coro}\label{injection in X_N non homogene cas general}
Let $m\in\N^{*}$ and $\Omega$ of class $\mathcal{C}^{m,1}$. Then the space $$ \textbf{\textit{X}}^{\,m,p}(\Omega)\!=\!\{\textbf{\textit{v}}\!\in\! \textbf{\textit{L}}^{p}(\Omega);\,\vdiv\,\textbf{\textit{v}}\!\in\!{W}^{\,m-1,p}(\Omega),\,\mathbf{curl}\,\textbf{\textit{v}}\!\in\!\textbf{\textit{W}}^{\,m-1,p}(\Omega),\,\,\textbf{\textit{v}}\cdot\textbf{\textit{n}}\!\in\!\textbf{\textit{W}}^{\,m-\frac{1}{p},p}(\Gamma)\}$$ is continuously embedded in $ \textbf{\textit{W}}^{\,m,p}(\Omega)$ and we have the following estimate: for any function $\textbf{\textit{v}}$ in $\textbf{\textit{W}}^{\,m,p}(\Omega)$,
 \begin{eqnarray}\label{inequality injection in X_N non homogene cas general}
\!\!\!\!\|{\textbf{\textit{v}}}\|_{\textbf{\textit{W}}^{\,m,p}(\Omega)}\!\leq\!
C\Big(\|\textbf{\textit{v}}\|_{\textbf{\textit{L}}^{p}(\Omega)}\!\!+\!\! \|\mathbf{curl}\,{\textbf{\textit{v}}}\|_{\textbf{\textit{W}}^{\,m-1,p}(\Omega)}\!\!+\!\!\|
\vdiv\,{\textbf{\textit{v}}}\|_{{W}^{\,m-1,p}(\Omega)}\!\!+\!\!\|\textbf{\textit{v}}\times\textbf{\textit{n}}\|_{\textbf{\textit{W}}^{\,m-\frac{1}{p},p}(\Gamma)}\!\Big)
\end{eqnarray} 
\end{coro}

We also recall the following result (cf. \cite[Corollary 3.2]{AS_M3AS}) which gives a Poincar\'e inequality for every function $\bv\in\bW^{1,p}(\O)$ with $\bv\times\bn=\textbf{0}$ on $\Gamma$.

\begin{coro}\label{equivalence normes Poincaré}
  On the space $\textbf{\textit{X}}^{p}_{N}(\Omega)$, the seminorm
\begin{equation}\label{semi norm normal}
\textbf{\textit{v}}\mapsto\|\mathbf{curl}\,\textbf{\textit{v}}\|_{\textbf{\textit{L}}^{p}(\Omega)}+\|\rm{div}\,\textbf{\textit{v}}\|_{L^{p}(\Omega)}+\sum_{i=1}^{I}\vert\langle\textbf{\textit{v}}\cdot\textbf{\textit{n}},\,1\rangle_{\Gamma_{i}}\vert
\end{equation}
is equivalent to the norm $\|\cdot\|_{\textbf{\textit{X}}^{p}(\Omega)}$ for any $1<p<\infty$. In particular, we have the following Poincar\'e inequality for every function $\bv\in\bW^{1,p}(\O)$ with $\bv\times\bn=\textbf{0}$ on $\Gamma$:
\begin{eqnarray}\label{Poincaré inequality}
 \norm{\bv}_{\bW^{1,p}(\O)}\leq C_{\mathcal{P}} \big(\norm{\vdiv\bv}_{L^{p}(\O)}+\norm{\curl\bv}_{\bL^{p}(\O)}+\sum_{i=1}^{I}\vert\langle\textbf{\textit{v}}\cdot\textbf{\textit{n}},\,1\rangle_{\Gamma_{i}}\vert\big),
\end{eqnarray}
where $C_{\mathcal{P}}=C_{\mathcal{P}}(\O)>0$. Moreover, the norm \eqref{semi norm normal} is equivalent to the full norm $\norm{\,\cdot\,}_{\bW^{1,p}(\O)}$ on $\bX^{p}_{N}(\O)$. 
\end{coro}
Let us consider the following Stokes problem:

\begin{equation*}\label{S_N}
 (\mathcal{S_N})
	\left\lbrace
		\begin{aligned}
		&   - \, \Delta \bu  + \nabla P = \ff \quad \mathrm{and}\quad  \vdiv \bu = h  \quad \mathrm{in} \,\, \Omega, \\
       & \bu \times \bn = \textbf{0} \quad  \mathrm{on}\,\,\Gamma,\\
        & P = P_0 \quad \text{on} \, \,\Gamma_0\quad  \mathrm{and}\quad 
         P = P_0 + c_i \,\,\,\text{on} \,\, \Gamma_i,\\
         & \langle\bu \cdot \bn, 1\rangle_{\Gamma_i} = 0,\,\, 1\leq i \leq I.
		\end{aligned}\right.
		\end{equation*}
Then, the following proposition is an extension of that in \cite[Theorem 5.7]{AS_M3AS} to the case of non-zero divergence condition ($h\neq 0$). It is concerned with the existence and uniqueness of the weak and strong solutions for the Stokes problem $(\mathcal{S_N})$. 
\begin{proposition}\label{thm solution W1,p and W2,p Stokes divu=h}
We assume that $\Omega$ is of class $\mathcal{C}^{\,2,1}$. Let $\textbf{\textit{f}}$, $h$ and $P_{0}$ such that
$$
\textbf{\textit{f}}\in[\textbf{\textit{H}}^{\,r',p'}_{0}(\mathbf{curl},\,\Omega)]',\quad h\in W^{1,r}(\O)\quad \mathrm{and}\quad P_{0}\in W^{\,1-1/r,r}(\Gamma),
$$
with $r\leq p$ and $\frac{1}{r}\leq \frac{1}{p}+\frac{1}{3}$.
Then, the problem  $(\mathcal{S_N})$ has a unique solution $(\textbf{\textit{u}},\,P)\in\textbf{\textit{W}}^{\,1,p}(\Omega)\times W^{\,1,r}(\Omega)$ and constants $c_{1},\ldots,c_{I}$ satisfying the estimate:
\begin{eqnarray}\label{estimation solution faible W^1,p x W^1,r pour S_N espaces H^r,p}
 \Vert\textbf{\textit{u}}\Vert_{\textbf{\textit{W}}^{\,1,p}(\Omega)}+\Vert P\Vert_{W^{\,1,r}(\Omega)}\leq C\big(\Vert\,\textbf{\textit{f}}\,\Vert_{[\textbf{\textit{H}}^{\,r',p'}_{0}(\mathbf{curl},\,\Omega)]'}+\Vert h\Vert_{W^{\,1,r}(\O)}+\Vert P_{0}\Vert_{W^{\,1-1/r,r}(\Gamma)}\big),
\end{eqnarray}
and $c_{1},\ldots,c_{I}$ are given by
\begin{equation}\label{constantes Stokes divu=h}
 c_i=\langle\textbf{\textit{f}},\,\nabla\,q_i^N\rangle_{[\textbf{\textit{H}}^{\,r',p'}_{0}(\mathbf{curl},\,\Omega)]'\times\textbf{\textit{H}}^{\,r',p'}_{0}(\mathbf{curl},\,\Omega)}+\int_{\Gamma}\,(h-P_0)\,\nabla\,q_i^N\cdot\textbf{\textit{n}} \, d\sigma.
 \end{equation}
 Moreover, if $\ff\in \bL^{p}(\O)$, $h\in W^{1,p}(\O)$ and $P_0\in W^{1-\frac{1}{p}, p}(\Gamma)$, then $(\bu,\,P)$ belongs to $\bW^{2,p}(\O)\times W^{1,p}(\O)$ and satisfies the estimate 
\begin{equation}\label{estim W^2,p Stokes}
 \norm{\bu}_{\bW^{\,2,p}(\Omega)} + \norm{P}_{W^{\,1,p}(\Omega)} \leqslant C_{S} \Big(\norm{\ff}_{\bL^{p}(\Omega)}  +\Vert h\Vert_{W^{\,1,p}(\O)}+ \norm{P_0}_{W^{1-\frac{1}{p}, p}(\Gamma)}+\sum_{i=1}^{I}\vert c_{i}\vert \Big),
 \end{equation}
 where $C_S=C_S(\O)>0$.
 \end{proposition}
 \begin{proof}
 To reduce the non vanishing divergence problem $(\mathcal{S_N})$ to the case where $\vdiv \bu=0$ in $\O$, we consider the problem 
$$\Delta\,\theta=h\quad \mathrm{in}\,\Omega\quad\,\mathrm{and}\quad\, \theta=0\quad \mathrm{on}\,\Gamma.$$
Since $h\in W^{1,r}(\Omega)$, it has a unique solution $\theta\in W^{\,3,r}(\Omega)\hookrightarrow W^{2,p}(\Omega)$, with (cf. \cite[Theorem 1.8]{GR})
\begin{equation}\label{estimation solution theta}
\Vert\theta\Vert_{W^{\,2,p}(\Omega)}\leq C\Vert h\Vert_{W^{1,r}(\Omega)}.
 \end{equation}
 Taking $\textbf{\textit{w}}=\nabla\,\theta$ and defining 
 \begin{equation}\label{definition de w tild}
\widetilde{\textbf{\textit{w}}}=\textbf{\textit{w}}-\sum_{i=1}^{I}\langle\textbf{\textit{w}}\cdot\textbf{\textit{n}},\,1\rangle_{\Gamma_{i}}\,\mathbf{grad}\,q_{i}^{N},
\end{equation}
we see that $\widetilde{\textbf{\textit{w}}}\in\textbf{\textit{W}}^{\,1,p}(\Omega)$ with $\vdiv\,\widetilde{\textbf{\textit{w}}}=h$, $\mathbf{curl}\,\widetilde{\textbf{\textit{w}}}=\bf{0}$ in $\Omega$, $\widetilde{\textbf{\textit{w}}}\times\textbf{\textit{n}}=\bf{0}$ on $\Gamma$ and $\langle\widetilde{\textbf{\textit{w}}}\cdot\textbf{\textit{n}},\,1\rangle_{\Gamma_{i}}=0$ for any $1\leq i \leq I$. Finally, taking $\textbf{\textit{z}}=\textbf{\textit{u}}-\widetilde{\textbf{\textit{w}}}$, we see that the problem $(\mathcal{S_N})$ can be reduced to the following problem for $\textbf{\textit{z}}$ and $P$:

\begin{equation}\label{S_N div=0}
	\left\lbrace
		\begin{aligned}
		&   - \, \Delta \bz  + \nabla P = \ff+\Delta\widetilde{\bw}\ \quad \mathrm{and}\quad  \vdiv \bz = 0  \quad \mathrm{in} \,\, \Omega, \\
       & \bz \times \bn = \textbf{0} \quad  \mathrm{on}\,\,\Gamma,\\
        & P = P_0 \quad \text{on} \, \,\Gamma_0\quad  \mathrm{and}\quad 
         P = P_0 + c_i \,\,\,\text{on} \,\, \Gamma_i,\\
         & \langle\bz \cdot \bn, 1\rangle_{\Gamma_i} = 0,\,\, 1\leq i \leq I.
		\end{aligned}\right.
		\end{equation}
Since $\textbf{\textit{w}}=\nabla\,\theta$ and $\Delta(\nabla q_i^N)=0$, it follows from \eqref{definition de w tild} that $\Delta\,\widetilde{\textbf{\textit{w}}}=\nabla\,(\Delta \theta)\in\bL^{r}(\O)\hookrightarrow [\textbf{\textit{H}}^{\,r',p'}_{0}(\mathbf{curl},\,\Omega)]'$ and $\int_{\O}\Delta\widetilde{\bw}\cdot \nabla q_{i}^{N}\,d\bx=0$, we deduce from \cite[Theorem 5.7]{AS_M3AS}, the existence of a unique solution $(\bz,\,P,\bc)\in \bW^{1,p}(\O)\times W^{1,r}(\O)\times\R^I$ of \eqref{S_N div=0} with $\bc=(c_1,\ldots,c_I)$ given by \eqref{constantes Stokes divu=h}. Moreover, using \eqref{estimation solution theta}, we have $\norm{\Delta \widetilde{\bw}}_{[\bH_0^{r',p'}(\curl, \Omega)]'} \leq C \norm{h}_{W^{1,r}(\O)}$ and then $(\bz,P)$ satisfies the estimate:
\begin{eqnarray}\label{estimation z du relevement de la div}
\Vert\textbf{\textit{z}}\Vert_{\textbf{\textit{W}}^{\,1,p}(\Omega)}+\norm{P}_{W^{1,r}(\O)}\leq C \Big(\Vert\textbf{\textit{f}}\Vert_{[\textbf{\textit{H}}^{\,r',p'}_{0}(\mathbf{curl},\,\Omega)]'}+\Vert h \Vert_{W^{\,1,r}(\Omega)}+\Vert P_{0}\Vert_{W^{\,1-1/r,r}(\Gamma)}\Big).
\end{eqnarray}
As a consequence, $(\textbf{\textit{u}},P)=(\textbf{\textit{z}}+\widetilde{\textbf{\textit{w}}},P)\in\textbf{\textit{W}}^{\,1,p}(\Omega)\times W^{1,r}(\O)$ is the unique solution of $(\mathcal{S_N})$ and the estimate \eqref{estimation solution faible W^1,p x W^1,r pour S_N espaces H^r,p} follows from \eqref{estimation solution theta} and \eqref{estimation z du relevement de la div}. \\
Now, we suppose that $\ff\in\bL^ p(\O)$, $h\in W^{1,p}(\O)$ and $P_0\in  W ^{1-1/p ,p}(\Gamma)$. We know that $(\bu,P)$ belongs to $\bW^{1,p}(\O)\times W^{1,r}(\O)$. We set $\bz = \curl \bu$. Since $\bu \times\bn=\textbf{0}$ on $\Gamma$, we have $\bz\cdot\bn=0$ on $\Gamma$ and then $\bz$ belongs to $\bX_N(\O)$. By Theorem \ref{injection continue X_N}, the function $\bz$ belongs to $\bW^{1,p}(\O)$. Then, $\bu$ satisfies
$$\bu \in \bL^p(\O),\,\,\vdiv\bu = h\in W^{1,p}(\O),\,\,\,
\curl\bu \in \bW^{1,p}(\O)\,\,\,\mathrm{and}\,\,\, \bu\times\bn=\textbf{0}\,\,\mathrm{on}\,\,\Gamma.$$
We deduce from Corollary \ref{injection in X_N non homogene cas general} (with $m=2$) that $\bu$ belongs to $\bW^{2,p}(\O)$.
 \end{proof}
 \noindent We need also some regularity results for the following elliptic problem
 \begin{eqnarray*} \label{Probleme elliptique}
(\mathcal{E_N})\begin{cases}
- \Delta \bb = \bg\quad \mathrm{and}\quad
\vdiv \bb = 0 \quad \mathrm{in} \,\, \Omega, \\
\bb \times \bn = \textbf{0} \quad \mathrm{on} \, \Gamma \\
\langle \bb \cdot \bn, 1 \rangle_{\Gamma_i} = 0, \quad \forall 1 \leqslant i \leqslant I,
\end{cases}
\end{eqnarray*} 
 which can be seen as a Stokes problem without pressure. We note that $(\mathcal{E}_N)$ is well-posed. Indeed, observe that the condition $\vdiv\bg=0$ in $\O$ is necessary to solve $(\mathcal{E}_N)$ and then we can verify that it is equivalent to the following problem:
\begin{equation}\label{E_N equivalent1}
\begin{cases}
-\Delta\,\textbf{\textit{b}}=\bg\,\,\,\,&\mathrm{in}\,\,\Omega\\
\vdiv\,\textbf{\textit{b}}=0,\quad\mathrm{and}\quad
\textbf{\textit{b}}\times\textbf{\textit{n}}=\textbf{0}\,\,\,\,&\mathrm{on}\,\Gamma,\\
\langle \bb \cdot \bn, 1 \rangle_{\Gamma_i} = 0, \quad \forall 1 \leqslant i \leqslant I,
\end{cases}
\end{equation}
where we have replaced the condition $\vdiv\,\textbf{\textit{b}}=0$ in $\Omega$ by $\vdiv\,\textbf{\textit{b}}=0$ on $\Gamma$. Next, we know that for any $\textbf{\textit{b}}\in\textbf{\textit{W}}^{\,1,p}(\Omega)$ such that $\vdiv\,\textbf{\textit{b}}\in W^{\,1,p}(\Omega)$,  we have (cf. \cite{Amrouche-Girault} or \cite{Heron} for $\textbf{\textit{b}}\in\textbf{\textit{W}}^{\,2,p}(\Omega)$):
\begin{equation}\label{formule sur la surface}
 \vdiv\,\textbf{\textit{b}}= \vdiv_{\Gamma}\,\textbf{\textit{b}}_{t} +\frac{\partial\,\textbf{\textit{b}}}{\partial\,\textbf{\textit{n}}}\cdot\textbf{\textit{n}}-2 K\textbf{\textit{b}}\cdot\textbf{\textit{n}}\quad\,\,\mathrm{on}\,\,\Gamma,
\end{equation}
where $\bb_t$ is the tangential component of $\bb$, $K$ denotes the mean curvature of $\Gamma$ and $\vdiv_{\Gamma}\,$ is the surface divergence.\medskip

\noindent Then, using \eqref{formule sur la surface}, the problem \eqref{E_N equivalent1} is equivalent to: find $\textbf{\textit{b}}\in\textbf{\textit{W}}^{\,1,p}(\Omega)$ such that
\begin{equation}\label{problem dans X_T avec la formule surfacique}
 \begin{cases}
   -\Delta\,\textbf{\textit{b}} =\bg& \mathrm{in}\,\Omega,\\
    \textbf{\textit{b}}\times\textbf{\textit{n}} =0 \quad\mathrm{and}\quad\dfrac{\partial\,\textbf{\textit{b}}}{\partial\,\textbf{\textit{n}}}\cdot\textbf{\textit{n}}-2 K\textbf{\textit{b}}\cdot\textbf{\textit{n}}=0&\mathrm{on}\,\Gamma,\\
\langle \bb \cdot \bn, 1 \rangle_{\Gamma_i} = 0, \quad \forall 1 \leqslant i \leqslant I,
\end{cases}
 \end{equation}
where the  condition $\dfrac{\partial\,\textbf{\textit{b}}}{\partial\,\textbf{\textit{n}}}\cdot\textbf{\textit{n}}-2 K\textbf{\textit{b}}\cdot\textbf{\textit{n}}=0\,\,\mathrm{on}\,\Gamma$ is a Fourier-Robin type boundary condition.\medskip
 
 We begin with the following regularity result for $(\mathcal{E_N})$ which can be found in \cite[Corollary 5.4.]{AS_M3AS}.
 
\begin{theorem}\label{thm strong soulution elliptic}
Assume that $\O$ is of class $\mathcal{C}^{2,1}$. Let $\bg\in \bL^p(\O)$ satisfying the
compatibility conditions
\begin{eqnarray} \label{Condition compatibilite K_N Lp} 
 \forall \, \bv \in \bK_{N}^{p'}(\Omega),\quad\int_{\O}\bg\cdot \bv \,d\bx= 0, \\
 \vdiv \bg = 0 \quad \mathrm{in}  \, \Omega.
 \end{eqnarray}
Then the elliptic problem $(\mathcal{E_N})$ has a unique solution $\bb\in \bW^{2,p}(\O)$  satisfying the estimate 
\begin{equation}\label{estim W^2,p elliptic}
  \norm{\bb}_{\bW^{\,2,p}(\Omega)} \leqslant C_{E} \norm{\bg}_{\bL^{p}(\Omega)}.
\end{equation}
 \end{theorem}
 
  We need also the following useful result for $\mathcal{(E_N)}$ which gives an improvement of that in \cite[Proposition 5.1]{AS_M3AS}. Indeed, we consider the dual space $[\bH_0^{r',p'}(\curl, \Omega)]'$  with $\frac{1}{r}=\frac{1}{p}+\frac{1}{3}$ (c.f. \eqref{espace Hr,p} and Proposition \eqref{charac dual H^r,p}) for data in the right-hand side instead of $[\bH_0^{p',p'}(\curl,\Omega)]'$. 

\begin{lemma} \label{Lemme elliptique H r' p'}
Let $\Omega$ of class $\mathcal{C}^{2,1}$. Let $\bg \in [\bH_0^{r',p'}(\curl,\Omega)]'$ satisfying the compatibility conditions 
\begin{eqnarray} \label{Condition compatibilite K_N Lp lemme} 
 \forall \, \bv \in \bK_{N}^{p'}(\Omega),\quad\langle \bg, \bv \rangle_{[\bH_0^{r',p'}(\curl, \Omega)]'\times \bH_0^{r',p'}(\curl, \Omega) } = 0, \\
 \vdiv \bg = 0 \quad \mathrm{in}  \, \Omega.
 \end{eqnarray}
Then, the elliptic problem $\mathcal{(E_N)}$ has a unique solution $\bb \in \bW^{1,p}(\Omega)$ satisfying the estimate: 
\begin{eqnarray} \label{Estimation lemme elliptique}
\norm{\bb}_{\bW^{1,p}(\Omega)} \leqslant C \norm{\bg}_{[\bH_0^{r',p'}(\curl, \Omega)]'} 
\end{eqnarray}
\end{lemma}

\begin{proof}
Using the characterization of the dual space $[\bH_0^{r',p'}(\curl, \Omega)]'$ given in Proposition \ref{charac dual H^r,p}, we can write $\bg$ as: 
\begin{eqnarray} \label{Decompo G}
\bg = \bG + \curl \Psi,\quad \mathrm{where}\quad \bG \in \bL^r(\Omega)\quad \mathrm{and}\quad \Psi \in \bL^p(\Omega). 
\end{eqnarray}
Note that, from \eqref{Definition des q_i^N}, for any $1 \leqslant i \leqslant I$, $\langle \curl \Psi, \nabla q_i^N \rangle_\Omega = 0$, then it follows from \eqref{Condition compatibilite K_N Lp lemme} and \eqref{Decompo G} that $\bG$ also satisfies the compatibility condition \eqref{Condition compatibilite K_N Lp lemme}. Similarly, by \eqref{Decompo G}, we have $\vdiv \bG = 0$. Thanks to Theorem \ref{thm strong soulution elliptic}, the following problem: 
\begin{eqnarray*}
\begin{cases}
- \Delta \bb_1 = \bG \quad \mathrm{in} \, \Omega \quad \mathrm{and}\quad 
\vdiv \bb_1 = 0\quad \mathrm{in} \, \Omega, \\
\bb_1 \times \bn = 0 \quad \mathrm{on} \, \Gamma, \\
\langle \bb_1 \cdot \bn, 1 \rangle_{\Gamma_i} = 0, \quad \forall 1 \leqslant i \leqslant I
\end{cases}
\end{eqnarray*}
\noindent has a unique solution $\bb_1 \in \bW^{2,r}(\Omega)$ satisfying the estimate: 
\begin{eqnarray} \label{Estimation bb1 lemme elliptique}
\norm{\bb_1}_{\bW^{2,r}(\Omega)} \leqslant C \norm{\bG}_{\bL^r(\Omega)}.
\end{eqnarray}
\noindent Next, since $\curl \Psi \in [\bH_0^{p'}(\curl, \Omega)]'$ and satisfies the compatibility conditions \eqref{Condition compatibilite K_N Lp lemme}, by \cite[Proposition 5.1.]{AS_M3AS} the following problem 
\begin{eqnarray*}
\begin{cases}
- \Delta \bb_2 = \curl \Psi \quad \mathrm{in} \, \Omega \quad \mathrm{and}\quad
\vdiv \bb_2 = 0 \quad \mathrm{in} \, \Omega, \\
\bb_2 \times \bn = 0 \quad \mathrm{on} \, \Gamma,\\
\langle \bb_2 \cdot \bn, 1 \rangle_{\Gamma_i} = 0, \quad \forall 1 \leqslant i \leqslant I
\end{cases}
\end{eqnarray*}
\noindent has a unique solution $\bb_2 \in \bW^{1,p}(\Omega)$ satisfying the estimate: 
\begin{eqnarray} \label{Estimation bb2 lemme elliptique}
\norm{\bb_2}_{\bW^{1,p}(\Omega)} \leqslant C \norm{\curl \Psi}_{\bL^p(\Omega)}.
\end{eqnarray}
Since $\frac{1}{r}=\frac{1}{p} +\frac{1}{3}$, $\bW^{2,r}(\Omega) \hookrightarrow \bW^{1,p}(\Omega)$. Then, $\bb = \bb_1 + \bb_2$ belongs to $\bW^{1,p}(\Omega)$ and it is the unique solution of $(\mathcal{E_N})$. The estimate \eqref{Estimation lemme elliptique} follows from \eqref{Estimation bb1 lemme elliptique} and \eqref{Estimation bb2 lemme elliptique}. 
\end{proof}

% For the following result gives a version of De Rham's Lemma for functionals acting on vector fields with vanishing tangential components. This result was introduced by X-B. Pan, see \cite[Lemma 2.2]{Pan} where .......
% 
% \begin{Lemma}\label{Lemma De Rham for chi equal to zero on Gamma}
%  
% \end{Lemma}

\begin{rmk} ${}$
\\
\indent $\textit{(1).}$ We note that the regularity $\mathcal{C}^{2,1}$ in Lemma \ref{Lemme elliptique H r' p'} can be reduced to $\mathcal{C}^{1,1}$. Indeed, we can verify that the Stokes problem $(\mathcal{E_N})$ is equivalent to the following variational frmulation (c.f. \cite[Proposition 5.1]{AS_M3AS}): Find $\bb\in\bW^{1,p}(\O)$ such that for any $\ba\in\bV^{p'}_N(\O)$:
\begin{equation}\label{VF Stokes EN}
 \int_{\O}\curl\bb\cdot\curl\ba\,\mathrm{d}\bx =\langle \bg,\,\ba\rangle_{[\bH_0^{r',p'}(\curl, \Omega)]'\times \bH_0^{r',p'}(\curl, \Omega) }.
\end{equation}
Thanks to \cite[Lemma 5.1]{AS_M3AS}, if $\O$ is of class $\mathcal{C}^{1,1}$, the following infi-sup condition holds: there exists a constant $\beta>0$, such that:
\begin{equation}\label{condition inf sup on V^pN}
 \inf_{\substack{\textbf{\textit{a}}\in\textbf{\textit{V}}^{\,p'}_{N}(\Omega)\\\textbf{\textit{a}}\neq 0}}\sup_{\substack{\textbf{\textit{b}}\in\textbf{\textit{V}}^{\,p}_{N}(\Omega)\\\textbf{\textit{b}}\neq 0}}\dfrac{\int_{\Omega}\mathbf{curl}\,\textbf{\textit{b}}\cdot\mathbf{curl}\,\textbf{\textit{a}}\,\mathrm{d}\textbf{\textit{x}}}{\Vert \textbf{\textit{b}}\Vert_{\textbf{\textit{W}}^{\,1,p}(\Omega)}\Vert \textbf{\textit{a}}\Vert_{\textbf{\textit{W}}^{\,1,p'}(\Omega)}}\geq \beta.
\end{equation}
So, problem \eqref{VF Stokes EN} has a unique solution $\bu\in\bV^p_N(\O)\subset \bW^{1,p}(\O)$ since the right-hand sides defines an element of $(\bV^p_N(\O))'$.\\

$\textit{(2).}$ In the classical study of the Stokes and Navier-Stokes equations,
the pressure $P$ is obtained thanks to a variant of De Rham’s theorem (see \cite[Theorem 2.8]{Amrouche-Girault}). Indeed, let $\O\subset \R^{3}$ be a bounded Lipschitz domain and $\ff\in \bW^{−1,p}(\O)$, $1 < p < \infty$ satisfying 
\begin{equation*}
 \forall \bv\in \boldsymbol{\mathcal{D}}_{\sigma}(\O),\quad \langle\ff,\bv\rangle_{ \boldsymbol{\mathcal{D}}'(\O)\times  \boldsymbol{\mathcal{D}}(\O)}=0.
\end{equation*}
Then there exists $P\in L^{p}(\O)$ such that $\ff=\nabla P$. Unlike the case of Dirichlet boundary condition, the pressure in the \eqref{eqn:MHD} problem can be found independently of the velocity $\bu$ and the magnetic field $\bb$. Indeed, the pressure $P$ is a solution of the problem 
\begin{equation*}
\begin{cases}
\Delta P=\vdiv\ff-\vdiv((\curl \bw)\times \bu) +\vdiv( (\curl\bb)\times \bdd)\quad \mathrm{in}\,\, \O\\
 P=P_0 \quad \mathrm{on}\,\,\Gamma_0 \quad \mathrm{and} \quad P=P_0+\alpha_i \quad \mathrm{on }\,\,\Gamma_i.
\end{cases}
\end{equation*}
So, when we talk about the regularity $\bW^{1,p}(\O)$ or $\bW^{2,p}(\O)$ it concerns $(\bu, \bb)$ and we mean that
$(\bu, \bb, P)$ is the weak or strong solution of the \eqref{eqn:MHD} problem.
\end{rmk}

\section{The linearized MHD system: $L^{2}$-theory}
In this section we take $\bw$ and $\bdd$ such that:

\begin{eqnarray}
 \curl \bw\in \bL^{3/2}(\O),\quad \bdd\in \bL^{3}(\O),& \quad\vdiv\,\bdd=0\,\,\,\mathrm{in}\,\,\O,\label{hypothesis on w and d}
\end{eqnarray}
and we consider the following linearized MHD system: Find ($\bu$, $\bb$,  $\mathrm{P}$, $\bc$) with $\bc=(c_1,\ldots, c_I)$ such that for $1\leq i \leq I$:
\begin{equation}\label{linearized MHD-pressure}
	\left\lbrace
		\begin{split}
		   - \, \Delta \bu + (\curl \bw) \times \bu + \nabla \mathrm{P} -  (\curl \bb) \times \bdd = \ff \quad &\mathrm{and}\quad  \vdiv \bu = h   \quad \mathrm{in} \,\, \Omega, \\
           \, \curl \curl \bb -  \curl (\bu \times \bdd) = \bg  \quad & \mathrm{and}\quad \vdiv \bb = 0 \quad \text{in} \,\, \Omega, \\
        \bu \times \bn = \textbf{0} \quad &\mathrm{and}\quad \bb \times \bn = \textbf{0}\quad \mathrm{on}\,\,\Gamma,\\
        \mathrm{P} = P_0 \quad \text{on} \, \,\Gamma_0\quad & \mathrm{and}\quad 
         \mathrm{P} = P_0 + c_i \,\,\,\text{on} \,\, \Gamma_i,\\
         \langle\bu \cdot \bn, 1\rangle_{\Gamma_i} = 0, \quad &\mathrm{and}\quad \langle\bb \cdot \bn, 1\rangle_{\Gamma_i} = 0.
		\end{split}\right.
\end{equation}
The aim of this section is to show, under minimal regularity assumptions on $\ff$, $\bg$, $h$ and $P_0$, the existence and the uniqueness of weak solutions $(\bu,\bb,P,\bc)$ in $\bH^{1}(\O)\times \bH^{1}(\O)\times L^{2}(\O)\times \R^I$. Classically, the idea is to write an equivalent variational formulation and use Lax Milgram if the bilinear form involved in the variational formulation is coercive. It is natural to look for a solution $(\bu,\bb)$ in $\bV_{N}(\O)\times \bV_{N}(\O)$ with 
$$\bV_N(\Omega) := \{ \bv \in \bH^1(\Omega); \, \vdiv \bv = 0 \, \, \text{in} \, \Omega, \, \, \bv \times \bn = 0 \, \, \text{on} \, \Gamma, \, \langle \bv \cdot \bn, 1 \rangle_{\Gamma_i} = 0, \, \, \forall 1 \leqslant i \leqslant I \} .$$
Unlike the case of Dirichlet type boundary conditions, the space $\bH^{-1}(\O)$ is not suitable for source terms in the right hand side to find solutions in $\bH^{1}(\O)$. Let us analyse the case of $\ff$, it holds true also for $\bg$. Since $\bv\in\bV_{N}(\O)$, then we can firstly consider the duality pairing $\langle\ff,\,\bv\rangle_{[\bH_0^{2,2}(\curl, \Omega)]'\times\bH_0^{2,2}(\curl, \Omega)}$ in view to write an equivalent variational formulation. Then, we must suppose that $\ff$ belongs to $[\bH_0^{2,2}(\curl, \Omega)]'$. But, we have $\bv$ belongs to $\bH^{1}(\O)\hookrightarrow \bL^{6}(\O)$. Then, the previous hypothesis on $\ff$ can be weakened by considering the space $[\bH_0^{6,2}(\curl, \Omega)]'$ which is a subspace of $\bH^{-1}(\O)$. Indeed, thanks to the characterization 
given in Proposition \ref{charac dual H^r,p}, we have for $r=6$ and $p=2$, \begin{equation}\label{charac dual H^6,2}
[\bH_0^{6,2}(\curl, \Omega)]'=\{\bF+\curl\boldsymbol{\psi};\,\,\,\bF\in\bL^{6/5}(\O),\,\,\boldsymbol{\psi}\in \bL^{2}(\O)\}.\end{equation}
                                                                                                                            
Then, since  $\bV_{N}(\O)\hookrightarrow \bH_0^{6,2}(\curl, \Omega)$, the previous duality is replaced by 
$$\langle\ff,\,\bv\rangle_{[\bH_0^{6,2}(\curl, \Omega)]'\times\bH_0^{6,2}(\curl, \Omega)}=\int_{\O}\bF\cdot\bv\,d\bx+\int_{\O}\boldsymbol{\psi}\cdot\curl \bv\,d\bx.$$
 In the sequel, we will consider the space $[\bH_0^{6,2}(\curl,\Omega)]'$ for $\ff$ and $\bg$ to obtain solutions in $\bH^{1}(\O)$.
\begin{proposition} \label{Proposition Equivalence sol hilb sol var}
Let us suppose $h=0$. Let $\ff, \bg \in [\bH_0^{6,2}(\curl, \Omega)]'$ and $P_0 \in H^{-\frac{1}{2}}(\Gamma)$ with the compatibility conditions 
\begin{equation} \label{Condition compatibilite K_N hilbert} 
 \forall \, \bv \in \bK_N^2(\Omega),\quad\langle \bg, \bv \rangle_{\O_{_{6,2}} } = 0,  
 \end{equation}
 \begin{equation}\label{Condition div nulle}
\vdiv \bg = 0 \quad \mathrm{on}  \, \Omega,
\end{equation}
where $\langle\cdot,\cdot\rangle_{\O_{_{6,2}}}$ denotes the duality product between $[\bH_0^{6,2}(\curl, \Omega)]'$ and $\bH_0^{6,2}(\curl, \Omega)$.\\ 
Then the following two problems are equivalent: \\
(i) Find $(\bu, \bb, P, \bc) \in \bH^1(\Omega) \times \bH^1(\Omega) \times L^2(\Omega) \times \R^I$ solution of \eqref{linearized MHD-pressure}. \vspace{.1cm}\\
(ii) Find $(\bu, \bb) \in \bV_N(\Omega) \times \bV_N(\Omega)$ and $\bc\in\R^{I}$ such that: for all $(\bv, \boldsymbol{\Psi} ) \in \bV_N(\Omega) \times \bV_N(\Omega)$  
\small
\begin{equation*} 
 \int_\Omega \curl \bu \cdot \curl \bv \, dx  + \int_\Omega (\curl \bw) \times \bu \cdot \bv \, dx - \int_\Omega (\curl \bb) \times \bdd \cdot \bv \, dx +
   \int_\Omega \curl \bb \cdot \curl \boldsymbol{\Psi} \, dx\end{equation*}\vspace{-.49cm}\small
 \begin{equation}\label{Var Form}
  + \int_\Omega (\curl \boldsymbol{\Psi}) \times \bdd \cdot \bu \, dx  = \langle \ff, \bv \rangle_{\O_{_{6,2}}} + \langle \bg, \boldsymbol{\Psi} \rangle_{\O_{_{6,2}}} - \langle P_0, \bv \cdot \bn \rangle_{\Gamma_0} - \sum_{i=1}^I \langle P_0 + c_i, \bv \cdot \bn \rangle_{\Gamma_i}
\end{equation}
and $\bc = (c_1, \cdots, c_I)$ satisfying for $1\leq i \leq I$: 
\small
\begin{eqnarray} \label{Constantes form var}
\begin{aligned}
c_i &= \langle \ff, \nabla q_i^N \rangle_{\O_{_{6,2}}} - \langle P_0, \nabla q_i^N \cdot \bn \rangle_\Gamma - \int_\Omega (\curl \bw) \times \bu \cdot \nabla q_i^N \, dx +  \int_\Omega (\curl \bb) \times \bdd \cdot \nabla q_i^N \, dx,
\end{aligned}
\end{eqnarray}
where $\langle\cdot,\cdot\rangle_{\Gamma}$ denotes the duality product between $H^{-1/2}(\Gamma)$ and $H^{1/2}(\Gamma)$.
\end{proposition}

\begin{proof}
Using the same arguments as in \cite[Lemma 5.5]{AS_M3AS}, we can prove that $\mathcal{D}_{\sigma}(\overline{\O})\times \mathcal{D}(\overline{\O})$ is dense in the space
 $$\mathcal{E}(\O)=\{(\bu,\,P)\in\bH^{1}_{\sigma}(\O)\times L^{2}(\O);\,\,\,\Delta \bu+\nabla P\in [ {\textbf{\textit{H}}}^{\,6,2}_{0}(\mathbf{curl},\Omega)]'\}.$$
  Moreover, we have the following Green formula: For any $(\textbf{\textit{u}},\,P)\in \mathcal{E}(\O)$ and $\boldsymbol{\varphi}\in\textbf{\textit{H}}^{\,1}_{\sigma}(\Omega)$ with $\boldsymbol{\varphi}\times\textbf{\textit{n}}=\textbf{0}$ on $\Gamma$:
\begin{equation}\label{Green Formula for VF}
 \langle-\Delta\,\textbf{\textit{u}}+\nabla\,P,\,\boldsymbol{\varphi}\rangle_{\O_{_{6,2}}}=\int_{\Omega}\mathbf{curl}\,\textbf{\textit{u}}\cdot\mathbf{curl}\,\boldsymbol{\varphi}\,\mathrm{d}\textbf{\textit{x}}+\langle P,\,\boldsymbol{\varphi}\cdot\textbf{\textit{n}}\rangle_{\Gamma}. 
\end{equation}

Using Green formula \eqref{Green Formula for VF}, we deduce that any $(\bu, \bb, P, \bc) \in \bH^1(\Omega) \times \bH^1(\Omega) \times L^2(\Omega) \times \R^I$ satisfying \eqref{linearized MHD-pressure} also solves \eqref{Var Form}.
It remains to recover the relation \eqref{Constantes form var}. Let us take $\bv \in \bH_\sigma^1(\Omega)$ with $\bv \times \bn = \boldsymbol{0}$ on $\Gamma$ and set: 
\begin{eqnarray}\label{decomposition of v}
\bv_0 = \bv - \sum_{i=1}^I \langle \bv \cdot \bn, 1 \rangle_{\Gamma_i} \nabla q_i^N 
\end{eqnarray}
Observe that $\bv_0\in \bH^{1}_{\sigma}(\O)$, $\bv_0\times\bn=\textbf{0}$ on $\Gamma$ and due to the properties of $q_i^N$, we have for all $1\leq i \leq I$, $\langle\bv_0\cdot\bn,\,1\rangle_{\Gamma_i}=0$. Then $\bv_0$ belongs to $\bV_{N}(\O)$. Multiplying the first equation on the left of the problem \eqref{linearized MHD-pressure} with $\bv=\bv_0+ \sum_{i=1}^I \langle \bv \cdot \bn, 1 \rangle_{\Gamma_i} \nabla q_i^N $, integrating by parts in $\Omega$, we obtain
\small
\begin{eqnarray*} 
\begin{aligned}
&\int_\Omega  \curl \bu \cdot \curl \bv \, dx + \int_\Omega (\curl \bw) \times \bu \cdot \bv \, dx 
-  \int_\Omega (\curl \bb) \times \bdd \cdot \bv \, dx - \langle \ff, \bv \rangle_\Omega+ \langle P_0, \bv \cdot \bn \rangle_{\Gamma_0}\\&+\sum_{i=1}^{I}\langle P_0+c_i, \bv \cdot \bn \rangle_{\Gamma_i}=  \int_\Omega \curl \bu \cdot \curl \bv_0 \, dx + \int_\Omega (\curl \bw) \times \bu \cdot \bv_0 \, dx -  \int_\Omega (\curl \bb) \times \bdd \cdot \bv_0 \, dx \\
&- \langle \ff, \bv_0 \rangle_\Omega + \langle P_0, \bv_0 \cdot \bn \rangle_{\Gamma_0}+\sum_{i=1}^{I}\langle P_0+c_i, \bv_0 \cdot \bn \rangle_{\Gamma_i}+\sum_{i=1}^I \langle \bv \cdot \bn, 1 \rangle_{\Gamma_i} \Big[ \int_\Omega (\curl \bw) \times \bu \cdot \nabla q_i^N \, dx \Big]\\
&+\sum_{i=1}^I \langle \bv \cdot \bn, 1 \rangle_{\Gamma_i}\Big[ - \int_\Omega (\curl \bb) \times \bdd \cdot \nabla q_i^N \, dx - \langle \ff, \nabla q_i^N \rangle_\Omega + \langle P_0, \nabla q_i^N \cdot \bn \rangle_{\Gamma} \Big]+\sum_{i=1}^I c_i \, \langle \bv \cdot \bn, 1 \rangle_{\Gamma_i}=0
\end{aligned}
\end{eqnarray*}
Comparing with the variational formulation $\eqref{Var Form}$ for the test function $(\bv_0, \boldsymbol{0})$, we obtain for all $1 \leqslant i \leqslant I$: 
\begin{eqnarray*}
\begin{aligned}
\sum_{i=1}^I c_i \, \langle \bv \cdot \bn, 1 \rangle_{\Gamma_i} &= \sum_{i=1}^I \langle \bv \cdot \bn, 1 \rangle_{\Gamma_i} \Big[ - \int_\Omega (\curl \bw) \times \bu \cdot \nabla q_i^N \, dx +  \int_\Omega (\curl \bb) \times \bdd \cdot \nabla q_i^N \, dx \\
&+ \langle \ff, \nabla q_i^N \rangle_\Omega - \langle P_0, \nabla q_i^N \cdot \bn \rangle_{\Gamma} \Big]
\end{aligned}
\end{eqnarray*} 
Finally, taking $\bv = \nabla q_j^N$, due to the properties of $q^N_i$ in \eqref{Definition des q_i^N}, we obtain the relation  \eqref{Constantes form var} for all $1 \leqslant j \leqslant I$.\medskip

Conversely, let $(\bu, \bb) \in \bV_N(\Omega) \times \bV_N(\Omega)$ be a solution of \eqref{Var Form} and $\bc = (c_1, \cdots , c_I)$ satisfying \eqref{Constantes form var}. We want to show it implies $(i)$. We note that $\boldsymbol{\mathcal{D}}_{\sigma}(\O)$ is not a subspace of $\bV_{N}(\O)$, so it is not possible to prove directly that \eqref{Var Form}-\eqref{Constantes form var} implies $(i)$. In particular, we can not apply the De Rham's lemma to recover the pressure. As a consequence, we need to extend \eqref{Var Form} for all divergence free functions $(\bv, \boldsymbol{\Psi}) \in \bX_N(\Omega) \times \bX_N(\Omega)$. For this purpose, let $(\bv, \boldsymbol{\Psi}) \in \bX_N(\Omega) \times \bX_N(\Omega)$ and we consider the decomposition \eqref{decomposition of v} for $\bv$ to obtain $\bv_0\in\bV_{N}(\O)$. Similarly, we set  \begin{eqnarray}\label{decomposition psi}{\boldsymbol{\Psi}_{0}} = \boldsymbol{\Psi} - \sum_{i=1}^I \langle \boldsymbol{\Psi} \cdot \bn, 1 \rangle_{\Gamma_i} \nabla q_i^N,\end{eqnarray}
which implies that ${\boldsymbol{\Psi}_{0}}$ is a function of $\bV_{N}(\O)$. Replacing in \eqref{Var Form}, we obtain: 
\small
\begin{eqnarray*}
 \int_\Omega \curl \bu \cdot \curl \bv \, dx + \int_\Omega (\curl \bw) \times \bu \cdot \bv\, dx -  \int_\Omega (\curl \bb) \times \bdd \cdot \bv \, dx 
+   \int_\Omega \curl \bb \cdot \curl \boldsymbol{\Psi} \, dx \qquad\quad \\+  \int_\Omega (\curl \boldsymbol{\Psi}) \times \bdd \cdot \bu \, dx - \langle \ff, \bv \rangle_\Omega
- \langle \bg, \boldsymbol{\Psi} \rangle_\Omega + \langle P_0, {\bv} \cdot \bn \rangle_{\Gamma_0} + \sum_{i=1}^I \langle P_0 + c_i, {\bv} \cdot \bn \rangle_{\Gamma_i} \qquad\qquad \qquad\quad  \qquad\\
= \sum_{i=1}^I \langle {\bv} \cdot \bn, 1 \rangle_{\Gamma_i}  \Big[\underbrace{ -  \int_\Omega (\curl \bb) \times \bdd \cdot \nabla q_i^N \, dx + \int_\Omega (\curl \bw) \times \bu \cdot \nabla q_i^N \, dx
- \langle \ff, \nabla q_i^N \rangle_\Omega + \langle P_0, \nabla q_i^N \cdot \bn \rangle_{\Gamma}}_{\Large{=\,-c_i}}\Big]\\ + \sum_{i=1}^I \langle {\bv} \cdot \bn, 1 \rangle_{\Gamma_i} c_i+ \sum_{i=1}^I \langle {\boldsymbol{\Psi}} \cdot \bn, 1 \rangle_{\Gamma_i} \langle \bg, \nabla q_i^N \rangle_\Omega, \qquad\qquad \qquad\qquad \qquad\qquad \qquad\qquad \qquad\qquad \qquad\qquad
\end{eqnarray*}
where we have used the fact that for all $1 \leqslant i \leqslant I$:  
$\sum_{j=1}^I c_j \,  \langle \nabla q_i^N \cdot \bn, 1 \rangle_{\Gamma_j} = c_i $. Note that the compatibility condition \eqref{Condition compatibilite K_N hilbert} implies that $\langle \bg, \nabla q_i^N \rangle_\Omega = 0$ for all $1 \leqslant i \leqslant I$.
Thus, the right hand side of the above relation is equal to zero and then for any  $(\bv, \boldsymbol{\Psi}) \in \bX_N(\Omega) \times \bX_N(\Omega)$, we have 
\begin{eqnarray} \label{Equation lemme equivalence form var}
 \int_\Omega \curl \bu \cdot \curl \bv \, dx + \int_\Omega (\curl \bw) \times \bu \cdot \bv\, dx -  \int_\Omega (\curl \bb) \times \bdd \cdot \bv \, dx 
+   \int_\Omega \curl \bb \cdot \curl \boldsymbol{\Psi} \, dx\nonumber \qquad\quad \\+  \int_\Omega (\curl \boldsymbol{\Psi}) \times \bdd \cdot \bu \, dx = \langle \ff, \bv \rangle_\Omega
+ \langle \bg, \boldsymbol{\Psi} \rangle_\Omega - \langle P_0, {\bv} \cdot \bn \rangle_{\Gamma_0} - \sum_{i=1}^I \langle P_0 + c_i, {\bv} \cdot \bn \rangle_{\Gamma_i}. \qquad\qquad \qquad\quad  \qquad
\end{eqnarray}
That means that problem \eqref{Equation lemme equivalence form var} and \eqref{Var Form} are equivalent. So, in the sequel, we will prove that problem \eqref{Equation lemme equivalence form var} implies $(i)$.\medskip

Choosing $(\bv, \boldsymbol{0})$ with $\bv \in \boldsymbol{\mathcal{D}}_\sigma(\Omega)$ as a test function in \eqref{Equation lemme equivalence form var}, we have
\begin{eqnarray*}
\langle - \Delta \bu + (\curl \bw) \times \bu -  (\curl \bb) \times \bdd - \ff, \bv \rangle_{\boldsymbol{\mathcal{D}'}(\Omega) \times \boldsymbol{\mathcal{D}}(\Omega)} = 0
\end{eqnarray*}
So by De Rham's theorem, there exists a distribution $P \in \mathcal{D}'(\Omega)$, defined uniquely up to an additive
constant such that
\begin{eqnarray} \label{Eq De Rham}
- \Delta \bu + (\curl \bw) \times \bu -  (\curl \bb) \times \bdd - \ff = - \nabla P\quad \mathrm{in}\,\,\O.
\end{eqnarray}
Since $\bu$ and $\bb$ belong to $\bH^{1}(\O)\hookrightarrow \bL^{6}(\O)$, the terms $(\curl \bw)\times \bu $ and $(\curl \bb) \times \bdd$  belong to $\bL^\frac{6}{5}(\Omega) \hookrightarrow \bH^{-1}(\Omega)$. As $\ff \in [\bH_0^{6,2}(\curl, \Omega)]' \hookrightarrow \bH^{-1}(\Omega)$, we deduce that $\nabla P\in H^{-1}(\O)$ and then $P \in L^2(\Omega)$ with a trace in $H^{-\frac{1}{2}}(\Gamma)$ (we refer to \cite{ARB_ARMA2011}).
Next, choosing $(\boldsymbol{0}, \boldsymbol{\Psi})$ with $\boldsymbol{\Psi} \in \boldsymbol{\mathcal{D}}_\sigma(\Omega)$ in \eqref{Equation lemme equivalence form var}, we have
\begin{eqnarray*}
\langle   \curl \curl \bb -  \curl(\bu \times \bdd) - \bg, \boldsymbol{\Psi} \rangle_{\boldsymbol{\mathcal{D}'}(\Omega) \times \boldsymbol{\mathcal{D}}(\Omega)}= 0 
\end{eqnarray*}
Then, applying \cite[Lemma 2.2]{Pan}, we have $\chi \in L^{2}(\Omega)$ defined uniquely up to an additive constant such that 
\begin{eqnarray*}
  \curl \curl \bb -  \curl (\bu \times \bdd) - \bg = \nabla \chi \quad \mathrm{in}\,\,\O \qquad \mathrm{and}\qquad \chi=0\quad \mathrm{on}\,\,\Gamma
\end{eqnarray*}
We note that the trace of $\chi$ is well defined and belongs to $H^{-1/2}(\Gamma)$. Taking the divergence of the above equation, the function $\chi$ is solution of the harmonic problem 
\begin{eqnarray*}
\Delta \chi = 0 \quad \mathrm{in} \, \, \Omega \qquad \mathrm{and} \qquad \chi = 0 \quad \mathrm{on} \, \, \Gamma. 
\end{eqnarray*}
So, we deduce that $\chi = 0$ in $\Omega$ which gives the second equation in \eqref{linearized MHD-pressure}. Moreover, by the fact that $\bu$ and $\bb$ belong to the space $\bV_{N}(\O)$, we have $\vdiv\bu=\vdiv\bb=0$ in $\O$ and $\bu\times\bn=\bb\times\bn=\mathbf{0}$ on $\Gamma$. \medskip

\noindent It remains to show the boundary conditions on the pressure. Multiplying equation \eqref{Eq De Rham} by $\bv \in \bX_N(\Omega)$, using the decomposition \eqref{decomposition of v} and integrating on $\O$, we obtain 
\small{
\begin{eqnarray*} \label{Num 5}
 \int_\Omega \curl \bu \cdot \curl \bv_0 \, dx + \int_\Omega (\curl \bw) \times \bu \cdot \bv_0 \, dx -  \int_\Omega (\curl \bb) \times \bdd \cdot \bv_0 \, dx 
- \langle \ff, \bv_0 \rangle_\Omega + \langle P, \bv_0 \cdot \bn \rangle_\Gamma \nonumber\\
= \sum_{i=1}^I \langle \bv \cdot \bn, 1 \rangle_{\Gamma_i} \Big[ - \int_\Omega (\curl \bw) \times \bu \cdot \nabla q_i^N \, dx +  \int_\Omega (\curl \bb) \times \bdd \cdot \nabla q_i^N \, dx 
+ \langle \ff, \nabla q_i^N \rangle_\Omega - \langle P, \nabla q_i^N \cdot \bn \rangle_\Gamma \Big]
\end{eqnarray*}}
Taking $(\bv, \boldsymbol{0})$ test function in \eqref{Var Form}, we have: 
\small
\begin{eqnarray*}
 \int_\Omega \curl \bu \cdot \curl \bv_0 \, dx + \int_\Omega (\curl \bw) \times \bu \cdot \bv_0 \, dx -  \int_\Omega (\curl \bb) \times \bdd \cdot \bv_0 \, dx 
- \langle \ff, \bv_0 \rangle_\Omega + \langle P_0, \bv_0 \cdot \bn \rangle_{\Gamma_0} \hspace{12cm}\\+ \sum_{i=1}^I \langle P_0 + c_i, \bv_0 \cdot \bn \rangle_{\Gamma_i}
= \sum_{i=1}^I \langle \bv \cdot \bn, 1 \rangle_{\Gamma_i} \Big[ - \int_\Omega (\curl \bw) \times \bu \cdot \nabla q_i^N \, dx +  \int_\Omega (\curl \bb) \times \bdd \cdot \nabla q_i^N \, dx 
\hspace{12cm}\\ + \langle \ff, \nabla q_i^N \rangle_\Omega- \langle P_0, \nabla q_i^N \cdot \bn \rangle_{\Gamma_0} - \sum_{j=1}^I \langle P_0 + c_j , \nabla q_i^N \cdot \bn \rangle_{\Gamma_j} \Big]\hspace{18cm}
\end{eqnarray*}
Substracting both equations and using again the decomposition \eqref{decomposition of v}, we obtain:
\begin{eqnarray*}
\begin{aligned}
&\underbrace{\langle P, \bv_0 \cdot \bn \rangle_\Gamma +\sum_{i=1}^I \langle \bv \cdot \bn, 1 \rangle_{\Gamma_i}\langle P, \nabla q_i^N \cdot \bn \rangle_\Gamma}_{\langle P, \bv \cdot \bn \rangle_\Gamma}  \\
&= \underbrace{\langle P_0, \bv_0 \cdot \bn \rangle_{\Gamma_0} + \sum_{i=1}^I \langle P_0 + c_i, \bv_0 \cdot \bn \rangle_{\Gamma_i}}_{\langle P_0, \bv_0 \cdot \bn \rangle_{\Gamma}+ c_i\sum_{i=1}^I \langle \bv_{0} \cdot \bn, 1 \rangle_{\Gamma_i} } +\sum_{i=1}^I \langle \bv \cdot \bn, 1 \rangle_{\Gamma_i} \Big[\underbrace{ \langle P_0, \nabla q_i^N \cdot \bn \rangle_{\Gamma_0} + \sum_{j=1}^I \langle P_0 + c_j, \nabla q_i^N \cdot \bn \rangle_{\Gamma_j} }_{\langle P_0, \nabla q_i^N \cdot \bn \rangle_{\Gamma}+c_i} \Big] 
\end{aligned}
\end{eqnarray*}
Since $\langle \bv_0 \cdot \bn, 1 \rangle_{\Gamma_i} =0$, $1\leq i\leq I$, we have
\begin{eqnarray*}
\langle P, \bv \cdot \bn \rangle_\Gamma = \langle P_0, \bv \cdot \bn \rangle_{\Gamma} + \sum_{i=1}^I \langle c_i, \bv \cdot \bn \rangle_{\Gamma_i} = \langle P_0, \bv \cdot \bn \rangle_{\Gamma_0} + \sum_{i=1}^I \langle P_0+c_i, \bv \cdot \bn \rangle_{\Gamma_i} 
\end{eqnarray*}
Next, the argmument to deduce that $P=P_0$ on $\Gamma_0$ and $P = P_0 + c_j$ on $\Gamma_j$ is very similar to that of \cite[Proposition 3.7]{AS_DCDS}, hence we omit it.
\end{proof}

\begin{rmk}
(i) Note that the compatibility condition \eqref{Condition compatibilite K_N hilbert} is necessary. Indeed, if we choose $\bv=\textbf{\emph{0}}$ and $\boldsymbol{\psi}=\nabla q_{i}^{N}$ in \eqref{Var Form}, we have $\langle\bg,\,\nabla q_{i}^ {N}\rangle_{\O_{_{6,2}} }=0$, $1\leq i \leq I$. Observe that since $\O$ is of class $\mathcal{C}^{1,1}$, the functions $q_{i}^{N}$ belong to $H^{2}(\O)$ and then the vectors $\nabla q_{i}^{N}$ belong to $\bH^{6,2}_{0}(\curl,\O)$. From the characterization \eqref{charac dual H^6,2}, this condition is actually written as $\int_{\O}\bF\cdot\nabla q_{i}^ {N}\,d\bx=0$, $1\leq i \leq I$. In the case where $\Omega$ is simply connected, the compatibility condition \eqref{Condition compatibilite K_N hilbert} is not necessary to solve \eqref{linearized MHD-pressure} because the kernel $\bK_N^2(\Omega)=\{ \textbf{\emph{0}} \}$. \medskip

(ii) If $\bg$ is the $\curl$ of an element $\boldsymbol{\xi}\in \bL^{2}(\O)$, then $\bg$ is still an element of $[\bH^{6,2}_{0}(\curl, \O)]'$. Moreover, since $\vdiv\bg = 0$ in $\O$, it always satisfies the compatibility condition \eqref{Condition compatibilite K_N hilbert}.
\end{rmk}

We now prove the solvability of the problem \eqref{Var Form}.

\begin{theorem} \label{thm weak solution H1 linear MHD}
Let $\O$ be $\mathcal{C}^{1,1}$ and we suppose $h=0$. Let 
$$\ff,\, \bg \in [\bH_0^{6,2}(\curl, \Omega)]' \quad \mathrm {and} \quad P_0 \in H^{-\frac{1}{2}}(\Gamma)$$ with the compatibility conditions \eqref{Condition compatibilite K_N hilbert}-\eqref{Condition div nulle}. Then the problem \eqref{linearized MHD-pressure} has a unique weak solution $(\bu, \bb, P, \bc) \in \bH^1(\Omega) \times \bH^1(\Omega) \times L^2(\Omega) \times \R^I$ which satisfies the estimates:
\small
\begin{eqnarray}\label{estim u,b H^1}
\norm{\bu}_{\bH^1(\Omega)} + \norm{\bb}_{\bH^1(\Omega)} \leqslant C (\norm{\ff}_{[\bH_0^{6,2}(\curl, \Omega)]'} + \norm{\bg}_{[\bH_0^{6,2}(\curl, \Omega)]'} + \norm{P_{0}}_{H^{-\frac{1}{2}}(\Gamma)})
\end{eqnarray}
\small
\begin{eqnarray}\label{estim P L^2}
\norm{P}_{L^{2}(\Omega)} \!\leqslant\! C \!(1\!+\!\norm{\curl\bw}_{\bL^{3/2}(\O)}\!+\!\norm{\bdd}_{\bL^{3}(\O)}\!)(\!\norm{\ff}_{[\bH_0^{6,2}(\curl, \Omega)]'} \!+\! \norm{\bg}_{[\bH_0^{6,2}(\curl, \Omega)]'} \!+\! \norm{P_{0}}_{H^{-\frac{1}{2}}(\Gamma)})
\end{eqnarray}
Moreover, if $\O$ is $\mathcal{C}^{2,1}$, $\ff, \bg \in \bL^{6/5}(\Omega)$ and $P_0\in W^{1/6,6/5}(\Gamma)$, then $(\bu, \bb, P) \in \bW^{2,\frac{6}{5}}(\Omega) \times \bW^{2,\frac{6}{5}}(\Omega) \times W^{1,\frac{6}{5}}(\Omega)$ and we have the
following estimate:

\begin{eqnarray}\label{estim u,b W^2,6/5 and P W^1,6/5}
\begin{aligned}
\norm{\bu}_{\bW^{2,\,6/5}(\Omega)} + \norm{\bb}_{\bW^{2,\,6/5}(\Omega)}+\norm{P}_{W^{1,\,6/5}(\Omega)} &\leqslant C (1+\norm{\curl\bw}_{\bL^{3/2}(\O)}+\norm{\bdd}_{\bL^{3}(\O)})\\&\times (\norm{\ff}_{\bL^{6/5}(\O)} + \norm{\bg}_{\bL^{6/5}(\O)} + \norm{P_{0}}_{W^{1/6,6/5}(\Gamma)})
\end{aligned}
\end{eqnarray}
\end{theorem}

\begin{proof}
We know, according to Proposition \ref{Proposition Equivalence sol hilb sol var}, that the linearized problem \eqref{linearized MHD-pressure} is equivalent to \eqref{Var Form}-\eqref{Constantes form var}.
The existence and uniqueness of weak solution $(\bu,\bb)\in \bH^{1}(\O)\times\bH^{1}(\O)$ follow from Lax-Milgram theorem. Let us define the bilinear continuous forms  
$a : \bZ_N(\Omega) \times \bZ_N(\Omega) \rightarrow \R$ and $a_{\bw, \bdd} : \bZ_N(\Omega) \times \bZ_N(\Omega) \rightarrow \R$ as follows: 
\begin{eqnarray}
\begin{aligned}
&a((\bu, \bb), (\bv, \boldsymbol{\Psi})) =  \int_\Omega \curl \bu \cdot \curl \bv \, dx +   \int_\Omega \curl \bb \cdot \curl \boldsymbol{\Psi} \, dx \label{forms a}\\
&a_{\bw, \bdd}((\bu, \bb), (\bv, \boldsymbol{\Psi})) = \int_\Omega (\curl \bw) \times \bu \cdot \bv \, dx +  \int_\Omega (\curl \boldsymbol{\Psi}) \times \bdd \cdot \bu \, dx -  \int_\Omega (\curl \bb) \times \bdd \cdot \bv \, dx 
\end{aligned}
\end{eqnarray}
where $\bZ_N(\O)=\bV_{N}(\O)\times\bV_{N}(\O)$ equipped with the product norm 
\begin{equation}\label{norm on Z_N}
 \norm{(\bv,\,\boldsymbol{\Psi})}_{\bZ_N(\O)}^2=\norm{\bv}_{\bH^{1}(\O)}^2+\norm{\boldsymbol{\Psi}}_{\bH^1(\O)}^2.
\end{equation}
Next, we introduce the linear form $\mathcal{L} : \bZ_N(\Omega) \rightarrow \R$ defined as follows
\begin{eqnarray*}
\mathcal{L}(\bv, \boldsymbol{\Psi}) = \langle \ff, \bv \rangle_{\O_{_{6,2}} }+ \langle \bg, \boldsymbol{\Psi} \rangle_{\O_{_{6,2}} } - \langle P_0, \bv \cdot \bn \rangle_{\Gamma_0} - \sum_{i=1}^I \langle P_0 + c_i, \bv \cdot \bn \rangle_{\Gamma_i}
\end{eqnarray*}
So, the variational formulation \eqref{Var Form} can be rewritten as: for any $(\bv,\boldsymbol{\Psi})\in \bZ_{N}(\O)$
\begin{eqnarray}\label{Var form compact}
 \mathcal{A}((\bu, \bb), (\bv, \boldsymbol{\Psi}))=a((\bu, \bb), (\bv, \boldsymbol{\Psi})) + a_{\bw, \bdd}((\bu, \bb), (\bv, \boldsymbol{\Psi})) = \mathcal{L} (\bv, \boldsymbol{\Psi}) 
\end{eqnarray}
Since $((\curl\bw)\times\bu)\cdot \bu=\textbf{0}$, then we have  $a_{\bw, \bdd}((\bu, \bb), (\bu, \bb)) = 0$ for all $(\bu, \bb) \in \bZ_N(\Omega)$.\\ 
Since $\bv$ and $\boldsymbol{\Psi}$ belong to $\bV_{N}(\O)$, we have from \cite[Corollary 3.2.]{AS_M3AS} that the application $\bv\mapsto \norm{\curl \bv}_{\bL^{2}(\O)}$ (respectively $\boldsymbol{\Psi}\mapsto \norm{\curl \boldsymbol{\Psi}}_{\bL^{2}(\O)}$) is a norm on $\bV_{N}(\O)$ equivalent to the norm $\norm{\bv}_{\bH^{1}(\O)}$ (respectively  $\norm{\boldsymbol{\Psi}}_{\bH^{1}(\O)}$). As a consequence, 
% \begin{eqnarray*}
% \norm{\bv}_{\bH^1(\Omega)}^2 + \norm{\boldsymbol{\Psi}}_{\bH^1(\Omega)}^2 \leqslant C_P \Big( \norm{\curl \bv}_{\bL^2(\Omega)}^2 + \norm{\curl \boldsymbol{\Psi}}_{\bL^2(\Omega)}^2 \Big) 
% \end{eqnarray*}
\begin{eqnarray}\label{coercivity A}
\begin{aligned}
\mathcal{A}((\bu, \bb), (\bv, \boldsymbol{\Psi}))=\abs{a((\bv, \boldsymbol{\Psi}), (\bv, \boldsymbol{\Psi}))} &=  \norm{\curl \bv}_{\bL^2(\Omega)}^2 +   \norm{\curl \boldsymbol{\Psi}}_{\bL^2(\Omega)}^2 \\
&\geqslant\frac{2}{C_{\mathcal{P}}^2} \norm{(\bv, \boldsymbol{\Psi})}_{\bZ_N(\Omega)}^2 
\end{aligned}
\end{eqnarray}
where $C_{\mathcal{P}}$ is the constant given in \eqref{Poincaré inequality}. This shows that the bilinear form $\mathcal{A}(\cdot,\cdot)$ is coercive on $\bZ_{N}(\O)$. Moreover, applying Cauchy-Schwarz inequality, we have: 
\begin{eqnarray}\label{continuity form a}
\abs{a((\bu, \bb), (\bv, \boldsymbol{\Psi}))} &\leqslant&  \norm{\curl \bu}_{\bL^2(\Omega)} \norm{\curl \bv}_{\bL^2(\Omega)} +   \norm{\curl \bb}_{\bL^2(\Omega)} \norm{\curl \boldsymbol{\Psi}}_{\bL^2(\Omega)} \nonumber\\
&\leqslant & C\norm{(\bu, \bb)}_{\bZ_N(\Omega)} \norm{(\bv, \boldsymbol{\Psi})}_{\bZ_N(\Omega)} 
\end{eqnarray}
Now, using Hölder inequality, we have 
\small
\begin{eqnarray*} 
\abs{a_{\bw, \bdd}((\bu, \bb), (\bv, \boldsymbol{\Psi}))} &
\leqslant& \norm{\curl \bw}_{\bL^\frac{3}{2}(\Omega)} \norm{\bu}_{\bL^6(\Omega)} \norm{\bv}_{\bL^6(\Omega)} +  \norm{\curl \boldsymbol{\Psi}}_{\bL^2(\Omega)} \norm{\bu}_{\bL^6(\Omega)} \norm{\bdd}_{\bL^3(\Omega)}\nonumber  \\
&+&  \norm{\curl \bb}_{\bL^2(\Omega)} \norm{\bdd}_{\bL^3(\Omega)} \norm{\bv}_{\bL^6(\Omega)} 
\end{eqnarray*}
Now, using again the equivalence of norms $\norm{\cdot}_{\bV_N(\O)}$ and $\norm{\cdot}_{\bH^{1}(\O)}$, we obtain for $C_1>0$ the constant of the embedding $\bH^1(\Omega) \hookrightarrow \bL^6(\Omega) $
\begin{eqnarray}\label{Continuity of a_wd}
\abs{a_{\bw, \bdd}((\bu,\bb), (\bv, \boldsymbol{\Psi}))}
\leqslant\big( C_1^2 \norm{\curl \bw}_{\bL^\frac{3}{2}(\Omega)} +  C_1  \norm{\bdd}_{\bL^{3}(\Omega)}\big)\norm{(\bu,\,\bb)}_{\bZ_N(\O)}\norm{(\bv,\,\boldsymbol{\Psi})}_{\bZ_N(\O)}
\end{eqnarray}
From \eqref{continuity form a} and \eqref{Continuity of a_wd}, we can deduce that the form $\mathcal{A}(\cdot,\cdot)$ is continuous. Using similar arguments, we can verify that the right hand side of \eqref{Var form compact} defines an element in the dual space of $\bZ_{N}(\O)$. Thus, by Lax-Milgram lemma, there exists a unique $(\bu,\bb)\in \bZ_{N}(\O)$ satisfying \eqref{Var form compact}. So, due to Theorem \ref{injection continue X_N},  we obtain the existence of a unique weak solution $(\bu,\bb)\in\bH^{1}(\O)\times \bH^{1}(\O)$. Using \eqref{coercivity A}, the variational formulation and trace theorem, we obtain the estimate \eqref{estim u,b H^1}. The existence of the pressure follows from De Rham's theorem. Moreover for the pressure estimate, we can write
\small
\begin{eqnarray*}
 \norm{P}_{L^{2}(\O)}\!\leq\! C\norm{\nabla P}_{\bH^{-1}(\O)}\leq\! C\!\big( \norm{\ff}_{\bH^{-1}(\O)}\!+\!\norm{\Delta\bu}_{\bH^{-1}(\O)}\!+\!\norm{(\curl \bw)\times \bu}_{\bH^{-1}(\O)}\!+\!\norm{(\curl \bb)\times \bdd}_{\bH^{-1}(\O)}\big).
\end{eqnarray*}
We know that $\norm{\ff}_{\bH^{-1}(\O)}\leq C \norm{\ff}_{[\bH^{6,2}_{0}(\curl,\O)]'}$ and 
$\norm{\Delta\bu}_{\bH^{-1}(\O)}\leq C \norm{\bu}_{\bH^{1}(\O)}$. For the two remaining terms, we proceed as follows 
\small
\begin{eqnarray*}
\norm{(\curl \bw)\times \bu}_{\bH^{-1}(\O)}\leq C\norm{(\curl \bw)\times \bu}_{\bL^{6/5}(\O)}\leq C\norm{\curl \bw}_{\bL^{3/2}(\O)}\norm{\bu}_{\bL^{6}(\O)},
\end{eqnarray*}
so, we have 
\begin{equation*}\label {control curl w x u in H^-1}
 \norm{(\curl \bw)\times \bu}_{\bH^{-1}(\O)}\leq CC_1\norm{\curl \bw}_{\bL^{3/2}(\O)}\norm{\bu}_{\bH^{1}(\O)}.
\end{equation*}
Proceeding similarly, we get 
\begin{equation*}\label {control curl b x d in H^-1}
 \norm{(\curl \bb)\times \bdd}_{\bH^{-1}(\O)}\leq C\norm{\bdd}_{\bL^{3}(\O)}\norm{\bb}_{\bH^{1}(\O)}
\end{equation*}
Hence, using the above estimates together with the estimate \eqref{estim u,b H^1}, we deduce the pressure estimate \eqref{estim P L^2}.\\
%%%%%%%%%%%%%%% Regularity W^2,6/5 %%%%%%%%%%%%%%%%%%%%%%%
\noindent Now, if $\ff, \bg \in \bL^\frac{6}{5}(\Omega)$ and $P_0 \in W^{1/6,6/5}(\Gamma)$, then we already know that $(\bu, \bb, P) \in \bH^1(\Omega) \times \bH^1(\Omega) \times L^2(\Omega)$ is solution of \eqref{linearized MHD-pressure}. We deduce that $(\curl \bw) \times \bu +  (\curl \bb) \times \bdd$ belongs to $\bL^{6/5}(\O)$. Similarly, we have  $ \curl (\bu \times \bdd) = (\bdd \cdot \nabla) \bu - (\bu \cdot \nabla) \bdd$ belongs to $\bL^{6/5}(\O)$. Observe that $(\bu,P,\bc)$ is solution of the following Stokes problem
\begin{eqnarray*}
\begin{cases}
-  \Delta \bu + \nabla P = \bF \quad \mathrm{and}\quad \vdiv \bu = 0\quad \text{in} \, \, \Omega, \\
\bu \times \bn = \textbf{0}\quad \text{on} \, \, \Gamma \quad  \mathrm {and}\quad P= P_0 \quad \text{on} \, \, \Gamma_0, \quad P=P_0 + c_i \quad \text{on} \, \, \Gamma_i, \\
\langle \bu \cdot \bn, 1 \rangle_{\Gamma_i} = 0, \quad \forall 1 \leqslant i \leqslant I,
\end{cases}
\end{eqnarray*}
with $\bF = \ff - (\curl \bw) \times \bu +  (\curl \bb) \times \bdd$ in $\bL^{6/5}(\O)$. Thanks to the regularity of Stokes problem $(\mathcal{S_N})$ (see Proposition \ref{thm solution W1,p and W2,p Stokes divu=h}), $(\bu,P)$ belongs to $\bW^{2,6/5}(\O)\times W^{1,6/5}(\O)$ with the corresponding estimate. Next, since $\bb$ is a solution of the following elliptic problem 
\begin{eqnarray*}
\begin{cases}
   \curl \curl \bb=\bG\quad \mathrm{and}\quad \vdiv \bb = 0\quad \text{in} \, \, \Omega, \\
\bb \times \bn = \textbf{0}\quad \text{on} \, \, \Gamma, \\
\langle \bb \cdot \bn, 1 \rangle_{\Gamma_i} = 0, \quad \forall 1 \leqslant i \leqslant I,
\end{cases}
\end{eqnarray*}
with $\bG = \bg +  \curl(\bu \times \bdd)$ in $\bL^{6/5}(\O)$ satisfying the compatibility conditions \eqref{Condition compatibilite K_N hilbert}-\eqref{Condition div nulle}, we deduce from Theorem \ref{thm strong soulution elliptic} that $\bb$ belongs to $\bW^{2,6/5}(\O)$. The estimate \eqref{estim u,b W^2,6/5 and P W^1,6/5} then follows from the regularity estimates of the above Stokes problem on $(\bu,P)$ and elliptic problem on $\bb$.
\end{proof}

\section{The linearized MHD system: $L^{p}$-theory}
After the study of weak solutions in the case of Hilbert spaces, we are interested in the study of weak and strong solutions in $L^p$-theory for the linearized system \eqref{linearized MHD-pressure}. We begin by studying strong solutions. If $p \geq 6/5$, it follows that $\bL^p(\O)\hookrightarrow \bL^{6/5}(\O)$, $ W^{1-1/p,p}(\Gamma)\hookrightarrow W^{1/6,6/5}(\Gamma)$. Then, due to  Theorem \ref{thm weak solution H1 linear MHD}, we have $(\bu,\bb)\in\bW^{2,6/5}(\O)\times\bW^{2,6/5}(\O)$. In the next subsection we will prove that this solution belongs to $\bW^{2,p}(\O)\times\bW^{2,p}(\O)$ for any $p>6/5$. 

\subsection{Strong solution in $\bW^{\,2,p}(\Omega)$ with $p\geq 6/5$}
The aim of this section is to give an answer to the question of the existence of a regular solution $(\bu,\bb,P)\in \bW^{2,p}(\O)\times\bW^{2,p}(\O)\times W^{1,p}(\O)$ for the linearized MHD problem \eqref{linearized MHD-pressure}. When $p<3$, we have the embedding $\bW^{2,p}(\O)\hookrightarrow \bW^{1,p^{*}}(\O)$ with $\frac{1}{p^*}=\frac{1}{p}-\frac{1}{3}$. Then, supposing $\bdd\in \bW^{\,1,3/2}(\O)\hookrightarrow \bL^{3}(\O)$ implies that the term $(\curl\bb)\times\bdd$ belongs to $\bL^{p}(\O)$. If $p<3/2$, $\bW^{1,p^*}(\O)\hookrightarrow \bL^{{p}^{**}}(\O)$ with $\frac{1}{p^{**}}=\frac{1}{p^*}-\frac{1}{3}$ and then supposing $\curl \bw \in \bL^{3/2}(\O)$ implies that the term $(\curl\bw)\times\bu$ belongs to $\bL^{p}(\O)$. So, we can use the well-known regularity of the Stokes problem $(\mathcal{S_N})$ in order to prove the regularity $ \bW^{2,p}(\O)\times W^{1,p}(\O)$ with $p<3/2$ for $(\bu,P)$ since the right hand side $\ff-(\curl\bw)\times\bu+ (\curl\bb)\times \bdd$ belongs to $\bL^{p}(\O)$. Similarly, the term $\curl(\bu\times\bdd)=(\bdd\cdot\nabla)\bu-(\bu\cdot\nabla)\bdd$ belongs to $\bL^{p}(\O)$ and we can use the regularity of the elliptic problem $(\mathcal{E}_N)$ to prove the regularity $\bW^{2,p}(\O)$ with $p<3/2$ for $\bb$. Now, if $3/2\leq p<3$, the terms $(\curl\bb)\times\bdd$ and $(\bdd\cdot\nabla)\bu$ still belong to $\bL^{p}(\O)$ but the situation is different for the terms $(\curl\bw)\times\bu$ and $(\bu\cdot\nabla)\bdd$ if $\curl\bw$ and $\nabla\bdd$ belong only to $\bL^{3/2}(\O)$. Indeed, we must suppose $\curl\bw\in\bL^{s}(\O)$ and $\nabla\bdd\in\bL^{s}(\O)$ with 
\begin{equation*}\label{def of s}
s= \frac{3}{2} \quad \mathrm{if} \,\,p < \frac{3}{2},\quad s>\frac{3}{2}\quad \mathrm{if}\,\,p=\frac{3}{2}\quad \mathrm{and}\quad s=p\quad \mathrm{if}\,\,p>\frac{3}{2}. 
\end{equation*}

Now, if $p\geq 3 $, the problem arises for the terms $(\curl\bb)\times\bdd$ and $(\bdd\cdot\nabla)\bu$ if we suppose only $\bdd$ in $\bL^{3}(\O)$. So, we must suppose that $\bdd\in\bL^{s'}(\O)$ with 
\begin{equation*}\label{def of s'}
s'= 3 \quad \mathrm{if} \,\,p < 3,\quad s'> 3\quad \mathrm{if}\,\,p=3\quad \mathrm{and}\quad s'=p\quad \mathrm{if}\,\,p> 3. 
\end{equation*}

So, to conserve the assumptions $\curl\bw\in\bL^{3/2}(\O)$ and $\bdd\in\bW^{1,3/2}(\O)$ and prove strong solutions $(\bu,\bb,P)$ in $\bW^{2,p}(\O)\times \bW^{2,p}(\O)\times W^{1,p}(\O)$, we first assume that $\bw$ and $\bdd$ are more regular and belong to $\boldsymbol{\mathcal{D}}(\overline{\O})$. We will then prove a priori estimates allowing to remove this latter regularity. We refer to  \cite[Theorem 2.4]{ARB_M2AN2014}  for a similar proof for the Oseen problem. The details are given in the following regularity result in a solenoidal framework.
\begin{theorem}\label{thm strong solution p>6/5}
Let $\O$ be $\mathcal{C}^{2,1}$ and $p\geq 6/5$. Assume that $h=0$, and let $\ff,\bg$, $\bw$, $\bdd$ and $P_0$ satisfying \eqref{Condition div nulle},
$$\ff\in \bL^p(\Omega),\,\ \bg \in \bL^p(\Omega),\,\,\curl\bw\in \bL^{3/2}(\O),\,\,\bdd\in\bW^{\,1,3/2}_{\sigma}(\O),\,\,\,\mathrm{and} \,\,\,P_0 \in W^{1-\frac{1}{p}, p}(\Gamma)$$
with the compatibility condition 
\begin{equation} \label{Condition compatibilite K_N Lp p>6/5} 
 \forall \, \bv \in \bK_{N}^{p'}(\Omega),\quad \int_{\O}\bg\cdot\bv \,d\bx = 0. 
 \end{equation}
Then, the weak solution  $(\bu, \bb, P)$ of the problem \eqref{linearized MHD-pressure} given by Theorem \ref{thm weak solution H1 linear MHD} belongs to $\bW^{2,p}(\Omega) \times \bW^{2,p}(\Omega) \times W^{1,p}(\Omega) $ which also satisfies
the estimate:
\begin{eqnarray}
\begin{aligned}\label{estim u,b,P strong p>6/5}
&\norm{\bu}_{\bW^{2,p}(\Omega)} + \norm{\bb}_{\bW^{2,p}(\Omega)}+ \norm{P}_{W^{1,p}(\Omega)} \leq C(1 + \norm{\curl\bw}_{\bL^\frac{3}{2}(\Omega)} + \norm{\bdd}_{\bW^{1,\frac{3}{2}}(\Omega)}) \\
&\times \big(\norm{\ff}_{\bL^p(\Omega)} + \norm{\bg}_{\bL^p(\Omega)} + \norm{P_0}_{W^{1-\frac{1}{p},p}(\Gamma)}\big) 
\end{aligned}
\end{eqnarray}
\end{theorem}

\begin{proof}We prove it in two steps:\\
\underline{\textbf{ First step:}} We consider the case of $\bw\in \boldsymbol{\mathcal{D}}(\overline{\O})$ and $\bdd\in \boldsymbol{\mathcal{D}}_{\sigma}(\overline{\O})$. We know that for all $p \geq 6/5$ we have

$$\bL^p(\Omega)\hookrightarrow \bL^{6/5}(\O)\quad \mathrm{and} \quad W^{\,1-1/p,p}(\Gamma)\hookrightarrow H^{1/2}(\Gamma).$$
Thanks to Theorem \ref{thm weak solution H1 linear MHD}, there exists a unique solution $(\bu, \bb, P, \bc) \in  \bW^{2,\frac{6}{5}}(\Omega) \times \bW^{2,\frac{6}{5}}(\Omega) \times W^{1,\frac{6}{5}}(\Omega)\times \R^{I}$ verifying the estimates \eqref{estim u,b H^1}-\eqref{estim P L^2}.\\
Since $\bu\in \bW^{2,\frac{6}{5}}(\Omega) \hookrightarrow \bL^6(\Omega)$ and $\curl\bb \in \bW^{1,\frac{6}{5}}(\Omega) \hookrightarrow \bL^2(\Omega)$ , it follows that $(\curl\bw)\times \bu \in \bL^6(\Omega)$ and  $(\curl \bb) \times \bdd \in \bL^2(\Omega)$. Note that $\bL^{2}(\O)\hookrightarrow \bL^{p}(\O)$ if $p\leq 2$, then we have three cases:\\
{\textbf{Case $\frac{6}{5} < p \leqslant 2$:}} Since $\ff-(\curl\bw) \times \bu +(\curl \bb) \times \bdd\in \bL^{p}(\O) $, thanks to the existence of strong solutions
for Stokes equations (see Theorem \ref{thm solution W1,p and W2,p Stokes divu=h}), we have that $(\bu, P)\in  \bW^{\,2,p}(\Omega)\times W^{\,1,p}(\Omega)$. Moreover, we have $\bg+\curl(\bu\times\bdd)\in \bL^{p}(\O)$. Thanks to the regularity of elliptic problem (see Theorem \ref{thm strong soulution elliptic}), we have that $\bb\in\bW^{\,2,p}(\O)$.\\
{\textbf{Case $2 \leqslant p \leqslant 6$:}} From the previous case, $(\bu, \bb, P) \in \bH^2(\Omega) \times \bH^2(\Omega) \times H^1(\Omega)$. Since 
\begin{eqnarray*}
\bH^2(\Omega) \hookrightarrow \bW^{1,6}(\Omega) \hookrightarrow \bL^\infty(\Omega),
\end{eqnarray*}
then $(\curl\bw) \times \bu \in \bL^\infty(\Omega) $ and $(\curl \bb) \times \bdd \in \bL^6(\Omega)$. Hence we hava that $\ff-\curl\bw\times \bu +(\curl \bb) \times \bdd\in \bL^{p}(\O)$. Again, by Theorem \ref{thm solution W1,p and W2,p Stokes divu=h}, it follows that $(\bu,P) \in \bW^{2,p}(\Omega) \times W^{1,p}(\Omega)$. Moreover, we have that $\bg+\curl(\bu\times\bdd)\in \bL^{6}(\O)$. Thanks to Theorem  \ref{thm strong soulution elliptic}, we have that $\bb \in \bW^{2,p}(\Omega)$.\\ 
{\textbf{Case $p >6$:}} We know that $(\bu, \bb, P) \in \bW^{2,6}(\Omega) \times \bW^{2,6}(\Omega) \times W^{1,6}(\Omega)$. Since 
\begin{eqnarray*}
\bW^{2,6}(\Omega) \hookrightarrow \bW^{1,\infty}(\Omega)
\end{eqnarray*}
then $(\curl\bw) \times \bu\in \bL^{q}(\O)$,  $(\curl \bb) \times \bdd\in\bL^{q}(\O)$ and $\curl(\bu\times\bdd)\in\bL^{q}(\O)$ for any $q\geq 1$. Again, according to the regularity of
Stokes and elliptic problems, we have $(\bu, \bb, P) \in \bW^{2,p}(\Omega)\times \bW^{2,p}(\Omega) \times W^{1,p}(\Omega)$ and we have the following estimate:
\begin{eqnarray} \label{Estimation forte Stokes Elliptique avec Lambda}
\begin{aligned}
&\norm{\bu}_{\bW^{2,p}(\Omega)} +  \norm{\bb}_{\bW^{2,p}(\Omega)}+\norm{P}_{W^{1,p}(\Omega)} \\
&\leqslant C_{_{SE}} \Big( \norm{\ff}_{\bL^p(\Omega)} + \norm{(\curl\bw) \times \bu}_{\bL^p(\Omega)} + \norm{(\curl \bb) \times \bdd}_{\bL^p(\Omega)} + \norm{P_0}_{W^{1-\frac{1}{p}, p}(\Gamma)} \\
&+ \sum_{i=1}^I \vert c_i \vert + \norm{\bg}_{\bL^p(\Omega)} + \norm{\curl (\bu \times \bdd)}_{\bL^p(\Omega)} \Big),
\end{aligned}
\end{eqnarray}
where $C_{_{SE}} = \max(C_{_{S}}, C_{_{E}})$ with $C_{_{S}}$ the constant given in \eqref{estim W^2,p Stokes} and $C_{_{E}}$ the constant given in \eqref{estim W^2,p elliptic}. \\
To prove the estimate \eqref{estim u,b,P strong p>6/5},  we must bound the terms $\norm{(\curl\bw) \times \bu}_{\bL^p(\Omega)}$, $\norm{(\curl \bb) \times \bdd}_{\bL^p(\Omega)}$, $\norm{\curl (\bu \times \bdd)}_{\bL^p(\Omega)}$ and $\sum_{i=1}^I | c_i |$ in the right hand side of \eqref{Estimation forte Stokes Elliptique avec Lambda}.\\ For this, let  $\epsilon > 0$ and $\rho_{\epsilon/2} $ the classical mollifier. We consider $\widetilde{\by}=\widetilde{\curl \bw}$ and  $ \widetilde{\bdd}$ the extensions by $\mathbf{0}$ of $\by$ and $\bdd$ to $\R^3$, respectively. We decompose $\curl\bw$ and $\bdd$: 
\begin{eqnarray}
\curl\bw= \by_1^\epsilon + \by_{2}^\epsilon\quad \mathrm{where}&\quad \by_1^\epsilon = \widetilde{\curl\bw} \ast \rho_{\epsilon/2} \quad \mathrm{and}\quad \by_{2}^\epsilon = \curl\bw- \by_1^\epsilon,\label{decomposition y}\\
\bdd= \bdd_1^\epsilon + \bdd_{2}^\epsilon \quad \mathrm{where}& \bdd_1^\epsilon = \widetilde{\bdd} \ast \rho_{\epsilon/2} \quad {}\label{decomposition d}
\end{eqnarray}
\noindent  \textbf{(i) Estimate of the term} $\norm{(\curl\bw)\times \bu}_{\bL^p(\Omega)}$. 
First, we look for the estimate depending on $\by^{\varepsilon}_{2}$. 
Observe that $\bW^{2,p}(\Omega) \hookrightarrow \bL^m(\Omega)$ with $$\frac{1}{m}=\frac{1}{p}-\frac{2}{3}\,\,\,\mathrm{if} \,\,p<3/2,\,\,\,  \,\,m=\dfrac{3s}{2s-3} \in[1,\,\infty[\,\,\,\mathrm{if}\,\, p=3/2\quad \mathrm{and} \quad  \,\,m=\infty\,\,\, \mathrm{if}\,\,p>3/2.$$
Using the Hölder inequality and Sobolev embedding, we have
\begin{eqnarray*}
\norm{\by_{2}^\epsilon\times \bu}_{\bL^p(\Omega)} \leqslant \norm{\by_{2}^\epsilon}_{\bL^s(\Omega)} \norm{\bu}_{\bL^m(\Omega)} \leq C\norm{\by_{2}^\epsilon}_{\bL^s(\Omega)} \norm{\bu}_{\bW^{\,2,p}(\Omega)}
\end{eqnarray*}
where $\frac{1}{p}=\frac{1}{m}+\frac{1}{s}$ and $s$ the real number defined as:
\begin{equation}\label{def of s}
s= \frac{3}{2} \quad \mathrm{if} \,\,p < \frac{3}{2},\quad s>\frac{3}{2}\quad \mathrm{if}\,\,p=\frac{3}{2}\quad \mathrm{and}\quad s=p\quad \mathrm{if}\,\,p>\frac{3}{2}. 
\end{equation}
 Moreover, we have 
\begin{eqnarray*}
\norm{\by_{2}^\epsilon}_{\bL^s(\Omega)} =  \norm{\curl\bw - \widetilde{\curl\bw} \ast \rho_{\epsilon/2}}_{\bL^s(\O)}\leq \epsilon.
\end{eqnarray*}
Then, it follows that
\begin{eqnarray}\label{estim y_2}
\begin{aligned}
\norm{\by_{2}^\epsilon \times \bu}_{\bL^p(\Omega)}\leqslant C \epsilon \norm{\bu}_{\bW^{2,p}(\Omega)}.
\end{aligned}
\end{eqnarray}
To get the estimate depending on $\by_1^\epsilon$, we consider two steps (similar to \cite[Theorem 3.5]{AS_DCDS}):\\

$\bullet$ \,{\textbf{Case $\frac{6}{5}\leq p \leqslant 6\,$:}} there exists $q \in [\frac{3}{2},\infty]$ such that $\frac{1}{p}=\frac{1}{q}+\frac{1}{6}$. By Hölder inequality, we have 
\begin{eqnarray*}
\begin{aligned}
\norm{\by_1^\epsilon \times \bu}_{\bL^p(\Omega)} &\leqslant \norm{\by_1^\epsilon}_{\bL^q(\Omega)} \norm{\bu}_{\bL^6(\Omega)}.
\end{aligned}
\end{eqnarray*} 
Let $t\in[1,\,3]$ such that  $1+\frac{1}{q}=\frac{2}{3}+\frac{1}{t}$, we obtain 
\begin{eqnarray*}
\begin{aligned}
\norm{\by_1^\epsilon \times \bu}_{\bL^p(\Omega)} 
\leqslant \norm{\curl\bw}_{\bL^\frac{3}{2}(\Omega)} \norm{\rho_\epsilon}_{\bL^t(\Omega)} \norm{\bu}_{\bL^6(\Omega)}
\leq C_{\varepsilon}\norm{\curl \bw}_{\bL^\frac{3}{2}(\Omega)} \norm{\bu}_{\bL^6(\Omega)},
\end{aligned}
\end{eqnarray*} 
where $C_{\varepsilon}$ is the constant absorbing the norm of the mollifier. Since $\bH^1(\Omega) \hookrightarrow \bL^6(\Omega)$, it follows from \eqref{estim u,b H^1} that 
\begin{eqnarray*}
\norm{\by_1^\epsilon \times \bu}_{\bL^p(\Omega)} \leqslant C_2 C_\epsilon \norm{\curl \bw}_{\bL^\frac{3}{2}(\Omega)} \!\!\big( \norm{\ff}_{[\bH_0^{6,2}(\curl, \Omega)]'} + \norm{\bg}_{[\bH_0^{6,2}(\curl, \Omega)]'} + \norm{P_0}_{H^{-\frac{1}{2}}(\Gamma)} \big),
\end{eqnarray*}
where $C_{2}$ is the constant of the Sobolev embedding $\bH^1(\Omega) \hookrightarrow \bL^6(\Omega)$.
Since, $p \geqslant \frac{6}{5}$, we deduce that 
\begin{eqnarray}\label{first estim y ^1 epsilon}
\norm{\by_1^\epsilon \times \bu}_{\bL^p(\Omega)} \leqslant C_3 C_\epsilon \norm{\curl \bw}_{\bL^\frac{3}{2}(\Omega)} \Big( \norm{\ff}_{\bL^p(\Omega)} + \norm{\bg}_{\bL^p(\Omega)} + \norm{P_0}_{W^{1-\frac{1}{p},p}(\Gamma)} \Big),
\end{eqnarray}
where $C_{3}$ is a constant which depends on $C_2$, $\bL^p(\Omega) \hookrightarrow [\bH_0^{6,2}(\curl, \Omega)]'$ and $ W^{1-\frac{1}{p},p}(\Gamma) \hookrightarrow H^{-\frac{1}{2}}(\Gamma)$.\\
$\bullet$ \textbf{Case} $p > 6$: we know that the embedding 
$$\bW^{2,p}(\Omega) \hookrightarrow \bW^{1,m}(\Omega),$$
is compact for any $m\in[1,p^*[$ if $p<3$, for any $m\in[1,\,\infty[$ if $p=3$ and for $m\in[1,\infty]$ if $p>3$.\\We choose the exponent $m$ such that $6< m< +\infty$. So, we have:
$$
\bW^{2,p}(\Omega) \underset{\mathrm{compact}}{\hookrightarrow} \bW^{1,m}(\Omega) \underset{\mathrm{continuous}}{\hookrightarrow} \bL^{\,6}(\O).
$$                                                                             
Hence, for any $\varepsilon'>0$, we know that there exists a constant $C_{\varepsilon '}$ such that the following interpolation inequality holds:
\begin{eqnarray}\label{interpolation ineq}
\norm{\bu}_{\bW^{1,m}(\Omega)} \leqslant \varepsilon '\norm{\bu}_{\bW^{2,p}(\Omega)} + C_{\varepsilon '} \norm{\bu}_{\bH^1(\Omega)} 
\end{eqnarray}
For $t >2$ such that $1+\frac{1}{p}=\frac{2}{3}+\frac{1}{t}$, we have
\begin{eqnarray*}
\begin{aligned}
\norm{\by_1^\epsilon \times \bu}_{\bL^p(\Omega)} &\leqslant C\norm{\by_1^\epsilon}_{\bL^p(\Omega)} \norm{\bu}_{\bW^{\,1,m}(\Omega)} \\
&\leqslant C\norm{\curl\bw}_{\bL^\frac{3}{2}(\Omega)} \norm{\rho_{\epsilon/2}}_{\bL^t(\R^3)} \norm{\bu}_{\bW^{\,1,m}(\Omega)} .
\end{aligned}
\end{eqnarray*}
Using \eqref{interpolation ineq}, we obtain
\begin{eqnarray}\label{second estim y ^1 epsilon}
\begin{aligned}
\norm{\by_1^\epsilon \times \bu}_{\bL^p(\Omega)} &\leqslant C C_\epsilon \norm{\curl\bw}_{\bL^\frac{3}{2}(\Omega)} (\varepsilon ' \norm{\bu}_{\bW^{2,p}(\Omega)} + C_{\varepsilon^{'}} \norm{\bu}_{\bH^1(\Omega)} )
\end{aligned}
\end{eqnarray}
Thus, choosing $\varepsilon '>0$ small enough, we can deduce from \eqref{first estim y ^1 epsilon} or \eqref{second estim y ^1 epsilon} that
\begin{eqnarray}\label{final estim y ^1 epsilon}
\norm{\by_1^\epsilon \times \bu}_{\bL^p(\Omega)} \leqslant C C_\epsilon \norm{\curl \bw}_{\bL^\frac{3}{2}(\Omega)}(\varepsilon ' \norm{\bu}_{\bW^{2,p}(\Omega)} + C_{\varepsilon^{'}} \norm{\bu}_{\bH^1(\Omega)} ).
\end{eqnarray}
 \textbf{(ii) Estimate of the term} $\norm{(\curl \bb) \times \bdd}_{\bL^p(\Omega)}$. Using the decomposition \eqref{decomposition d}, as previously, we have for the part $\bdd_{2}^{\varepsilon}$:
\begin{eqnarray} \label{Estim d_2}
\norm{\bdd_{ 2}^\epsilon}_{\bL^{s'}(\Omega)} \leqslant \norm{\bdd-\widetilde{\bdd} \ast \rho_{\epsilon/2}}_{\bL^{s'}(\O)} \leq \epsilon.
\end{eqnarray}
Recall that $ \bW^{\,2,p}(\O) \hookrightarrow  \bW^{\,1,k}(\O)$ for $k=p^{*}=\dfrac{3p}{3-p}$ if $p<3$, for any $k\in[1,\infty[$ if $p=3$ and $k=\infty$ if $p > 3$.
Using the Hölder inequality and \eqref{Estim d_2}, we have
\begin{eqnarray}\label{estim d_ lambda, 2}
\begin{aligned}
\norm{(\curl \bb) \times \bdd_{ 2}^\epsilon}_{\bL^p(\Omega)} &\leqslant \norm{\curl \bb}_{\bL^ k(\Omega)} \norm{\bdd_{2}^\epsilon}_{\bL^{s'}(\Omega)}\leqslant C \epsilon \norm{\bb}_{\bW^{2,p}(\Omega)},
\end{aligned}
\end{eqnarray}
where $\frac{1}{p}=\frac{1}{k}+\frac{1}{s'}$ for $s'$ given by:
\begin{equation}\label{def of s'}
s'= 3 \quad \mathrm{if} \,\,p < 3,\quad s'> 3\quad \mathrm{if}\,\,p=3\quad \mathrm{and}\quad s'=p\quad \mathrm{if}\,\,p> 3. 
\end{equation}
It remains to prove the estimate depending on $\bdd_{1}^{\epsilon}$. We have three cases:

$\bullet$ \,{\textbf{Case $ p \leq 2$}}: Using the Hölder inequality, we have 
\begin{eqnarray*}
\begin{aligned}
\norm{(\curl \bb) \times \bdd_{1}^\epsilon}_{\bL^p(\Omega)} &\leqslant \norm{\bdd_{1}^\epsilon}_{\bL^{k}(\Omega)} \norm{\curl\bb}_{\bL^{2}(\Omega)},
\end{aligned}
\end{eqnarray*}
where $\frac{1}{p}=\frac{1}{k}+\frac{1}{2}$. Let $t\in[1,\,3/2]$ such that  $1+\frac{1}{k}=\frac{1}{3}+\frac{1}{t}$, we obtain (since $r\geq 3$)
\begin{eqnarray}\label{estim curlb x d_1 case p<2}
\begin{aligned}
\norm{(\curl \bb) \times \bdd_{1}^\epsilon}_{\bL^p(\Omega)} &\leqslant \norm{\bdd}_{\bL^3(\Omega)} \norm{\rho_{\epsilon/2}}_{\bL^t(\R^3)}\norm{\curl \bb}_{\bL^2(\Omega)} \\& \leq C_{\epsilon}\norm{\bdd}_{\bL^3(\Omega)}\norm{ \bb}_{\bH^{\,1}(\Omega)}.
\end{aligned}
\end{eqnarray}
% 
% Consequently, using \eqref{estim u,b H^1}, we obtain 
% \begin{eqnarray}\label{estim curlb x d_1 case p<2}
% \begin{aligned}
% \norm{(\curl \bb) \times \bdd_{1}^\epsilon}_{\bL^p(\Omega)} &\leqslant C_4C_{\epsilon}\norm{\bdd}_{\bL^3(\Omega)}
% (\norm{\ff}_{\bL^p(\Omega)} + \norm{\bg}_{\bL^p(\Omega)} + \norm{P_0}_{W^{1-\frac{1}{p},p}(\Gamma)}),
% \end{aligned}
% \end{eqnarray}
where $C_4$ is a constant which depends on $\bL^p(\Omega) \hookrightarrow [\bH_0^{6,2}(\curl, \Omega)]'$ and $ W^{1-\frac{1}{p},p}(\Gamma) \hookrightarrow H^{-\frac{1}{2}}(\Gamma)$.\\
$\bullet$ \,{\textbf{Case $ 2<p < 3$}}: Assuming $2<q<p^{*}$, from the relation
\begin{equation*}\bW^{2,p}(\Omega) \underset{\mathrm{compact}}{\hookrightarrow}\bW^{1,q}(\Omega)\underset{\mathrm{continuous}}{\hookrightarrow} \bH^{\,1}(\O)\end{equation*}
we have for any $\epsilon'>0$, there exists a constant $C_{\epsilon}$ such that 
\begin{eqnarray}\label{interpolation ineq epsilon'}
\norm{\bb}_{\bW^{1,q}(\Omega)} \leqslant \varepsilon '\norm{\bb}_{\bW^{2,p}(\Omega)} + C_{\varepsilon '} \norm{\bb}_{\bH^1(\Omega)} 
\end{eqnarray}
Let $k$ be defined by  $\frac{1}{p} = \frac{1}{q}+\frac{1}{k}$ and $t \geqslant 1$ defined by  $1+\frac{1}{k} = \frac{1}{3} + \frac{1}{t}$. Thus, since $k>3$, the following estimate holds:
\begin{eqnarray*}
\begin{aligned}
\norm{(\curl \bb) \times \bdd_1^\epsilon}_{\bL^p(\Omega)} &\leqslant \norm{\bdd_1^\epsilon}_{\bL^k(\Omega)}\norm{\curl \bb}_{\bL^q(\Omega)} \leq \norm{\bdd}_{\bL^3(\Omega)} \norm{\rho_{\epsilon/2}}_{\bL^t(\R^3)} \norm{\curl \bb}_{\bL^q(\Omega)}.
\end{aligned}
\end{eqnarray*}

Next using \eqref{interpolation ineq epsilon'} yields,
\begin{eqnarray}\label{estim curlb x d_1 case 2<p<3}
\begin{aligned}
\norm{(\curl \bb) \times \bdd_1^\epsilon}_{\bL^p(\Omega)}\leq C_\epsilon\norm{\bdd}_{\bL^3(\Omega)}( \varepsilon '\norm{\bb}_{\bW^{2,p}(\Omega)} + C_{\varepsilon '} \norm{\bb}_{\bH^1(\Omega)} ).
\end{aligned}
\end{eqnarray}
$\bullet$ \,{\textbf{Case $ p \geq 3$}}: For $\frac{1}{p}=\frac{1}{s'}+\frac{1}{p*}$ with $s'$ defined in \eqref{def of s'}, we have
\begin{eqnarray*}
\begin{aligned}
\norm{(\curl \bb) \times \bdd_1^\epsilon}_{\bL^p(\Omega)} &\leqslant \norm{\bdd_1^\epsilon}_{\bL^{s'}(\Omega)}\norm{\curl \bb}_{\bL^{p*}(\Omega)}.
\end{aligned}
\end{eqnarray*}
Let $t$ be defined by $1+\frac{1}{s'}=\frac{1}{3}+\frac{1}{t}$. Thus, using \eqref{interpolation ineq epsilon'} with $q=p*$, we obtain:
\begin{eqnarray}\label{estim curlb x d_1 case p>3}
\begin{aligned}
\norm{(\curl \bb) \times \bdd_1^\epsilon}_{\bL^p(\Omega)} \leq C_\epsilon\norm{\bdd}_{\bL^3(\Omega)}( \varepsilon '\norm{\bb}_{\bW^{2,p}(\Omega)} + C_{\varepsilon '} \norm{\bb}_{\bH^1(\Omega)} ).
\end{aligned}
\end{eqnarray}
Choosing $\epsilon '>0$ small enough, we deduce from \eqref{estim curlb x d_1 case p<2}, \eqref{estim curlb x d_1 case 2<p<3} or \eqref{estim curlb x d_1 case p>3} that 
\begin{eqnarray}\label{final estim curlb x d_1}
\begin{aligned}
\norm{(\curl \bb) \times \bdd_1^\epsilon}_{\bL^p(\Omega)} \leq C_\epsilon\norm{\bdd}_{\bL^3(\Omega)}( \varepsilon '\norm{\bb}_{\bW^{2,p}(\Omega)} + C_{\varepsilon '} \norm{\bb}_{\bH^1(\Omega)} ).
\end{aligned}
\end{eqnarray}

\noindent  \textbf{(iii) Estimate of the term} $\norm{\curl (\bu\times \bdd)}_{\bL^p(\Omega)}$. Note that, since $\vdiv\,\bu=0$ and $\vdiv\,\bdd=0$ :

$$\curl\,(\bu\times\bdd)=\bdd\cdot\nabla\bu-\bu\cdot\nabla\bdd.$$
$\bullet$ \,{\textbf{The term}} $\norm{\bdd\cdot\nabla\bu}_{\bL^{p}(\O)}$. Using the decomposition \eqref{decomposition d} and exactly the same analysis as in \textbf{(ii)} for the term $(\curl\bb)\times\bdd$ with $\curl\bb$ replaced by $\nabla\bu$, we obtain the following estimates:
\begin{eqnarray}\label{estim d_2 for d.nabla u}
\begin{aligned}
\norm{ \bdd_{ 2}^\epsilon\cdot\nabla\bu}_{\bL^p(\Omega)} &\leqslant \norm{\nabla \bu}_{\bL^ k(\Omega)} \norm{\bdd_{2}^\epsilon}_{\bL^{s'}(\Omega)}\leqslant C \epsilon \norm{\bu}_{\bW^{2,p}(\Omega)},
\end{aligned}
\end{eqnarray}
where $s'$ is defined in \eqref{def of s'} and 
\begin{eqnarray}\label{final estim d_1 nabla u}
\begin{aligned}
\norm{\bdd_1^\epsilon\cdot \nabla \bu}_{\bL^p(\Omega)} \leq C_\epsilon\norm{\bdd}_{\bL^3(\Omega)}( \varepsilon '\norm{\bu}_{\bW^{2,p}(\Omega)} + C_{\varepsilon '} \norm{\bu}_{\bH^1(\Omega)} ).
\end{aligned}
\end{eqnarray}
$\bullet$ \,{\textbf{The term}} $\norm{\bu\cdot\nabla\bdd}_{\bL^{p}(\O)}$. The analysis is similar to the case \textbf{(i)}. We consider:
\begin{eqnarray}
\nabla\bdd= \bz_1^\epsilon + \bz_{2}^\epsilon &\quad\mathrm{where}&\quad \bz_1^\epsilon = \widetilde{\nabla\bdd} \ast \rho_{\epsilon/2}\quad\,\, \mathrm{and} \quad \bz_{2}^\epsilon = \nabla\bdd - \widetilde{\nabla\bdd} \ast \rho_{\epsilon/2},\label{decomposition nabla d}
\end{eqnarray}
 $\widetilde{\nabla\bdd}$ is the extension by zero of $\nabla\bdd$ to $\R^3$. Observe that 
\begin{eqnarray*} \label{Estim eta2 and z_2}
\norm{\bz_{ 2}^\epsilon}_{\bL^{s}(\Omega)} \leqslant \Vert\nabla\bdd-\widetilde{\nabla\bdd} \ast \rho_{\epsilon/2}\Vert_{\bL^{s}(\O)} \leq \epsilon,
\end{eqnarray*}
with $s$ given in \eqref{def of s}. Using the above estimates and the same arguments as in the case \textbf{(i)}, the influence of $\bz_2^\epsilon$ in the bound of $\norm{\bu\cdot\nabla\bdd}_{\bL^{p}(\O)}$ is given by:
\begin{eqnarray}\label{estim u z_2}
\norm{\bu\cdot\bz_{2}^\epsilon}_{\bL^p(\Omega)} \leq C\norm{\bz_{2}^\epsilon}_{\bL^s(\Omega)} \norm{\bu}_{\bL^{m}(\Omega)} \leq C \epsilon \norm{\bu}_{\bW^{2,p}(\Omega)}.
\end{eqnarray}
And for the bound depending on $\bz_1^\epsilon$, proceeding in the same way as in the case \textbf{(i)}, we derive:
\begin{eqnarray}\label{final estim z_1 u}
\norm{\bz_1^\epsilon \cdot \bu}_{\bL^p(\Omega)} \leqslant C C_\epsilon \norm{\nabla \bdd}_{\bL^\frac{3}{2}(\Omega)}(\varepsilon ' \norm{\bu}_{\bW^{2,p}(\Omega)} + C_{\varepsilon^{'}} \norm{\bu}_{\bH^1(\Omega)} )
\end{eqnarray}
%%%%%%%%%%%%%% estimation des constantes c_i %%%%%%%%%%%%%%%%%%%%%%%%
\textbf{(iv) Estimate of the constants} $\sum_{i=1}^{I}\vert c_i\vert$. 
We note that \small{
\begin{eqnarray*}
\begin{aligned}
|c_i| &\leqslant \vert\int_{\O}\ff \cdot\nabla q_i^N \,d\bx| + |\int_{\O}(\curl \bb \times \bdd)\cdot \nabla q_i^N\,d\bx| + |\int_{\O}(\curl\bw \times \bu)\cdot \nabla q_i^N)\,d\bx | + |\int_{\Gamma} P_0 \nabla q_i^N \cdot \bn\,d\sigma | \\
&\leqslant \norm{\ff}_{\bL^\frac{6}{5}(\Omega)} \norm{\nabla q_i^N}_{\bL^6(\Omega)} + \norm{\curl \bb}_{\bL^2(\Omega)} \norm{\bdd}_{\bL^3(\Omega)} \norm{\nabla q_i^N}_{\bL^6(\Omega)}  \\
&+ \norm{\curl\bw}_{\bL^\frac{3}{2}(\Omega)} \norm{\bu}_{\bL^6(\Omega)} \norm{\nabla q_i^N}_{\bL^6(\Omega)} + \norm{P_0}_{H^{-\frac{1}{2}}(\Gamma)} \norm{\nabla q_i^N \cdot \bn}_{H^\frac{1}{2}(\Gamma)}
\end{aligned}
\end{eqnarray*}}
Thanks to \cite[Corollary 3.3.1]{Sel-thesis}, we know that the functions $\nabla q_i^N $ belong to $\bW^{\,1,q}(\O)$ for any $q\geq 2$ where each $q_{i}^{N}$ is the unique solution of the problem \eqref{Definition des q_i^N}.\\
Now, the estimate \eqref{estim u,b H^1} yields:
\begin{eqnarray}\label{estim c_i}
\begin{aligned}
\sum_{i=1}^{I}|c_i|
&\leqslant C( 1+\norm{\bdd}_{\bL^3(\Omega)}+ \norm{\curl\bw}_{\bL^\frac{3}{2}(\Omega)})(\norm{\ff}_{\bL^p(\Omega)} + \norm{\bg}_{\bL^p(\Omega)} + \norm{P_0}_{W^{1-\frac{1}{p},p}(\Gamma)})\end{aligned}
\end{eqnarray}
%%%%%%%%%%%%%%%%% fin estimation des constantes c_i %%%%%%%%%%%%%%%%%%%%%
Using the embeddings $$\bW^{\,1,3/2}(\O)\hookrightarrow \bL^{3}(\O),\,\,\, \bL^p(\Omega) \hookrightarrow [\bH_0^{6,2}(\curl, \Omega)]',\,\,\,  W^{1-\frac{1}{p},p}(\Gamma) \hookrightarrow H^{-\frac{1}{2}}(\Gamma),$$
choosing $\epsilon$ and $\epsilon '$ such that $$\epsilon 'C_{_{SE}}C_{\epsilon}(\norm{\curl\bw}_{\bL^{3/2}(\O)}+\norm{\bdd}_{\bW^{1,\frac{3}{2}}(\O)})< \frac{1}{2},$$
we deduce from \eqref{estim y_2},\eqref{final estim y ^1 epsilon}, \eqref{estim curlb x d_1 case p<2}, \eqref{final estim curlb x d_1}-\eqref{final estim d_1 nabla u}, \eqref{estim u z_2}-\eqref{estim c_i}, the weak estimate \eqref{estim u,b H^1} and the embedding $\bW^{1,\frac{3}{2}}(\Omega) \hookrightarrow \bL^3(\Omega)$ that the estimate \eqref{estim u,b,P strong p>6/5} holds in all cases. \\

\noindent \underline{\textbf{ Second step:}} \textbf{The case of} $\curl\bw\in\bL^{3/2}(\O)$ \textbf{and} $\bdd\in \bW^{\,1,3/2}_{\sigma}(\O)$.\\
Let $\bw_{\lambda} \in \boldsymbol{\mathcal{D}}(\Omega)$ and $\bdd_\lambda\in \boldsymbol{\mathcal{D}}(\O)$ such that 
$\curl\bw_\lambda \rightarrow \curl \bw\,\,\text{in} \,\, \bL^{3/2}(\Omega)$ and 
$\bdd_\lambda \rightarrow \bdd \,\, \text{in} \,\, \bW^{1,3/2}(\Omega)$.\\
Consequently, the following problem: 
\begin{equation*}\label{linearized MHD-pressure lambda}
	\left\lbrace
		\begin{split}
		   - \, \Delta \bu_{\lambda} + (\curl\bw_{\lambda})\times \bu_{\lambda} + \nabla \mathrm{P}_{\lambda} -  (\curl \bb_{\lambda}) \times \bdd_{\lambda} = \ff \quad &\mathrm{and}\quad  \vdiv \bu_{\lambda} = 0   \quad \mathrm{in} \,\, \Omega, \\
           \, \curl \curl \bb_{\lambda} -  \curl (\bu_{\lambda} \times \bdd_{\lambda}) = \bg  \quad & \mathrm{and}\quad \vdiv \bb_{\lambda} = 0 \quad \text{in} \,\, \Omega, \\
        \bu_{\lambda} \times \bn = 0 \quad &\mathrm{and}\quad \bb_{\lambda} \times \bn = 0\quad \mathrm{on}\,\,\Gamma,\\
        \mathrm{P}_{\lambda} = P_0 \quad \text{on} \, \,\Gamma_0\quad & \mathrm{and}\quad 
         \mathrm{P}_{\lambda} = P_0 + c_i \,\,\,\text{on} \,\, \Gamma_i,\\
         \langle\bu_{\lambda} \cdot \bn, 1\rangle_{\Gamma_i} = 0, \quad &\mathrm{and}\quad \langle\bb_{\lambda} \cdot \bn, 1\rangle_{\Gamma_i} = 0,\quad 1\leq i\leq I.
		\end{split}\right.
\end{equation*}
has a unique solution $(\bu_\lambda, \bb_\lambda, P_\lambda, \bc_\lambda)\in  \bW^{\,2,p}(\Omega) \times \bW^{\,2,p}(\Omega) \times W^{\,1,p}(\Omega)\times \R^{I}$ and satisfies:
\begin{eqnarray}\label{estim u_lambda,b_lambda,P_lambda strong p>6/5}
\begin{aligned}
&\norm{\bu_\lambda}_{\bW^{2,p}(\Omega)} + \norm{\bb_\lambda}_{\bW^{2,p}(\Omega)}+ \norm{P_\lambda}_{W^{1,p}(\Omega)}  \leq  C(1 + \norm{\curl\bw_\lambda}_{\bL^\frac{3}{2}(\Omega)} + \norm{\bdd_\lambda}_{\bW^{1,\frac{3}{2}}(\Omega)}) \\
&\times\big(\norm{\ff}_{\bL^p(\Omega)} + \norm{\bg}_{\bL^p(\Omega)} + \norm{P_0}_{W^{1-\frac{1}{p},p}(\Gamma)}\big) ,
\end{aligned}
\end{eqnarray}
where $C$ is independent of $\lambda$. Finally, these uniform bounds enable us to pass to the limit $\lambda\rightarrow 0$. As a consequence, $(\bu_\lambda, \bb_\lambda, P_\lambda, \bc_\lambda)$ converges to $(\bu, \bb, P, \bc)$ the solution of the linearized MHD problem \eqref{linearized MHD-pressure} and satisfies the estimate \eqref{estim u,b,P strong p>6/5}.
\end{proof}

%%%%%%%%%%%%%%%%%%%%%%%%%%% Fin preuve strong solution p > 6/5 %%%%%%%%%%%%%%%%%%%%%%%%%%%%%%%%%%%%%

\subsection{Weak solution in $\bW^{\,1,p}(\Omega)$ with $1<p < +\infty$} 
In this subsection, we study the regularity $\bW^{\,1,p}(\Omega)$ of the weak solution for the linearized MHD problem \eqref{linearized MHD-pressure}. We begin with the case $p>2$. The next theorem will be improved in Corollary \ref{Corollary improvement pressure} where we consider a data $P_0$ less regular. \medskip

In the following, we denote by $\langle\cdot,\cdot\rangle_{\O_{_{r,p}}}$ the duality product between $[\bH_0^{r,p}(\curl, \Omega)]'$ and $\bH_0^{r,p}(\curl, \Omega)$. 

\begin{theorem} (\textbf{Generalized solution in} $\bW^{\,1,p}(\O)$ with $p>2$)\label{Solutions faibles p grand}. Suppose that $\O$ is of class $\mathcal{C}^{1,1}$ and $p > 2$. Assume that $h=0$, and let $\ff, \bg \in [\bH_0^{r',p'}(\curl, \Omega)]'$ and $P_0 \in W^{1-\frac{1}{r}, r}(\Gamma)$ with the compatibility condition 
\begin{eqnarray} 
 \forall \, \bv \in \bK_{N}^{p'}(\Omega),\quad\langle \bg, \bv \rangle_{\O_{_{r',p'}}}= 0,\label{Condition compatibilite K_N Lp}  \\
 \vdiv \bg = 0 \quad \mathrm{in}  \, \Omega.\label{condition div g=0 sol W^1,p linear MHD}
 \end{eqnarray} and 
\begin{eqnarray}\label{hypothesis on wurl w and d}
 \curl\bw\in\bL^{s}(\O),\,\,\, \bdd\in \bW^{1,s}_{\sigma}(\O),
\end{eqnarray}
with
\begin{equation}\label{def of s for weak p>2}
s= \frac{3}{2} \quad \mathrm{if} \,\,2<p < 3,\quad s>\frac{3}{2}\quad \mathrm{if}\,\,p=3\quad \mathrm{and}\quad s=r\quad \mathrm{if}\,\,p>3.
\end{equation}
\begin{equation}\label{def r}
 r\geq1 \qquad \mathrm{such \,\,that}\qquad \frac{1}{r}=\frac{1}{p}+\frac{1}{3}.
\end{equation}
% \begin{equation*}\label{def of s' for weak p>2}
% s'= 3 \quad \mathrm{if} \,\,p < 3,\quad s'> 3\quad \mathrm{if}\,\,p=3\quad \mathrm{and}\quad s'=p\quad \mathrm{if}\,\,p> 3. 
%\end{equation*}
Then the linearized MHD problem \eqref{linearized MHD-pressure} has a unique solution $(\bu, \bb, P, \bc) \in \bW^{1,p}(\Omega) \times \bW^{1,p}(\Omega) \times W^{1,r}(\Omega) \times \R^I$. Moreover, we have the following estimate:  
\small
\begin{eqnarray}\label{estim u,b W^1,p for p>2}
\begin{aligned}
&\norm{\bu}_{\bW^{1,\,p}(\Omega)} + \norm{\bb}_{\bW^{1,\,p}(\Omega)}+\norm{P}_{W^{1,r}(\Omega)} \leqslant C \big(1+\norm{\curl\bw}_{\bL^{s}(\O)}+\norm{\bdd}_{\bW^{1,s}(\O)}\big)\\&\times\big(\norm{\ff}_{[\bH_0^{r',p'}(\curl, \Omega)]'} + \norm{\bg}_{[\bH_0^{r',p'}(\curl, \Omega)]'} + \norm{P_{0}}_{W^{\,1-1/r,r}(\Gamma)}\big).
\end{aligned}
\end{eqnarray}
\end{theorem}

\begin{proof}
\textbf{A) Existence:}
Applying Proposition \ref{thm solution W1,p and W2,p Stokes divu=h}, there exists a unique solution $(\bu_1, P_1, \boldsymbol{\alpha}^{(1)}) \in \bW^{1,p}(\Omega) \times W^{1,r}(\Omega) \times \R^I$ solution of the following problem: 
\begin{eqnarray*}
\begin{cases}
- \Delta \bu_1 + \nabla P_1 = \ff \quad \mathrm{and}\quad
\vdiv \bu_1 = 0 \quad \mathrm{in} \,\, \Omega, \\
\bu_1 \times \bn = \textbf{0} \quad \mathrm{on} \,\, \Gamma, \\
P_1 = P_0 \quad \mathrm{on} \,\, \Gamma_0, \, \, P_1 = P_0 + \alpha_i^{(1)} \quad \mathrm{on} \,\, \Gamma_i, \\
\langle \bu_1 \cdot \bn, 1 \rangle_{\Gamma_i} = 0, \, \, \forall 1 \leqslant i \leqslant I
\end{cases}
\end{eqnarray*}
where $\alpha_i^{(1)} = \langle \ff, \nabla q_i^N \rangle_{\O_{_{r',p'}}}- \displaystyle\int_{\Gamma} P_0 \nabla q_i^N \cdot \bn\,d\sigma $ and satisfying the estimate: 
\begin{eqnarray}\label{estim u_1,p_1}
\norm{\bu_1}_{\bW^{1,p}(\Omega)} + \norm{P_1}_{W^{1,r}(\Omega)} \leqslant C \Big( \norm{\ff}_{[\bH_0^{r',p'}(\curl, \Omega)]'} + \norm{P_0}_{W^{1-\frac{1}{r},r}(\Gamma)} \Big).
\end{eqnarray}
Next, since $\bg$ satisfies the compatibility conditions \eqref{Condition compatibilite K_N Lp}-\eqref{condition div g=0 sol W^1,p linear MHD}, due to Lemma \ref{Lemme elliptique H r' p'}, the following problem: 
\begin{eqnarray*}
\begin{cases}
- \Delta \bb_1 = \bg \quad \mathrm{and}\quad
\vdiv \bb_1 = 0 \quad \mathrm{in} \,\, \Omega, \\
\bb_1 \times \bn = \textbf{0} \quad \mathrm{on} \,\, \Gamma,\\
\langle \bb_1 \cdot \bn, 1 \rangle_{\Gamma_i} = 0 \quad \forall 1 \leqslant i \leqslant I
\end{cases}
\end{eqnarray*}
has a unique solution $\bb_1 \in \bW^{1,p}(\Omega)$ satisfying the estimate:
\begin{eqnarray}\label{estim b_1}
\norm{\bb_1}_{\bW^{1,p}(\Omega)} \leqslant C \norm{\bg}_{[\bH_0^{r',p'}(\curl, \Omega)]'}
\end{eqnarray}
Then, since $\curl \bw \in \bL^{s}(\Omega)$ and $\bu_1 \in \bW^{1,p}(\Omega) $, we have $(\mathbf{curl}\,\textbf{\textit{w}})\times\textbf{\textit{u}}_1\in\textbf{\textit{L}}^r(\Omega)$. Indeed, if $p<3$, then $ \textbf{\textit{W}}^{\,1,p}(\Omega)\hookrightarrow \textbf{\textit{L}}^{p*}(\Omega)$ with $\frac{1}{p*}=\frac{1}{p}-\frac{1}{3}$ and $\frac{1}{s}+\frac{1}{p*}=\frac{1}{r}$. If $p=3$, then  there exists $\varepsilon>0$ such that $\dfrac{1}{\frac{3}{2}+\varepsilon}+\frac{1}{p*}=\frac{2}{3}$. Finally, if $p>3$, then $p*=\infty$ and $r=s$. Next, since $s \geqslant \frac{3}{2}$, then $\bW^{1,s}(\Omega) \hookrightarrow \bL^3(\Omega)$ so $\bdd\in\bL^{3}(\O)$, and by the definition of $r$ in \eqref{def r}, we have  $(\curl \bb_1) \times \bdd \in \bL^r(\Omega)$. Then $\ff_{1}=-(\curl \bw) \times \bu_1+(\curl \bb_1) \times \bdd $ belongs to $\bL^r(\Omega)$. 
Furthermore, we set $\bg_1= \curl (\bu_1 \times \bdd) = (\bdd \cdot \nabla) \bu_1 - (\bu_1 \cdot \nabla) \bdd $. By the same way, using the definitions of $s$ and $r$, we can check that $\bg_1\in \bL^r(\Omega)$. Moreover, $\bg_1$  satisfies the compatibility conditions \eqref{Condition compatibilite K_N Lp}-\eqref{condition div g=0 sol W^1,p linear MHD}. Observe that with the values of $s$ given in \eqref{def of s for weak p>2} for $p>2$, we have  $r\in[\frac{6}{5},\,3)$ and satisfies
\begin{equation}\label{def of s for weak r>6/5}
s= \frac{3}{2} \quad \mathrm{if} \,\,\frac{6}{5}<r < \frac{3}{2},\quad s>\frac{3}{2}\quad \mathrm{if}\,\,r=\frac{3}{2}\quad \mathrm{and}\quad s=r\quad \mathrm{if}\,\,r>\frac{3}{2}.
\end{equation}
So, $s\geq \frac{3}{2}$ and then $\curl\bw$ is at least in $\bL^\frac{3}{2}(\Omega)$ and $\bdd$ is at least in $\bW^{1,\frac{3}{2}}(\Omega)$. We deduce from Theorem \ref{thm strong solution p>6/5} that the following problem:
\begin{equation}\label{linearized MHD u_2,b_2,P_2}
	\left\lbrace
		\begin{split}
		   - \, \Delta \bu_{2} + (\curl\bw)\times \bu_{2} + \nabla \mathrm{P}_{2} -  (\curl \bb_{2}) \times \bdd = \ff_{1} \quad &\mathrm{and}\quad  \vdiv \bu_{2} = 0   \quad \mathrm{in} \,\, \Omega, \\
           \, \curl \curl \bb_{2} -  \curl (\bu_{2} \times \bdd) = \bg_{1}  \quad & \mathrm{and}\quad \vdiv \bb_{2} = 0 \quad \mathrm{in} \,\, \Omega, \\
        \bu_{2} \times \bn = \textbf{0} \quad &\mathrm{and}\quad \bb_{2} \times \bn = \textbf{0}\quad \mathrm{on}\,\,\Gamma,\\
        \mathrm{P}_{2} = 0 \quad \mathrm{on} \, \,\Gamma_0\quad & \mathrm{and}\quad 
         \mathrm{P}_{2} = \alpha_i^{(2)} \,\,\,\mathrm{on} \,\, \Gamma_i,\\
         \langle\bu_{2} \cdot \bn, 1\rangle_{\Gamma_i} = 0, \quad &\mathrm{and}\quad \langle\bb_{2} \cdot \bn, 1\rangle_{\Gamma_i} = 0,\quad 1\leq i\leq I.
		\end{split}\right.
\end{equation}
has a unique solution $(\bu_2, \bb_2, P_2, \boldsymbol{\alpha}^{(2)}) \in \bW^{2,r}(\Omega) \times \bW^{2,r}(\Omega) \times W^{1,r}(\Omega) \times \R^I$ satisfying the estimate:
\small
\begin{eqnarray}\label{estim u2,b2,P2 strong r>6/5}
\begin{aligned}
&\norm{\bu_2}_{\bW^{2,r}(\Omega)} + \norm{\bb_2}_{\bW^{2,r}(\Omega)}+ \norm{P_2}_{W^{1,r}(\Omega)}  \\
&\leqslant C(1 + \norm{\curl\bw}_{\bL^{\frac{3}{2}}(\Omega)} + \norm{\bdd}_{\bW^{1,\frac{3}{2}}(\Omega)})(\norm{\ff_{1}}_{\bL^r(\Omega)} + \norm{\bg_{1}}_{\bL^r(\Omega)} ) 
\end{aligned}
\end{eqnarray}
with 
\small
\begin{eqnarray}\label{def alpha_i^2}
\begin{aligned}
\alpha_i^{(2)}& =  \langle (\curl (\bb_1+\bb_2)) \times \bdd, \nabla q_i^N \rangle _{\O_{_{r',p'}}}
- \langle (\curl \bw) \times (\bu_1+\bu_2), \nabla q_i^N \rangle_{\O_{_{r',p'}}} .
\end{aligned}
\end{eqnarray}
Finally, using the embedding $\bW^{2,r}(\Omega) \hookrightarrow \bW^{1,p}(\Omega)$, the solution of  the linearized MHD problem \eqref{linearized MHD-pressure} is given by  $(\bu_1 + \bu_2, \bb_1 + \bb_2, P_1 + P_2, \boldsymbol{\alpha}^{(1)} + \boldsymbol{\alpha}^{(2)})\in\bW^{1,p}(\Omega) \times \bW^{1,p}(\Omega) \times W^{1,r}(\Omega) \times \R^I$.\\

In particular, the constants $c_i=\alpha_i^{(1)}+\alpha_i^{(2)}$ are given by 
\begin{equation}\label{constantes c_i cas W^1,p avec h=0}
 c_i= \langle \ff, \nabla q_i^N \rangle_{\O_{_{r',p'}}}- \displaystyle\int_{\Gamma} P_0 \nabla q_i^N \cdot \bn\,d\sigma+ \langle (\curl \bb) \times \bdd, \nabla q_i^N \rangle _{\O_{_{r',p'}}}
- \langle (\curl \bw) \times \bu, \nabla q_i^N \rangle_{\O_{_{r',p'}}} .
\end{equation}

\textbf{B) Estimates:}
The terms on $\ff_1$ and $\bg_1$ in \eqref{estim u2,b2,P2 strong r>6/5} can be controlled as:
\begin{eqnarray}\label{estim f_1}
 \norm{\ff_1}_{\bL^{r}(\O)}\leq  C\big(\norm{\curl\bw}_{\bL^{s}(\O)}\norm{\bu_1}_{\bW^{1,p}(\O)} +\norm{\bdd}_{\bL^{3}(\O)}\norm{\bb_1}_{\bW^{1,p}(\O)} \big).
\end{eqnarray} 
\begin{eqnarray}\label{estim g_1}
 \norm{\bg_1}_{\bL^{r}(\O)}\leq  C\big(\norm{\bdd}_{\bL^{3}(\O)}\norm{\bu_1}_{\bW^{1,p}(\O)} +\norm{\nabla\bdd}_{\bL^{s}(\O)}\norm{\bu_1}_{\bW^{1,p}(\O)} \big).
\end{eqnarray} 

 Then, using the above estimates and the embeddings $\bW^{1,s}(\Omega) \hookrightarrow \bW^{1,\frac{3}{2}}(\Omega) \hookrightarrow \bL^3(\Omega)$ for $s\geq 3/2$, the estimate \eqref{estim u2,b2,P2 strong r>6/5} becomes 
 
\begin{eqnarray}\label{estim u2,b2,P2 strong r>6/5 final}
\begin{aligned}
&\norm{\bu_2}_{\bW^{2,r}(\Omega)} + \norm{\bb_2}_{\bW^{2,r}(\Omega)}+ \norm{P_2}_{W^{1,r}(\Omega)}\leqslant C \big(1+\norm{\curl\bw}_{\bL^{s}(\O)}+\norm{\bdd}_{\bW^{1,s}(\O)}\big)\\&\times\big(\norm{\curl\bw}_{\bL^{s}(\O)}+\norm{\bdd}_{\bW^{1,s}(\O)}\big) (\norm{\bu_1}_{\bW^{1,p}(\Omega)} + \norm{\bb_1}_{\bW^{1,p}(\Omega)}).
\end{aligned}
\end{eqnarray} 
Thanks to \eqref{estim u_1,p_1}, \eqref{estim b_1} and \eqref{estim u2,b2,P2 strong r>6/5 final}, the solution  $(\bu,\bb,P)$ satisfies 
\begin{eqnarray}\label{estim u,b W^1,p for p>2 not optimal}
\begin{aligned}
&\norm{\bu}_{\bW^{1,\,p}(\Omega)} + \norm{\bb}_{\bW^{1,\,p}(\Omega)}+\norm{P}_{W^{1,r}(\Omega)} \leqslant C \big(1+\norm{\curl\bw}_{\bL^{s}(\O)}+\norm{\bdd}_{\bW^{1,s}(\O)}\big)^{2}\\&\times\big(\norm{\ff}_{[\bH_0^{r',p'}(\curl, \Omega)]'} + \norm{\bg}_{[\bH_0^{r',p'}(\curl, \Omega)]'} + \norm{P_{0}}_{W^{\,1-1/r,r}(\Gamma)}\big).
\end{aligned}
\end{eqnarray}

This estimate is not optimal and can be improved. For this, we will consider $(\bu, \bb, P) \in \bW^{1,p}(\Omega) \times \bW^{1,p}(\Omega) \times W^{1,r}(\Omega)$ the solution of \eqref{linearized MHD-pressure} obtained in the existence part. \\\noindent 
Note that, due to the hypothesis on $\bdd$ and $\curl \bw$, the terms $(\curl \bw) \times \bu$, $(\curl \bb) \times \bdd$ and $\curl (\bu \times \bdd)$ belong to $\bL^r(\Omega)$ with $\frac{1}{r} = \frac{1}{p} + \frac{1}{3}$. Thus, according to the regularity of the Stokes problem $(\mathcal{S_N})$ (see Proposition \ref{thm solution W1,p and W2,p Stokes divu=h}) and the elliptic problem $(\mathcal{E}_N)$ (see Lemma \ref{Lemme elliptique H r' p'}), we have: 
\begin{eqnarray} \label{Estimation Stokes Elliptique preuve estim optimale W1p}
\begin{aligned}
&\norm{\bu}_{\bW^{1,p}(\Omega)} + \norm{\bb}_{\bW^{1,p}(\Omega)} + \norm{P}_{W^{1,r}(\Omega)} \\
&\leqslant C \Big( \norm{\ff}_{[\bH_0^{r',p'}(\curl, \Omega)]'} + \norm{\bg}_{[\bH_0^{r',p'}(\curl, \Omega)]'} + \norm{P_0}_{W^{1-\frac{1}{r},r}(\Gamma)}  + \norm{(\curl \bw) \times \bu}_{\bL^r(\Omega)} \\
&+ \norm{(\curl \bb) \times \bdd}_{\bL^r(\Omega)} + \norm{\curl (\bu \times \bdd)}_{\bL^r(\Omega)} \Big)
\end{aligned}
\end{eqnarray}

We proceed in a way similar to the proof of the Theorem \ref{thm strong solution p>6/5}: we bound the three last terms of \eqref{Estimation Stokes Elliptique preuve estim optimale W1p}, using the decomposition of $\curl \bw$, $\bdd$ and $\nabla \bdd$ given in \eqref{decomposition y}-\eqref{decomposition d} and \eqref{decomposition nabla d} respectively. \\

\noindent \textbf{(i) The term $\norm{(\curl \bw) \times \bu}_{\bL^r(\Omega)}$:}
Using the decomposition \eqref{decomposition y} for $\by = \curl \bw$, we obtain:
\begin{eqnarray} \label{Estimation y2epsi x u preuve estim W1p}
\norm{\by_2^\epsilon \times \bu}_{\bL^r(\Omega)} \leqslant \norm{\by_2^\epsilon}_{\bL^s(\Omega)} \norm{\bu}_{\bL^{p^*}(\Omega)} \leqslant C \epsilon \norm{\bu}_{\bW^{1,p}(\Omega)}
\end{eqnarray}
where $\bW^{1,p}(\Omega) \hookrightarrow \bL^{p^*}(\Omega)$ with $\frac{1}{p^*} = \frac{1}{p} - \frac{1}{3}$ is $p < 3$, for $p^* = \frac{3s}{2s-3}$ if $p=3$ and $p^* = \infty$ if $p > 3$. Next, for the term $\by_1^\epsilon \times \bu$, let us consider the first the case $p<3$. We have 
\begin{eqnarray*}
\norm{\by_1^\epsilon \times \bu}_{\bL^r(\Omega)} \leqslant \norm{\by_1^\epsilon}_{\bL^t(\Omega)} \norm{\bu}_{\bL^m(\Omega)} \leqslant \norm{\by}_{\bL^\frac{3}{2}(\Omega)} \norm{\rho_{\epsilon/2}}_{\bL^k(\Omega)} \norm{\bu}_{\bL^m(\Omega)} 
\end{eqnarray*}
with $\frac{1}{r} = \frac{1}{t} + \frac{1}{m}$ and $1 + \frac{1}{t} = \frac{2}{3} + \frac{1}{k}$. Choosing $6 < m < p^*$, the embedding $\bW^{1,p}(\Omega)\hookrightarrow \bL^{m}(\O)$ is compact. Following this choice, we have $ t \in ]\frac{3}{2},\,\frac{2p}{2+p}[$ and $k \in ]1,\frac{6p}{5p+6}[$. Then, for any $\varepsilon'>0$, there exists $C_{\varepsilon'} >0$ such that 
\begin{equation*}
 \norm{\bu}_{\bL^{m}(\O)}\leq \varepsilon' \norm{\bu}_{\bW^{1,p}(\O)}+ C_{\varepsilon'}\norm{\bu}_{\bL^{6}(\O)}.
\end{equation*}
So, we deduce 
\begin{eqnarray} \label{Estimation y1epsi x u preuve estim W1p}
\norm{\by_1^\epsilon \times \bu}_{\bL^r(\Omega)} \leqslant \varepsilon' C_{\varepsilon} \norm{\by}_{\bL^{3/2}(\O)} \norm{\bu}_{\bW^{1,p}(\O)}+ C_1C_{\varepsilon'} C_\epsilon \norm{\by}_{\bL^\frac{3}{2}(\Omega)} \norm{\bu}_{\bH^1(\Omega)} 
\end{eqnarray}
where $C_\epsilon$ is the constant which absorbes the norm of the mollifier and $C_1$ is the constant of the Sobolev embedding $\bH^{1}(\O)\hookrightarrow \bL^{6}(\O)$. If $p\geq 3$, we have
\begin{eqnarray}\label{Estimation BIS y1epsi x u preuve estim W1p}
\norm{\by_1^\epsilon \times \bu}_{\bL^r(\Omega)} \leqslant \norm{\by_1^\epsilon}_{\bL^s(\Omega)} \norm{\bu}_{\bL^m(\Omega)} \leqslant \norm{\by}_{\bL^s (\Omega)} \norm{\rho_{\epsilon/2}}_{\bL^1(\Omega)} \norm{\bu}_{\bL^m(\Omega)}, 
\end{eqnarray}
where we choose $m=\infty$ if $p>3$ and $m\in(1,\infty)$ if $p=3$. \\
\noindent \textbf{(ii) The term $\norm{(\curl \bb) \times \bdd}_{\bL^r(\Omega)}$}: Using the decomposition \eqref{decomposition d} for $\bdd$, we have: 

\begin{eqnarray} \label{Estimation curl b times bd2epsi preuve estim W1p}
\norm{(\curl \bb) \times \bdd_2^\epsilon}_{\bL^r(\Omega)} \leqslant \norm{\curl \bb}_{\bL^p(\Omega)} \norm{\bdd_2^\epsilon}_{\bL^3(\Omega)} \leqslant C \epsilon \norm{\bb}_{\bW^{1,p}(\Omega)}
\end{eqnarray}

Next, in order to bound the term $(\curl \bb) \times \bdd_1^\epsilon$, we have two cases:

\noindent ${\bullet}$  The case $2 < p \leq  6$:  we have 
\small
\begin{eqnarray} \label{Estimation curl b times d1epsi preuve estim W1p p<6}
\norm{(\curl \bb) \times \bdd_1^\epsilon}_{\bL^r(\Omega)} &\leqslant& \norm{\curl \bb}_{\bL^2(\Omega)} \norm{\bdd_1^\epsilon}_{\bL^t(\Omega)} 
\leqslant \norm{\curl \bb}_{\bL^2(\Omega)} \norm{\bdd}_{\bL^3(\Omega)} \norm{\rho_{\epsilon/2}}_{\bL^k(\R^3)} \nonumber \\
&\leqslant & C_\epsilon \norm{\bdd}_{\bL^3(\Omega)} \norm{\bb}_{\bH^1(\Omega)} 
\end{eqnarray}

with $\frac{1}{r} = \frac{1}{2} + \frac{1}{t}$ and $1 + \frac{1}{t} = \frac{1}{3} + \frac{1}{k}$, so we have to take $t = \frac{6p}{6-p}$ and $k = \frac{2p}{2 + p}$ which are well-defined. \\
\noindent ${\bullet}$  The case $p > 6$. We have:
\begin{eqnarray} \label{Estimation curl b times d1epsi preuve estim W1p p>6 num 1}
\norm{(\curl \bb) \times \bdd_1^\epsilon}_{\bL^r(\Omega)} \leqslant \norm{\curl \bb}_{\bL^q(\Omega)} \norm{\bdd_1^\epsilon}_{\bL^t(\Omega)} \leqslant \norm{\curl \bb}_{\bL^q(\Omega)} \norm{\bdd}_{\bL^3(\Omega)} \norm{\rho_{\epsilon/2}}_{\bL^k(\R^3)}
\end{eqnarray}
with $\frac{1}{r} = \frac{1}{q} + \frac{1}{t}$ and $1 + \frac{1}{t} = \frac{1}{k} + \frac{1}{3}$. We choose $3 < q < p$, and then we have that $t \in ]3,\,p[$ and $k \in ]1, \,\frac{3p}{2p+3}[$. The interpolation estimate of $\bW^{1,q}(\O)$ between $\bH^1(\Omega)$ and $\bW^{1,p}(\Omega)$ gives (cf. \cite{Niremberg}): 
\begin{eqnarray*}
\norm{\curl \bb}_{\bL^q(\Omega)} \leqslant \norm{\bb}_{\bW^{1,p}(\Omega)}^{\theta} \norm{\bb}_{\bH^1(\Omega)}^{1-\theta} ,
\end{eqnarray*}
with $\theta =\frac{p(q-2)}{q(p-2)}$. Applying the Young inequality, we obtain for small $\epsilon' > 0$:
\begin{eqnarray*}
\norm{\curl \bb}_{\bL^q(\Omega)} \leqslant \epsilon' \norm{\bb}_{\bW^{1,p}(\Omega)} + C_{\epsilon'} \norm{\bb}_{\bH^1(\Omega)} .
\end{eqnarray*}
Replacing this estimate in \eqref{Estimation curl b times d1epsi preuve estim W1p p>6 num 1}, we obtain:
\begin{eqnarray} \label{Estimation curl b times d1epsi preuve estim W1p p>6 num 2}
\norm{(\curl \bb) \times \bdd_1^\epsilon}_{\bL^r(\Omega)} \leqslant \epsilon' C_\epsilon \norm{\bdd}_{\bL^3(\Omega)} \norm{\bb}_{\bW^{1,p}(\Omega)} + C_{\epsilon'} C_\epsilon \norm{\bdd}_{\bL^3(\Omega)} \norm{\bb}_{\bH^1(\Omega)} .
\end{eqnarray}

\noindent \textbf{(iii) The term $\norm{\curl (\bu \times \bdd)}_{\bL^r(\Omega)}$:} Since $\vdiv \bu = 0$ and $\vdiv \bdd = 0$ in $\Omega$, thus we rewrite $\curl (\bu \times \bdd) = (\bdd \cdot \nabla) \bu - (\bu \cdot \nabla) \bdd$. \\
$\bullet$ \textbf{The term $\norm{(\bdd \cdot \nabla) \bu}_{\bL^r(\Omega)}$:} following the same proof as for the term $(\curl \bb) \times \bdd$ by replacing $\curl \bb$ with $\nabla \bu$, we obtain
\begin{eqnarray} \label{Estimation d2epsi nabla u preuve estim W1p}
\norm{(\bdd_2^\epsilon \cdot \nabla) \bu}_{\bL^r(\Omega)} \leqslant C \epsilon \norm{\bu}_{\bW^{1,p}(\Omega)} 
\end{eqnarray}\vspace{-.1cm}
and
\begin{eqnarray} \label{Estimation d1epsi nabla u preuve estim W1p p>6}
\norm{(\bdd_1^\epsilon \cdot \nabla) \bu}_{\bL^r(\Omega)} \leqslant C_\epsilon \norm{\bdd}_{\bL^3(\Omega)}\big( \varepsilon'\norm{\bu}_{\bW^{1,p}(\Omega)} + C_{\epsilon'} \norm{\bu}_{\bH^1(\Omega)}\big)
\end{eqnarray}
$\bullet$ \textbf{The term $\norm{(\bu \cdot \nabla) \bdd}_{\bL^r(\Omega)}$:}
In the same way, we remark that we can control this term as for $\norm{(\curl \bw) \times \bu}_{\bL^r(\Omega)}$ by replacing $\curl \bw$ with $\nabla \bdd$. Applying the decomposition \eqref{decomposition nabla d} for $\nabla \bdd$, we thus prove that: 
\begin{eqnarray} \label{Estimation u z2epsi preuve estim W1p}
\norm{\bz_2^\epsilon \cdot \bu}_{\bL^r(\Omega)} \leqslant C \epsilon \norm{\bu}_{\bW^{1,p}(\Omega)} 
\end{eqnarray}
and 
\begin{eqnarray} \label{Estimation z1epsi u preuve estim W1p p<3}
\norm{\bz_1^\epsilon \cdot \bu}_{\bL^r(\Omega)} \leqslant C C_\epsilon \norm{\nabla \bdd}_{\bL^s(\Omega)}\big( \varepsilon' \norm{\bu}_{\bW^{1,p}(\O)}+ C_{\varepsilon'}\norm{\bu}_{\bH^1(\Omega)}\big)
\end{eqnarray}
Finally, taking the estimates \eqref{Estimation y2epsi x u preuve estim W1p}-\eqref{Estimation z1epsi u preuve estim W1p p<3} together with the embedding $\bW^{1,s}(\Omega)\hookrightarrow\bL^3(\Omega)$ for $s\geq 3/2$ and \eqref{Estimation Stokes Elliptique preuve estim optimale W1p}, then  choosing $\epsilon, \epsilon' > 0$ small enough and using the estimate \eqref{estim u,b H^1}, we thus obtain \eqref{estim u,b W^1,p for p>2}.
\end{proof}
\begin{rmk}
${}$\\
\textbf{\emph{(i)}} The case $p>2$ can be analyzed in a similar way to the case $p=2$ to prove that the space $[\bH_0^{r',p'}(\curl,\Omega)]'$ with $\frac{1}{r}=\frac{1}{p}+\frac{1}{3}$ is optimal to obtain the regularity $\bW^{1,p}(\O)$.\\
 \textbf{\emph{(ii)}} Why do we take $P_0\in W^{1-1/r,r}(\Gamma)$ instead of $W^{-1/p,p}(\Gamma)$? If we take $P_0\in W^{-1/p,p}(\Gamma)$, we obtain that $P\in L^p(\O)$ as in the classic case of Navier-Stokes equations with Dirichlet boundary conditions. But we are not able to solve the Stokes problem $(\mathcal{S_N})$ because, in this case, $\ff = \curl(\curl \bu)+\nabla P  \notin [\bH^{r',p'}(\curl, \O)]'$. 
\end{rmk}
We also need to study the case where the divergence is not free for the velocity field. The following problem appears as the dual problem associated to the linearized MHD problem \eqref{linearized MHD-pressure} in the study of weak solutions for $p<2$: 

%%%%%%%%%%%%%%%%%%%%%%%%%%% Problem with chi %%%%%%%%%%%%%%%%%%%%%%%%%%%%%%%
\begin{equation}\label{MHD with chi  and div u not free}
	\left\lbrace
		\begin{split}
		   - \, \Delta \bu + (\curl \bw) \times \bu + \nabla \mathrm{P} -  (\curl \bb) \times \bdd = \ff \quad &\mathrm{and}\quad  \vdiv \bu = h   \quad \mathrm{in} \,\, \Omega, \\
           \, \curl \curl \bb -  \curl (\bu \times \bdd) +\nabla \chi= \bg  \quad & \mathrm{and}\quad \vdiv \bb = 0 \quad \text{in} \,\, \Omega, \\
       \bu \times \bn = \textbf{0}, \quad  \bb \times \bn = \textbf{0} \quad &\mathrm{and}\quad \chi= 0\quad\quad \,\,\, \mathrm{on}\,\,\Gamma,\\
        \mathrm{P} = P_0 \quad \text{on} \, \,\Gamma_0\quad & \mathrm{and}\quad 
         \mathrm{P} = P_0 + c_i \,\,\,\text{on} \,\, \Gamma_i,\\
         \langle\bu \cdot \bn, 1\rangle_{\Gamma_i} = 0, \quad &\mathrm{and}\quad \langle\bb \cdot \bn, 1\rangle_{\Gamma_i} = 0.
		\end{split}\right.
\end{equation}
Observe that the second equation in \eqref{linearized MHD-pressure} is replaced by $ \, \curl \curl \bb - \curl (\bu \times \bdd) +\nabla \chi= \bg\,\,\,\mathrm{in}\,\,\Omega$ with $\chi=0$ on $\Gamma$. The scalar $\chi$ represents the Lagrange multiplier associated with magnetic divergence constraint. Note that, taking the divergence in the above equation, $\chi$ is a solution of the following Dirichlet problem:
\begin{eqnarray}\label{Dirichlet prob for chi}
\Delta \chi = \vdiv \bg \quad \mathrm{in} \,\, \Omega\quad\mathrm{and}\quad  \chi = 0 \quad\mathrm{on} \, \Gamma.
\end{eqnarray}  In particular, if $\bg$ is divergence-free, we have $\chi=0$. Nevertheless, the introduction of $\chi$ will be useful to enforce zero divergence condition over the magnetic field.
First, we give the following result for the case  $h=0$. 
\begin{corollary} \label{Corollaire double stokes}
Suppose that $p \geqslant 2$ and $h=0$. Let $\ff, \bg \in [\bH_0^{r',p'}(\curl, \Omega)]'$, $P_0 \in W^{1-\frac{1}{r}, r}(\Gamma)$ with the compatibility condition \eqref{Condition compatibilite K_N Lp} and $\bw$, $\bdd$ defined with \eqref{hypothesis on wurl w and d}-\eqref{def r}. Then the problem \eqref{MHD with chi  and div u not free} has a unique solution $(\bu, \bb, P, \chi, \bc) \in \bW^{1,p}(\Omega) \times \bW^{1,p}(\Omega) \times W^{1,r}(\Omega) \times W^{1,r}(\Omega) \times \R^I$ where $\bc=(c_1,\ldots,c_I)$ is given by \eqref{constantes c_i cas W^1,p avec h=0}. Moreover, we have the following estimates:  
\small
\begin{eqnarray}\label{estim weak p>2 chi and div u not free}
\begin{aligned}
&\norm{\bu}_{\bW^{1,p}(\Omega)} + \norm{\bb}_{\bW^{1,p}(\Omega)} + \norm{P}_{W^{1,r}(\Omega)}  \leqslant C \big(1+\norm{\curl\bw}_{\bL^{s}(\O)}+\norm{\bdd}_{\bW^{1,s}(\O)}\big)\\&\times\big(\norm{\ff}_{[\bH_0^{r',p'}(\curl, \Omega)]'} + \norm{\bg}_{[\bH_0^{r',p'}(\curl, \Omega)]'} + \norm{P_{0}}_{W^{\,1-1/r,r}(\Gamma)}\big).
\end{aligned}
\end{eqnarray}
\begin{equation}\label{estim chi}
 \norm{\chi}_{W^{1,r}(\O)}\leq C \norm{\bg}_{[\bH_0^{r',p'}(\curl, \Omega)]'}
\end{equation}

\end{corollary}

 \begin{proof}
% {\color{red} J'ai mis dans la Proposition \eqref{proposition to improv estim W^1,p for p>2 with phi}  l'idée de relever la condition $\div\bu=\phi_u$ en passant par le problème de Dirichlet \eqref{Dirichlet theta} vérifié par $\theta$, dans cette preuve j'ai mis celle en passant par des problèmes de Stokes, les estimations sont obtenus rapidement et ça montre en partuclier la nécessité de la condition de compatibilité \eqref{Condition compatibilite K_N} pour $\bg$.}\\
As mentioned before, the scalar $\chi$ can be found directly as a solution of the Dirchlet problem \eqref{Dirichlet prob for chi}. Since $\bg \in [\bH_0^{r',p'}(\curl, \Omega)]'$, $\vdiv \bg \in W^{-1,r}(\Omega)$ and then $\chi$ belongs to $W^{1,r}(\Omega)$ and satisfies the estimate \eqref{estim chi}. We set $\bg' = \bg - \nabla \chi$. It is clear that $\bg'$ is an element of the dual space $[\bH_0^{r',p'}(\curl,\Omega)]'$. Moreover, it is clear that $\vdiv\,\bg'=0$ in $\O$ and $\bg'$ satisfies the compatibility conditions \eqref{Condition compatibilite K_N Lp}.  So, problem \eqref{MHD with chi  and div u not free} becomes 
\begin{equation}\label{Probleme MHD g'}
	\left\lbrace
		\begin{split}
		   - \, \Delta \bu + (\curl \bw) \times \bu + \nabla \mathrm{P} -  (\curl \bb) \times \bdd = \ff \quad &\mathrm{and}\quad  \vdiv \bu = 0   \quad \mathrm{in} \,\, \Omega, \\
           \, \curl \curl \bb -  \curl (\bu \times \bdd)= \bg'  \quad & \mathrm{and}\quad \vdiv \bb = 0 \quad \text{in} \,\, \Omega, \\
       \bu \times \bn = \textbf{0} \quad &\mathrm{and}\quad \bb \times \bn = \textbf{0}\quad \mathrm{on}\,\,\Gamma,\\
        \mathrm{P} = P_0 \quad \text{on} \, \,\Gamma_0\quad & \mathrm{and}\quad 
         \mathrm{P} = P_0 + c_i \,\,\,\text{on} \,\, \Gamma_i,\\
         \langle\bu \cdot \bn, 1\rangle_{\Gamma_i} = 0, \quad &\mathrm{and}\quad \langle\bb \cdot \bn, 1\rangle_{\Gamma_i} = 0.
		\end{split}\right.
\end{equation}
Thanks to Theorem \ref{Solutions faibles p grand}, problem \eqref{Probleme MHD g'} 
has a unique solution $(\bu, \bb, P, \bc)$ in $\bW^{1,p}(\Omega) \times \bW^{1,p}(\Omega) \times W^{1,r}(\Omega) \times \R^I$ satisfying the estimate:
\begin{eqnarray}\label{estim u,b W^1,p for p>2 MHD g'}
\begin{aligned}
&\norm{\bu}_{\bW^{1,\,p}(\Omega)} + \norm{\bb}_{\bW^{1,\,p}(\Omega)}+\norm{P}_{W^{1,r}(\Omega)} \leqslant C \big(1+\norm{\curl\bw}_{\bL^{s}(\O)}+\norm{\bdd}_{\bW^{1,s}(\O)}\big)\\&\times\big(\norm{\ff}_{[\bH_0^{r',p'}(\curl, \Omega)]'} + \norm{\bg'}_{[\bH_0^{r',p'}(\curl, \Omega)]'} + \norm{P_{0}}_{W^{\,1-1/r,r}(\Gamma)}\big).
\end{aligned}
\end{eqnarray}
Using \eqref{estim chi}, the previous estimate still holds when $\bg'$ is replaced by $\bg$. 
\end{proof}
The next theorem gives a generalization for the case $h\neq 0$.

\begin{theorem}\label{thm weak sol W^1,p p>2 with phi_u and chi}
Suppose that $p \geqslant 2$. Let $\ff, \bg \in [\bH_0^{r',p'}(\curl, \Omega)]'$, $P_0 \in W^{1-\frac{1}{r}, r}(\Gamma)$ and $ h\in W^{1,r}(\O)$ with the compatibility condition \eqref{Condition compatibilite K_N Lp} and $\bw$, $\bdd$ defined with \eqref{hypothesis on wurl w and d}-\eqref{def r}. Then the problem \eqref{MHD with chi  and div u not free} has a unique solution $(\bu, \bb, P, \chi, \bc) \in \bW^{1,p}(\Omega) \times \bW^{1,p}(\Omega) \times W^{1,r}(\Omega) \times W^{1,r}(\Omega) \times \R^I$.
Moreover, we have the following estimate for $(\bu,\bb,P)$:  
\small
\begin{eqnarray}\label{estim1 weak p>2 chi and div u not free}
\begin{aligned}
&\norm{\bu}_{\bW^{1,p}(\Omega)} + \norm{\bb}_{\bW^{1,p}(\Omega)} + \norm{P}_{W^{1,r}(\Omega)}  \leqslant C \big(1+\norm{\curl\bw}_{\bL^{s}(\O)}+\norm{\bdd}_{\bW^{1,s}(\O)}\big)^{2}\\&\times\big(\norm{\ff}_{[\bH_0^{r',p'}(\curl, \Omega)]'} + \norm{\bg}_{[\bH_0^{r',p'}(\curl, \Omega)]'} + \norm{h}_{W^{1,r}(\O)}+\norm{P_{0}}_{W^{\,1-1/r,r}(\Gamma)}\big).
\end{aligned}
\end{eqnarray}
% \begin{equation}\label{estim chi}
%  \norm{\chi}_{W^{1,r}(\O)}\leq C \norm{\bg}_{[\bH_0^{r',p'}(\curl, \Omega)]'}
% \end{equation}
\end{theorem}

\begin{proof}
The idea is to lift the data $h$ by using the Stokes problem:
\begin{eqnarray*}
\begin{cases}
- \Delta \bu_1 + \nabla P_1 = \ff \quad \mathrm{and}\quad
\vdiv \bu_1 = h \quad \mathrm{in} \,\, \Omega, \\
\bu_1 \times \bn = 0 \quad \mathrm{on} \,\, \Gamma, \\
P_1 = P_0 \quad \mathrm{on} \,\, \Gamma_0\quad \mathrm{and}\quad  P_1 = P_0 + \alpha_i^{(1)} \quad \mathrm{on} \,\, \Gamma_i, \\
\langle \bu_1 \cdot \bn, 1 \rangle_{\Gamma_i} = 0, \, \, \forall 1 \leqslant i \leqslant I
\end{cases}
\end{eqnarray*}
Thanks to Proposition \ref{thm solution W1,p and W2,p Stokes divu=h}, there exists a unique solution $(\bu_1, P_1, \boldsymbol{\alpha}^{(1)}) \in \bW^{1,p}(\Omega) \times W^{1,r}(\Omega) \times \R^I$ satisfying the estimate: 
\begin{eqnarray}\label{estim u_1,p_1 with phi_u}
\norm{\bu_1}_{\bW^{1,p}(\Omega)} + \norm{P_1}_{W^{1,r}(\Omega)} \leqslant C \Big( \norm{\ff}_{[\bH_0^{r',p'}(\curl, \Omega)]'} + \norm{h}_{W^{1,r}(\O)}+\norm{P_0}_{W^{1-\frac{1}{r},r}(\Gamma)} \Big).
\end{eqnarray}

where $\alpha_i^{(1)} = \langle \ff, \nabla q_i^N \rangle_\Omega +\displaystyle\int_{\Gamma}(h- P_0)\, \nabla q_i^N \cdot \bn \,d\sigma$.\\

Next, since $\bg$ satisfies the compatibility condition \eqref{Condition compatibilite K_N Lp}, due to \cite[Theorem 5.2]{AS_M3AS}, the following problem: 
\begin{eqnarray*}
\begin{cases}
- \Delta \bb_1+\nabla \chi = \bg \quad \mathrm{and}\quad
\vdiv \bb_1 = 0 \quad \mathrm{in} \,\, \Omega, \\
\bb_1 \times \bn = 0 \quad \mathrm{and}\quad \chi=0 \quad  \mathrm{on} \,\, \Gamma,\\
\langle \bb_1 \cdot \bn, 1 \rangle_{\Gamma_i} = 0 \quad \forall 1 \leqslant i \leqslant I
\end{cases}
\end{eqnarray*}
has a unique solution $(\bb_1 ,\chi)\in \bW^{1,p}(\Omega)\times W^{1,r}(\O)$ satisfying the estimate:
\begin{eqnarray}\label{estim b_1 and chi}
\norm{\bb_1}_{\bW^{1,p}(\Omega)}+\norm{\chi}_{W^{1,r}(\O)} \leqslant C \norm{\bg}_{[\bH_0^{r',p'}(\curl, \Omega)]'}
\end{eqnarray}
Finally, we consider $(\bu_2,\,\bb_2,\,P_2,\,\boldsymbol{\alpha}^{(2)})\in\bW^{2,r}(\Omega) \times \bW^{2,r}(\Omega) \times W^{1,r}(\Omega) \times \R^I$ the solution of \eqref{linearized MHD u_2,b_2,P_2} satisfying \eqref{def alpha_i^2} and \eqref{estim u2,b2,P2 strong r>6/5 final}. Therefore, $(\bu_1+\bu_2,\,\bb_1+\bb_2,P_1+P_2,\chi, \boldsymbol{\alpha}^{(1)}+\boldsymbol{\alpha}^{(2)})$ is the solution of \eqref{MHD with chi  and div u not free}. Estimate \eqref{estim1 weak p>2 chi and div u not free} follows from \eqref{estim u_1,p_1 with phi_u},\eqref{estim b_1 and chi} and \eqref{estim u2,b2,P2 strong r>6/5 final}. 
\end{proof}
 Note that the estimate \eqref{estim1 weak p>2 chi and div u not free} is not optimal and will be improved in the next result. 
 
 \begin{proposition}\label{proposition to improv estim W^1,p for p>2 with phi}
Under the assumptions of Theorem \ref{thm weak sol W^1,p p>2 with phi_u and chi}, the problem \eqref{MHD with chi  and div u not free} has a unique solution $(\bu, \bb, P, \chi, \bc) \in \bW^{1,p}(\Omega) \times \bW^{1,p}(\Omega) \times W^{1,r}(\Omega) \times W^{1,r}(\Omega) \times \R^I$ satisfying \eqref{estim chi} and the following estimate:  
\small
\begin{eqnarray}\label{estim weak p>2 chi and div u not free}
&{}&\norm{\bu}_{\bW^{1,p}(\Omega)} + \norm{\bb}_{\bW^{1,p}(\Omega)} + \norm{P}_{W^{1,r}(\Omega)} \nonumber\\&\leqslant&\!\! \!\!C\!\big(1+\norm{\curl\bw}_{\bL^{s}(\O)}+\norm{\bdd}_{\bW^{1,s}(\O)}\big)\Big(\norm{\ff}_{[\bH_0^{r',p'}(\curl, \Omega)]'} + \norm{\bg}_{[\bH_0^{r',p'}(\curl, \Omega)]'} \nonumber\\&+&\norm{P_{0}}_{W^{\,1-1/r,r}(\Gamma)} + \norm{h}_{W^{1,r}(\O)}\big(1+\norm{\curl\bw}_{\bL^{s}(\O)}+\norm{\bdd}_{\bW^{1,s}(\O)}\big)\Big).
\end{eqnarray}
 \end{proposition}
\begin{proof}
We can reduce the non vanishing divergence problem \eqref{MHD with chi  and div u not free} for the velocity to the case where $\vdiv \bu = 0$ in $\Omega$, by solving the following Dirichlet problem: 
\begin{eqnarray}\label{Dirichlet theta}
\Delta \theta = h \quad \text{in} \,\, \Omega \quad \mathrm{and}\quad 
\theta = 0 \quad \text{on} \, \,\Gamma
\end{eqnarray}
For $h \in W^{1,r}(\Omega)$, problem \eqref{Dirichlet theta} has a unique solution $\theta \in W^{3,r}(\Omega) \hookrightarrow W^{2,p}(\Omega)$ satisfying the following estimate: 
\begin{eqnarray}\label{estim theta}
\norm{\theta}_{W^{2,p}(\Omega)} \leqslant C \norm{h}_{W^{1,r}(\Omega)} 
\end{eqnarray}
Setting $\bz = \bu - \nabla \theta$, then \eqref{MHD with chi  and div u not free} becomes: Find $(\bz, \bb, P, \chi, \bc)$ solution of problem: 
\begin{eqnarray} \label{Probleme z b preuve cor}
\begin{cases}
- \Delta \bz + (\curl \bw) \times \bz + \nabla P -  (\curl \bb) \times \bdd = \ff + \nabla  h- (\curl \bw) \times \nabla \theta\quad \mathrm{in}\,\,\O\\
\curl \curl \bb -  \curl(\bz \times \bdd) + \nabla \chi = \bg + \curl (\nabla \theta \times \bdd)\quad \mathrm{in}\,\,\O \\
\vdiv \bz = 0, \quad \vdiv \bb = 0 \quad \text{in} \, \Omega \\
\bz \times \bn = \textbf{0},\quad  \bb \times \bn = \textbf{0}\quad \mathrm{and}\quad  \chi = 0 \quad \text{on} \, \Gamma \\
P = P_0 \quad \text{on} \, \Gamma_0, \quad P= P_0 + c_i \quad \text{on} \, \Gamma_i \\
\langle \bz \cdot \bn, 1 \rangle_{\Gamma_i} = \langle \bb \cdot \bn, 1 \rangle_{\Gamma_i} = 0, \quad \forall 1 \leqslant i \leqslant I,
\end{cases}
\end{eqnarray}
which is a problem treated in the proof of Corollary \ref{Corollaire double stokes}. Since $\nabla \theta\in \bW^{1,p}(\O)\hookrightarrow \bL^{p*}(\O)$, by using the definition of $s$ in \eqref{def of s}, we have  $(\curl \bw) \times \nabla \theta \in \bL^r(\Omega)$ with $\frac{1}{p^*}=\frac{1}{p}-\frac{1}{3}$ if $p<3$, $p^*=\frac{rs}{s-r}$ if $p=3$ and $p^*=\infty$ if $p>3$. So $\ff +\nabla h- (\curl \bw) \times \nabla \theta \in [\bH_0^{r',p'}(\curl, \Omega)]'$. Now, we consider the term $\curl (\nabla \theta \times \bdd)=\bdd\cdot\nabla\nabla \theta- \nabla\theta \cdot\nabla\bdd$ .   Since $\bdd\in\bW^{1,s}(\O) \hookrightarrow \bL^3(\O)$ ($s \geqslant 3/2$), then  $\bdd\cdot\nabla\nabla \theta$ belongs to $\bL^{r}(\O)$. Moreover, since $\nabla\bdd\in \bL^s(\O)$, using the same arguments for the term  $(\curl \bw) \times \nabla \theta$, we deduce that $\nabla\theta \cdot\nabla\bdd$ belongs to $\bL^r(\O)$. So, $\bg + \curl (\nabla \theta \times \bdd) \in [\bH_0^{r',p'}(\curl, \Omega)]'$ and satisfies \eqref{Condition compatibilite K_N Lp}. Thanks to Theorem \ref{Solutions faibles p grand}, there exists a unique solution $(\bz, \bb, P, \chi, \bc) \in \bW^{1,p}(\Omega) \times \bW^{1,p}(\Omega) \times W^{1,r}(\Omega) \times W^{1,r}(\Omega) \times \R^I$ satisfying \eqref{estim chi} and 
\small
\begin{eqnarray}\label{estim z,b W^1,p for p>2 MHD with chi}
\begin{aligned}
&\norm{\bz}_{\bW^{1,\,p}(\Omega)} + \norm{\bb}_{\bW^{1,\,p}(\Omega)}+\norm{P}_{W^{1,r}(\Omega)} \leqslant C \big(1+\norm{\curl\bw}_{\bL^{s}(\O)}+\norm{\bdd}_{\bW^{1,s}(\O)}\big)\\&\times\big(\norm{\ff}_{[\bH_0^{r',p'}(\curl, \Omega)]'} + \norm{\nabla h}_{\bL^{r}(\O)} +\norm{ (\curl \bw) \times \nabla \theta}_{\bL^{r}(\O)}+\norm{\bg}_{[\bH_0^{r',p'}(\curl, \Omega)]'}\\& + \norm{\curl(\nabla\theta\times \bdd)}_{\bL^{r}(\O)}+\norm{P_{0}}_{W^{\,1-1/r,r}(\Gamma)}\Big).
\end{aligned}
\end{eqnarray}
with \small
\begin{eqnarray*}
\begin{aligned}
&c_i = \langle \ff, \nabla q_i^N \rangle_{\O_{_{r',p'}}}+ \int_{\Gamma}(h -P_0) \nabla q^{N}_{i}\cdot\bn\,\,d\sigma - \langle (\curl \bw) \times \nabla \theta, \nabla q_i^N \rangle_{\O_{_{r',p'}}} \\&- \langle (\curl \bw) \times \bz, \nabla q_i^N \rangle_{\O_{_{r',p'}}} + \langle (\curl \bb) \times \bdd, \nabla q_i^N \rangle_{\O_{_{r',p'}}} 
\end{aligned}
\end{eqnarray*}
To bound the terms $\norm{ (\curl \bw) \times \nabla \theta}_{\bL^{r}(\O)}$ and $ \norm{\curl(\nabla\theta\times \bdd)}_{\bL^{r}(\O)}$ in \eqref{estim z,b W^1,p for p>2 MHD with chi}, we write by using \eqref{estim theta}
\begin{eqnarray}\label{estim curl w X nabla theta}
 \norm{ (\curl \bw) \times \nabla \theta}_{\bL^{r}(\O)}&\leq& \norm{ \curl \bw}_{\bL^s(\O)}\norm{ \nabla \theta}_{\bL^{p^{*}}(\O)}\leq  \norm{ \curl \bw}_{\bL^s(\O)}\norm{ \nabla \theta}_{\bW^{1,p}(\O)}\nonumber\\\leq &C& \norm{ \curl \bw}_{\bL^s(\O)}\norm{h}_{W^{1,r}(\O)}.
\end{eqnarray}
In addition we have
\begin{eqnarray}\label{estim curl (nabla theta x d)}
 \norm{\curl(\nabla\theta\times\bdd)}_{\bL^{r}(\O)}&\leq & \norm{ \bdd\cdot\nabla\nabla\theta}_{\bL^r(\O)}+\norm{\nabla\bdd\cdot\nabla\theta}_{\bL^{r}(\O)} \nonumber\\ &\leq & \norm{\bdd}_{\bL^3(\O)}\norm{ \nabla \nabla\theta}_{\bL^{p}(\O)}+ \norm{ \nabla\bdd}_{\bL^s(\O)}\norm{ \nabla \theta}_{\bL^{p^{*}}(\O)} \nonumber\\&\leq&C  \norm{ \bdd}_{\bW^{1,s}(\O)}\norm{h}_{W^{1,r}(\O)}.
\end{eqnarray}
Now plugging the estimates \eqref{estim curl w X nabla theta} and \eqref{estim curl (nabla theta x d)} in \eqref{estim z,b W^1,p for p>2 MHD with chi} gives 
\small
\begin{eqnarray}\label{estim z,b W^1,p for p>2 MHD with chi final}
\begin{aligned}
&\norm{\bz}_{\bW^{1,\,p}(\Omega)} + \norm{\bb}_{\bW^{1,\,p}(\Omega)}+\norm{P}_{W^{1,r}(\Omega)} \leqslant C \big(1+\norm{\curl\bw}_{\bL^{s}(\O)}+\norm{\bdd}_{\bW^{1,s}(\O)}\big)\\&\times\Big(\norm{\ff}_{[\bH_0^{r',p'}(\curl, \Omega)]'} +\norm{\bg}_{[\bH_0^{r',p'}(\curl, \Omega)]'} +\norm{P_{0}}_{W^{\,1-1/r,r}(\Gamma)} \\&+\norm{ h}_{W^{1,r}(\O)} \big(1+\norm{\curl\bw}_{\bL^{s}(\O)}+\norm{\bdd}_{\bW^{1,s}(\O)} \big)\Big).
\end{aligned}
\end{eqnarray}
Thus, summing the resulting estimate \eqref{estim z,b W^1,p for p>2 MHD with chi final} along with estimate \eqref{estim theta}, we get the bound for $(\bu,\bb, P)$ in \eqref{estim weak p>2 chi and div u not free}.
\end{proof}

%%%%%%%%%%%%%%%%%%%%%%%%%%%%%%%%%%%%%%%%%%%  Dualité pour p<2 %%%%%%%%%%%%%%%%%%%%%%%%%%%

We are interested now on the existence of solution for the linearized problem \eqref{linearized MHD-pressure} in $\bW^{1,p}(\O)$ with $p<2$. Since the problem is linear, we will use a duality argument developed by Lions-Magenes \cite{Lions-Magenes}. This way ensures the uniqueness of solutions. For this, we must derive that problem \eqref{linearized MHD-pressure} has an equivalent variational formulation. We then need adequate density lemma and Green formulae, adapted to our functional framework, to define rigorously each term. We introduce the space
\small
\begin{eqnarray*}
\begin{aligned}
&\boldsymbol{\mathcal{V}}(\Omega) := \Big\{ (\bv, \ba, \theta, \tau) \in \bW^{1,p'}(\Omega) \times \bW^{1,p'}(\Omega) \times W^{1, (p^*)'}(\Omega) \times W_0^{1,(p^*)'}(\Omega); \,\,\,\vdiv \bv \in W_0^{1,(p^*)'}(\Omega),\\& 
\bv \times \bn = \ba \times \bn = \textbf{0} \, \, \text{on} \, \, \Gamma, \, \, \theta = 0, \, \, \text{on} \, \, \Gamma_0\,\,\mathrm{and}\,\,\theta = cste \, \, \text{on} \, \, \Gamma_i,\,\,\,\langle\bv \cdot \bn,1\rangle_{\Gamma_i} =\langle \ba \cdot \bn,1\rangle_{\Gamma_i} = 0, \,\,\, \forall 1 \leqslant i \leqslant I\Big\}
\end{aligned}
\end{eqnarray*}
and we recall that $\langle \cdot, \cdot \rangle_{\Omega_{p^*,p}}$ denotes the duality between $\bH_0^{p^*,p}(\curl, \Omega)$ and $[\bH_0^{p^*,p}(\curl, \Omega)]'$ with $\frac{1}{p*}=\frac{1}{p}-\frac{1}{3}$.

\begin{lemma} \label{Lemme equiv sol pour Lp*}
We suppose $\O$ of class $\mathcal{C}^{1,1}$. Let $\frac{3}{2}<p < 2$. Assume that $\ff, \bg \in [\bH_0^{r',p'}(\curl, \Omega)]'$, $h=0$ and $P_0 \in W^{1-\frac{1}{r}, r}(\Gamma)$ satisfying the compatibility conditions \eqref{Condition compatibilite K_N Lp}-\eqref{condition div g=0 sol W^1,p linear MHD}, together with $\curl\bw\in\bL^{3/2}(\O)$ and $\bdd\in\bW^{1,3/2}_{\sigma}(\O)$. Then, the following two problems are equivalent: 

\noindent \textbf{\emph{(1).}} $(\bu, \bb, P, \bc) \in \bW^{1,p}(\Omega) \times \bW^{1,p}(\Omega) \times W^{1,r}(\Omega) \times \R^I$  satisifies the linearized problem \eqref{linearized MHD-pressure}.\\ 
\textbf{\emph{(2).}} Find $(\bu, \bb, P, \bc) \in \bW^{1,p}_{\sigma}(\Omega) \times \bW^{1,p}_{\sigma}(\Omega) \times W^{1,r}(\Omega) \times \R^I$  with $\bu \times \bn = \textbf{\emph{0}}$ and $\bb \times \bn = \textbf{\emph{0}}$ on $\Gamma$, $\langle \bu \cdot \bn, 1 \rangle_{\Gamma_i} = 0$ and $\langle \bb \cdot \bn, 1 \rangle_{\Gamma_i} = 0$ for any $1 \leqslant i \leqslant I$ such that: \\
     For any $ (\bv, \ba, \theta, \tau) \in \boldsymbol{\mathcal{V}}(\Omega)$,
     \small
\begin{eqnarray} \label{Variational lemma formulation}
\langle \bu, - \Delta \bv - (\curl \bw) \times \bv + (\curl \ba) \times \bdd + \nabla \theta \rangle_{\Omega p^*, p} - \int_\Omega P \vdiv \bv \, dx \qquad\qquad\qquad\qquad\nonumber\\+\,\langle \bb, \curl \curl \ba + \curl (\bv \times \bdd) + \nabla \tau \rangle_{\Omega p^*,p}= \langle \ff, \bv \rangle_{\Omega r', p'} + \langle \bg, \ba \rangle_{\Omega r', p'} - \int_\Gamma P_0 \bv \cdot \bn \, d\sigma, \end{eqnarray}
\begin{eqnarray}\label{def ci weak VF}
c_i = \langle \ff, \nabla q_i^N \rangle_{\Omega r',p'} - \int_\Gamma P_0 \nabla q_i^N \cdot \bn \, d\sigma + \!\int_\Omega (\curl \bb) \times \bdd \cdot \nabla q_i^N \, dx\!- \!\int_\Omega (\curl \bw) \times \bu \cdot \nabla q_i^N \, dx, \end{eqnarray}
\end{lemma}

\begin{proof}

 $\textbf{\emph{(1)}}\boldsymbol{\Rightarrow  } \textbf{\emph{(2)}}$ Let $(\bu, \bb, P, \bc) \in \bW^{1,p}_{\sigma}(\Omega) \times \bW^{1,p}_{\sigma}(\Omega) \times W^{1,r}(\Omega) \times \R^I$ solution of the linearized problem \eqref{linearized MHD-pressure}. Let us take $(\bv, \ba, \theta, \tau) \in \boldsymbol{\mathcal{V}}(\Omega)$. We want to   multiply the system \eqref{linearized MHD-pressure} by $(\bv, \ba, \theta, \tau)$ and integrate by parts. Let us study these terms one by one. \\
Firstly, we note that the duality pairing $\langle -\Delta \bu, \bv \rangle_{\Omega r',p'}$ is well defined. Indeed, since $-\Delta \bu = \curl \curl \bu$ and $\curl \bu \in \bL^p(\Omega)$, it follows that $- \Delta \bu \in [\bH_0^{r',p'}(\curl, \Omega)]'$. Besides, recall that $\bv \in \bW^{1,p'}(\Omega)$, so $\curl \bv \in \bL^{p'}(\Omega)$, and since
\begin{eqnarray} \label{Relation r' et p'}
\frac{1}{r'}=1-\frac{1}{r}=1-\frac{3+p}{3p} = \frac{1}{p'} - \frac{1}{3},
\end{eqnarray}
then we have $\bW^{1,p'}(\Omega) \hookrightarrow \bL^{r'}(\Omega)$. Hence $\bv \in \bH_0^{r',p'}(\curl, \Omega)$. 
Now, by the density of $\boldsymbol{\mathcal{D}}({\Omega})$ in $\bH_0^{r',p'}(\curl, \Omega)$ and $\bH_0^{p^*, p}(\curl, \Omega)$, we have
\begin{eqnarray} \label{Premier terme lemme}
\langle - \Delta \bu, \bv \rangle_{\Omega r',p'} = \int_\Omega \curl \bu \cdot \curl \bv \, dx = \langle \bu, \curl \curl \bv \rangle_{\Omega p^*,p}
\end{eqnarray}
The last duality pairing is again well defined: the embedding $\bW^{1,p}(\Omega) \hookrightarrow \bL^{p^*}(\Omega)$ implies $\bu \in \bH_0^{p^*,p}(\curl,\Omega)$, and since $\curl \bv \in \bL^{p'}(\Omega)$ then $\curl \curl \bv \in [\bH_0^{p^*,p}(\curl, \Omega)]'$. 
Thus, due to the relation $\curl \curl \bv = - \Delta \bv + \nabla \vdiv \bv$, we deduce that:
\begin{eqnarray*}
\langle - \Delta \bu, \bv \rangle_{\Omega r', p'} = \langle \bu, - \Delta \bv + \nabla \vdiv \bv \rangle_{\Omega p^*,p}
\end{eqnarray*}
Observe that since $\bv\in\boldsymbol{\mathcal{V}}(\Omega)$, we have $\vdiv \bv \in \bW_0^{1, (p^*)'}(\Omega)$, and it follows that $\nabla \vdiv \bv \in \bL^{(p^*)'}(\Omega) \hookrightarrow [\bH_0^{p^*,p}(\curl, \Omega)]'$. Therefore, since $\curl\curl\bv$ belongs to $[\bH_0^{p^*,p}(\curl, \Omega)]'$, we deduce that $\Delta \bv$ also belongs to $[\bH_0^{p^*,p}(\curl, \Omega)]'$. This proves that the last duality makes sense. 
Next, since $\vdiv \bv = 0$ on $\Gamma$ and $\vdiv \bu = 0$ in $\Omega$, we have
\begin{eqnarray*}
\begin{aligned}
\langle \bu, \nabla \vdiv \bv \rangle_{\Omega p^*,p} &=\int_{\O}\bu\cdot\nabla\vdiv\bv \,d\bx= - \int_\Omega \vdiv \bu \vdiv \bv \, d\bx +\int_\Gamma \bu \cdot \bn \, \vdiv \bv \, d\sigma= 0 
\end{aligned}.
\end{eqnarray*}
We conclude that 
\begin{eqnarray*}
\langle - \Delta \bu, \bv \rangle_{\Omega r',p'} = \langle \bu, - \Delta \bv \rangle_{\Omega p^*, p}
\end{eqnarray*}

\vspace{5mm}

\noindent We now treat the term $\langle (\curl \bw) \times \bu, \bv \rangle_{\Omega r',p'}$. Since we have $\curl \bw \in \bL^\frac{3}{2}(\Omega)$ and $\bu \in \bW^{1,p}(\Omega) \hookrightarrow \bL^{p^*}(\Omega)$, then, by definition of $r$, $(\curl \bw) \times \bu \in \bL^r(\Omega)$. Besides, $\bv \in \bW^{1,p'}(\Omega) \hookrightarrow \bL^{(p')^*}(\Omega)$. So
\begin{eqnarray*}
\begin{aligned}
\langle (\curl \bw) \times \bu , \bv \rangle_{\Omega r',p'} = \int_\Omega (\curl \bw) \times \bu \cdot \bv \, dx = - \int_\Omega (\curl \bw) \times \bv \cdot \bu \, dx  
\end{aligned}
\end{eqnarray*}
with the integral well defined thanks to \eqref{Relation r' et p'} and
\begin{eqnarray*}
\frac{1}{r} + \frac{1}{(p')^*} = \frac{1}{r} + \frac{1}{p'} - \frac{1}{3} = \frac{1}{r} + \frac{1}{r'} = 1.
\end{eqnarray*} 

\vspace{2mm}

\noindent For the term $\langle - (\curl \bb) \times \bdd,\, \bv \rangle_{\Omega_{r',p'}}$, we proceed as for the previous pairing: we have $\curl \bb \in \bL^p(\Omega)$ and $\bdd \in \bW^{1,\frac{3}{2}}(\Omega) \hookrightarrow \bL^3(\Omega)$, so $(\curl \bb) \times \bdd \in \bL^r(\Omega)$. Therefore:
\begin{eqnarray*}
\langle - (\curl \bb) \times \bdd, \bv \rangle_{\Omega r',p'} = - \int_\Omega (\curl \bb) \times \bdd \cdot \bv \, dx = - \int_\Omega \curl \bb \cdot (\bdd \times \bv) \, dx 
\end{eqnarray*}
Using again the density of $\boldsymbol{\mathcal{D}}({\Omega})$ in $\bH_0^{p^*,p}(\curl, \Omega)$, we obtain
\begin{eqnarray*}
\begin{aligned}
-\int_\Omega \curl \bb \cdot (\bdd \times \bv) \, dx  = \int_\Omega \bb \cdot \curl (\bv \times \bdd) \, dx.
\end{aligned}
\end{eqnarray*} 
It remains us to treat the pressure term. Since $\nabla P \in \bL^r(\Omega)$, we have as for the both previous terms the well defined of the integral:
\begin{eqnarray*}
\langle \nabla P, \bv \rangle_{\Omega r', p'} = \int_\Omega \nabla P \cdot \bv \, dx 
\end{eqnarray*}
Using the same arguments as in \cite[Proposition 3.7]{AS_DCDS}, and taking into account the boundary conditions on the pressure $P$, we have
\begin{eqnarray*}
\begin{aligned}
\langle \nabla P, \bv \rangle_{\Omega r',p'} &= \int_\Gamma P \, \bv \cdot \bn \, d\sigma - \int_\Omega P \, \vdiv \bv \, dx \\
&= \int_{\Gamma_0} P_0 \, \bv \cdot \bn \, d\sigma + \sum_{i=1}^I \int_{\Gamma_i} (P_0 + c_i) \, \bv \cdot \bn \, d\sigma - \int_\Omega P \, \vdiv \bv \, dx 
\end{aligned}
\end{eqnarray*}
However, since $\langle\bv \cdot \bn ,\,1\rangle_{\Gamma_{i}} = 0$ for all $1 \leqslant i \leqslant I$, we have $\displaystyle\sum_{i=1}^I \int_{\Gamma_i} c_i \, \bv \cdot \bn \, d\sigma=0$. Therefore, we obtain: 
\begin{eqnarray*}
\begin{aligned}
\langle \nabla P, \bv \rangle_{\Omega r',p'} = \int_\Gamma P_0 \, \bv \cdot \bn \, d\sigma - \int_\Omega P \, \vdiv \bv \, dx 
\end{aligned}
\end{eqnarray*}
Now, multiplying the equation $\vdiv \bu = 0$ in $\Omega$ by $\theta$, we obtain by the density of $\boldsymbol{\mathcal{D}}(\bar{\Omega})$ in $\bH^{p^*,p}(\vdiv, \Omega)$\begin{eqnarray*}
0 = - \int_\Omega \theta \vdiv \bu \, dx = \int_\Omega \bu \cdot \nabla \theta \, dx - \int_\Gamma \theta \, \bu \cdot \bn \, d\sigma,
\end{eqnarray*}
where we have used the fact that $\bu\in\bW^{1,p}(\O)\hookrightarrow \bL^{p*}(\O)$ which implies that $\bu\in\bH^{p^{*},\,p}(\vdiv,\O)$. Combining the boundary conditions of $\theta$ on $\Gamma_i$, $0 \leqslant i \leqslant I$, with zero fluxs of the velocity $\langle\bu \cdot \bn,\,1\rangle_{\Gamma_{i}} = 0$ for $1 \leqslant i \leqslant I$, we have:
$\displaystyle\int_\Gamma \theta \, \bu \cdot \bn \, d\sigma = 0.$
Thus, summing the above resulting terms, we obtain: 
\begin{eqnarray}\label{part 1 FV p<2}
\begin{aligned}
\langle \ff, \bv \rangle_{\Omega r', p'} &= \langle \bu, \curl \curl \bv \rangle_{\Omega p^*, p} - \int_\Omega (\curl \bw) \times \bv \cdot \bu \, dx -  \int_\Omega \curl (\bdd \times \bv) \cdot \bb \, dx \\
&- \int_\Omega P \vdiv \bv \, dx  + \int_\Gamma P_0 \, \bv \cdot \bn \, d\sigma + \int_\Omega \bu \cdot \nabla \theta \, dx.
\end{aligned}
\end{eqnarray}

Now, we treat the terms of the second equation of \eqref{linearized MHD-pressure}. For the term $\langle \curl \curl \bb, \ba \rangle_{\Omega r', p'}$, the duality pairing is well defined and following the same reasoning than for \eqref{Premier terme lemme}, we have  \begin{eqnarray*}
\langle \curl \curl \bb, \ba \rangle_{\Omega_{r',p'}} = \int_\Omega \curl \bb \cdot \curl \ba \, dx = \langle \bb, \curl \curl \ba \rangle_{\Omega_{p^*, p}}
\end{eqnarray*}
Next, for the term $\langle \curl (\bu \times \bdd), \ba \rangle_{\Omega r', p'}$, the duality pairing is again well defined: since $\bu \in \bW^{1,p}(\Omega) \hookrightarrow \bL^{p^*}(\Omega)$ and $\bdd \in \bW^{1,\frac{3}{2}}(\Omega) \hookrightarrow \bL^3(\Omega)$, then  $\bu \times \bdd \in \bL^p(\Omega)$ so $\curl (\bu \times \bdd) \in [\bH_0^{r',p'}(\curl, \Omega)]'$. Similarly, by the density of $\boldsymbol{\mathcal{D}}(\Omega)$ in $\bH_0^{r',p'}(\curl, \Omega)$, we have
\begin{eqnarray*}
\begin{aligned}
\langle \curl (\bu \times \bdd), \ba \rangle_{\Omega r',p'} = \int_\Omega \curl \ba \cdot (\bu \times \bdd) \, dx = \int_\Omega \bu \cdot (\bdd \times \curl \ba) \, dx
\end{aligned}
\end{eqnarray*}
Here also the integrals are well defined. Indeed, for the first integral, $\curl \ba \in \bL^{p'}(\Omega)$, $\bu \in \bL^{p^*}(\Omega)$, $\bdd \in \bW^{1,\frac{3}{2}}(\Omega) \hookrightarrow \bL^3(\Omega)$ and $\frac{1}{p'} + \frac{1}{p^*} + \frac{1}{3} = 1$.

It remains us to multiply the equation $\vdiv\,\bb=0$ in $\O$ by $\tau$. By density of $\mathcal{D}(\Omega)$ in $W_0^{1,(p^*)'}(\Omega)$, we have the Green formula: for any $\tau\in W_0^{1,(p^*)'}(\Omega)$
\begin{eqnarray*}
- \int_\Omega (\vdiv \bb) \, \tau \, dx = \int_\Omega \bb \cdot \nabla \tau \, dx
\end{eqnarray*}
In summary, these terms provided by the second equation of \eqref{linearized MHD-pressure} give:
\begin{eqnarray}\label{Part 2 FV p<2}
\langle \bg, \ba \rangle_{\Omega r',p'} = \langle \bb, \curl \curl \ba \rangle_{\Omega p^*,p} -  \int_\Omega \bu \cdot (\bdd \times \curl \ba) \, dx + \int_\Omega \bb \cdot \nabla \tau \, dx
\end{eqnarray}

\noindent Finally, adding \eqref{part 1 FV p<2} and \eqref{Part 2 FV p<2}, we obtain the variational formulation  \eqref{Variational lemma formulation}. 

\vspace{5mm}

We now want to determine the constants $c_i$ in \eqref{def ci weak VF}. Let us take $\bv \in \bW^{1,p}_\sigma(\Omega)$ with $\bv \times \bn = 0$ on $\Gamma$, and set: 
\begin{eqnarray} \label{Def v0}
\bv_0 = \bv - \sum_{i=1}^I \Big( \int_{\Gamma_i} \bv \cdot \bn \, d\sigma \Big) \nabla q_i^N
\end{eqnarray}
So, $\bv_0$ belongs to $\bW^{1,p}_\sigma(\Omega)$ and satisfies  $\bv \times \bn = 0$ on $\Gamma$, $\langle\bv_0\cdot\bn,\,1\rangle_{\Gamma_{i}}=0$, $1\leq i \leq I$.
Multiplying the first equation of \eqref{linearized MHD-pressure} by $\bv$ and integrating by parts, we obtain 
\begin{eqnarray} \label{Egalite preuve lemme p^*}
\begin{aligned}
&\int_\Omega \curl \bu \cdot \curl \bv \, dx - \int_\Omega (\curl \bw) \times \bv \cdot \bu \, dx +  \int_\Omega \curl (\bv \times \bdd) \cdot \bb \, dx \\
&+ \int_{\Gamma_0} P_0 \, \bv \cdot \bn \, d\sigma + \sum_{i=1}^I \int_{\Gamma_i} (P_0 + c_i) \bv \cdot \bn \, d\sigma =  \int_\Omega \ff \cdot \bv \, dx
\end{aligned}
\end{eqnarray}
We now take a test function $(\bv_0, 0, 0, 0)$ in \eqref{Variational lemma formulation}. Note that it is possible because of the definition of \eqref{Def v0}:
\begin{eqnarray*}
\begin{aligned}
&\langle \bu, - \Delta \bv_0 \rangle_{\Omega p^*,p} - \int_\Omega \bu \cdot (\curl \bw) \times \bv_0 \, dx - \int_\Omega P \vdiv \bv_0 \, dx +  \int_\Omega \bb \cdot \curl (\bv_0 \times \bdd) \, dx \\
&= \langle \ff, \bv_0 \rangle_{\Omega r',p'} - \int_\Gamma P_0 \, \bv_0 \cdot \bn \, d\sigma
\end{aligned}
\end{eqnarray*}
By definition of $q_i^N$, we have $\vdiv \bv_0 = 0$. Besides, $\curl \bv_0 = \curl \bv$, so it follows from the same density argument used previously:
\begin{eqnarray*}
\begin{aligned}
\int_\Omega \curl \bu \cdot \curl \bv \, dx - \int_\Omega (\curl \bw) \times \bv_0 \cdot \bu \, dx +  \int_\Omega \bb \cdot \curl(\bv_0 \times \bdd) \, dx = \int_\Omega \ff \cdot \bv_0 \, dx - \int_\Gamma P_0 \, \bv_0 \cdot \bn \, d\sigma 
\end{aligned}
\end{eqnarray*}
Decomposing $\bv_0$ with \eqref{Def v0} in the previous equality, we have:
{\footnotesize
\begin{eqnarray*}
\begin{aligned}
& \int_\Omega \curl \bu \cdot \curl \bv \, dx - \int_\Omega (\curl \bw) \times \bv \cdot \bu +  \int_\Omega \bb \cdot \curl(\bv \times \bdd) \, dx 
-\int_\Omega \ff \cdot \bv \, dx + \int_\Gamma P_0 \, \bv \cdot \bn \, d\sigma \\\small
&= \sum_{i=1}^I (\int_{\Gamma_i} \bv \cdot \bn \, d\sigma) \Big[ -\int_\Omega (\curl \bw) \times \nabla q_i^N \cdot \bu \, dx +  \int_\Omega \bb \cdot \curl (\nabla q_i^N \times \bdd) \, dx - \int_\Omega \ff \cdot \nabla q_i^N \, dx + \int_\Gamma P_0 \, \nabla q_i^N \cdot \bn \, d\sigma \Big]
\end{aligned}
\end{eqnarray*}}
Injecting \eqref{Egalite preuve lemme p^*} in this calculus, we thus obtain: 
\begin{eqnarray*}
- \sum_{i=1}^I c_i \int_{\Gamma_i} \bv \cdot \bn \, d\sigma & =&\sum_{i=1}^I (\int_{\Gamma_i} \bv \cdot \bn \, d\sigma) \Big[ - \int_\Omega (\curl \bw) \times \nabla q_i^N \cdot \bu \, dx +  \int_\Omega \bb \cdot \curl (\nabla q_i^N \times \bdd) \, dx
\\&-& \int_\Omega \ff \cdot \nabla q_i^N \, dx + \int_\Gamma P_0 \, \nabla q_i^N \cdot \bn \, d\sigma  \Big]
\end{eqnarray*}
Therefore, taking $\bv = \nabla q_i^N$ and since, for all $1 \leqslant i,k \leqslant I$, $\langle \nabla q_i^N \cdot \bn, 1 \rangle_{\Gamma_k} = \delta_{i,k}$, we have: 
\begin{eqnarray*}
\begin{aligned}
c_i = - \int_\Omega (\curl \bw) \times \bu \cdot \nabla q_i^N \, dx - \underbrace{ \int_\Omega \bb \cdot \curl (\nabla q_i^N \times \bdd) \, dx }_{ \displaystyle\int_\Omega (\curl \bb) \times \bdd \cdot \nabla q_i^N \, dx }+ \int_\Omega \ff \cdot \nabla q_i^N \, dx - \int_\Gamma P_0 \, \nabla q_i^N \cdot \bn \, d\sigma
\end{aligned}
\end{eqnarray*}
which gives the relation \eqref{def ci weak VF}.\medskip

 $\textbf{\emph{(2)}}\boldsymbol{\Rightarrow  } \textbf{\emph{(1)}}$ Conversely, let $(\bu, \bb, P, \bc) \in \bW^{1,p}_{\sigma}(\Omega) \times \bW^{1,p}_{\sigma}(\Omega) \times W^{1,r}(\Omega) \times \R^I$ solution of \eqref{Variational lemma formulation}-\eqref{def ci weak VF} with $\bu \times \bn =\bb \times \bn = \textbf{0}$ on $\Gamma$, and $\langle \bu \cdot \bn, 1 \rangle_{\Gamma_i} = \langle \bb \cdot \bn, 1 \rangle_{\Gamma_i} = 0$ for all $1 \leqslant i \leqslant I$. We want to prove that $(\bu, \bb, P, \bc)$ satisfies \eqref{linearized MHD-pressure}.
Let us take $\bv \in \boldsymbol{\mathcal{D}}(\Omega)$, $\ba = \boldsymbol{0}$, $\theta = \tau = 0$ as test functions in \eqref{Variational lemma formulation}. We obtain 
\begin{eqnarray*}
\langle - \Delta \bu + (\curl \bw) \times \bu -  (\curl \bb) \times \bdd  - \ff, \bv \rangle_{\boldsymbol{\mathcal{D}'}(\Omega) \times \boldsymbol{\mathcal{D}}(\Omega)} = 0, \quad \forall \bv \in \boldsymbol{\mathcal{D}}(\Omega)
\end{eqnarray*}
So by De Rham's theorem, there exists $P\in \L^{p}(\O)$ such that 
\begin{eqnarray*}
- \Delta \bu + (\curl \bw) \times \bu -  (\curl \bb) \times \bdd + \nabla P = \ff 
\end{eqnarray*}
So, $(\bu, \bb, P)$ satisfies the first equation of \eqref{linearized MHD-pressure}. 
Let us now take $\ba \in \bW^{1,p'}_{\sigma}(\Omega)$ with $\ba \times \bn = 0$ on $\Gamma$, $\bv = \boldsymbol{0}$, $\theta = \tau = 0$ as test functions in \eqref{Variational lemma formulation}. We obtain
\begin{eqnarray*}
\langle \curl \curl \bb -  \curl (\bu \times \bdd) - \bg, \ba \rangle_{\Omega_{p^*,p}}
\end{eqnarray*}
Applying a De Rham Lemma version for functionals acting on vector fields with vanishing
tangential components (see \cite[Lemma 2.2]{Pan}), there exists $\chi \in L^2(\Omega)$ defined uniquely up to an additive constant such that: 
\begin{eqnarray} \label{2eme equation MHD lemme}
\curl \curl \bb -  \curl (\bu \times \bdd) + \nabla \chi = \bg \quad \mathrm{in}\,\,\O\quad \mathrm{and}\quad \chi=0\quad \mathrm{on}\,\,\Gamma.
\end{eqnarray}
But taking the divergence in \eqref{2eme equation MHD lemme}, $\chi$ is solution of the following Dirichlet problem: 
\begin{eqnarray*}
\Delta \chi = \vdiv \bg=0 \quad \text{in} \, \, \Omega  \quad \mathrm{and}\quad  \chi= 0 \, \, \text{on} \, \, \Gamma
\end{eqnarray*}
Since $\bg$ satisfies the compatibility condition \eqref{condition div g=0 sol W^1,p linear MHD}, then $\chi$ is equal to zero and we have: 
\begin{eqnarray*}
\curl \curl \bb -  \curl (\bu \times \bdd) = \bg 
\end{eqnarray*}
So $(\bu, \bb)$ satisfies the second equation in \eqref{linearized MHD-pressure}. 

\noindent Next, if we choose $\bv = \ba = \boldsymbol{0}$, $\tau = 0$ and $\theta \in \mathcal{D}(\Omega)$, then we obtain $\vdiv \bu = 0$ in $\Omega$. Similarly, if we choose $\bv = \ba = \boldsymbol{0}$, $\theta = 0$ and $\tau \in \mathcal{D}(\Omega)$, we obtain $\vdiv \bb = 0$ in $\Omega$. 

\vspace{2mm}

It remains to prove the boundary condition given on the pressure $P$. To this end, we follow a method from \cite[Proposition 3.7]{AS_DCDS}. Let us take as test functions $\bv \in \bW^{1,p'}(\Omega)$ with $\bv \times \bn = 0$ on $\Gamma$ and $\vdiv \bv \in W_0^{1,p^*}(\Omega)$, $\ba = \boldsymbol{0}$, $\theta \in W^{1,(p^*)'}(\Omega)$ and $\tau = 0$ in the variational formulation \eqref{Variational lemma formulation}. Thus, applying Green formulae as previously, we have: 
\begin{eqnarray*}
\begin{aligned}
\langle \ff, \bv \rangle_{\Omega_{r',p'}} &= \langle \bu, - \Delta \bv \rangle_{\Omega_{r',p'}} - \int_\Omega (\curl \bw) \times \bv \cdot \bu \, dx +  \int_\Omega \bb \cdot \curl (\bv \times \bdd) \, dx \\
& +\int_\Omega \nabla \theta \cdot \bu \, dx - \int_\Omega \nabla \theta \cdot \bu \, dx - \int_\Omega P \, \vdiv \bv \, dx + \int_\Gamma P \, \bv \cdot \bn \, d\sigma
\end{aligned}
\end{eqnarray*}
We decompose $\bv$ as in \eqref{decomposition of v} and to simplify the presentation, we set $\bz=\displaystyle\sum_{i=1}^I \Big( \int_{\Gamma_i} \bv \cdot \bn \, d\sigma \Big) \nabla q_i^N$. So, $\bv= \bv_0 + \bz$. By definition of $q_i^N$ for $1 \leqslant i \leqslant I$, we have $\Delta \bz = 0$ and $\vdiv \bz = 0$ in $\Omega$. Thus: 
\small
\begin{eqnarray} \label{Eq *** constantes}
\langle \ff, \bv_0 \rangle_{\Omega_{r',p'}} + \langle \ff, \bz \rangle_{\Omega_{r',p'}}\! =\! \langle \bu, - \Delta \bv_0 - (\curl \bw) \times \bv_0 + \nabla \theta \rangle_{\Omega_{p^*,p}} 
+  \langle \bb, \curl (\bv_0 \times \bdd) \rangle_{\Omega_{p^*,p}}\!\!-\!\int_\Omega \!\nabla \theta \cdot \bu \, dx \hspace{4.2cm}\\-\! \int_\Omega \!(\curl \bw) \times \bz \cdot \bu \, dx 
+  \int_\Omega \bb \cdot \curl (\bz \times \bdd) \, dx - \int_\Omega P \, \vdiv \bv_0 \, dx + \int_\Gamma P \bv_0 \cdot \bn \, d\sigma + \int_\Omega P \, \bz \cdot \bn \, d\sigma \hspace{5cm}
\end{eqnarray}
Taking $(\bv_0, \boldsymbol{0}, \theta, 0)$ as a test function in the variational formulation \eqref{Variational lemma formulation}, we obtain
\begin{eqnarray*}
\int_\Gamma P \, \bv_0 \cdot \bn \, d\sigma - \int_\Gamma P_0 \, \bv_0 \cdot \bn \, d\sigma - \int_\Omega \nabla \theta \cdot \bu \, dx = 0
\end{eqnarray*}
Note that, since $\vdiv \bu = 0$ in $\Omega$, $\theta = 0$ on $\Gamma_0$, $\theta = \beta_i$ on $\Gamma_i$ and $\langle \bu \cdot \bn, 1 \rangle_{\Gamma_i} = 0$, it follows that $\displaystyle\int_\Omega \nabla \theta \cdot \bu \, dx = - \int_\Omega \theta \vdiv \bu \, dx + \int_\Gamma \theta \bu \cdot \bn \, d\sigma = 0$. Therefore: 
\begin{eqnarray*}
\int_\Gamma (P - P_0) \bv_0 \cdot \bn \, d\sigma = 0
\end{eqnarray*}
Then, we deduce that
\begin{eqnarray} \label{Eq ** lemme}
\begin{aligned}
\langle \ff, \bz \rangle_{\Omega_{r',p'}} + \int_\Omega (\curl \bw) \times \bz \cdot \bu \, dx -  \int_\Omega \bb \cdot \curl (\bz \times \bdd) \, dx - \int_\Omega P \, \bz \cdot \bn \, d\sigma = 0.
\end{aligned}
\end{eqnarray}
Now, using \eqref{Eq ** lemme} and the fact that $\displaystyle\int_{\Gamma_i} \bv_0 \cdot \bn \, d\sigma = 0$ for all $1 \leqslant i \leqslant I$, we have
\begin{eqnarray} \label{Eq preuve cstes lemme}
\begin{aligned}
\int_\Gamma P \bv \cdot \bn \, d\sigma &= \int_\Gamma P \, \bv_0 \cdot \bn \, d\sigma + \int_\Gamma P \, \bz \cdot \bn \, d\sigma \\
&= \int_\Gamma P_0 \, \bv_0 \cdot \bn \, d\sigma + \langle \ff, \bz \rangle_{\Omega_{r',p'}} + \int_\Omega (\curl \bw) \times \bz \cdot \bu \, dx -  \int_\Omega \bb \cdot \curl (\bz \times \bdd) \, dx \\
&= \int_\Gamma P_0 \, \bv_0 \cdot \bn \, d\sigma + \sum_{i=1}^I \Big( \int_{\Gamma_i} \bv \cdot \bn \, d\sigma \Big)\Big [ \langle \ff, \nabla q_i^N \rangle_{\Omega_{r',p'}} - \int_\Omega (\curl \bw) \times \bu \cdot \nabla q_i^N \, dx \\
&-  \int_\Omega \bb \cdot \curl (\nabla q_i^N \times \bdd) \, dx \Big].
\end{aligned}
\end{eqnarray}
However, we have from \eqref{Variational lemma formulation}, for all $1 \leqslant i \leqslant I$, 
\begin{eqnarray*}
\begin{aligned}
c_i &= \langle \ff, \nabla q_i^N \rangle_{\Omega_{r',p'}} - \int_\Gamma P_0 \, \nabla q_i^N \cdot \bn \, d\sigma +  \int_\Omega (\curl \bb) \times \bdd \cdot \nabla q_i^N \, dx - \int_\Omega (\curl \bw) \times \bu \cdot \nabla q_i^N \, dx \\
&= \langle \ff, \nabla q_i^N \rangle_{\Omega_{r',p'}} - \int_\Gamma P_0 \, \nabla q_i^N \cdot \bn \, d\sigma -  \int_\Omega \bb \cdot \curl (\nabla q_i^N \times \bdd) \, dx - \int_\Omega (\curl \bw) \times \bu \cdot \nabla q_i^N \, dx .
\end{aligned}
\end{eqnarray*}
Therefore, replacing in \eqref{Eq preuve cstes lemme}, we have: 
\begin{eqnarray*}
\int_\Gamma P \, \bv \cdot \bn \, d\sigma = \int_\Gamma P_0 \, \bv_0 \cdot \bn d\sigma + \sum_{i=1}^I \Big( \int_{\Gamma_i} \bv \cdot \bn \, d\sigma \Big) \Big[c_i + \int_\Gamma P_0 \, \nabla q_i^N \cdot \bn \, d\sigma \Big]
\end{eqnarray*}
Moreover, applying directly the decomposition \eqref{decomposition of v}, we have: 
\begin{eqnarray*}
\int_\Gamma P_0 \, \bv \cdot \bn \, d\sigma = \int_\Gamma P_0 \, \bv_0 \cdot \bn d\sigma + \sum_{i=1}^I \Big( \int_{\Gamma_i} \bv \cdot \bn \, d\sigma \Big)\int_\Gamma P_0 \, \nabla q_i^N \cdot \bn \, d\sigma 
\end{eqnarray*}
Thus, combining the last two equations, we obtain: 
\begin{eqnarray*}
\int_\Gamma P \, \bv \cdot \bn \, d\sigma = \int_\Gamma P_0 \bv \cdot \bn \, d\sigma + \sum_{i=1}^I \Big( \int_{\Gamma_i} \bv \cdot \bn \, d\sigma \Big) c_i = \int_\Gamma (P_0 + c) \bv \cdot \bn \, d\sigma,
\end{eqnarray*}
with $c = 0$ on $\Gamma_0$ and $c = c_i$ on $\Gamma_i$ for all $1 \leqslant i \leqslant I$. 
We conclude as in \cite[Theorem 3.2.]{AS_DCDS} to prove that  $P = P_0$ on $\Gamma_0$ and $P=P_0+c_i$ on $\Gamma_i$ . 
% let us take $ \in W^{-\frac{1}{p},p}(\Gamma_j)$ for $0 \leqslant j \leqslant I$ and set: 
% 
% \begin{eqnarray*}
% ' = \begin{cases}
%  \quad \text{on} \, \, \Gamma_j \\
% 0 \quad \text{on} \, \, \Gamma \setminus \Gamma_j
% \end{cases}
% \end{eqnarray*}
% 
% Taking $\bv \in \bW^{1,p}(\Omega)$ with $\vdiv \bv = 0$ in $\Omega$, such that: 
% 
% \begin{eqnarray*}
% \bv = \begin{cases}
% ( - \frac{1}{|\Gamma_j|} \int_{\Gamma_j}  ) \bn \quad \text{on} \, \, \Gamma_j \\
% \boldsymbol{0} \quad \text{on} \, \, \Gamma \setminus \Gamma_j
% \end{cases}
% \end{eqnarray*}
% 
% Therefore we obtain: 
% 
% \begin{eqnarray*}
% \langle P - P_0,  \rangle_{\Gamma_j} = \langle P - P_0, ' \rangle_{\Gamma} = \langle P - P_0, \bv \cdot \bn \rangle_{\Gamma} = \langle c_j,  \rangle_{\Gamma_j}
% \end{eqnarray*}
% 
% hence $P = P_0 + c_i$ on $\Gamma_i$ for $0 \leqslant i \leqslant I$ with $c_0 = 0$, and it achieves the proof. 
\end{proof}

%%%%%%%%%%%%%%%%%%%%%%%%%%%%%%% Fin VF for p<2 %%%%%%%%%%%%%%%%%%%%%%%%
We are now in position to prove the following theorem 

\begin{theorem}\label{thm: weak W1,p for p<2}
  We suppose $\O$ of classe $\mathcal{C}^{1,1}$. Let $\frac{3}{2}<p < 2$. Assume that $\ff, \bg \in [\bH_0^{r',p'}(\curl, \Omega)]'$, $P_0 \in W^{1-\frac{1}{r}, r}(\Gamma)$, $h\in W^{1,r}(\O)$ with the compatibility conditions \eqref{Condition compatibilite K_N Lp}-\eqref{condition div g=0 sol W^1,p linear MHD}, together with $\curl\bw\in\bL^{3/2}(\O)$ and $\bdd\in\bW^{1,3/2}_{\sigma}(\O)$ . Then the linearized problem \eqref{linearized MHD-pressure} has a unique solution $(\bu, \bb, P, \bc) \in \bW^{1,p}(\Omega) \times \bW^{1,p}(\Omega) \times W^{1,r}(\Omega) \times \R^I$. Moreover, we have the following estimates:  
 \begin{eqnarray}\label{estim u,b W^1,p for p<2}
 \begin{aligned} 
&\norm{\bu}_{\bW^{1,\,p}(\Omega)} + \norm{\bb}_{\bW^{1,\,p}(\Omega)}\leq  C(1 + \norm{\curl\bw}_{\bL^{3/2}(\Omega)} + \norm{\bdd}_{\bW^{1,3/2}(\O)})\Big(\norm{\ff}_{[\bH_0^{r',p'}(\curl,\Omega)]'}\\& +  \norm{P_0}_{W^{1-\frac{1}{r},r}(\Gamma)} +\norm{\bg}_{[\bH_0^{r',p'}(\curl, \Omega)]}+(1 +  \norm{\curl\bw}_{\bL^{3/2}(\Omega)} + \norm{\bdd}_{\bW^{1,3/2}(\O)})\norm{h}_{W^{1,r}(\O)}\Big) 
\end{aligned}
\end{eqnarray}
 \begin{eqnarray}\label{estim P W^1,r for p<2}
 \begin{aligned} 
 &\norm{P}_{W^{1,r}(\Omega)} \leq C(1 + \norm{\curl\bw}_{\bL^{3/2}(\Omega)} + \norm{\bdd}_{\bW^{1,3/2}(\O)})^2\times\Big(\norm{\ff}_{[\bH_0^{r',p'}(\curl,\Omega)]'} \\&+ \norm{\bg}_{[\bH_0^{r',p'}(\curl, \Omega)]}+ \norm{P_0}_{W^{1-\frac{1}{r},r}(\Gamma)}+(1+\norm{\curl\bw}_{\bL^{3/2}(\Omega)} + \norm{\bdd}_{\bW^{1,3/2}(\O)})\norm{h}_{W^{1,r}(\O)}\Big).
\end{aligned}
\end{eqnarray} 
\end{theorem}
\begin{proof}
Since $2<p'<3$, thanks to Theorem \ref{thm weak sol W^1,p p>2 with phi_u and chi}, we have for any $(\bF, \bG, \phi) \in [\bH_0^{p^*,p}(\curl, \Omega)]'\times [\bH_0^{p^*,p}(\curl, \Omega)]'\bot \bK^p_N(\O)\times  W_0^{1,(p^*)'}(\Omega)$ that the following problem 
\begin{eqnarray}\label{dual MHD}
		\left\lbrace
		\begin{aligned}
		 & -\, \Delta \bv -(\curl \bw) \times \bv + \nabla \theta+ (\curl \ba) \times \bdd = \bF \quad \mathrm{and}\quad  \vdiv \bv = \phi   \,\,\ \mathrm{in} \,\, \Omega, \\
        & \, \curl \curl \ba + \curl (\bv \times \bdd)+\nabla \tau = \bG  \quad  \mathrm{and}\quad \vdiv \ba = 0 \quad \text{in} \,\, \Omega, \\
       &\bv \times \bn = \textbf{0}, \quad \ba \times \bn = \textbf{0}\quad \mathrm{and}\quad \tau=0\quad\mathrm{on}\,\,\Gamma,\quad
        \theta =0 \quad \text{on} \, \,\Gamma_0\quad \mathrm{and}\quad 
         \theta= \beta_i \,\,\,\text{on} \,\, \Gamma_i,\\
        &\langle\bv \cdot \bn, 1\rangle_{\Gamma_i} = 0, \quad \mathrm{and}\quad \langle\ba \cdot \bn, 1\rangle_{\Gamma_i} = 0,\,\,\,1\leq i \leq I.
        \end{aligned}\right.
			\end{eqnarray}
 has a unique solution $(\bv, \ba, \theta, \tau, \boldsymbol{\beta}) \in \bW^{1,p'}(\Omega) \times \bW^{1,p'}(\Omega) \times W^{1,(p^*)'}(\Omega) \times W^{1,(p^*)'}_{0}(\Omega) \times \R^I$ with $\vdiv \bv \in W_0^{1,(p^*)'}(\Omega)$ and $\boldsymbol{\beta} = (\beta_1, \cdots , \beta_I)$ such that: 
\begin{eqnarray*}
\beta_i = \langle \bF, \nabla q_i^N \rangle_{\Omega_{ p^*, p}} + \langle (\curl \ba)\times \bdd, \nabla q_i^N \rangle_{\Omega_{ p^*, p}} - \langle (\curl \bw) \times \bv, \nabla q_i^N \rangle_{\Omega_{ p^*, p}}+\int_{\Gamma}\phi\nabla q_i^N\cdot\bn\,d\sigma.
\end{eqnarray*}
Moreover, this solution also satisfies the estimates:
\begin{eqnarray} \label{Estimation provenant du lemme}
\begin{aligned}
&\norm{\bv}_{\bW^{1,p'}(\Omega)} + \norm{\ba}_{\bW^{1,p'}(\Omega)} + \norm{\theta}_{W^{1,(p^*)'}(\Omega)} \leqslant C \Big( 1 + \norm{\curl \bw}_{\bL^\frac{3}{2}(\Omega)} + \norm{\bdd}_{\bW^{1,\frac{3}{2}}(\Omega)} \Big) \\
&\times \Big( \norm{\bF}_{[\bH_0^{p^*,p}(\curl, \Omega)]'} + \norm{\bG}_{[\bH_0^{p^*,p}(\curl, \Omega)]'} + \big(1 + \norm{\curl \bw}_{\bL^\frac{3}{2}(\Omega)} + \norm{\bdd}_{\bW^{1,\frac{3}{2}}(\Omega)}\big) \norm{\phi}_{W^{1,(p^*)'}(\Omega)} \Big).
\end{aligned}
\end{eqnarray}
We note that, from Theorem \ref{thm weak sol W^1,p p>2 with phi_u and chi} for $2<p'<3$, the value of $s$ is $3/2$. Using \eqref{Estimation provenant du lemme}, we have 
{\footnotesize
\begin{eqnarray} \label{Estimation preuve prop}
\begin{aligned}
&| \langle \ff, \bv \rangle_{\Omega r',p'} + \langle \bg, \ba \rangle_{\Omega r', p'} - \int_\Gamma P_0 \, \bv \cdot \bn \, d\sigma | \\
& \leqslant \norm{\ff}_{[\bH_0^{r',p'}(\curl,\Omega)]'} \norm{\bv}_{\bH_0^{r',p'}(\curl,\Omega)} + \norm{\bg}_{[\bH_0^{r',p'}(\curl,\Omega)]'} \norm{\ba}_{\bH_0^{r',p'}(\curl,\Omega)} + \norm{P_0}_{W^{1-\frac{1}{r},r}(\Gamma)} \norm{\bv \cdot \bn }_{\bW^{1-\frac{1}{p'},p'}(\Gamma)} \\
&\leqslant C \Big( \norm{\ff}_{[\bH_0^{r',p'}(\curl, \Omega)]'} + \norm{\bg}_{[\bH_0^{r',p'}(\curl, \Omega)]'} + \norm{P_0}_{W^{1-\frac{1}{r},r}(\Gamma)} \Big) \Big( \norm{\bv}_{\bW^{1,p'}(\Omega)} + \norm{\ba}_{\bW^{1,p'}(\Omega)} \Big) \\
&\leqslant C \Big( \norm{\ff}_{[\bH_0^{r',p'}(\curl,\Omega)]'} + \norm{\bg}_{[\bH_0^{r',p'}(\curl, \Omega)]'} + \norm{P_0}_{W^{1-\frac{1}{r},r}(\Gamma)} \Big) \Big( 1 + \norm{\curl \bw}_{\bL^\frac{3}{2}(\Omega)} + \norm{\bdd}_{\bW^{1,\frac{3}{2}}(\Omega)} \Big) \\
&\times \Big( \norm{\bF}_{\bH_0^{p^*,p}(\curl, \Omega)} + \norm{\bG}_{\bH_0^{p^*,p}(\curl, \Omega)} + (1 + \norm{\curl \bw}_{\bL^\frac{3}{2}(\Omega)} + \norm{\bdd}_{\bW^{1,\frac{3}{2}}(\Omega)} ) \norm{\phi}_{W^{1,(p^*)'}(\Omega)} \Big)  
\end{aligned}
\end{eqnarray}}
We deduce that the linear mapping $(\bF, \bG, \phi) \rightarrow \langle \ff, \bv \rangle_{\Omega _{r',p'}} + \langle \bg, \ba \rangle_{\Omega_{r',p'}} - \int_\Gamma P_0 \, \bv \cdot \bn \, d\sigma$ defines an element of the dual space of $\bH_0^{p^*, p}(\curl, \Omega) \times \bH_0^{p^*,p}(\curl, \Omega) \times W^{-1,p^*}(\Omega)$. It follows from Riesz' representation theorem that there exists a solution $(\bu, \bb, P)$ in $\bH_0^{p^*, p}(\curl, \Omega) \times \bH_0^{p^*,p}(\curl, \Omega) \times W^{-1,p^*}(\Omega)$ of the problem
\begin{eqnarray*}
\begin{aligned}
& \langle \bu, \bF \rangle_{\Omega p^*, p} + \langle \bb, \bG \rangle_{\Omega p^*,p} - \langle P, \phi \rangle_{W^{-1,p^*}(\Omega) \times W_0^{1,(p^*)'}(\Omega)} = \langle \ff, \bv \rangle_{\Omega r',p'} + \langle \bg, \ba \rangle_{\Omega r',p'} - \int_\Gamma P_0 \, \bv \cdot \bn \, d\sigma
\end{aligned}
\end{eqnarray*}
which is the variational formulation \eqref{Variational lemma formulation}. Moreover, it satisfies the estimate:
\begin{eqnarray} \label{estim u,b,P with duality p<2}
\begin{aligned}
&\norm{\bu}_{\bH_0^{p^*,p}(\curl, \Omega)} + \norm{\bb}_{\bH_0^{p^*,p}(\curl, \Omega)} + \big( 1 +  \norm{\curl\bw}_{\bL^{3/2}(\Omega)} + \norm{\bdd}_{\bW^{1,3/2}(\Omega)} \big)^{-1}\norm{P}_{W^{-1,p^*}(\Omega)}\\
&\!\leqslant \!C\big( 1 \!+ \! \norm{\curl\bw}_{\bL^{3/2}(\Omega)} \!+\! \norm{\bdd}_{\bW^{1,3/2}(\Omega)}\big) \!\big( \norm{\ff}_{[\bH_0^{r',p'}(\curl,\Omega)]'} \!+\! \norm{\bg}_{[\bH_0^{r',p'}(\curl,\Omega)]'}\! +\! \norm{P_0}_{W^{1-\frac{1}{r},r}(\Gamma)} \big).
\end{aligned}
\end{eqnarray} 
In order to recover the solution of \eqref{linearized MHD-pressure} through the equivalence result given in Lemma \ref{Lemme equiv sol pour Lp*}, it remains us to prove that $\bu, \bb \in \bW^{1,p}(\Omega)$, $P \in W^{1,r}(\Omega)$, that $\langle \bu \cdot \bn, 1 \rangle_{\Gamma_i} = 0$, $\langle \bb \cdot \bn, 1 \rangle_{\Gamma_i} = 0 $ for all $1 \leqslant i \leqslant I$ and to recover the relation of \eqref{def ci weak VF}. 
We firstly want to show that $\displaystyle\int_{\Gamma_i} \bu \cdot \bn \, d\sigma = 0$ and $\displaystyle\int_{\Gamma_i} \bb \cdot \bn \, d\sigma = 0$. We choose $(\boldsymbol{0}, \boldsymbol{0}, \theta, 0)$ with $\theta \in W^{1,(p^*)'}(\Omega)$ satisfying $\theta = 0$ on $\Gamma_0$ and $\theta = \delta_{ij}$ on $\Gamma_j$ for all $1 \leqslant j \leqslant I$ and a fixed $1 \leqslant i \leqslant I$. Then: 
\begin{eqnarray*}
0 = \langle \bu, \nabla \theta \rangle_{\Omega_{p^*,p}} = \int_\Omega \bu \cdot \nabla \theta \, dx = \int_{\Gamma} \theta \bu \cdot \bn \, d\sigma - \int_\Omega \vdiv \bu \, \theta \, dx = \int_{\Gamma_i} \bu \cdot \bn \, d\sigma 
\end{eqnarray*}
For the condition $\displaystyle\int_{\Gamma_{i}} \bb \cdot \bn\,d\sigma = 0$ for all $1 \leqslant i \leqslant I$, we set $\tilde{\bb} = \bb - \displaystyle\sum_{i=1}^I \langle \bb \cdot \bn, 1 \rangle_{\Gamma_i} \nabla q_i^N$. Observe that by the definition of $q_i^N$,  $\tilde{\bb}$ is also solution of \eqref{Variational lemma formulation} and satisfies the condition $\langle \tilde{\bb} \cdot \bn, 1 \rangle_{\Gamma_i} = 0$.

\vspace{2mm} 

Next, taking test functions $(\boldsymbol{0},\boldsymbol{0},\theta,0)$ and $(\boldsymbol{0},\boldsymbol{0},0,\tau)$ with $\theta \in W^{1,(p^*)'}(\Omega)$ as above and $\tau \in \mathcal{D}(\Omega)$, we respectly recover $\vdiv \bu = 0$ and $\vdiv \bb = 0$ in $\Omega$. Besides, since $\bu, \bb \in \bH_0^{p^*,p}(\curl, \Omega)$, we have $\bu$ and $\bb$ belong to $\bX^{p}_{N}(\O)$. From Theorem \ref{injection continue X_N}, we deduce that $\bu, \bb \in \bW^{1,p}(\Omega)$. Thus, the estimate \eqref{estim u,b W^1,p for p<2} follows from \eqref{inequality injection in X_N non homogene cas general} and \eqref{estim u,b,P with duality p<2}. 
\vspace{2mm}
Finally, in order to prove that $P \in W^{1,r}(\Omega)$, we take the test functions $(\bv, \boldsymbol{0}, 0, 0)$ with $\bv \in \boldsymbol{\mathcal{D}}({\Omega})$, and we obtain as in the proof of Lemma \ref{Lemme equiv sol pour Lp*} that: 
\begin{eqnarray*}
\nabla P = \ff + \Delta \bu - (\curl \bw) \times \bu +  (\curl \bb) \times \bdd \quad \text{in} \, \, \Omega.
\end{eqnarray*}
Then taking the divergence, $P$ is solution of the following problem
\begin{eqnarray} \label{Delta P preuve prop}
\Delta P = \vdiv \ff + \vdiv ( (\curl \bb) \times \bdd - (\curl \bw) \times \bu) \quad \text{in} \, \, \Omega,\nonumber\\
P=P_0\quad \mathrm{on}\,\,\Gamma_0\quad \mathrm{and}\quad P=P_0+c_i\quad \mathrm{on}\,\,\Gamma_i.\hspace{1.3cm}
\end{eqnarray}
Since $\curl \bw \in \bL^\frac{3}{2}(\Omega)$ and $\bu \in \bW^{1,p}(\Omega) \hookrightarrow \bL^{p^*}(\Omega)$, then $(\curl \bw) \times \bu \in \bL^r(\Omega)$. Besides, $\curl \bb \in \bL^p(\Omega)$ and $\bdd \in \bW^{1,\frac{3}{2}}(\Omega) \hookrightarrow \bL^3(\Omega)$. So $(\curl \bb) \times \bdd \in \bL^r(\Omega)$. Hence, we obtain that $\Delta P \in W^{-1,r}(\Omega)$. Since $P_0$ belongs to $W^{1-1/r,r}(\Gamma)$, we deduce that the solution $P$ of \eqref{Delta P preuve prop} belongs to $\in W^{1,r}(\Omega)$. Moreover, it satisfies the estimate
\begin{eqnarray*}
\begin{aligned}
\norm{P}_{W^{1,r}(\Omega)} &\leqslant \norm{\vdiv \ff}_{W^{-1,r}(\Omega)} + \norm{\vdiv ( (\curl \bb) \times \bdd - (\curl \bw) \times \bu)}_{W^{-1,r}(\Omega)}+\norm{P_0}_{W^{1-1/r,r}(\Gamma)}
\end{aligned}
\end{eqnarray*}
Applying the characterization of $[\bH_0^{r',p'}(\curl, \Omega)]'$ given in Proposition \ref{charac dual H^r,p}, we write $\ff = \bF + \curl \boldsymbol{\Psi}$ with $\bF \in \bL^r(\Omega)$ and $\boldsymbol{\Psi} \in \bL^p(\Omega)$. So
\begin{eqnarray*}
\begin{aligned}
\norm{\vdiv \ff}_{W^{-1,r}(\Omega)} = \norm{\vdiv \bF}_{W^{-1,r}(\Omega)} = \sup_{\theta \in W_0^{1,r'}(\Omega)} \frac{\abs{\langle \vdiv \bF, \theta \rangle}}{\norm{\theta}_{W^{1,r'}(\Omega)}} = \sup_{\theta \in W_0^{1,r'}(\Omega)} \frac{\abs{\langle \bF, \nabla \theta \rangle}}{\norm{\theta}_{W^{1,r'}(\Omega)}} \leqslant \norm{\bF}_{\bL^r(\Omega)},
\end{aligned}
\end{eqnarray*} 
which implies that 
\begin{eqnarray} \label{Estim div f}
\norm{\vdiv \ff}_{W^{-1,r'}(\Omega)} \leqslant \norm{\ff}_{[\bH_0^{r',p'}(\curl, \Omega)]'} 
\end{eqnarray}
% $P$ satisfies \eqref{estim P W^1,r for p<2}. Indeed, we have:
In the same way, we have:
\small
\begin{eqnarray}\label{estim RHS1 for estim pressure}
 \norm{\vdiv ( (\curl \bb) \times \bdd )}_{W^{-1,r}(\Omega)} &\leqslant& \norm{ (\curl \bb) \times \bdd}_{\bL^r(\Omega)}\leq  \norm{\curl \bb}_{\bL^p(\Omega)} \norm{\bdd}_{\bL^3(\Omega)}\nonumber\\&\leq & C_d  \norm{\bb}_{\bW^{1,p}(\Omega)} \norm{\bdd}_{\bW^{1,\frac{3}{2}}(\Omega)},
\end{eqnarray}
where $C_d$ is the constant related to the Sobolev embedding $\bW^{1,\frac{3}{2}}(\Omega) \hookrightarrow \bL^3(\Omega)$. 
Next, 
\begin{eqnarray}\label{estim RHS2 for estim pressure}
 \norm{\vdiv((\curl \bw) \times \bu)}_{W^{-1,r}(\Omega)} &\leqslant& \norm{(\curl \bw) \times \bu}_{\bL^r(\Omega)}\leqslant  \norm{\curl \bw}_{\bL^\frac{3}{2}(\Omega)} \norm{\bu}_{\bL^{p^*}(\Omega)} \nonumber\\&\leq& C  \norm{\curl \bw}_{\bL^\frac{3}{2}(\Omega)} \norm{\bu}_{\bW^{1,p}(\Omega)}, 
\end{eqnarray}
where we have used the Sobolev embedding $\bW^{1,p}(\Omega) \hookrightarrow \bL^{p^*}(\Omega)$. Using estimates \eqref{Estim div f}, \eqref{estim RHS1 for estim pressure}, \eqref{estim RHS2 for estim pressure} combined with the estimate \eqref{estim u,b W^1,p for p<2}, we obtain the estimate \eqref{estim P W^1,r for p<2} for the pressure. 
\end{proof}
The following result gives the regularity $\bW^{1,p}(\O)$ with $p<2$ when the divergence of the velocity field $\bu$ does not vanish.
\begin{corollary}\label{cor: weak W1,p for p<2}
  Let $\frac{3}{2}<p < 2$. Assume that $\ff, \bg \in [\bH_0^{r',p'}(\curl, \Omega)]'$, $P_0 \in W^{1-\frac{1}{r}, r}(\Gamma)$, $h\in W^{1,r}(\O)$ with the compatibility conditions \eqref{Condition compatibilite K_N Lp}-\eqref{condition div g=0 sol W^1,p linear MHD}, together with $\curl\bw\in\bL^{3/2}(\O)$ and $\bdd\in\bW^{1,3/2}_{\sigma}(\O)$ . Then the linearized problem \eqref{linearized MHD-pressure} has a unique solution $(\bu, \bb, P, \bc) \in \bW^{1,p}(\Omega) \times \bW^{1,p}(\Omega) \times W^{1,r}(\Omega) \times \R^I$. Moreover, we have the following estimates:  
 \begin{eqnarray}\label{estim cor u,b W^1,p for p<2}
 \begin{aligned} 
&\norm{\bu}_{\bW^{1,\,p}(\Omega)} + \norm{\bb}_{\bW^{1,\,p}(\Omega)}\leq  C(1 + \norm{\curl\bw}_{\bL^{3/2}(\Omega)} + \norm{\bdd}_{\bW^{1,3/2}(\O)})\Big(\norm{\ff}_{[\bH_0^{r',p'}(\curl,\Omega)]'}+\\& +  \norm{P_0}_{W^{1-\frac{1}{r},r}(\Gamma)} +\norm{\bg}_{[\bH_0^{r',p'}(\curl, \Omega)]}+(1 +  \norm{\curl\bw}_{\bL^{3/2}(\Omega)} + \norm{\bdd}_{\bW^{1,3/2}(\O)})\norm{h}_{W^{1,r}(\O)}\Big) 
\end{aligned}
\end{eqnarray}
and
 \begin{eqnarray}\label{estim cor P W^1,r for p<2}
 \begin{aligned} 
 &\norm{P}_{W^{1,r}(\Omega)} \leq C(1 + \norm{\curl\bw}_{\bL^{3/2}(\Omega)} + \norm{\bdd}_{\bW^{1,3/2}(\O)})^2\times\Big(\norm{\ff}_{[\bH_0^{r',p'}(\curl,\Omega)]'} +\\&+ \norm{\bg}_{[\bH_0^{r',p'}(\curl, \Omega)]}+ \norm{P_0}_{W^{1-\frac{1}{r},r}(\Gamma)}+(1+\norm{\curl\bw}_{\bL^{3/2}(\Omega)} + \norm{\bdd}_{\bW^{1,3/2}(\O)})\norm{h}_{W^{1,r}(\O)}\Big).
\end{aligned}
\end{eqnarray} 
\end{corollary}

\begin{proof}
We can reduce the non-vanishing divergence problem for the velocity to the case where $\vdiv \bu = 0$, by solving the Dirichlet problem:
\begin{eqnarray*}
\Delta \theta = h \quad \mathrm{in} \, \, \Omega \quad \mathrm{and} \quad \theta = 0 \, \, \mathrm{on} \, \, \Gamma.
\end{eqnarray*}
For $h\in W^{1,r}(\O)$, the solution $\theta$ belongs to $W^{3,r}(\O)$ and satisfies 
the estimate
\begin{eqnarray}\label{estim theta liftign h}
\norm{\theta}_{W^{3,r}(\O)}\leq C \norm{h}_{W^{1,r}(\O)}.
 \end{eqnarray}
Setting $\bz = \bu - \nabla \theta$, we obtain that $(\bz, \bb, P, \bc)$ is the solution of the problem treated in the Theorem \ref{thm: weak W1,p for p<2} with $\ff$ and $\bg$ replaced by $\tilde{\ff} = \ff + \nabla h - (\curl \bw) \times \nabla \theta$ and $\tilde{\bg} = \bg + \curl (\nabla \theta \times \bdd)$ respectively. 

\noindent Indeed, we have $\nabla h \in \bL^r(\Omega) \hookrightarrow [\bH_0^{r',p'}(\curl, \Omega)]'$, and since $\nabla \theta \in \bW^{2,r}(\Omega) \hookrightarrow \bL^{p^*}(\Omega)$, then $(\curl \bw) \times \nabla \theta \in \bL^r(\Omega) \hookrightarrow [\bH_0^{r',p'}(\curl, \Omega)]'$, and $\nabla \theta \times \bdd \in \bL^p(\Omega)$  so $\curl (\nabla \theta \times \bdd) \in [\bH_0^{r',p'}(\curl, \Omega)]'$. Therefore, $\tilde{\ff}, \tilde{\bg} \in [\bH_0^{r',p'}(\curl, \Omega)]'$. Besides, since we add a $\curl$, then $\tilde{\bg}$ still satisfies the compatibility conditions \eqref{Condition compatibilite K_N Lp}-\eqref{condition div g=0 sol W^1,p linear MHD}. Thus, applying Theorem \ref{thm: weak W1,p for p<2}, we have the estimate:
\begin{eqnarray*}
\begin{aligned}
\norm{\bz}_{\bW^{1,p}(\Omega)} + \norm{\bb}_{\bW^{1,p}(\Omega)} &\leqslant C \Big( 1 + \norm{\curl \bw}_{\bL^\frac{3}{2}(\Omega)} + \norm{\bdd}_{\bW^{1,\frac{3}{2}}(\Omega)} \Big) \\
&\times \Big( \norm{\tilde{\ff}}_{[\bH_0^{r',p'}(\curl, \Omega)]'} + \norm{\tilde{\bg}}_{[\bH_0^{r',p'}(\curl, \Omega)]'} + \norm{P_0}_{W^{1-\frac{1}{r},r}(\Gamma)} \Big)
\end{aligned}
\end{eqnarray*}
We want to control the terms on $\tilde{\ff}$ and $\tilde{\bg}$.
\begin{eqnarray*}
\begin{aligned}
\norm{\tilde{\ff}}_{[\bH_0^{r',p'}(\curl, \Omega)]'} &\leqslant \norm{\ff}_{[\bH_0^{r',p'}(\curl,\Omega)]'} + \norm{\nabla h}_{[\bH_0^{r',p'}(\curl,\Omega)]'} + \norm{(\curl \bw) \times \nabla \theta}_{[\bH_0^{r',p'}(\curl,\Omega)]'} \\
&\leqslant \norm{\ff}_{[\bH_0^{r',p'}(\curl,\Omega)]'} + \norm{\nabla h}_{\bL^r(\Omega)} + \norm{(\curl \bw) \times \nabla \theta}_{\bL^r(\Omega)} \\
&\leqslant \norm{\ff}_{[\bH_0^{r',p'}(\curl,\Omega)]'} + \norm{h}_{W^{1,r}(\Omega)} + C \norm{\curl \bw}_{\bL^\frac{3}{2}(\Omega)} \norm{\nabla \theta}_{\bW^{1,p}(\Omega)} 
\end{aligned}
\end{eqnarray*}
and 
\begin{eqnarray*}
\begin{aligned}
\norm{\tilde{\bg}}_{[\bH_0^{r',p'}(\curl,\Omega)]'} &\leqslant \norm{\bg}_{[\bH_0^{r',p'}(\curl,\Omega)]'} + \norm{\curl (\bdd \times \nabla \theta)}_{[\bH_0^{r',p'}(\curl,\Omega)]'} \\
&\leqslant \norm{\bg}_{[\bH_0^{r',p'}(\curl,\Omega)]'} + \norm{\bdd \times \nabla \theta}_{\bL^p(\Omega)} \\
&\leqslant \norm{\bg}_{[\bH_0^{r',p'}(\curl,\Omega)]'} + C_d \norm{\bdd}_{\bW^{1,\frac{3}{2}}(\Omega)} \norm{\nabla \theta}_{\bL^{p^*}(\Omega)}. 
\end{aligned}
\end{eqnarray*}
Therefore, from these two last estimates combined with \eqref{estim theta liftign h}, we obtain the estimate \eqref{estim cor u,b W^1,p for p<2}. 

Moreover, we also obtain from the Theorem \ref{thm: weak W1,p for p<2}: 
\begin{eqnarray*}
\begin{aligned}
\norm{P}_{W^{1,r}(\Omega)} &\leqslant C \Big(1 + \norm{\curl \bw}_{\bL^\frac{3}{2}(\Omega)} + \norm{\bdd}_{\bW^{1,\frac{3}{2}}(\Omega)} \Big)^2 \\
&\times \Big( \norm{\tilde{\ff}}_{[\bH_0^{r',p'}(\curl,\Omega)]'} + \norm{\tilde{\bg}}_{[\bH_0^{r',p'}(\curl, \Omega)]'} + \norm{P_0}_{W^{1-\frac{1}{r},r}(\Gamma)} \Big).
\end{aligned}
\end{eqnarray*}
Using the same arguments as previously, we can obtain the estimate \eqref{estim cor P W^1,r for p<2}.
\end{proof}

%%%%%%%%%%%%%%%%%%%%%%%%%%%%%  Amélioartion pression %%%%%%%%%%%%%%%%%%%%%%%%%%%
In this subsection, we always take $P_0 \in W^{1-\frac{1}{r},r}(\Gamma)$ to obtain $P \in W^{1,r}(\Omega)$. However, since the pressure is decoupled from the system, we can improve its regularity given in the previous results by choosing a convenient boundary condition. For this, we begin by the following regularity concerning the Stokes problem $(\mathcal{S_N})$  which is an improvement of \cite[Theorem 2.2.6]{Bri-thesis}: 

\begin{theorem} \label{thm:pression Stokes div non nulle}Let $\O$ be of class $\mathcal{C}^{1,1}$. Let us assume $\ff \in [\bH_0^{r',p'}(\curl, \Omega)]'$ and $h \in W^{1,r}(\Omega)$ with $1 \leqslant r \leqslant p$ and $\frac{1}{r} \leqslant \frac{1}{p} + \frac{1}{3}$. Then 
\begin{enumerate}
\item If $r < 3$ and $P_0 \in W^{-\frac{1}{r^*}, r^*}(\Gamma)$, the Stokes problem $(\mathcal{S_N})$ has a unique solution $(\bu,\,P) \in \bW^{1,p}(\O)\times L^{r^*}(\Omega)$. 
\item If $r \geqslant 3$ and $P_0 \in W^{-\frac{1}{q},q}(\Gamma)$ for any finite number $q > 1$, the Stokes problem $(\mathcal{S_N})$ has a unique solution $(\bu,\,P) \in \bW^{1,p}(\O)\times L^q(\Omega)$.
\end{enumerate}
\end{theorem}

\begin{proof}
Taking the divergence in the first equation of $(\mathcal{S_N})$, we have: 
\begin{eqnarray*}\begin{cases}
\Delta P = \vdiv \ff + \Delta h \qquad \mathrm{in} \, \, \Omega, \\P=P_0 \quad \mathrm{on} \, \, \Gamma_0 \quad \mathrm{and}\quad  P=P_0 + c_i \, \, \mathrm{on} \, \, \Gamma_i.\end{cases}
\end{eqnarray*} 
We split this problem in two parts:
find $P_1$ such that
\begin{eqnarray*}
(\mathcal{P}_1)\qquad \qquad \Delta P_1 = \vdiv \ff + \Delta h \quad\text{in} \, \, \Omega \qquad \mathrm{and}\quad P_1 = 0 \, \, \text{on} \, \, \Gamma,
\end{eqnarray*}
and find $P_2$ such that
\begin{eqnarray*}
(\mathcal{P}_2)\qquad \qquad \Delta P_2 = 0 \quad \text{in} \, \, \Omega,\,\,\quad P_2 = P_0 \, \, \text{on} \,\, \Gamma_0 \quad \mathrm{and}\quad P_2 = P_0 + c_i \, \, \text{on} \, \, \Gamma_i.
\end{eqnarray*}
We note that the regularity of $P_1$ is only dependent of $\vdiv \ff$ and $\Delta h$, and then we choose $P_0$ in order to recover for $P_2$ the same regularity as than for $P_1$. Then, we obtain the regularity of  $P$ by adding $P_1$ and $P_2$. 
Let us analyze  problem $(\mathcal{P}_1)$. Since $\ff \in [\bH_0^{r',p'}(\curl, \Omega)]'$, there exists $\bF \in \bL^r(\Omega)$ and $\boldsymbol{\Psi} \in \bL^p(\Omega)$ such that $\ff = \bF + \curl \boldsymbol{\Psi}$. So $\vdiv \ff = \vdiv \bF \in W^{-1,r}(\Omega)$. Moreover, we have $\Delta h \in W^{-1,r}(\Omega)$. Then, $\vdiv \ff + \Delta h$ belongs to  $\in W^{-1,r}(\Omega)$ which implies that problem $(\mathcal{P}_1)$  has a unique solution $P_1 \in W^{1,r}(\Omega)$. Next, we determine the regularity of $P_2$ with respect to the data $P_0$. We note that $P_0$ must be chosen so that the solution $P_2$ of
$(\mathcal{P}_2)$ could belong to a class of spaces containing spaces for $P_1$. We distinguish the following cases:

\underline{\textbf{Case $r<3$:}}
If $P_0 \in W^{-\frac{1}{r^*},r^*}(\Gamma)$ with $r^*=\frac{3r}{3-r}$, the solution $P_2$ of problem $(\mathcal{P}_2)$ belongs to $L^{r^*}(\Omega)$. Since $P_1 \in W^{1,r}(\Omega) \hookrightarrow L^{r^*}(\Omega)$, we deduce that $P=P_1+P_2$ belongs to $L^{r^*}(\O)$. 

\underline{\textbf{Case $r \geqslant 3$:}}

We have $P_1 \in W^{1,r}(\Omega) \hookrightarrow L^q(\Omega)$ for any finite number $q > 1$ if $r=3$, for $q = \infty$ if $r > 3$. Thus, taking $P_0 \in W^{-\frac{1}{q},q}(\Gamma)$ for any $q > 1$ we have $P_2 \in L^q(\Omega)$ and then $P \in L^q(\Omega)$. 
\end{proof}
\begin{rmk}
 Observe that, using the above splitting, if $P_0\in W^{1-1/r,r}(\Gamma)$, we have immediately $P\in W^{1,r}(\O)$. 
\end{rmk}

The regularity result given in Theorem \ref{thm:pression Stokes div non nulle} enables us to improve the pressure in the linearized MHD system \eqref{linearized MHD-pressure}. In particular, we have the following result
\begin{corollary}\label{Corollary improvement pressure}
Let $p>\frac{3}{2}$, $\ff, \bg \in [\bH_0^{r',p'}(\curl, \Omega)]'$, $h \in W^{1,r}(\Omega)$, $P_0 \in W^{-\frac{1}{r^*}, r^*}(\Gamma)$ satisfying the compatibility condition \eqref{Condition compatibilite K_N Lp}-\eqref{condition div g=0 sol W^1,p linear MHD}. We suppose that  
\begin{itemize}
\item $\curl \bw \in \bL^\frac{3}{2}(\Omega)$ and $\bdd \in \bW^{1,\frac{3}{2}}_{\sigma}(\Omega)$ if $\frac{3}{2} < p < 2$. 
\item $\curl \bw \in \bL^s(\Omega)$ and $\bdd \in \bW_\sigma^{1,s}(\Omega)$ if $p \geqslant 2$, where $s$ is defined in \eqref{def of s} 
\end{itemize}

Then, the solution $(\bu, \bb, P)$ given in Proposition \ref{proposition to improv estim W^1,p for p>2 with phi} and Theorem \ref{thm: weak W1,p for p<2} of the linearized MHD problem \eqref{linearized MHD-pressure} belongs to $\bW^{1,p}(\Omega) \times \bW^{1,p}(\Omega) \times L^{r^*}(\Omega) $. 
\end{corollary}

\begin{proof}
We are going to take advantage of the regularity results for the Stokes problem $(\mathcal{S_N})$ given in Theorem \ref{thm:pression Stokes div non nulle}. Then, we can rewrite \eqref{linearized MHD-pressure} in the following way:
Find $(\bu,P,\bc)$ such that 

\begin{equation*}\label{MHD with part1 Stokes problem}
(\mathcal{\widetilde{S_N}})
	\left\lbrace
		\begin{aligned}
		&   - \, \Delta \bu  + \nabla P = \widetilde{\ff}\quad \mathrm{and}\quad  \vdiv \bu = h  \quad \mathrm{in} \,\, \Omega, \\
       & \bu \times \bn = \textbf{0} \quad  \mathrm{on}\,\,\Gamma,\\
        & P = P_0 \quad \text{on} \, \,\Gamma_0\quad  \mathrm{and}\quad 
         P = P_0 + c_i \,\,\,\text{on} \,\, \Gamma_i,\\
         & \langle\bu \cdot \bn, 1\rangle_{\Gamma_i} = 0,\,\, 1\leq i \leq I,
		\end{aligned}\right.
		\end{equation*}
		with $\widetilde{\ff}=\ff - (\curl \bw) \times \bu +  (\curl \bb) \times \bdd $ and find $\bb$ solution of the following elliptic problem
				\begin{eqnarray*} \label{Probleme elliptique}
(\mathcal{\widetilde{E_N}})\begin{cases}
 \curl\curl\bb= \widetilde{\bg} \quad \mathrm{in} \,\, \Omega \quad \mathrm{and}\quad
\vdiv \bb = 0 \quad \mathrm{in} \,\, \Omega, \\
\bb \times \bn = 0 \quad \mathrm{on} \, \Gamma \\
\langle \bb \cdot \bn, 1 \rangle_{\Gamma_i} = 0, \quad \forall 1 \leqslant i \leqslant I,
\end{cases}
\end{eqnarray*} 
with $\widetilde{\bg}=\bg+\curl(\bu\times\bdd)$.
As in the previous proofs, we can easily verify that the assumptions on $\ff$, $\curl \bw$ and $\bdd$ imply that the term $\widetilde{\ff}$ belongs to $[\bH_0^{r',p'}(\curl, \Omega)]'$ for both cases $p<2$ and $p\geq 2$. Thanks to Theorem \ref{thm:pression Stokes div non nulle}, there exists a unique solution $(\bu,P,\bc)\in \bW^{1,p}(\Omega) \times L^{r^*}(\Omega) \times \R^I$ for the problem $(\mathcal{\widetilde{S_N}})$. Besides, the existence of $\bb$ is independent of the pressure. Indeed, $\widetilde{\bg}$ belongs to $[\bH_0^{r',p'}(\curl, \Omega)]'$ and satisfies the compatibility conditions \eqref{Condition compatibilite K_N Lp}-\eqref{condition div g=0 sol W^1,p linear MHD}. It follows from Lemma \ref{Lemme elliptique H r' p'} that problem $(\mathcal{\widetilde{E_N}})$ has a unique solution $\bb\in \bW^{1,p}(\O)$. 
\end{proof}

%%%%%%%%%%%%%%%%%%%%%%%%%%%%%% Solutions fortes p petit partie linéaire %%%%%%%%%%%%%%%%%%%%%%%%%%%%%%%

\subsection{Strong solution in $\bW^{\,2,p}(\Omega)$; $1<p< 6/5$}
The aim of this subsection is to complete the $L^p-$theory for the linearized MHD problem \eqref{linearized MHD-pressure} by the proof of strong solutions in $\bW^{2,p}(\O)$ with $1<p< 6/5$. One of the approach that we can use is to consider $\bw$ and $\bdd$ more regular in a first step and then remove this regularity in a second step. Since the proof with this approach highly mimics that of Theorem \ref{thm strong solution p>6/5}, we put it in the Appendix (see Section \ref{section Appendix}). We are going to give a shorter and different proof where we take advantage of the regularity $\bW^{1,p}(\O)$ with $1< p < 2$ for the linearized MHD problem \eqref{linearized MHD-pressure}.

\begin{theorem}[Strong solution in $\bW^{2,p}(\Omega)$ with $1 < p < \frac{6}{5}$] \label{thm: strong W2,p for 1<p<6/5}
Suppose that $\O$ is of class $\mathcal{C}^{2,1}$ and  $1 < p < \frac{6}{5}$. Assume $h=0$, and let $\ff, \bg \in \bL^p(\Omega)$, $P_0 \in W^{1-\frac{1}{p},p}(\Gamma)$, $\curl\bw\in \bL^{3/2}(\O)$ and $\bdd\in \bW^{1,3/2}_{\sigma}(\O)$ with the compatibility conditions \eqref{Condition div nulle}-\eqref{Condition compatibilite K_N Lp p>6/5}. 

Then the linearized problem \eqref{linearized MHD-pressure} has a unique solution $(\bu, \bb, P, c) \in \bW^{2,p}(\Omega) \times \bW^{2,p}(\Omega) \times W^{1,p}(\Omega) \times \R^I$ satisfying the following estimate:
\begin{eqnarray} \label{Estimation theorem strong solution p small}
\begin{aligned}
\norm{\bu}_{\bW^{2,p}(\Omega)} + \norm{\bb}_{\bW^{2,p}(\Omega)} + \norm{P}_{W^{1,p}(\Omega)} &\leqslant C \Big( 1 + \norm{\curl \bw}_{\bL^\frac{3}{2}(\Omega)} + \norm{\bdd}_{\bW^{1,\frac{3}{2}}(\Omega)} \Big)^2 \\
&\times \Big( \norm{\ff}_{\bL^p(\Omega)} + \norm{\bg}_{\bL^p(\Omega)} + \norm{P_0}_{W^{1-\frac{1}{p},p}(\Gamma)} \Big)
\end{aligned}
\end{eqnarray}

\end{theorem}

\begin{proof}
Observe that 
\begin{equation}\label{embedding to pass to p*}
 \bL^{p}(\O)\hookrightarrow [\bH^{r',(p^{*})'}(\curl,\O)]',\quad P_0\in W^{1-1/p,p}(\Gamma)\hookrightarrow W^{1-1/r,r}(\Gamma),\end{equation}
 where $\frac{1}{p^{*}}=\frac{1}{p}-\frac{1}{3}$ and $\frac{1}{r}=\frac{1}{p}+\frac{1}{3}$.  Since $1<p<6/5$, we have $\frac{3}{2} < p^{*} <2$. Then applying  the regularity $\bW^{1,p}(\O)$ for the MHD system \eqref{linearized MHD-pressure} for small values of $p$ (see Theorem \ref{thm: weak W1,p for p<2} with $h=0$), we can deduce the existence of a solution $(\bu,\bb,P,\bc)\in \bW^{1,p^{*}}(\O)\times \bW^{1,p^{*}}(\O)\times W^{1,r}(\O)\times \R^{I}$ satisfying the estimate
 \small
 \begin{eqnarray}\label{estim u,b W^1,p*}
&{}&\norm{\bu}_{\bW^{1,\,p^*}(\Omega)} + \norm{\bb}_{\bW^{1,\,p^*}(\Omega)}\nonumber\\&\leq&  C(1 + \norm{\curl\bw}_{\bL^{3/2}(\Omega)} + \norm{\bdd}_{\bW^{1,3/2}(\O)})\Big(\norm{\ff}_{\bL^{p}(\O)}+  \norm{P_0}_{W^{1-\frac{1}{r},r}(\Gamma)} +\norm{\bg}_{\bL^{p}(\O)}\Big), \qquad
\end{eqnarray}
where we have used the embedding \eqref{embedding to pass to p*}
and 
\small
 \begin{eqnarray*}\label{estim P W^1,r for p*}
 \begin{aligned} 
 &\norm{P}_{W^{1,r}(\Omega)} \!\leq\! C(1 + \norm{\curl\bw}_{\bL^{3/2}(\Omega)} + \norm{\bdd}_{\bW^{1,3/2}(\O)})^2\Big(\!\norm{\ff}_{\bL^{p}(\O)}+  \norm{P_0}_{W^{1-\frac{1}{r},r}(\Gamma)} +\norm{\bg}_{\bL^{p}(\O)}\!\Big), \,\,
 \end{aligned}
\end{eqnarray*} 
Since $\bW^{1,p^*}(\O)\hookrightarrow \bL^{p^{**}}(\O)$ with $\frac{1}{p^{**}}=\frac{1}{p^*}-\frac{1}{3}$, the terms $(\curl\bw)\times \bu$ and $(\curl\bb)\times \bdd$ belong to $\bL^{p}(\O)$. $(\bu,\,P,\bc)$ is then a solution of the problem $(\mathcal{\widetilde{S_N}})$ with $h=0$ and a RHS $\widetilde{\ff}$ in $\bL^{p}(\O)$. We deduce from Proposition \ref{thm solution W1,p and W2,p Stokes divu=h} that $(\bu,P)$ belongs to $\bW^{\,2,p}(\Omega)\times W^{\,1,p}(\Omega)$ and satisfies the estimate 
\begin{equation}\label{estim W^2,p b after p*}
 \norm{\bu}_{\bW^{\,2,p}(\Omega)} + \norm{P}_{W^{\,1,p}(\Omega)} \leqslant C_{S} \Big(\norm{\ff}_{[\bL^{p}(\Omega)} +\norm{(\curl\bw)\times\bu}_{\bL^p(\O)}+\norm{(\curl\bb)\times\bdd}_{\bL^p(\O)}+\norm{P_0}_{W^{1-\frac{1}{p}, p}(\Gamma)}\Big).
\end{equation}
Next, $\bb$ is a solution of the elleptic problem $(\mathcal{\widetilde{E_N}})$ with a RHS $\bg+\curl(\bu\times\bdd)$ which belongs to $\bL^p(\O)$. Thanks to Theorem \ref{thm strong soulution elliptic}, the solution $\bb$ belongs to $\bW^{2,p}(\O)$ and satisfies the estimate
\begin{equation}\label{estim W^2,p elliptic after p*}
  \norm{\bb}_{\bW^{\,2,p}(\Omega)} \leqslant C_{E} \big(\norm{\bg}_{[\bL^{p}(\Omega)}+\norm{\curl(\bu\times \bdd)}_{\bL^p(\O)}\big).
\end{equation}
Moreover, we have the following bounds 
\begin{eqnarray}
 \norm{(\curl\bw)\times\bu}_{\bL^p(\O)}\leq \norm{\curl\bw}_{\bL^{3/2}(\O)}\norm{\bu}_{\bL^{{p}^{**}}(\O)}\leq C \norm{\curl\bw}_{\bL^{3/2}(\O)}\norm{\bu}_{\bW^{1,{p^{*}}}(\O)},
\end{eqnarray}
\begin{eqnarray}
 \norm{(\curl\bb)\times\bdd}_{\bL^p(\O)}\leq \norm{\curl\bb}_{\bL^{{p}^{*}}(\O)}\norm{\bdd}_{\bL^{3}(\O)}\leq C \norm{\bdd}_{\bL^{3}(\O)}\norm{\bb}_{\bW^{1,p^*}(\O)},
\end{eqnarray}
and similarly,
\begin{eqnarray*}
 \norm{\curl(\bu\times\bdd)}_{\bL^{p}(\O)}\leq \norm{(\bdd\cdot\nabla)\bu }_{\bL^p(\O)}+\norm{(\bu\cdot\nabla)\bdd}_{\bL^p(\O)}\leq C\norm{\bu}_{\bW^{1,{p}^{*}}(\O)}(\norm{\bdd}_{\bL^3(\O)}+\norm{\nabla\bdd}_{\bL^{3/2}(\O)}).
\end{eqnarray*}

Collecting the above bounds together with \eqref{estim u,b W^1,p*} in \eqref{estim W^2,p b after p*}-\eqref{estim W^2,p elliptic after p*} leads to the bound \eqref{Estimation theorem strong solution p small}.

\end{proof}

%%%%%%%%%%%%%%%%%%%% FIN %%%%%%%%%%%%%%%%%%%%%%%
Proceeding as in the proof of Corollary \ref{cor: weak W1,p for p<2}, we can extend the previous result to the non-vanishing divergence case. The proof of the following result can be found in the Appendix (see Section \ref{section Appendix}). 

\begin{corollary}\label{corollary strong sol p<6/5 with h}
Let $\O$ be of class $\mathcal{C}^{2,1}$ and $1 < p < \frac{6}{5}$. Assume that $\ff, \bg \in \bL^p(\Omega)$, $P_0 \in W^{1-\frac{1}{p},p}(\Gamma)$, $h \in W^{1,p}(\Omega)$, $\curl\bw\in \bL^{3/2}(\O)$ and $\bdd\in \bW^{1,3/2}_{\sigma}(\O)$ with the compatibility conditions \eqref{Condition div nulle}-\eqref{Condition compatibilite K_N Lp p>6/5}. Then the linearized problem \eqref{linearized MHD-pressure} has a unique solution $(\bu,\bb,P,\bc) \in \bW^{2,p}(\Omega) \times \bW^{2,p}(\Omega) \times W^{1,p}(\Omega) \times \R^I$. Moreover, we have the following estimates: 
\begin{eqnarray} \label{estim cor sol fortes div non nulle p petit}
\begin{aligned}
&\norm{\bu}_{\bW^{2,p}(\Omega)} + \norm{\bb}_{\bW^{2,p}(\Omega)} + \norm{P}_{W^{1,p}(\Omega)} \\
&\leqslant C (1 + \norm{\curl \bw}_{\bL^\frac{3}{2}(\Omega)} + \norm{\bdd}_{\bW^{1,\frac{3}{2}}(\Omega)})^2 (\norm{\ff}_{\bL^p(\Omega)} + \norm{\bg}_{\bL^p(\Omega)} + \norm{P_0}_{W^{1-\frac{1}{p},p}(\Gamma)}\\
& + (1 + \norm{\curl \bw}_{\bL^\frac{3}{2}(\Omega)} + \norm{\bdd}_{\bW^{1,\frac{3}{2}}(\Omega)}) \norm{h}_{W^{1-\frac{1}{p},p}(\Gamma)})
\end{aligned}
\end{eqnarray}
\end{corollary}

%%%%%%%%%%%%%%%%%%%%%%%%%%%%%%%%%%%%%%%%%%%%%%%%%%%%% Début partie non linéaire %%%%%%%%%%%%%%%%%%%%%
 
\section{The nonlinear MHD system }
In this section, we consider the nonlinear problem and we study the existence of generalized and strong solutions for \eqref{eqn:MHD}.

\subsection{Existence and uniqueness: $L^{2}$-theory}\label{subsection L2 theory nonlinear}
In this subsection, we establish the existence and uniqueness for the weak solution in the Hilbert case for the problem \eqref{eqn:MHD}. 
The following result is one of the main results. First, we recall that $\langle\cdot,\cdot\rangle_{\O_{_{r,p}}}$ denotes the duality product between $[\bH_0^{r,p}(\curl, \Omega)]'$ and $\bH_0^{r,p}(\curl, \Omega)$ and $\langle\cdot,\cdot\rangle_{\Gamma}$ denotes the duality product between $H^{-1/2}(\Gamma)$ and $H^{1/2}(\Gamma)$.
\begin{theo}\emph{(Weak solutions of \eqref{eqn:MHD} system in $\bH^{1}(\O)$)}\label{thm:existence_weak_sol}. Let $\O$ be of classe $\mathcal{C}^{1,1}$ and let 
$$ \ff, \,\,\bg \in [\bH_0^{6,2}(\curl, \Omega)]',\quad h=0\quad \mathrm{ and}\quad P_0 \in H^{-\frac{1}{2}}(\Gamma)$$
satisfying the compatibility conditions 
\begin{eqnarray} 
\forall \, \bv \in \bK_N^2(\Omega),\quad\langle \bg, \bv \rangle_{\O_{_{6,2}}} = 0,  \label{Condition compatibilite K_N hilbert for non linear} \\
\vdiv \bg = 0 \quad \mathrm{in}  \,\, \Omega.\label{Condition div nulle for non linear}
 \end{eqnarray}
Then the \eqref{eqn:MHD} problem has at least one weak solution $
(\bu, \bb, P, \boldsymbol{\alpha}) \in \bH^1(\Omega) \times \bH^1(\Omega) \times L^2(\Omega) \times \R^{I}$ such that 
\begin{equation}\label{estim weak H1 solution MHD}
 \norm{\bu}_{\bH^1(\O)}+\norm{\bb}_{\bH^1(\O)}+\norm{P}_{\bL^2(\O)}\leq M,
\end{equation}
where $M=C(\norm{\ff}_{[\bH_0^{6,2}(\curl, \Omega)]'} + \norm{\bg}_{[\bH_0^{6,2}(\curl, \Omega)]'} + \norm{P_0}_{H^{-1/2}(\Gamma)}) $ with 
\begin{equation}\label{def alpha_i}
 \alpha_i=\langle \ff, \nabla q_i^N, \rangle_{\O_{_{6,2}}} - \langle P_0, \nabla q_i^N \cdot \bn\rangle_{\Gamma} +  \int_\Omega (\curl \bb) \times \bb \cdot \nabla q_i^N \, dx- \int_\Omega (\curl \bu) \times \bu \cdot \nabla q_i^N \, dx.
 \end{equation}
In addition, suppose that $\ff$, $\bg$ and $P_0$ are small in the sense that 
\begin{eqnarray}\label{condition uniqueness sol H^1 nonlinear}
C_1C_2^2 M \leq \frac{2}{3 C_{\mathcal{P}} ^2},
\end{eqnarray}
where ${C_{\mathcal{P}}}$ is the constant in \eqref{Poincaré inequality} and $C_1$, $C_2$ are given in \eqref{def C_1 and C_2}. Then the weak solution $(\bu,\bb,P)$ of \eqref{eqn:MHD} is unique.
\end{theo}
We first recall that the space $\bV_N(\O)$ denotes 
$$\bV_N(\Omega) := \{ \bv \in \bH^1(\Omega); \, \vdiv \bv = 0 \, \, \text{in} \, \Omega, \, \, \bv \times \bn = 0 \, \, \text{on} \, \Gamma, \, \langle \bv \cdot \bn, 1 \rangle_{\Gamma_i} = 0, \, \, \forall 1 \leqslant i \leqslant I \}$$
 and we give the following definition. 

\begin{defi}\label{definition weak solution nonlinear}
 Given  $\ff, \bg \in [\bH_0^{6,2}(\curl, \Omega)]'$ and $P_0 \in H^{-\frac{1}{2}}(\Gamma)$ with the compatibility conditions \eqref{Condition compatibilite K_N hilbert for non linear}-\eqref{Condition div nulle for non linear}, $(\bu,\bb, \alpha_i)\in \bV_N(\Omega) \times \bV_N(\Omega)\times \R$ is called a weak solution of \eqref{eqn:MHD} problem if it satisfies: for all  $(\bv, \boldsymbol{\Psi})\in \bV_N(\Omega) \times \bV_N(\Omega)$,
 \begin{eqnarray}
\label{FormVar1a}
\begin{aligned}
& \int_{\O}\curl \bu\cdot\curl \bv\,d\bx+ \int_{\O}(\curl \bu \times \bu)\cdot \bv\,d\bx -  \int_{\O}(\curl \bb \times \bb)\cdot \bv\,d\bx +   \int_{\O}\curl \bb\cdot \curl \boldsymbol{\Psi}\,d\bx \\
&+  \int_{\O}(\curl \boldsymbol{\Psi} \times \bb)\cdot \bu\,d\bx 
= \langle \ff, \bv \rangle_{\O_{_{6,2}}} + \langle \bg, \boldsymbol{\Psi} \rangle_{\O_{_{6,2}}} - \langle P_0, \bv \cdot \bn \rangle_{\Gamma_0} - \sum_{i=1}^I \langle P_0 + \alpha_i, \bv \cdot \bn \rangle_{\Gamma_i}
\end{aligned}
\end{eqnarray}
and
\begin{eqnarray}
\label{FormVar1b}
\begin{aligned}
\alpha_i = \langle \ff, \nabla q_i^N \rangle_{\O_{_{6,2}}}\! -\! \langle P_0, \nabla q_i^N \cdot \bn \rangle_{\Gamma} \!+\!  \int_{\O}(\curl \bb \times \bb)\cdot\nabla q_i^N\,d\bx \!-\! \int_{\O}(\curl \bu \times \bu)\cdot \nabla q_i^N\,d\bx.
\end{aligned}
\end{eqnarray}
\end{defi}
To interpret \eqref{FormVar1a}-\eqref{FormVar1b}, it is convenient to remove the constraint of fluxs of the test functions through $\Gamma_i$. In the following Lemma, we prove that \eqref{FormVar1a} can be extended to any test function $(\bv, \boldsymbol{\Psi})\in\bX_N^{2}(\Omega) \times \bX_N^{2}(\Omega)$. 

\begin{lemma} \label{Equivalence de formulations var}
Let $\ff, \bg \in [\bH_0^{6,2}(\curl, \Omega)]'$ and $P_0 \in H^{-\frac{1}{2}}(\Gamma)$ with the compatibility conditions \eqref{Condition compatibilite K_N hilbert for non linear}-\eqref{Condition div nulle for non linear}. Then, the following two statements are equivalent:

(i) $(\bu,\bb,\alpha_i)\in\bV_N(\Omega) \times \bV_N(\Omega)\times \R$ satisfies \eqref{FormVar1a}-\eqref{FormVar1b} for any $(\bv, \boldsymbol{\Psi})\in\bV_N(\Omega) \times \bV_N(\Omega)$.

(ii) $(\bu,\bb,\alpha_i)\!\in\bV_N(\Omega) \times \bV_N(\Omega)\times \R$ satisfies \eqref{FormVar1a}-\eqref{FormVar1b} for any $(\bv, \boldsymbol{\Psi})\in \bX_N^{2}(\Omega) \times \bX_N^{2}(\Omega)$ .
\end{lemma}

\begin{proof} 

Since $\bV_N(\Omega) \subset \textbf{\textit{X}}_N^{2}(\Omega)$, then (ii) implies (i). Conversely, let $(\bu,\bb,\alpha_i)\in\bV_N(\Omega) \times \bV_N(\Omega)\times \R$ satisfying \eqref{FormVar1a}-\eqref{FormVar1b} for any $(\bv, \boldsymbol{\Psi})\in\bV_N(\Omega) \times \bV_N(\Omega)$ and we want to show it implies (ii). The proof is similar to than given for the linearized problem (see Proposition \ref{Proposition Equivalence sol hilb sol var}). \\
Let $(\widetilde{\boldsymbol{\Psi}},\widetilde{\bv})\in\textbf{\textit{X}}_N^{2}(\Omega)\times\textbf{\textit{X}}_N^{2}(\Omega)$. We set 

$$\boldsymbol{\Psi} = \widetilde{\boldsymbol{\Psi}} - \sum_{i=1}^I \langle \widetilde{\boldsymbol{\Psi}} \cdot \bn, 1 \rangle_{\Gamma_i} \nabla q_i^N\quad \mathrm{and}\quad \bv = \widetilde{\bv} - \sum_{i=1}^I \langle \widetilde{\bv} \cdot \bn, 1 \rangle_{\Gamma_i} \nabla q_i^N.$$ 
\vspace{2mm}
Then, $(\boldsymbol{\Psi},\bv)$ belongs to $\bV_N(\Omega) \times \bV_N(\Omega)$. Replacing in (\ref{FormVar1a})-(\ref{FormVar1b}), we obtain 
\begin{eqnarray}
\label{EqLem1}
\begin{aligned}
&  \int_{\O}\curl \bb\cdot \curl \widetilde{\boldsymbol{\Psi}}\,d\bx +  \int_{\O}(\curl \widetilde{\boldsymbol{\Psi}} \times \bb)\cdot \bu\,d\bx - \langle \bg,\widetilde{\boldsymbol{\Psi}} \rangle_{\O_{_{6,2}}} 
+ \int_{\O}\curl \bu\cdot\curl \widetilde{\bv}\,d\bx \\
&+ \int_{\O}(\curl \bu \times \bu)\cdot\widetilde{\bv}\,d\bx -  \int_{\O}(\curl \bb \times \bb)\cdot \widetilde{\bv}\,d\bx  -\langle \ff, \widetilde{\bv} \rangle_{\O_{_{6,2}}}  + \langle P_0, \widetilde{\bv} \cdot \bn \rangle_{\Gamma_0} + \sum_{i=1}^I \langle P_0 + \alpha_i, \widetilde{\bv} \cdot \bn \rangle_{\Gamma_i} \\
&= \sum_{i=1}^I \langle \widetilde{\boldsymbol{\Psi}} \cdot \bn, 1 \rangle_{\Gamma_i} \Big(   \int_{\O}\curl \bb\cdot\underbrace{\curl(\nabla q_i^N)}_{=0}\,d\bx +  \int_{\O} \underbrace{(\curl \nabla q_i^N)}_{=0} \times \bb\cdot \bu\,d\bx - \underbrace{\langle \bg, \nabla q_i^N \rangle_{\O_{_{6,2}}}}_{=0 \,\mathrm{due\,\, to}\,\,\eqref{Condition compatibilite K_N hilbert for non linear}}\Big) \\
&+\sum_{i=1}^I \langle \widetilde{\bv} \cdot \bn, 1 \rangle_{\Gamma_i} \Big( \int_ {\O} \curl \bu\cdot \underbrace{(\curl \nabla q_i^N)}_{=0} \, d\bx -  \int_{\O}(\curl \bb \times \bb)\cdot \nabla q_i^N\,d\bx 
+ \int_{\O}(\curl \bu \times \bu)\cdot \nabla q_i^N\,d\bx \\
&- \langle \ff, \nabla q_i^N \rangle_{\O_{_{6,2}}} +\langle P_0, \nabla q_i^N \cdot \bn \rangle_{\Gamma_0}
+ \sum_{j=1}^I  \langle P_0 + \alpha_j, \nabla q_i^N \cdot \bn \rangle_{\Gamma_j} \Big)
\end{aligned}
\end{eqnarray}
So, we have 
\begin{eqnarray*}
\begin{aligned}
&\sum_{i=1}^I \langle \tilde{\bv} \cdot \bn, 1 \rangle_{\Gamma_i} \Big ( -  \int_{\O}(\curl \bb \times \bb)\cdot \nabla q_i^N\,d\bx + \int_{\O}(\curl \bu \times \bu)\cdot\nabla q_i^N\,d\bx - \langle \ff, \nabla q_i^N \rangle_{\O_{_{6,2}}} \\
& +\langle P_0, \nabla q_i^N \cdot \bn \rangle_{\Gamma_0} + \sum_{j=1}^I \langle P_0 + \alpha_j, \nabla q_i^N \cdot \bn \rangle_{\Gamma_j} \Big)
\end{aligned}
\end{eqnarray*}
Since $q_i^N$ satisfies \eqref{Definition des q_i^N}, we have in particular that  $\sum_{j=1}^I \langle \alpha_j, \nabla q_i^n \cdot \bn \rangle_{\Gamma_j} = \alpha_i \delta_{ij}$. Then, we obtain 
\begin{eqnarray}\label{calcul on alpha_i}
\sum_{j=1}^I \langle P_0 + \alpha_j, \nabla q_i^N \cdot \bn \rangle_{\Gamma_j} = \sum_{j=1}^I \langle P_0, \nabla q_i^N \cdot \bn \rangle_{\Gamma_j} + \alpha_i
\end{eqnarray} 

Replacing \eqref{FormVar1b} in \eqref{calcul on alpha_i}, we obtain from \eqref{EqLem1} that  
\begin{eqnarray}
\label{FormVar1a with phi and v tild}
\begin{aligned}
& \int_{\O}\curl \bu\cdot\curl \widetilde{\bv}\,d\bx+ \int_{\O}(\curl \bu \times \bu)\cdot \widetilde{\bv}\,d\bx -  \int_{\O}(\curl \bb \times \bb)\cdot \widetilde{\bv}\,d\bx +   \int_{\O}\curl \bb\cdot \curl \widetilde{\boldsymbol{\Psi}}\,d\bx \\
&+  \int_{\O}(\curl \widetilde{\boldsymbol{\Psi}} \times \bb)\cdot \bu\,d\bx 
= \langle \ff, \widetilde{\bv} \rangle_{\O_{_{6,2}}} + \langle \bg, \widetilde{\boldsymbol{\Psi}} \rangle_{\O_{_{6,2}}} - \langle P_0, \bv \cdot \bn \rangle_{\Gamma_0} - \sum_{i=1}^I \langle P_0 + \alpha_i, \widetilde{\bv} \cdot \bn \rangle_{\Gamma_i}
\end{aligned}
\end{eqnarray}
which is \eqref{FormVar1a} with test functions $(\widetilde{\boldsymbol{\Psi}},\widetilde{\bv})\in\textbf{\textit{X}}_N^{2}(\Omega)\times\textbf{\textit{X}}_N^{2}(\Omega)$. This completes the proof. 

\end{proof}

\vspace{4mm}
%%%%%%%%%%%%%%%%%%%%%%%% Deja dit avant %%%%%%%%%%%%%%%%%%%%%%%%%%%%
% \begin{remark} \label{Remarque compatibilité curl}
% Avant de procéder à la preuve du théorème, notons que la condition $\langle \bg, \bv \rangle_\Omega = 0$ pour tout $\bv \in \bK_N^6(\Omega)$ est naturelle. En effet, si $(\bu, \bb, P, C) \in \bH^1(\Omega) \times \bH^1(\Omega) \times L^2(\Omega)/\R \times \R^I$ est solution de \eqref{MHD}, alors: 
% 
% \begin{eqnarray*}
% \begin{aligned}
% \langle \bg, \bv \rangle_\Omega &=   \langle \curl \curl \bb, \bv \rangle_{[\bH_0(\curl, \Omega)]' \times \bH_0(\curl, \Omega)} \\
% &-  \langle \curl (\bu \times \bb), \bv \rangle_{[\bH_0(\curl, \Omega)]' \times \bH_0(\curl, \Omega)}
% \end{aligned}
% \end{eqnarray*}
% 
% Par densit\'e de $\boldsymbol{D}(\Omega)$ dans $\bH_0(\curl, \Omega)$ et intégration par partie, on obtient: 
% 
% \begin{eqnarray*}
% \langle \bg, \bv \rangle_\Omega =   \int_\Omega \curl \bb \cdot \curl \bv \, dx -  \int_\Omega (\bu \times \bb) \cdot \curl \bv \, dx 
% \end{eqnarray*}
% 
% Or, $\bK_N^6(\Omega)$ est engendré par des gradients, donc: 
% 
% \begin{eqnarray*}
% \langle \bg, \bv \rangle_\Omega = 0, \quad \forall \, \bv \in \bK_N^6(\Omega)
% \end{eqnarray*}
% \end{remark}
%%%%%%%%%%%%%%%%%%%%%%%%%%%%%%%%%%%%%%%%%%%%%%%%%%%%%%%%%%%%%%%%%%%%%%%%%%%%
%%%%%%%%%%%%%%%%%%%%%%%%%%%%%%%%%%%%%%%%%%%%%%%%%%%%%%%%%%%%%%%%%%%%%%
We can now prove the following result

\begin{theorem}\label{Thm equi def and MHD nonlinear}
 Let $\ff, \bg \in [\bH_0^{6,2}(\curl, \Omega)]'$ and $P_0 \in H^{-\frac{1}{2}}(\Gamma)$ with the compatibility conditions \eqref{Condition compatibilite K_N hilbert for non linear}-\eqref{Condition div nulle for non linear}. Then, the following two statements are equivalent:

 \noindent \textbf{(i)} $(\bu,\bb,P,\alpha_i)\in\bH^{1}(\Omega) \times \bH^{1}(\Omega)\times L^{2}(\O)\times \R$ is a solution of \eqref{eqn:MHD},
 
\noindent \textbf{(ii)} $(\bu,\bb,\alpha_i)\in\bV_N(\Omega) \times \bV_N(\Omega)\times \R$ is a weak solution of \eqref{eqn:MHD}, in the sense of Definition \eqref{definition weak solution nonlinear}.

\end{theorem}

\begin{proof}
The proof that \textbf{(i)} implies \textbf{(ii)} is very similar to that of Proposition \ref{Proposition Equivalence sol hilb sol var}, hence we omit it. Let $(\bu,\bb,\alpha_i)\in\bV_N(\Omega) \times \bV_N(\Omega)\times \R$ satisfying \eqref{FormVar1a}-\eqref{FormVar1b} for any $(\bv,\boldsymbol{\Psi}) \in \bV_N(\Omega) \times \bV_N(\Omega)$. Due to Lemma \ref{Equivalence de formulations var}, we have that $(\bu,\bb,\alpha_i)\in\bV_N(\Omega) \times \bV_N(\Omega)\times \R$ satisfies also \eqref{FormVar1a}-\eqref{FormVar1b} for any $(\bv,\boldsymbol{\Psi}) \in \bX_N^{2}(\Omega) \times \bX_N^{2}(\Omega)$. 
Choosing $\bv \in \boldsymbol{\mathcal{D}}_\sigma(\Omega)$  and $\boldsymbol{\Psi} = \textbf{0}$ as test functions in \eqref{FormVar1a}, we have 
\begin{eqnarray}
\nonumber
\begin{aligned}
\langle -  \Delta \bu + (\curl \bu) \times \bu -  (\curl \bb) \times \bb - \ff, \bv \rangle_{\boldsymbol{\mathcal{D}}'(\Omega) \times \boldsymbol{\mathcal{D}}(\Omega)} = 0.
\end{aligned}
\end{eqnarray}
So, by De Rham's theorem, there exists an unique $ P \in \bL^{2}(\O)$ such that 
\begin{eqnarray}
\label{ExistPress}
\begin{aligned}
-  \Delta \bu +\nabla P+ (\curl \bu) \times \bu -  (\curl \bb) \times \bb - \ff =0 \quad \mathrm{in}\,\,\Omega.
\end{aligned}
\end{eqnarray}
Next,  choosing $(\boldsymbol{0}, \boldsymbol{\Psi})$ with $\boldsymbol{\Psi} \in \boldsymbol{\mathcal{D}}_\sigma(\Omega)$ in \eqref{FormVar1a}, we have
\begin{eqnarray*}
\langle   \curl \curl \bb -  \curl(\bu \times \bb) - \bg, \boldsymbol{\Psi} \rangle_{\boldsymbol{\mathcal{D}'}(\Omega) \times \boldsymbol{\mathcal{D}}(\Omega)}= 0.
\end{eqnarray*}
Then, applying \cite[Lemma 2.2]{Pan}, we have $\chi \in L^{2}(\Omega)$ defined uniquely up to an additive constant such that 
\begin{eqnarray*}
  \curl \curl \bb -  \curl (\bu \times \bb) - \bg = \nabla \chi \quad \mathrm{in}\,\,\O \qquad \mathrm{and}\qquad \chi=0\quad \mathrm{on}\,\,\Gamma
\end{eqnarray*}
Since $\chi$ is solution of the following harmonic problem 
\begin{eqnarray*}
\Delta \chi = 0 \quad \mathrm{in} \, \, \Omega \qquad \mathrm{and} \qquad \chi = 0 \quad \mathrm{on} \, \, \Gamma. 
\end{eqnarray*}
We deduce that $\chi = 0$ in $\Omega$ which gives the second equation in \eqref{eqn:MHD}. Moreover, by the fact that $\bu$ and $\bb$ belong to the space $\bV_{N}(\O)$, we have $\vdiv\bu=\vdiv\bb=0$ in $\O$ and $\bu\times\bn=\bb\times\bn=\mathbf{0}$ on $\Gamma$.  The proof of the boundary conditions on the pressure is fairly similar to that given in \cite[Proposition 3.7]{AS_DCDS}.
\end{proof}

\textbf{Proof of Theorem \ref{thm:existence_weak_sol}:}

We use the Schauder fixed point Theorem. We make use of the product space $\textit{\textbf{Z}}_N(\O)=\bV_{N}(\O)\times\bV_{N}(\O)$ defined in \eqref{norm on Z_N}. We define the mapping  $G : \textit{\textbf{Z}}_N(\O) \rightarrow \textit{\textbf{Z}}_N(\O)$ such that  $G(\textit{\textbf{w}},\textit{\textbf{d}})=(\bu,\bb)$ with  $(\bu,\bb)\in\bZ_N(\O)$ a solution of the linearized problem \eqref{Var form compact}. By Theorem \ref{thm weak solution H1 linear MHD}, for each pair $(\bw,\bdd)\in\bZ_N(\O)$ the solution $(\bu,\bb)\in \bZ_N(\O)$ of problem \eqref{Var form compact} exists, is unique and satisfies the following estimate:
\begin{eqnarray}\label{estim u,b H^1 for non linear}
\norm{(\bu,\bb)}_{_{\textbf{\textit{Z}}_N(\O)}}^{2}\!\!\!=\!(\norm{\bu}_{_{\bH^1(\Omega)}}^{2}\!\! +\! \norm{\bb}_{_{\bH^1(\Omega)}}^2)^{1/2} \!\leqslant C\! \big(\norm{\ff}_{_{[\bH_0^{6,2}(\curl, \Omega)]'}}\! \!\!+ \norm{\bg}_{_{[\bH_0^{6,2}(\curl, \Omega)]'}} \!\!\!+\! \norm{P_{0}}_{_{H^{-\frac{1}{2}}(\Gamma)}}\!\!\!\!\!\!\big):=M\quad 
\end{eqnarray}
for some constant $C>0$ independent of $\bw$ and $\bdd$.

We define the ball $$\bB_r =\displaystyle\{ (\bv, \boldsymbol{\Psi}) \in \textit{\textbf{Z}}_N(\O);\quad \,\, \norm{(\bv, \boldsymbol{\Psi})}_{\textbf{\textit{Z}}_N(\O)}\leqslant r \displaystyle\},$$
where $r=M=C \big(\norm{\ff}_{[\bH_0^{6,2}(\curl, \Omega)]'} + \norm{\bg}_{[\bH_0^{6,2}(\curl, \Omega)]'} + \norm{P_{0}}_{H^{-\frac{1}{2}}(\Gamma)}\big)$. By the definition  of $G$ and \eqref{estim u,b H^1 for non linear}, it follows that 
$G(\bB_r) \subset \bB_r$ and then $G$ is a mapping of the ball $\bB_r$ into itself. Now, we want to prove that the operator $G$ is compact on $\bB_r$.  For this, let  
$\{(\bw_k, \bdd_k)\}_{k\geq 1}$ be an arbitrary sequence in $\bB_r$. Since $\bH^{1}(\O)$ is reflexive, there exists a subsequence still denoted $\{(\bw_k, \bdd_k)\}_{k\geq 1}$ and a pair $(\bw,\bdd)$ in $\bB_r$ such that $(\textbf{\textit{w}}_k, \textbf{\textit{d}}_k)$ converges weakly to  $(\textbf{\textit{w}},\textbf{\textit{d}})$ in $\bH^{1}(\O)\times\bH^{1}(\O)$ as $k\rightarrow \infty$. We set $(\bu_k,\bb_k)=G(\textbf{\textit{w}}_k, \textbf{\textit{d}}_k)$ and $(\widetilde{\bu},\widetilde{\bb})=G(\bw,\bdd)$. We have to prove that $(\textbf{\textit{u}}_k, \textbf{\textit{b}}_k)$ converges strongly to  $(\widetilde{\bu},\widetilde{\bb})$ in $\bH^{1}(\O)\times\bH^{1}(\O)$ as $k\rightarrow \infty$.
By the compactness of the embedding $\bH^1(\Omega) \hookrightarrow \bL^4(\Omega)$, we have that $\textbf{\textit{w}}_k \rightarrow \textbf{\textit{w}}$ strongly in $\bL^4(\O)$ and  $\textbf{\textit{d}}_k \rightarrow \textbf{\textit{d}}$ strongly in $\bL^4(\O)$.
By definition $(\bu_k,\bb_k)$ and $(\widetilde{\bu},\widetilde{\bb})$ are solutions of
\begin{eqnarray}
\nonumber
\begin{aligned}
a((\bu_k,\bb_k),(\bv,\boldsymbol{\Psi})) + a_{\textbf{\textit{w}}_k,\textbf{\textit{d}}_k}((\bu_k,\bb_k),(\bv,\boldsymbol{\Psi})) = \mathcal{L} (\bv,\boldsymbol{\Psi}),
\end{aligned}
\end{eqnarray}
and 
\begin{eqnarray}\label{non linear problem on u tild and b tild}
\begin{aligned}
a((\widetilde{\bu},\widetilde{\bb}),(\bv,\boldsymbol{\Psi})) + a_{\textbf{\textit{w}},\textbf{\textit{d}}}((\widetilde{\bu},\widetilde{\bb}),(\bv,\boldsymbol{\Psi})) = \mathcal{L} (\bv,\boldsymbol{\Psi}). 
\end{aligned}
\end{eqnarray} 
where the forms $a$ and $a_{\textbf{\textit{w}},\textbf{\textit{d}}}$ are defined in \eqref{forms a}. By subtracting the above problems, we obtain 
\begin{eqnarray}\label{Compacity}
&{}& (\curl (\bu_k-\widetilde{\bu}), \curl \bv) +   (\curl (\bb_k-\widetilde{\bb}),\curl \boldsymbol{\Psi})+ ((\curl \textbf{\textit{w}}_k) \times \bu_k - (\curl \textbf{\textit{w}}) \times \widetilde{\bu}, \bv)\hspace{1cm}\nonumber\\& +\!\!&  (\bu_k, (\curl \boldsymbol{\Psi}) \times \textbf{\textit{d}}_k) -  (\widetilde{\bu}, (\curl \boldsymbol{\Psi}) \times \textbf{\textit{d}})-  ((\curl \bb_k) \times \textbf{\textit{d}}_k, \bv) +  ((\curl \widetilde{\bb}) \times \textbf{\textit{d}}, \bv) = 0\hspace{2cm}
\end{eqnarray}
Since 
\begin{eqnarray*}
((\curl \textbf{\textit{w}}_k) \times \bu_k - (\curl \textbf{\textit{w}}) \times \widetilde{\bu},\bv) = (\curl \textbf{\textit{w}} \times (\bu_k - \widetilde{\bu}), \bv) + (\curl (\textbf{\textit{w}}_k - \textbf{\textit{w}}) \times \bu_k, \bv)
\end{eqnarray*}
and
\begin{eqnarray*}
&{}& (\bu_k, (\curl \boldsymbol{\Psi}) \times \textbf{\textit{d}}_k) -  (\widetilde{\bu}, (\curl \boldsymbol{\Psi}) \times \textbf{\textit{d}}) -  ((\curl \bb_k) \times \textbf{\textit{d}}_k, \bv) + ((\curl \widetilde{\bb}) \times \textbf{\textit{d}}, \bv) \\
&=&  (\bu_k - \widetilde{\bu}, (\curl \boldsymbol{\Psi}) \times \textbf{\textit{d}}) -  (\bu_k, (\curl \boldsymbol{\Psi}) \times \textit{\textbf{d}}) -  (\curl (\bb_k - \widetilde{\bb}) \times \textit{\textbf{d}}, \bv) \\
&+& ((\curl \bb_k) \times \textit{\textbf{d}}, \bv) +  (\bu_k, (\curl \boldsymbol{\Psi}) \times \textit{\textbf{d}}_k) -  ((\curl \bb_k) \times \textbf{\textit{d}}_k, \bv).
\end{eqnarray*}
Then, replacing in \eqref{Compacity}, we obtain
\begin{eqnarray}
\nonumber
\begin{aligned}
&a((\bu_k - \widetilde{\bu}, \bb_k - \widetilde{\bb}),(\bv,\boldsymbol{\Psi})) + a_{\textbf{\textit{w}},\textbf{\textit{d}}}((\bu_k - \widetilde{\bu}, \bb_k - \widetilde{\bb}),(\bv,\boldsymbol{\Psi})) \\
&= -(\curl (\textbf{\textit{w}}_k - \textbf{\textit{w}}) \times \bu_k, \bv) +  (\bu_k, \curl \boldsymbol{\Psi} \times (\textbf{\textit{d}}-\textbf{\textit{d}}_k))-  (\curl \bb_k \times (\textit{\textbf{d}} - \textit{\textbf{d}}_k),\bv)
\end{aligned}
\end{eqnarray}

By H\"{o}lder inequality and the fact that $(\bu_k,\bb_k)$ belongs to  $\bB_r$, we have
 \begin{eqnarray}
\nonumber
\begin{aligned}
\abs{(\curl(\textbf{\textit{w}}_k - \textbf{\textit{w}}) \times \bu_k, \bv)} &\leqslant C_2\norm{\textbf{\textit{w}}_k - \textbf{\textit{w}}}_{\bL^4(\Omega)} \norm{\bu_k}_{\bH^{1}(\Omega)} \norm{\bv}_{\bH^{1}(\Omega)} \\
&\leqslant C_2 M\norm{\textbf{\textit{w}}_k - \textbf{\textit{w}}}_{\bL^4(\Omega)} \norm{\bv}_{\bH^1(\Omega)},
\end{aligned}
\end{eqnarray}

\begin{eqnarray}
\nonumber
\begin{aligned}
\abs{ (\bu_k, \curl \boldsymbol{\Psi} \times (\textbf{\textit{d}} - \textbf{\textit{d}}_k))} &\leqslant \norm{\bu_k}_{\bL^4(\Omega)} \norm{\curl \boldsymbol{\Psi}}_{\bL^2(\Omega)} \norm{\textbf{\textit{d}} - \textbf{\textit{d}}_k}_{\bL^4(\Omega)} \\
&\leqslant C_1C_2 M \norm{\boldsymbol{\Psi}}_{\bH^1(\Omega)} \norm{\textbf{\textit{d}} - \textbf{\textit{d}}_k}_{\bL^4(\Omega)}
\end{aligned}
\end{eqnarray}
and
\begin{eqnarray}
\nonumber
\begin{aligned}
\abs{ (\curl \bb_k \times (\textbf{\textit{d}} - \textbf{\textit{d}}_k), \bv)} &\leqslant C \norm{\curl \bb_k}_{\bL^2(\Omega)} \norm{\textbf{\textit{d}} - \textbf{\textit{d}}_k}_{\bL^4(\Omega)} \norm{\bv}_{\bL^4(\Omega)} \\
&\leqslant C_1C_2 M \norm{\textbf{\textit{d}} - \textbf{\textit{d}}_k}_{\bL^4(\Omega)} \norm{\bv}_{\bH^1(\Omega)}
\end{aligned}
\end{eqnarray}
where $C_1>0$ and $C_2>0$ are such that
\begin{eqnarray}\label{def C_1 and C_2}
\norm{\curl\bv}_{\bL^{2}(\O)}\leq C_1\norm{\bv}_{\bH^{1}(\O)}\quad  \mathrm{and}\quad \norm{\bv}_{\bL^{4}(\O)}\leq C_2\norm{\bv}_{\bH^{1}(\O)}.
\end{eqnarray}

Choosing $(\bv, \boldsymbol{\Psi}) = (\bu_k - \widetilde{\bu},\bb_k - \widetilde{\bb})$, thanks to the coercivity of the form $a$ in \eqref{coercivity A} and the fact that $a_{\textbf{\textit{w}},\textbf{\textit{d}}}((\bu_k - \widetilde{\bu},\bb_k - \widetilde{\bb}), (\bu_k - \widetilde{\bu}, \bb_k - \widetilde{\bb})) = 0$, we obtain 
\begin{eqnarray}
\begin{aligned}
& \frac{2}{C_{\mathcal{P}}^2}\big(\norm{\bu_k - \widetilde{\bu}}^2_{\bH^1(\Omega)} + \Vert\bb_k - \widetilde{\bb}\Vert^2_{\bH^1(\Omega)}\big) \leqslant \displaystyle\vert\,a((\bu_k - \widetilde{\bu}, \bb_k - \widetilde{\bb}),(\bu_k - \widetilde{\bu},\bb_k - \widetilde{\bb}))\,\vert \nonumber\\
&\leqslant  C_2 M \norm{\bw_k - \bw}_{\bL^4(\Omega)}\norm{\bu_k - \tilde{\bu}}_{\bH^1(\Omega)}+C_1C_2 M\norm{\bdd - \bdd_k}_{\bL^4(\Omega)} (\norm{\bu_k - \widetilde{\bu}}_{\bH^1(\Omega)} + \Vert\bb_k - \widetilde{\bb}\Vert_{\bH^1(\Omega)}) \nonumber \\
&\leqslant
\frac{1}{2} \norm{\bu_k - \tilde{\bu}}_{\bH^1(\Omega)}^2 + \frac{1}{2} \Vert\bb_k - \widetilde{\bb}\Vert_{\bH^1(\Omega)}^2 +C_2^2 M^2 \norm{\bw_k - \bw}_{\bL^4(\Omega)}^{2}   
+ \frac{3}{2} C_1^2 C_2^2 M^2\norm{\bdd_k - \bdd}_{\bL^4(\Omega)}^2.
\end{aligned}
\end{eqnarray}
So, we have
\begin{eqnarray*}
\norm{\bu_k - \widetilde{\bu}}^2_{\bH^1(\Omega)} +  \Vert\bb_k - \widetilde{\bb}\Vert^2_{\bH^1(\Omega)} \leqslant C\big( \norm{\bw_k - \bw}_{\bL^4(\Omega)}^2 + \norm{\bdd_k - \bdd}_{\bL^4(\Omega)}^2\big)\longrightarrow 0\quad \mathrm{as}\,\,k\longrightarrow 0,
\end{eqnarray*}
where $C=C_{\mathcal{P}}^2\max(C_2^2 M^2, \,\frac{3}{2}C_1^2 C_2^2  M^2 )$ is independent of $k$. Hence, this gives the compactness of $G$. From Schauder's theorem we then find that $G$ has a fixed point $(\widetilde{\bu},\widetilde{\bb})=G(\widetilde{\bu},\widetilde{\bb})\in \bB_r$. This fixed point is solution of \eqref{non linear problem on u tild and b tild}. Moreover, $(\widetilde{\bu},\widetilde{\bb})$ satisfies \eqref{estim u,b H^1 for non linear}.

Now, we establish the uniqueness of the solution of \eqref{eqn:MHD}. For this, let $(\bu_1,\bb_1,P_1)$ and $(\bu_2,\bb_2,P_2)$ in $\bV_N(\Omega) \times \bV_N(\Omega) \times L^2(\Omega)$ be two solutions of \eqref{eqn:MHD}. We set  $\bu = \bu_1 - \bu_2$, $\bb = \bb_1 - \bb_2$ et $P = P_1 - P_2$ and we want to prove that $\bu=\bb=\textbf{0}$ and $P=0$. Choose $\bv = \bu$ and $\boldsymbol{\Psi} = \bb$ in \eqref{FormVar1a}, then $(\bu,\bb)$ satisfies
 \begin{eqnarray}
 \nonumber
 \begin{aligned}
 & \norm{\curl \bu}_{\bL^2(\Omega)}^2 +   \norm{\curl \bb}_{\bL^2(\Omega)}^2 + ((\curl \bu_1) \times \bu_1 - (\curl \bu_2) \times \bu_2, \bu) \\
 &-  ((\curl \bb_1 )\times \bb_1 - (\curl \bb_2) \times \bb_2, \bu) +  ((\curl \bb) \times \bb_1, \bu_1) -  ((\curl \bb) \times \bb_2,\bu_2) = 0
 \end{aligned}
 \end{eqnarray}
Observe that 
\begin{eqnarray*}
((\curl \bu_1) \times \bu_1, \bu) - ((\curl \bu_2) \times \bu_2, \bu) = ((\curl \bu) \times \bu_1, \bu).
\end{eqnarray*}
and
\begin{eqnarray*}
\begin{aligned}
&- ((\curl \bb_1) \times \bb_1 - (\curl \bb_2) \times \bb_2, \bu)  +  ((\curl \bb) \times \bb_1, \bu_1) -  ((\curl \bb) \times \bb_2, \bu_2) \\
&=  ((\curl \bb) \times \bb, \bu_2) -  ((\curl \bb_2) \times \bb, \bu)
\end{aligned}
\end{eqnarray*}
which gives 
\begin{eqnarray}
\nonumber
\begin{aligned}
& \norm{\curl \bu}_{\bL^2(\Omega)}^2 +   \norm{\curl \bb}_{\bL^2(\Omega)}^2 + ((\curl \bu_1) \times \bu_1 - (\curl \bu_2) \times \bu_2, \bu) \\
&-  ((\curl \bb_1) \times \bb_1 - (\curl \bb_2) \times \bb_2, \bu) +  ((\curl \bb) \times \bb_1, \bu_1) -  ((\curl \bb) \times \bb_2, \bu_2) \\
&=  \norm{\curl \bu}^2 +   \norm{\curl \bb}^2 + ((\curl \bu) \times \bu_1, \bu) \\
&+  ((\curl \bb) \times \bb, \bu_2) -  ((\curl \bb_2) \times \bb, \bu)  
\end{aligned}
\end{eqnarray}
Then,
\begin{eqnarray}
\label{EqUni}
\begin{aligned}
 \norm{\curl \bu}^2 +   \norm{\curl \bb}^2 =  ((\curl \bb_2) \times \bb, \bu)-((\curl \bu) \times \bu_1, \bu) -  ((\curl \bb) \times \bb, \bu_2) 
\end{aligned}
\end{eqnarray}
We want to bound the terms in the RHS of \eqref{EqUni}. We have
 \begin{eqnarray}
\nonumber
\begin{aligned}
\abs{((\curl \bu) \times \bu_1, \bu)} &\leqslant \norm{\curl \bu}_{\bL^2(\Omega)} \norm{\bu_1}_{\bL^4(\Omega)} \norm{\bu}_{\bL^4(\Omega)}  \\
&\leqslant C_1C_2 \norm{\bu}_{\bH^1(\Omega)}^2 \norm{\bu_1}_{\bL^4(\Omega)} \\
&\leqslant C_1C_2^2 \norm{\bu}_{\bH^1(\Omega)}^2 \norm{\bu_1}_{\bH^1(\Omega)} \leqslant C_1C_2^2 M\norm{\bu}_{\bH_1(\Omega)}^2,
\end{aligned}
\end{eqnarray}
\begin{eqnarray}
\nonumber
\begin{aligned}
\abs{ ((\curl \bb) \times \bb, \bu_2)} &\leqslant \norm{\curl \bb}_{\bL^2(\Omega)} \norm{\bb}_{\bL^4(\Omega)} \norm{\bu_2}_{\bL^4(\Omega)} \\
&\leqslant C_1C_2^2 \norm{\bb}_{\bH^1(\Omega)}^2 \norm{\bu_2}_{\bH^1(\Omega)} \\
&\leqslant C_1 C_2^2  M \norm{\bb}_{\bH^1(\Omega)}^2,
\end{aligned}
\end{eqnarray}
and
\begin{eqnarray}
\nonumber
\begin{aligned}
\abs{ ((\curl \bb_2) \times \bb, \bu)} &\leqslant  \norm{\curl \bb_2}_{\bL^2(\Omega)} \norm{\bb}_{\bL^4(\Omega)} \norm{\bu}_{\bL^4(\Omega)} \\ 
&\leqslant C_1C_2^2 \norm{\bb_2}_{\bH^1(\Omega)} \norm{\bb}_{\bH^1(\Omega)} \norm{\bu}_{\bH^1(\Omega)} \\
&\leqslant \frac{1}{2}C_1C_2^2 M(\norm{\bu}_{\bH^1(\Omega)}^2 + \norm{\bb}_{\bH^1(\Omega)}^2).
\end{aligned}
\end{eqnarray}
Using these estimates in \eqref{EqUni} together with Poincar\'e's type inequality \eqref{Poincaré inequality}, we obtain 
\begin{eqnarray}
\nonumber
\begin{aligned}
\frac{1}{C_{\mathcal{P}}^2}\big(\norm{\bu}_{\bH^1(\Omega)}^2 +\norm{\bb}_{\bH^1(\Omega)}^2\big) &\leqslant \norm{\curl \bu}^2 + \norm{\curl \bb}^2\leqslant \frac{3}{2}C_1C_2^2 M ( \norm{\bu}^2_{\bH^1(\Omega)} + \norm{\bb}_{\bH^1(\Omega)}^2 ),
\end{aligned}
\end{eqnarray}
where $C_{\mathcal{P}}$ is the constant in \eqref{Poincaré inequality}. From this relation, we obtain
\begin{eqnarray*}
 (\frac{1}{C_{\mathcal{P}}^2}-\frac{3}{2}C_1C_2^2M)\big(\norm{\bu}_{\bH^1(\Omega)}^2 +\norm{\bb}_{\bH^1(\Omega)}^2\big)\leq 0.
\end{eqnarray*}

This, together with condition \eqref{condition uniqueness sol H^1 nonlinear}, implies that $\bu=\bb=\textbf{0}$. 
\vspace{2mm}
The construction of the pressure $P\in L^{2}(\O)$ follows from  De Rham's Theorem (see  \ref{Thm equi def and MHD nonlinear}).

%%%%%%%%%%%%%%%%%%%%%%%%%%%%%%%%%%%%%%%%%%%%%%%%%%%%%%%%%%%%%%%%%%%%%%%%%%%%%%%%%%%%%%%%%%%%%%%%%%%

\subsection{Weak solution: $L^{p}$-theory,\,\, $p\geq2$}
 In this subsection, we study the regularity of weak solution of system \eqref{eqn:MHD} in $L^p$-theory. We start with the case $p\geq 2$. The proof is done essentially using the existence of weak solution in the hilbertian case and a bootstrap argument. To take advantage of the regularity of the Stokes problem $(\mathcal{S_N})$ and the elliptic problem $(\mathcal{E_N})$, we can rewrite the \eqref{eqn:MHD} problem in the following way:
 %%%%%%%%%%%%%%%%%%%%%%%%%
 \begin{equation*}
  \begin{aligned}
\begin{cases}
       &   - \, \Delta \bu  + \nabla P = \ff - (\curl \bu) \times \bu +  (\curl \bb) \times \bdd \quad  \quad \mathrm{in} \,\, \Omega, \\
       & \vdiv \bu = h \quad \mathrm{in} \,\, \Omega,\\
       & \bu \times \bn = \textbf{0} \quad  \mathrm{on}\,\,\Gamma,\\
        & P = P_0 \quad \text{on} \, \,\Gamma_0\quad  \mathrm{and}\quad 
         P = P_0 + c_i \,\,\,\text{on} \,\, \Gamma_i,\\
         & \langle\bu \cdot \bn, 1\rangle_{\Gamma_i} = 0,\,\, 1\leq i \leq I,\end{cases}   
  \end{aligned}
 \end{equation*}
and 
 \begin{equation*}
  \begin{aligned}
\begin{cases}
\curl\curl\bb= {\bg}+ \curl(\bu\times\bb)\quad \mathrm{in} \,\, \Omega, \\ 
\vdiv \bb = 0 \quad \mathrm{in} \,\, \Omega, \\
\bb \times \bn = \textbf{0} \quad \mathrm{on} \, \Gamma \\
\langle \bb \cdot \bn, 1 \rangle_{\Gamma_i} = 0, \quad \forall 1 \leqslant i \leqslant I.
 \end{cases}   
  \end{aligned}
 \end{equation*}
The following result can be improved in the same way as in Corollary \ref{Corollary improvement pressure} by considering a data $P_0$ less regular. 

\begin{theorem}[Regularity $\bW^{1,p}(\O)$ with $p>2$]\label{thm:regularity_weak_sol1} Let $p>2$ and $r=\frac{3p}{3+p}$. Suppose that $\ff, \bg\in [\bH_0^{r',p'}(\curl,\Omega)]'$, $h=0$,  $P_0 \in W^{1-\frac{1}{r},r}(\Gamma)$ with the compatibility conditions \eqref{Condition compatibilite K_N Lp}-\eqref{condition div g=0 sol W^1,p linear MHD}. Then the weak solution for the \eqref{eqn:MHD} system given by Theorem \ref{thm:existence_weak_sol} satisfies $$(\bu, \bb, P) \in \bW^{1,p}(\Omega) \times \bW^{1,p}(\Omega) \times W^{1,r}(\Omega).$$ Moreover, we have the following estimate:
{\footnotesize \begin{eqnarray}\label{estim weak p>2 MHD}
\begin{aligned}
\!\norm{\bu}_{\bW^{1,p}(\Omega)} \!+\! \norm{\bb}_{\bW^{1,p}(\Omega)}\!+\! \norm{P}_{W^{1,r}(\Omega)} \!\leqslant\! C (\norm{\ff}_{(\bH_0^{r',p'}(\curl, \Omega))'} \!+\! \norm{\bg}_{(\bH_0^{r',p'}(\curl, \Omega))'} \!+\! \norm{P_0}_{W^{1/r',r}(\Gamma)}) 
\end{aligned}
\end{eqnarray}}
\end{theorem}

\begin{proof}
Since $p>2$, we have  $[\bH_0^{r', p'}(\curl, \Omega)]' \hookrightarrow [\bH_0^{6,2}(\curl, \Omega)]'$ and $W^{1-1/r,r}(\Gamma)\hookrightarrow H^{-1/2}(\Gamma)$. Thanks to Theorem \ref{thm:existence_weak_sol}, there exists $(\bu, \bb, P, \bc) \in \bH^1(\Omega) \times \bH^1(\Omega) \times L^2(\Omega) \times \R^I$ solution of \eqref{eqn:MHD}.
By using the embedding $\bH^1(\Omega) \hookrightarrow \bL^6(\Omega)$, it follows that $(\curl \bu) \times \bu$ and  $(\curl \bb) \times \bb$ belong to $\bL^\frac{3}{2}(\Omega)$. To apply the regularity of Stokes problem and obtain weak solutions $(\bu, P) \in \bW^{1,p}(\Omega) \times W^{1,r}(\Omega)$, we must justify that the RHS $\ff-(\curl\bu)\times\bu+(\curl\bb)\times\bb$ belongs to  $[\bH_0^{r',p'}(\curl, \Omega)]'$. Similarly, to apply the regularity of the elliptic problem $(\mathcal{E_N})$ and obtain a solution $\bb\in \bW^{1,p}(\O)$, we also need to justify that the RHS $\bg+\curl(\bu\times\bb)$ belongs to $[\bH_0^{r',p'}(\curl, \Omega)]'$. As a consequence, we distinguish according to the values of $p$ the following cases:\\

(i) If $p\leq 3$,  then $\ff-(\curl\bu)\times\bu+(\curl\bb)\times\bb\in\bL^{3/2}(\O)\hookrightarrow \bL^{\frac{3p}{3+p}}(\O)=\bL^{r}(\O)$. Since $\bL^{r}(\O) \hookrightarrow [\bH_0^{r',p'}(\curl, \Omega)]'$, thanks to the regularity of the Stokes problem $(\mathcal{S_N})$, see Proposition \ref{thm solution W1,p and W2,p Stokes divu=h}, we have that $\bu\in\bW^{1,p}(\O)$ and $P\in W^{1,r}(\O)$. Since $\curl(\bu\times\bb)=(\bb\cdot\nabla)\bu-(\bu\cdot\nabla)\bb$, we have with the same argument above that $\curl(\bu\times\bb)\in[\bH_0^{r',p'}(\curl, \Omega)]'$. It follows that $\bg+\curl(\bu\times\bb)\in [\bH_0^{r',p'}(\curl, \Omega)]' $. Moreover, $\bg+\curl(\bu\times\bb)$ satisfies the compatibility conditions \eqref{Condition compatibilite K_N Lp}-\eqref{condition div g=0 sol W^1,p linear MHD}. Consequently, thanks to the regularity of the elliptic problem $(\mathcal{E_N})$, see Lemma \ref{Lemme elliptique H r' p'}, we have that $\bb\in\bW^{1,p}(\O)$.\\

(ii) If $p >3$, from the previous case, we have that $(\bu,\,\bb,P)\in \bW^{1,3}(\O)\times \bW^{1,3}(\O)\times W^{1,3/2}(\O)$. Therefore, $(\bu,\bb)\in \bL^{q}(\O)\times \bL^{q}(\O)$, for any $1< q < \infty$. Then, $(\curl\bu)\times\bu\in \bL^{m}(\O)$ for all $1\leq m < 3$. In particular, we take $\frac{1}{m}=\frac{1}{p}+\frac{1}{3}$ with $3/2 < m <3$. So, we have that $(\curl\bu)\times\bu\in \bL^{\frac{3p}{3+p}}(\O)\hookrightarrow [\bH_0^{r',p'}(\curl, \Omega)]'$. Using the same arguments, we obtain that $(\curl\bb)\times\bb\in \bL^{\frac{3p}{3+p}}(\O)\hookrightarrow [\bH_0^{r',p'}(\curl, \Omega)]'$. Then, the required regularity for $(\bu,\bb,P)$ follows by applying the regularity of the Stokes problem $(\mathcal{S_N})$. Further, we have $\bu\times\bb\in \bL^{t}(\O)$ for any $1< t < \infty$. In particular, taking $t=p$ with $3<t< \infty$, we have that $\bu\times\bb\in \bL^{p}(\O)$. So, due to the characterization of the space $[\bH_0^{r',p'}(\curl, \Omega)]'$, we have $\curl(\bu\times\bb)$ belongs to $[\bH_0^{r',p'}(\curl, \Omega)]'$ and we finish the proof by applying the regularity of the elliptic problem $(\mathcal{E_N})$.

\end{proof}

\subsection{Strong solution: $L^{p}$-theory,\,\, $p\geq 6/5$}

In this subsection, we study the existence of strong solutions for more regular data. The following theorem gives the regularity $\bW^{2,p}(\O)$ with $p\geq 6/5$.

\begin{theorem}[Regularity $\bW^{2,p}(\O)$ with $p\geq\frac{6}{5}$]\label{thm:regularity_weak_sol1_W2p}
Let us suppose that $\O$ is of class $\mathcal{C}^{2,1}$ and $p\geq \frac{6}{5}$. Let $\ff$, $\bg$ and $P_0$ satisfy \eqref{Condition compatibilite K_N Lp}-\eqref{condition div g=0 sol W^1,p linear MHD} and  
$$\ff \in \bL^p(\Omega), \quad \bg \in \bL^p(\Omega), \quad h=0\quad  \mathrm{and}\quad P_0 \in W^{1-\frac{1}{p},p}(\Gamma).$$
Then the weak solution  $(\bu, \bb, P)$ for the \eqref{eqn:MHD} system given by Theorem \ref{thm:existence_weak_sol} belongs to \\ $ \bW^{2,p}(\Omega) \times \bW^{2,p}(\Omega) \times W^{\,1,p}(\Omega)$ and satisfies the following estimate:
\begin{eqnarray}\label{estim strong p>6/5 MHD}
\begin{aligned}
&\norm{\bu}_{\bW^{2,p}(\Omega)} + \norm{\bb}_{\bW^{2,p}(\Omega)}+ \norm{P}_{W^{1,p}(\Omega)} \leqslant C(\norm{\ff}_{\bL^p(\Omega)} + \norm{\bg}_{\bL^p(\Omega)} + \norm{P_0}_{W^{1-\frac{1}{p},p}(\Gamma)}) 
\end{aligned}
\end{eqnarray}
\end{theorem}
\begin{proof}
To start the proof, the idea is to use the regularity result for weak solutions in $\bW^{1,p}(\O)$ with $p>2$ given in Theorem \ref{thm:regularity_weak_sol1} instead of the weak solutions in the Hilbert case $\bH^{1}(\O)$. Observe that if $6/5 \leq p \leq 3/2$, we have $2<p^*<3$, where $\frac{1}{p^{*}}=\frac{1}{p}-\frac{1}{3}$. Now, denoting $r(p^{*})=\frac{3p^{*}}{3+p^{*}}$, we obtain that $r(p^{*})=p$. Then, we have from the hypothesis of this Theorem that $\ff\in\bL^{r(p^{*})}(\O)$, $\bg\in\bL^{r(p^{*})}(\O)$ and $P_0\in W^{1-1/r(p^{*}),r(p^{*})}(\Gamma)$. Since $\bL^{r(p^{*})}(\O)\hookrightarrow [\bH_0^{(r^{p^{*}})',(p^{*})'}(\curl, \Omega)]'$ and $p^{*}>2$, we deduce from the regularity result of the \eqref{eqn:MHD} problem (see Theorem \ref{thm:regularity_weak_sol1_W2p}) that $(\bu,\bb,P)\in\bW^{1,p^{*}}(\O)\times\bW^{1,p^{*}}(\O)\times W^{1,r(p^{*})}(\O)$. Then, we have the following three cases:

(i) Case $\frac{6}{5}\leq p < \frac{3}{2}$: We have $(\curl\bu)\times\bu\in \bL^{t}(\O)$ with $\frac{1}{t}=\frac{2}{p}-1$. Since $t>p$, it follows that $(\curl\bu)\times\bu$ belongs to $\bL^{p}(\O)$. The same argument gives that $(\curl\bb)\times\bb$ belongs to $\bL^{p}(\O)$. Consequently, thanks to the existence of strong solutions for the Stokes problem $(\mathcal{S_N})$ (see Proposition \ref{thm solution W1,p and W2,p Stokes divu=h}), we deduce that $(\bu,P)\in \bW^{2,p}(\O)\times W^{1,p}(\O)$. Since $\curl(\bu\times\bb)$ belongs also to $\bL^{t}(\O)$ with $t>p$, we deduce from the regularity result of the elliptic problem $(\mathcal{E_N})$ (see Theorem \ref{thm strong soulution elliptic}) that $\bb\in\bW^{2,p}(\O)$.    \\

(ii) Case $p=\frac{3}{2}$: We have in this case $p^{*}=3$. From above, we know that  $(\bu,\bb,P)\in\bW^{1,3}(\O)\times\bW^{1,3}(\O)\times W^{1,\frac{3}{2}}(\O)$. Since $\bW^{1,3}(\O)\hookrightarrow \bL^{q}(\O)$ for any $1< q< \infty$, we deduce that $(\curl \bu)\times \bu$ and $(\curl \bb)\times \bb$ belong to $\bL^{t}(\O)$ with $\frac{1}{t}=\frac{1}{3}+\frac{1}{q}$. Choosing $q>3$ gives $t>\frac{3}{2}$. Thanks to the regularity of the Stokes problem $(\mathcal{S_N})$, we have that $(\bu,P)\in\bW^{2,\frac{3}{2}}(\O)\times W^{1,\frac{3}{2}}(\O)$. Using the same arguments, we have $\curl(\bu\times\bu)\in \bL^{t}(\O)$ with $t>\frac{3}{2}$. Then, we apply the regularity of the elliptic problem $(\mathcal{E_N})$ to obtain $\bb\in \bW^{2,\frac{3}{2}}(\O)$.
%%%%%%%%%%%%%%%%%%%%%%%%%%%%%%%%%%%%%%%%%%%%%%%%%%%%%%%%

(iii) Case $p>\frac{3}{2}$ : We know that $(\bu,\bb)\in \bW^{2,\frac{3}{2}}(\O)\times \bW^{2,\frac{3}{2}}(\O) $. Then, $(\bu,\bb)\in\bL^{q}(\O)\times \bL^{q}(\O)$ with $1<q<\infty$. We deduce that the terms $(\curl\bu)\times\bu$ and $(\curl\bb)\times \bb$ belong to $\bL^{t}(\O)$ with $1 \leq t < 3$. So, we have the following two cases:

 \indent (a) If $\frac{3}{2} <p <3$, we have $(\curl\bu)\times\bu$, $(\curl\bb)\times\bb$ and $\curl(\bu\times\bb)$ belong to $\bL^{p}(\O)$. Thanks to the regularity of $(\mathcal{S_N})$ and $(\mathcal{E_N})$, we deduce that $(\bu,\bb,P)$ belongs to $\bW^{2,p}(\O)\times \bW^{2,p}(\O)\times W^{1,p}(\O)$. 

 \indent (b) If $p \geq 3$, from the above result, we have $(\bu,\bb,P)\in\bW^{2,3-\varepsilon}(\O)\times \bW^{2,3-\varepsilon}(\O)\times W^{1,3-\varepsilon}(\O)$ with $0<\varepsilon <\frac{3}{2}$. Observe that $(3-\varepsilon)^*=\frac{3(3-\varepsilon)}{\varepsilon}>3$. This implies that $\bu\in\bL^{\infty}(\O)$ and $\bb\in\bL^{\infty}(\O)$. Since $\curl\bu\in \bW^{1,3-\varepsilon}(\O)\hookrightarrow \bL^{(3-\varepsilon)^*}(\O)$, it follows that $(\curl\bu)\times\bu$ and $(\curl\bb)\times\bb$ belong to $\bL^{(3-\varepsilon)^*}(\O)\hookrightarrow\bL^{3}(\O)$. Again, according the regularity of
Stokes problem, we have $(\bu,P)\in\bW^{2,3}(\O)\times W^{1,3}(\O)$. Similarly, $\curl(\bu\times\bb)\in \bL^3(\O)$ and the regularity of the elliptic problem $(\mathcal{E_N})$ implies that $\bb\in\bW^{1,3}(\O)$. Finally, using the embeddings $\bW^{2,3}(\O)\hookrightarrow \bL^{\infty}(\O)$ and $ \bW^{1,3}(\O)\hookrightarrow \bL^q(\O)$ for $1<q<\infty$, all the terms $(\curl\bu)\times\bu$, $(\curl\bb)\times\bb$ and $\curl(\bu\times\bb)$ belong to $\bL^{q}(\O)$. To conclude, we apply once again the regularity
of Stokes problem $(\mathcal{S_N})$ and elliptic problem $(\mathcal{E_N})$.

\end{proof}

\subsection{Existence result of the MHD system for $3/2< p < 2$}

The next theorem tells us that it is possible to extend the regularity $\bW^{1,p}(\O)$ of the solution of the nonlinear \eqref{eqn:MHD} problem for $\frac{3}{2} < p<2$. For this, we apply Banach's fixed-point theorem over the linearized problem \eqref{linearized MHD-pressure}. 

%-------------------------------------------------------------------------------------------------------------------------------------------------

\begin{theorem}[Regularity $\bW^{1,p}(\O)$ with $\frac{3}{2}< p<2$]\label{thm:regularity_weak_ MHD p<2}Assume that $\frac{3}{2} < p < 2$ and let $r$ be defined by $\frac{1}{r}=\frac{1}{p}+\frac{1}{3}$. 
Let us consider $\ff, \bg \in [\bH_0^{r',p'}(\curl, \Omega)]'$, $P_0 \in W^{1-\frac{1}{r},r}(\Gamma)$ and $h \in W^{1,r}(\O)$ with the compatibility conditions \eqref{Condition compatibilite K_N Lp}-\eqref{condition div g=0 sol W^1,p linear MHD}. 

\emph{\textbf{(i)}} There exists a constant $\delta_1$ such that, if 
\begin{eqnarray*}
\norm{\ff}_{[\bH_0^{r',p'}(\curl, \Omega)]'} + \norm{\bg}_{[\bH_0^{r',p'}(\curl,\Omega)]'} + \norm{P_0}_{W^{1-\frac{1}{r},r}(\Gamma)}+\norm{h}_{W^{1,r}(\O)} \leq \delta_1
\end{eqnarray*}
Then, the \eqref{eqn:MHD} problem has at least a solution $(\bu, \bb, P, \boldsymbol{\alpha}) \in \bW^{1,p}(\Omega) \times \bW^{1,p}(\Omega) \times W^{1,r}(\O) \times \R^I$. Moreover, we have the following estimates:\small
\begin{eqnarray} \label{estim u,b non linear p<2}
\!\!\norm{\bu}_{_{\bW^{1,p}(\Omega)}}\!\! \!\!\!+\!\norm{\bb}_{_{\bW^{1,p}(\Omega)}}\!\!\!\!\!\leqslant \!C_1\! \big(\! \norm{\ff}_{[\bH_0^{r',p'}(\curl, \Omega)]'} \!+ \!\norm{\bg}_{[\bH_0^{r',p'}(\curl,\Omega)]'}\! +\! \norm{P_0}_{W^{1-\frac{1}{r},r}(\Gamma)} \!\!+ \!\norm{h}_{W^{1,r}(\O)} \!\!\big)
\end{eqnarray}
\begin{eqnarray} \label{estim P non linear p<2}
\begin{aligned}
\norm{P}_{_{W^{1,r}(\O)}}\!\!\!\!\leqslant \!C_1\!(1+C^*\eta)\big( \!\norm{\ff}_{[\bH_0^{r',p'}(\curl, \Omega)]'} \!+\! \norm{\bg}_{[\bH_0^{r',p'}(\curl,\Omega)]'} \!+\! \norm{P_0}_{W^{1-\frac{1}{r},r}(\Gamma)}\!\! +\!\norm{h}_{W^{1,r}(\O)}\!\! \big)
\end{aligned}
\end{eqnarray}
where $\delta_1=(2C^2C^*)^{-1}$, $C_1=C(1+C^*\eta)^2$ with $C>0$, $C^*>0$ are the constants given in \eqref{definition C^* preuve W1p p<2 non lin} and $\eta$ defined by \eqref{definition eta preuve W1p p<2 non lin}. Furthermore, $\boldsymbol{\alpha}=(\alpha_1,\ldots,\alpha_I)$ satisfies
\begin{eqnarray*}
\begin{aligned}
&\alpha_i = \langle \ff, \nabla q_i^N \rangle_{\O}- \int_{\O} (\curl \bw) \times \bu\cdot \nabla q_i^N \,d \bx + \int_{\O}(\curl \bb) \times \bdd \cdot \nabla q_i^N \,d\bx+ \int_{\Gamma}(h -P_0) \nabla q^{N}_{i}\cdot\bn\,\,d\sigma 
\end{aligned}
\end{eqnarray*}
\emph{\textbf{(ii)}} Moreover, if the data satisfy that 
\begin{eqnarray*}
\norm{\ff}_{[\bH_0^{r',p'}(\curl, \Omega)]'} + \norm{\bg}_{[\bH_0^{r',p'}(\curl,\Omega)]'} + \norm{P_0}_{W^{1-\frac{1}{r},r}(\Gamma)}  +\norm{h}_{W^{1,r}(\O)} \leq \delta_2,
\end{eqnarray*}
for some $\delta_2\in]0,\,\delta_1]$, then the weak solution of \eqref{eqn:MHD} is unique.
\end{theorem}

\begin{proof}
${}$\\
\noindent\textbf{(i) Existence}: Let us define the space 
\begin{eqnarray*}
\bZ^p(\Omega) = \bW^{1,p}(\Omega) \times \bW_\sigma^{1,p}(\Omega)
\end{eqnarray*}
For given $(\bw,\bdd)\in \bB_{\eta}$, define the operator $T$ by $T(\bw,\bdd)=(\bu,\bb)$ where $(\bu,\bb)$ is the unique solution of the linearized problem \eqref{linearized MHD-pressure} given by Theorem \ref{thm: weak W1,p for p<2} and the neighbourhood $\bB_{\eta}$ is defined by
$$\bB_{\eta}=\{(\bw,\bdd)\in \bZ^{p}(\O),\,\Vert(\bw,\bdd)\Vert_{\bZ^{p}(\O)}\leq \eta\}, \quad \eta >0.$$
 Here $\bZ^p(\O)$ is equipped with the norm 
$$\Vert(\bw,\bdd)\Vert_{\bZ^{p}(\O)}=\Vert\bw\Vert_{\bW^{1,p}(\O)}+\Vert\bdd\Vert_{\bW^{1,p}(\O)}.$$
 We have to prove that $T$ is a contraction from $\bB_\eta$ to itself. Let $(\bw_1, \bdd_1), (\bw_2, \bdd_2) \in \bB_\eta$. We show that there exists $\theta \in (0,1)$ such that:
\begin{eqnarray}\label{ineq for T}
\begin{aligned}
\norm{T(\bw_1, \bdd_1) - T(\bw_2, \bdd_2)}_{\bZ^{p}(\O)} &= \norm{(\bu_1, \bb_1) - (\bu_2, \bb_2)}_{\bZ^{p}(\O)} \leqslant \theta \norm{(\bw_1, \bdd_1) - (\bw_2, \bdd_2)}_{\bZ^{p}(\O)}	
\end{aligned}
\end{eqnarray}

Thanks to Corollary \ref{cor: weak W1,p for p<2}, each $(\bu_i, \bb_i, P_i, \bc^{i})$, $i=1,2$, belongs to $ \bW^{1,p}(\Omega) \times \bW^{1,p}(\Omega) \times W^{1,r}(\Omega)$ and verifies: 
\begin{eqnarray} \label{Systeme en uk et bk}
\begin{cases}
- \Delta \bu_i + (\curl \bw_i) \times \bu_i + \nabla P_i -  (\curl \bb_i) \times \bdd_i = \ff \, \, \text{and} \, \, \vdiv \bu_i = h \, \, \, \text{in} \, \,  \Omega \\
\curl \curl \bb_i - \curl(\bu_i \times \bdd_i) = \bg \, \, \text{and} \, \, \vdiv \bb_i = 0 \, \, \, \text{in} \, \, \Omega \\
\bu_i \times \bn = 0 \, \, \text{and} \, \, \bb_i \times \bn = 0 \, \, \, \text{on} \, \, \Gamma \\
P_i = P_0 \, \, \text{on} \, \, \Gamma_0 \, \, \text{and} \, \, P_i = P_0 + c_j^{(i)} \, \, \, \text{on} \, \, \Gamma_j \\
\langle \bu_i \cdot \bn, 1 \rangle_{\Gamma_j} = 0 \, \, \text{and} \, \, \langle \bb_i \cdot \bn, 1 \rangle_{\Gamma_j} = 0, \, \, \, \forall \, 1 \leqslant j \leqslant I
\end{cases}
\end{eqnarray}
together with following estimate for $i=1,2$: 
{\footnotesize
\begin{eqnarray} \label{Estimation de uk et bk}
\norm{(\bu_i, \bb_i)}_{_{\bZ^p(\Omega)}} \!\!\!\leqslant C \Big( 1 + \norm{\curl \bw_i}_{\bL^\frac{3}{2}(\Omega)} + \norm{\bdd_i}_{\bW^{1,\frac{3}{2}}(\Omega)} \!\!\Big)  
\Big( \gamma_1 + (1 + \norm{\curl \bw_i}_{_{\bL^\frac{3}{2}(\Omega)}} \!\!\!\!\!\!+ \norm{\bdd_i}_{\bW^{1,\frac{3}{2}}(\Omega)}) \gamma_2 \Big) 
\end{eqnarray}}
where $\gamma_1 = \norm{\ff}_{[\bH_0^{r',p'}(\curl, \Omega)]'} + \norm{\bg}_{[\bH_0^{r',p'}(\curl, \Omega)]'} + \norm{P_0}_{W^{1-\frac{1}{r},r}(\Gamma)}$ and $\gamma_2 = \norm{h}_{W^{1,r}(\Omega)}$. \\

\noindent Next, the differences $(\bu, \bb, P, \bc) = (\bu_1 - \bu_2, \bb_1 - \bb_2, P_1 - P_2, \bc^{1}-\bc^{2})$ satisfy
\begin{eqnarray} \label{Systeme differences de u1 avec u2}
\begin{cases}
- \Delta \bu + (\curl \bw_{1}) \times \bu + \nabla P - (\curl \bb) \times \bdd_1 = \ff_2 \, \, \text{and} \, \, \vdiv \bu = 0 \, \, \text{in} \, \, \Omega \\
\curl \curl \bb -  \curl (\bu \times \bdd_1) = \bg_2 \, \, \text{and} \, \, \vdiv \bb = 0 \, \, \text{in} \, \, \Omega \\
\bu \times \bn = \textbf{0} \, \, \text{and} \, \, \bb \times \bn = \textbf{0} \, \, \, \text{on} \, \, \Gamma \\
P= 0 \, \, \, \text{on} \, \, \Gamma_0\quad \mathrm{and}\quad P=c_j\quad \mathrm{on}\,\,\Gamma_j \\
\langle \bu \cdot \bn, 1 \rangle_{\Gamma_j} = 0 \, \, \text{and} \, \, \langle \bb \cdot \bn, 1 \rangle_{\Gamma_j} = 0 \, \, \, \forall \, 1 \leqslant j \leqslant I 
\end{cases}
\end{eqnarray}
with $\ff_2 = - (\curl \bw) \times \bu_2 + (\curl \bb_2) \times \bdd$ and $\bg_2 = \curl(\bu_2 \times \bdd)$. Observe that  $\ff_2$ and $\bg_2$ belong to $[\bH_0^{r',p'}(\curl, \Omega)]'$. Indeed, Since $\bu_2, \bb_2 \in \bW^{1,p}(\Omega) \hookrightarrow \bL^{p^*}(\Omega)$, $\curl \bw \in \bL^\frac{3}{2}(\Omega)$ and $\bdd \in \bW^{1,\frac{3}{2}}(\Omega)$, then $(\curl \bw) \times \bu_2$ and $(\curl \bb_2) \times \bdd$ belong to $\bL^r(\Omega) \hookrightarrow [\bH_0^{r',p'}(\curl, \Omega)]'$. Besides, $\bu_2 \times \bdd \in \bL^p(\Omega)$ so $\curl (\bu_2 \times \bdd) \in [\bH_0^{r',p'}(\curl, \Omega)]'$. Moreover, since $\bg_2$ is a $\curl$, then it satisfies the conditions \eqref{Condition compatibilite K_N Lp}-\eqref{condition div g=0 sol W^1,p linear MHD}. Hence, we apply the Theorem \ref{thm: weak W1,p for p<2} and we have the estimate: 
\begin{eqnarray*}
\begin{aligned}
\norm{(\bu, \bb)}_{\bZ^p(\Omega)} &\leqslant C ( 1 + \norm{\curl \bw_1}_{\bL^\frac{3}{2}(\Omega)} + \norm{\bdd_1}_{\bW^{1,\frac{3}{2}}(\Omega)}) (\norm{\ff_2}_{[\bH_0^{r',p'}(\curl, \Omega)]'} + \norm{\bg_2}_{[\bH_0^{r',p'}(\curl, \Omega)]'})
\end{aligned}
\end{eqnarray*}
By the definition of the norm on $[\bH_0^{r',p'}(\curl, \Omega)]'$, the Hölder inequality and the embeddings
\begin{eqnarray*}
\bW^{1,p}(\Omega) \hookrightarrow \bL^{p^*}(\Omega) \, \, \text{and} \, \, \bW^{1,p}(\Omega) \hookrightarrow \bW^{1,\frac{3}{2}}(\Omega) \hookrightarrow \bL^3(\Omega),
\end{eqnarray*}
it follows: 
\begin{eqnarray*}
\begin{aligned}
\norm{\ff_2}_{[\bH_0^{r',p'}(\curl, \Omega)]'} &\leqslant \norm{(\curl \bw) \times \bu_2}_{\bL^r(\Omega)} + \norm{(\curl \bb_2) \times \bdd}_{\bL^r(\Omega)} \\
&\leqslant \norm{\curl \bw}_{\bL^\frac{3}{2}(\Omega)} \norm{\bu_2}_{\bL^{p^*}(\Omega)} + \norm{\curl \bb_2}_{\bL^p(\Omega)} \norm{\bdd}_{\bL^3(\Omega)} \\
&\leqslant C (C_w \norm{\bw}_{\bW^{1,p}(\Omega)} \norm{\bu_2}_{\bW^{1,p}(\Omega)} + C_d \norm{\bdd}_{\bW^{1,p}(\Omega)} \norm{\bb_2}_{\bW^{1,p}(\Omega)} ) 
\end{aligned}
\end{eqnarray*}
and 
\begin{eqnarray*}
\begin{aligned}
\norm{\bg_2}_{[\bH_0^{r',p'}(\curl, \Omega)]'} = \norm{\bu_2 \times \bdd}_{\bL^p(\Omega)} &\leqslant \norm{\bu_2}_{\bL^{p^*}(\Omega)} \norm{\bdd}_{\bL^3(\Omega)} \leqslant C C_d \norm{\bu_2}_{\bW^{1,p}(\Omega)} \norm{\bdd}_{\bW^{1,p}(\Omega)} 
\end{aligned}
\end{eqnarray*}
where $C_w > 0$ and $C_d > 0$ are respectively defined by $\norm{\curl \bw}_{\bL^\frac{3}{2}(\Omega)} \leqslant C_w \norm{\bw}_{\bW^{1,p}(\Omega)}$ and $\norm{\bdd}_{\bL^3(\Omega)} \leqslant C_d \norm{\bdd}_{\bW^{1,p}(\Omega)}$. Therefore, recalling that $(\bw_1, \bdd_1)$ belongs to $\bB_\eta$, we have: 
\begin{eqnarray*}
\norm{\bu}_{\bW^{1,p}(\Omega)} + \norm{\bb}_{\bW^{1,p}(\Omega)} \!\!&\leqslant &\!\! C (1 + C^* \norm{(\bw_1, \bdd_1)}_{\bZ^p(\Omega)}) (C_w \norm{\bw}_{\bW^{1,p}(\Omega)} + C_d \norm{\bdd}_{\bW^{1,p}(\Omega)}) \norm{(\bu_2, \bb_2)}_{\bZ^p(\Omega)} \\
&\leqslant& \!\!C (1 + C^* \eta) C^* \norm{(\bw, \bdd)}_{\bZ^p(\Omega)} \norm{(\bu_2, \bb_2)}_{\bZ^p(\Omega)}
\end{eqnarray*}
with $C^* = C_w + C_d$. Combining with \eqref{Estimation de uk et bk}, we thus obtain: 
\begin{eqnarray} \label{definition C^* preuve W1p p<2 non lin}
\norm{\bu}_{\bW^{1,p}(\Omega)} + \norm{\bb}_{\bW^{1,p}(\Omega)} \leqslant C_1 C^* \norm{(\bw, \bdd)}_{\bZ^p(\Omega)} (\gamma_1 + (1+C^* \eta) \gamma_2) 
\end{eqnarray}
where $C_1 = C (1 + C^* \eta)^2$. Hence, if we choose, for instance: 
\begin{eqnarray} \label{definition eta preuve W1p p<2 non lin}
\begin{aligned}
\eta = (C^*)((2C C^* \gamma)^{-\frac{1}{3}} - 1) \, \, \text{and} \, \, \gamma = \gamma_1 + \gamma_2 < (2C C^*)^{-1}
\end{aligned}
\end{eqnarray} 
then $C_1 C^* (\gamma_1 + (1 + C^* \eta) \gamma_2) < \frac{1}{2}$. Therefore $T$ is a contraction and we obtain the unique fixed-point $(\bu^*, \bb^*) \in \bW^{1,p}(\Omega) \times \bW^{1,p}_{\sigma}(\Omega)$ satisfying 
%\eqref{Estimation de uk et bk}: 
\begin{eqnarray*}
\begin{aligned}
\norm{(\bu^*, \bb^*)}_{\bZ^p(\Omega)} &\leqslant C \Big( 1 + C^* \norm{(\bu^*, \bb^*)}_{\bZ^p(\Omega)} \Big) \Big(\gamma_1 + (1 + C^* \norm{(\bu^*, \bb^*)}_{\bZ^p(\Omega)}) \gamma_2 \Big)
\end{aligned}
\end{eqnarray*}
Next, since $(\bu^*, \bb^*) \in \bB_\eta$, we obtain 
\begin{eqnarray} \label{Estim u^* and b^*}
\begin{aligned}
\norm{(\bu^*, \bb^*)}_{\bZ^p(\Omega)} &\leqslant C (1 + C^* \eta) (\gamma_1 + (1 + C^* \eta) \gamma_2) \leqslant C_1 \gamma
\end{aligned}
\end{eqnarray}
which implies the estimate \eqref{estim u,b non linear p<2}: 

\noindent Now, we want to prove the estimate for the associated pressure. Taking the divergence in the first equation of problem \eqref{eqn:MHD}, we have that $P^*$ is a solution of the following problem: 
\begin{eqnarray*}
\begin{cases}
\Delta P^* = \vdiv \ff + \vdiv( (\curl \bb^*) \times \bb^* - (\curl \bu^*) \times \bu^*)+\Delta h\qquad \mathrm{in}\,\,\O,\\
P^*=P_0\quad \mathrm{on}\,\,\Gamma_0\quad \mathrm{and}\quad P^*=P_0+c_i\quad \mathrm{on}\,\,\Gamma_i\qquad \qquad
\end{cases}
\end{eqnarray*}
with 
{\footnotesize
\begin{eqnarray*}
\norm{P^*}_{W^{1,r}(\Omega)}\leqslant \norm{\vdiv \ff}_{W^{-1,r}(\Omega)} + \norm{\vdiv ((\curl \bb^*) \times \bb^* - (\curl \bu^*) \times \bu^*)}_{W^{-1,r}(\Omega)}+\norm{\Delta h}_{W^{-1,r}(\O)}+\norm{P_0}_{W^{1-1/r,r}(\Gamma)}
\end{eqnarray*}}
% \begin{eqnarray}
% \leqslant C( \norm{\ff}_{[\bH_0^{r',p'}(\curl, \Omega)]'} + C^* \norm{(\bu^*, \bb^*)}_{\bZ^p(\Omega)}^2) \\
% \leqslant C(\gamma + C^* \eta \norm{(\bu^*, \bb^*)}_{\bZ^p(\Omega)} ) \\
% \leqslant C_1 (1 + C^* \eta)^2 \gamma
% \end{eqnarray*}
Following the same calculus as in the proof of Theorem \ref{thm: weak W1,p for p<2}, we obtain
\begin{eqnarray*}
\norm{P^*}_{W^{1,r}(\Omega)}\leqslant C(1+C^*\eta)^2(\gamma_1+(1+C^*\eta)\gamma_2).
\end{eqnarray*}
which implies \eqref{estim P non linear p<2} and the proof of existence is completed. 
% than in \eqref{Estim div f}-\eqref{Estim div curl b x d et curl w x u}, we have $\norm{\vdiv \ff}_{W^{-1,r}(\Omega)} \leqslant \norm{\ff}_{[\bH_0^{r',p'}(\curl, \Omega)]'}$ and $
% \norm{\vdiv (\curl \bb^* \times \bb^* - \curl \bu^* \times \bu^*)}_{W^{-1,r}(\Omega)} \leqslant C C^* \norm{(\bu^*, \bb^*)}_{\bZ^p(\Omega)}^2$, and consequently we obtain:

\noindent \textbf{ (ii) Uniqueness}:

Let $(\bu_1, \bb_1, P_1)$ and $(\bu_2, \bb_2, P_2)$ two solutions of the problem \eqref{eqn:MHD}. Then, we set $(\bu, \bb, P) = (\bu_1 - \bu_2, \bb_1	 - \bb_2, P_1 - P_2)$ which satisfies the problem: 
\begin{eqnarray*}
\begin{cases}
- \Delta \bu + (\curl \bu_1) \times \bu + \nabla P -  (\curl \bb) \times \bb_1 = \ff_2 \, \, \text{and} \, \, \vdiv \bu = 0 \, \, \text{in} \, \, \Omega \\
\curl \curl \bb - \curl (\bu \times \bb_1) = \bg_2 \, \, \text{and} \, \, \vdiv \bb = 0 \, \, \text{in} \, \, \Omega \\
\bu \times \bn = \textbf{0} \, \, \text{and} \, \, \bb \times \bn = \textbf{0} \, \, \, \text{on} \, \, \Gamma \\
P = 0 \, \, \text{on} \, \, \Gamma_0\quad \mathrm{and}\quad P=\alpha_i^{(1)}-\alpha_i^{(2)}\quad \mathrm{on}\,\,\Gamma_i,\\
\langle \bu \cdot \bn, 1 \rangle_{\Gamma_i} = 0 \, \, \text{and} \, \, \langle \bb \cdot \bn, 1 \rangle_{\Gamma_i} = 0 \, \, \, \forall \, 1 \leqslant i \leqslant I
\end{cases}
\end{eqnarray*}
where $\ff_2, \bg_2 \in [\bH_0^{r',p'}(\curl, \Omega)]'$ are already given in \eqref{Systeme differences de u1 avec u2} and satisfy the hypothesis of Theorem \ref{thm: weak W1,p for p<2}. Applying this theorem, we have the estimate
{\footnotesize
\begin{eqnarray*}
\norm{(\bu, \bb)}_{\bZ^p(\Omega)}\!\!\!& \!\!\!\!\leqslant\!\! &\!\!\!\!\!C(1 + \norm{\curl \bu_1}_{\bL^\frac{3}{2}(\Omega)} + \norm{\bb_1}_{\bW^{1,\frac{3}{2}}(\Omega)}) (C_w \norm{\bu}_{\bW^{1,p}(\Omega)} \norm{\bu_2}_{\bW^{1,p}(\Omega)} 
+ C_d \norm{\bb}_{\bW^{1,p}(\Omega)} \norm{\bb_2}_{\bW^{1,p}(\Omega)} ) \\
&\!\!\leqslant& \!\!\!\!\!C(1 + C^* \norm{(\bu_1, \bb_1)}_{\bZ^p(\Omega)}) C^* \norm{(\bu_2, \bb_2)}_{\bZ^p(\Omega)} \norm{(\bu, \bb)}_{\bZ^p(\Omega)} .
\end{eqnarray*}}
From \eqref{Estim u^* and b^*}, we obtain that:
\begin{eqnarray*}
\norm{(\bu, \bb)}_{\bZ^p(\Omega)} \leqslant C (1 + C^* C_1 \gamma) C^* C_1 \gamma \norm{(\bu, \bb)}_{\bZ^p(\Omega)}
\end{eqnarray*}
Thus, for $\gamma$ small enough such that 
\begin{eqnarray*}
C (1 + C^* C_1 \gamma) C^* C_1 \gamma < 1 
\end{eqnarray*}
we deduce that $\bu=\bb=\textbf{0}$ and then we obtain the uniqueness of the velocity and the magnetic field which implies the uniqueness of the pressure $P$.
\end{proof}

\section{Appendix}\label{section Appendix}
We begin this section by giving another proof of Theorem \ref{thm: strong W2,p for 1<p<6/5}.\\
\textbf{A second proof of Theorem \ref{thm: strong W2,p for 1<p<6/5}:}  Let $\lambda > 0$, and let us assume $\ff_\lambda, \bg_\lambda \in \boldsymbol{D}(\Omega)$ such that $\ff_\lambda$ and $\bg_\lambda$ respectively converge to $\ff$ and $\bg$ in $\bL^p(\Omega)$, and $P_0^\lambda \in C^\infty(\Gamma)$ which converges to $P_0$ in $W^{1-\frac{1}{p},p}(\Gamma)$. 

We consider the problem: find $(\bu_\lambda, \bb_\lambda, P_\lambda, c_i^\lambda)$ solution of:
\begin{eqnarray} \label{MHD f_lambda g_lambda problem}
\begin{cases}
-\Delta \bu_\lambda + (\curl \bw) \times \bu_\lambda + \nabla P_\lambda -  (\curl \bb_\lambda) \times \bdd  = \ff_\lambda \, \, \text{and} \, \, \vdiv \bu_\lambda = 0 \, \, \text{in} \, \, \Omega \\
\curl \curl \bb_\lambda -  \curl(\bu_\lambda \times \bdd) = \bg_\lambda \, \, \text{and} \, \, \vdiv \bb_\lambda = 0 \, \, \text{in} \, \, \Omega \\
P_\lambda=P_0^\lambda, \quad \text{on} \, \Gamma_0, \, P_\lambda = P_0^\lambda + c_i^\lambda \quad \text{on} \, \Gamma_i\\
\bu_\lambda \times \bn = 0, \quad \bb_\lambda \times \bn = 0 \quad \text{on} \, \Gamma \\
\langle \bu_\lambda \cdot \bn, 1 \rangle_{\Gamma_i} = 0, \quad  \langle \bb_\lambda \cdot \bn, 1 \rangle_{\Gamma_i} = 0, \, \, \forall \, 1 \leqslant i \leqslant I
\end{cases}
\end{eqnarray}
Note that, since $\ff_\lambda, \bg_\lambda \in \boldsymbol{D}(\Omega)$, in particular they belong to $\bL^\frac{6}{5}(\Omega)$. Thus, applying Theorem \ref{thm strong solution p>6/5}, we have $(\bu_\lambda, \bb_\lambda) \in \bW^{2,\frac{6}{5}}(\Omega) \times \bW^{2,\frac{6}{5}}(\Omega) \hookrightarrow \bH^1(\Omega) \times \bH^1(\Omega) \hookrightarrow \bL^6(\Omega) \times \bL^6(\Omega)$. Therefore, $(\curl \bw) \times \bu_\lambda$, $(\curl \bb_\lambda) \times \bdd$ and $\curl (\bu_\lambda \times \bdd)$ belong to $\bL^\frac{6}{5}(\Omega) \hookrightarrow \bL^p(\Omega)$. Hence, using the regularity results of the Stokes and elliptic problems, we obtain that the problem \eqref{MHD f_lambda g_lambda problem} has a unique solution $(\bu_\lambda, \bb_\lambda, P_\lambda) \in \bW^{2,p}(\Omega) \times \bW^{2,p}(\Omega) \times W^{1,p}(\Omega)$ which also satisfies the estimate: 
\begin{eqnarray} \label{Estimation forte Stokes Ellipt preuve}
\begin{aligned}
&\norm{\bu_\lambda}_{\bW^{2,p}(\Omega)} + \norm{\bb_\lambda}_{\bW^{2,p}(\Omega)} + \norm{P_\lambda}_{W^{1,p}(\Omega)} \\
&\leqslant C_{SE} \Big( \norm{\ff_\lambda}_{\bL^p(\Omega)} + \norm{\bg_\lambda}_{\bL^p(\Omega)} + \norm{P_0^\lambda}_{W^{1-\frac{1}{p},p}(\Gamma)} + \sum_{i=1}^I |c_i^\lambda| + \norm{(\curl \bw) \times \bu_\lambda}_{\bL^p(\Omega)} \\
&+ \norm{(\curl \bb_\lambda) \times \bdd}_{\bL^p(\Omega)} + \norm{\curl (\bu_\lambda \times \bdd)}_{\bL^p(\Omega)} \Big)
\end{aligned}
\end{eqnarray}
with $C_{SE} = \max (C_S, C_E)$ where $C_S$ is the constant given in the Proposition \ref{thm solution W1,p and W2,p Stokes divu=h} and $C_E$ the constant given in the Theorem \ref{thm strong soulution elliptic}. We now want to bound the right hand side terms $\norm{(\curl \bw) \times \bu_\lambda}_{\bL^p(\Omega)}$, $\norm{(\curl \bb_\lambda) \times \bdd}_{\bL^p(\Omega)}$, $\norm{\curl (\bu_\lambda \times \bdd)}_{\bL^p(\Omega)}$ and $\sum_{i=1}^I \abs{c_i^\lambda}$ to obtain the estimate \eqref{Estimation theorem strong solution p small}. In this purpose, we decompose $\curl \bw$ and $\bdd$ as in \eqref{decomposition y}-\eqref{decomposition d}. 

\vspace{2mm}

Let $\epsilon > 0$ and $\rho_{\epsilon/2}$ the classical mollifier. We consider $\tilde{\by} = \widetilde{\curl \bw}$ and $\tilde{\bdd}$ the extensions by $0$ of $\by = \curl \bw$ and $\bdd$ in $\R^3$ respectively. We take: 

\begin{equation} \label{Decomposition of bw and bd}
\begin{aligned}
\curl \bw = \by_1^\epsilon + \by_2^\epsilon \, \, &\text{where} \, \, \by_1^\epsilon = \tilde{\by} \ast \rho_{\epsilon/2} \, \, \text{and} \, \, \by_2^\epsilon = \curl \bw - \by_1^\epsilon \\
\bdd = \bdd_1^\epsilon + \bdd_2^\epsilon \, \, &\text{where} \, \, \bdd_1^\epsilon = \tilde{\bdd} \ast \rho_{\epsilon/2} \, \, \text{and} \, \, \bdd_2^\epsilon = \bdd - \bdd_1^\epsilon
\end{aligned}
\end{equation}
For each term, we start by bounding the part depending on $\bdd_2^\epsilon$ (resp. $\by_2^\epsilon$), and then we look at $\bdd_1^\epsilon$ (resp. $\by_1^\epsilon$). 

\noindent \textbf{(i) Estimate of the term $\norm{(\curl \bw) \times \bu_\lambda}_{\bL^p(\Omega)}$:}

First, following the definition of the mollifier, we classically obtain: 
\begin{eqnarray} \label{Estimate of y_2^epsilon}
\norm{\by_2^\epsilon}_{\bL^\frac{3}{2}(\Omega)} = \norm{\by - \tilde{\by} \ast \rho_{\epsilon/2}}_{\bL^\frac{3}{2}(\Omega)} \leqslant \epsilon
\end{eqnarray}
Then, since we have $\bW^{2,p}(\Omega) \hookrightarrow \bL^{p^{**}}(\Omega)$ and the Hölder inequality, we obtain: 
\begin{eqnarray} \label{Estimate of y_2^epsilon times u_lambda}
\norm{\by_2^\epsilon \times \bu_\lambda}_{\bL^p(\Omega)} \leqslant \norm{\by_2^\epsilon}_{\bL^\frac{3}{2}(\Omega)} \norm{\bu_\lambda}_{\bL^{p^{**}}(\Omega)} \leqslant C_1 \norm{\by_2^\epsilon}_{\bL^\frac{3}{2}(\Omega)} \norm{\bu_\lambda}_{\bW^{2,p}(\Omega)}
\end{eqnarray}
where $C_1$ is the constant related to the previous Sobolev embedding $\bW^{2,p}(\Omega) \hookrightarrow \bL^{p^{**}}(\Omega)$.  Thus, injecting \eqref{Estimate of y_2^epsilon} in \eqref{Estimate of y_2^epsilon times u_lambda}, it follows:
\begin{eqnarray*}
\norm{\by_2^\epsilon \times \bu_\lambda}_{\bL^p(\Omega)} \leqslant C_1 \epsilon \norm{\bu_\lambda}_{\bW^{2,p}(\Omega)}
\end{eqnarray*}
Now, for the term in $\by_1^\epsilon$, we apply the Hölder inequality to obtain:  
\begin{eqnarray*}
\norm{\by_1^\epsilon \times \bu_\lambda}_{\bL^p(\Omega)} \leqslant \norm{\by_1^\epsilon}_{\bL^m(\Omega)} \norm{\bu_\lambda}_{\bL^q(\Omega)} \leqslant \norm{\by}_{\bL^\frac{3}{2}(\Omega)} \norm{\rho_{\epsilon/2}}_{\bL^t(\Omega)} \norm{\bu_\lambda}_{\bL^q(\Omega)} 
\end{eqnarray*}
with $m, q \geqslant p$ such that $\frac{1}{p} = \frac{1}{m} + \frac{1}{q}$ and $t > 1$ defined by $1 + \frac{1}{m} = \frac{2}{3} + \frac{1}{t}$. Note that this definition imposes $m > \frac{3}{2}$, and we can take in particular $m =3$, and hence $q=p^*$. Since, following the properties of the mollifier, there exists $C_\epsilon > 0$ such that, for all $t > 1$:  
\begin{eqnarray} \label{Mollifier bound}
\norm{\rho_{\epsilon/2}}_{\bL^t(\Omega)} \leqslant C_\epsilon
\end{eqnarray}
so we have
\begin{eqnarray*}
\norm{\by_1^\epsilon \times \bu_\lambda}_{\bL^p(\Omega)} \leqslant C_\epsilon \norm{\by}_{\bL^\frac{3}{2}(\Omega)} \norm{\bu_\lambda}_{\bL^{p^*}(\Omega)} 
\end{eqnarray*}

\noindent \textbf{(ii) Estimate of the term $\norm{(\curl \bb_\lambda) \times \bdd}_{\bL^p(\Omega)}$:}
As previously, we have from the definition of the mollifier: 
\begin{eqnarray} \label{Estim d_2^epsi sol p petit}
\norm{\bdd_2^\epsilon}_{\bW^{1,\frac{3}{2}}(\Omega)} = \norm{\bdd - \tilde{\bdd} \ast \rho_{\epsilon/2}}_{\bW^{1,\frac{3}{2}}(\Omega)} \leqslant \epsilon
\end{eqnarray}
Then, combining the Hölder inequality and the Sobolev embedding $\bW^{1,\frac{3}{2}}(\Omega) \hookrightarrow \bL^3(\Omega)$, we have:
\begin{eqnarray*}
\norm{(\curl \bb_\lambda) \times \bdd_2^\epsilon}_{\bL^p(\Omega)} \leqslant \norm{\curl \bb_\lambda}_{\bL^{p^*}(\Omega)} \norm{\bdd_2^\epsilon}_{\bL^3(\Omega)} \leqslant C_2 \epsilon \norm{\bb_\lambda}_{\bW^{1,p^*}(\Omega)}
\end{eqnarray*}
where $C_2$ is the constant related to the Sobolev embedding $\bW^{1,\frac{3}{2}}(\Omega) \hookrightarrow \bL^3(\Omega)$. 
We finally recall the Sobolev embedding $\bW^{2,p}(\Omega) \hookrightarrow \bW^{1,p^*}(\Omega)$ to obtain: 
\begin{eqnarray} \label{Estimate curl bb_lambda times d_2^epsilon} 
\norm{(\curl \bb_\lambda) \times \bdd_2^\epsilon}_{\bL^p(\Omega)} \leqslant C_2 C_3 \epsilon \norm{\bb_\lambda}_{\bW^{2,p}(\Omega)}
\end{eqnarray}
with $C_3$ the constant related to the Sobolev embedding $\bW^{2,p}(\Omega) \hookrightarrow \bW^{1,p^*}(\Omega)$. 
It remains to bound the term in $\bdd_1^\epsilon$. Applying the Holdër inequality, we have:
\begin{eqnarray} \label{Estimate curl bb_lambda times d_1^epsilon}
\begin{aligned}
\norm{(\curl \bb_\lambda) \times \bdd_1^\epsilon}_{\bL^p(\Omega)} &\leqslant \norm{\curl \bb_\lambda}_{\bL^q(\Omega)} \norm{\bdd_1^\epsilon}_{\bL^m(\Omega)} \leqslant \norm{\curl \bb_\lambda}_{\bL^q(\Omega)} \norm{\bdd}_{\bL^3(\Omega)} \norm{\rho_\epsilon}_{\bL^t(\Omega)} 
\end{aligned}
\end{eqnarray}
with $m, q \geqslant p$ such that $\frac{1}{p} = \frac{1}{m} + \frac{1}{q}$ and $t > 1$ such that $1 + \frac{1}{m} = \frac{1}{3} + \frac{1}{t}$. Note that these relations require $m > 3$, then $\frac{1}{p} < \frac{1}{q} + \frac{1}{3}$ so $q < p^*$. Therefore, we have the Sobolev embeddings: 
\begin{equation*}\bW^{2,p}(\Omega) \underset{\mathrm{compact}}{\hookrightarrow}\bW^{1,q}(\Omega)\underset{\mathrm{continuous}}{\hookrightarrow} \bL^{p^*}(\O)\end{equation*}
hence there exists $\eta > 0$ and $C_\eta > 0$ such that
\begin{eqnarray} \label{Estimation interpolation}
\norm{\bb_\lambda}_{\bW^{1,q}(\Omega)} \leqslant \eta \norm{\bb_\lambda}_{\bW^{2,p}(\Omega)} +C_\eta \norm{\bb_\lambda}_{\bL^{p^*}(\Omega)}
\end{eqnarray}
Injecting \eqref{Estimation interpolation} in \eqref{Estimate curl bb_lambda times d_1^epsilon}, combining with \eqref{Mollifier bound} and the Sobolev embedding $\bW^{1,\frac{3}{2}}(\Omega) \hookrightarrow \bL^3(\Omega)$, we obtain:
\begin{eqnarray*}
\begin{aligned}
\norm{(\curl \bb_\lambda) \times \bdd_1^\epsilon}_{\bL^p(\Omega)}\leqslant C_\epsilon C_2 \norm{\bdd}_{\bW^{1,\frac{3}{2}}(\Omega)} (\eta \norm{\bb_\lambda}_{\bW^{2,p}(\Omega)} + C_\eta \norm{\bb_\lambda}_{\bL^{p^*}(\Omega)} )
\end{aligned}
\end{eqnarray*}

\noindent \textbf{(iii) Estimate of the term $\norm{\curl (\bu_\lambda \times \bdd)}_{\bL^p(\Omega)}$:}

Since $\vdiv \bu_\lambda = 0$ and $\vdiv \bdd = 0$ in $\Omega$, then $\curl (\bu_\lambda \times \bdd) = (\bdd \cdot \nabla) \bu_\lambda - (\bu_\lambda \cdot \nabla) \bdd$. We thus bound these two terms. 
\begin{itemize}

\item \textbf{Estimate of the term $\norm{(\bdd \cdot \nabla) \bu_\lambda}_{\bL^p(\Omega)}$:}

The reasoning is exactly the same as for $(ii)$ with $\curl \bb_\lambda$ replacing by $\nabla \bu_\lambda$. Then we have: 
\begin{eqnarray*}
\norm{(\bdd_2^\epsilon \cdot \nabla) \bu_\lambda}_{\bL^p(\Omega)} \leqslant C_2 C_3 \epsilon \norm{\bu_\lambda}_{\bW^{2,p}(\Omega)} 
\end{eqnarray*}
and 
\begin{eqnarray*}
\norm{(\bdd_1^\epsilon \cdot \nabla) \bu_\lambda}_{\bL^p(\Omega)} \leqslant C_2 C_\epsilon \norm{\bdd}_{\bW^{1,\frac{3}{2}}(\Omega)} (\eta \norm{\bu_\lambda}_{\bW^{2,p}(\Omega)} + C_\eta \norm{\bu_\lambda}_{\bL^{p^*}(\Omega)})
\end{eqnarray*}

\item \textbf{Estimate of the term $\norm{(\bu_\lambda \cdot \nabla) \bdd}_{\bL^p(\Omega)}$:}

Again, the reasoning is the same as for $(i)$, with $\nabla \bdd$ instead of $\curl \bw$. Then we have:
\begin{eqnarray*}
\norm{(\bu_\lambda \cdot \nabla) \bdd_2^\epsilon}_{\bL^p(\Omega)} \leqslant C_1 \epsilon \norm{\bu_\lambda}_{\bW^{2,p}(\Omega)}
\end{eqnarray*}
and
\begin{eqnarray*}
\norm{(\bu_\lambda \cdot \nabla) \bdd_1^\epsilon}_{\bL^p(\Omega)}  \leqslant C_\epsilon \norm{\bdd}_{\bW^{1,\frac{3}{2}}(\Omega)} \norm{\bu_\lambda}_{\bL^{p^*}(\Omega)} 
\end{eqnarray*}
\end{itemize}

\noindent \textbf{(iv) Estimate of the term $\sum_{i=1}^I |c_i^\lambda|$}: 

With a triangle inequality, we have: 
\begin{eqnarray*}
\begin{aligned}
|c_i^\lambda| &\leqslant |\langle \ff_\lambda, \nabla q_i^N \rangle| + |\langle P_0^\lambda, \nabla q_i^N \cdot \bn \rangle | + |\int_\Omega (\curl \bw) \times \bu_\lambda \cdot \nabla q_i^N \, dx| \\
&+ |\int_\Omega (\curl \bb_\lambda) \times \bdd \cdot  \nabla q_i^N \, dx| 
\end{aligned}
\end{eqnarray*}
We can't directly bound $\abs{\int_\Omega (\curl \bw) \times \bu_\lambda \cdot \nabla q_i^N \, dx}$ and $\abs{\int_\Omega (\curl \bb_\lambda) \times \bdd \cdot \nabla q_i^N \, dx}$ with an Hölder inequality: we must use again the decomposition of $\curl \bw$ and $\bdd$ in \eqref{Decomposition of bw and bd}. 
\begin{itemize}
\item \textbf{Estimate of the term $\abs{\int_\Omega (\curl \bw) \times \bu_\lambda \cdot \nabla q_i^N \, dx}$:}

From \eqref{Estimate of y_2^epsilon}, the Sobolev embedding $\bW^{2,p}(\Omega) \hookrightarrow \bL^{p^{**}}(\Omega)$ and the Hölder inequality, we obtain:
\begin{eqnarray*}
\begin{aligned}
\abs{\int_\Omega \by_2^\epsilon \times \bu_\lambda \cdot \nabla q_i^N \, dx} &\leqslant \norm{\by_2^\epsilon}_{\bL^\frac{3}{2}(\Omega)} \norm{\bu_\lambda}_{\bL^{p^{**}}(\Omega)} \norm{\nabla q_i^N}_{\bL^{p'}(\Omega)} \\
&\leqslant \epsilon C_1 \norm{\bu_\lambda}_{\bW^{2,p}(\Omega)} \norm{\nabla q_i^N}_{\bL^{p'}(\Omega)}  
\end{aligned}
\end{eqnarray*}
with $C_1$ defined in \eqref{Estimate of y_2^epsilon times u_lambda}. Next, applying an Hölder inequality, we have: 
\begin{eqnarray*}
\begin{aligned}
\abs{\int_\Omega \by_1^\epsilon \times \bu_\lambda \cdot \nabla q_i^N \, dx} &\leqslant \norm{\by_1^\epsilon}_{\bL^q(\Omega)} \norm{\bu_\lambda}_{\bL^{p^*}(\Omega)} \norm{\nabla q_i^N}_{\bL^m(\Omega)} \\
&\leqslant \norm{\by}_{\bL^\frac{3}{2}(\Omega)} \norm{\rho_{\epsilon/2}}_{\bL^t(\Omega)} \norm{\bu_\lambda}_{\bL^{p^*}(\Omega)} \norm{\nabla q_i^N}_{\bL^m(\Omega)} \\
&\leqslant C_\epsilon \norm{\by}_{\bL^\frac{3}{2}(\Omega)} \norm{\bu_\lambda}_{\bL^{p^*}(\Omega)} \norm{\nabla q_i^N}_{\bL^m(\Omega)}
\end{aligned}
\end{eqnarray*}
with $m, q \geqslant p$ such that $\frac{1}{q} + \frac{1}{p^*} + \frac{1}{m} = 1$ and $t > 1$ such that $1 + \frac{1}{q} = \frac{2}{3} + \frac{1}{t}$, and $C_\epsilon$ already defined in \eqref{Mollifier bound}. 

\item \textbf{Estimate of the term $\int_\Omega (\curl \bb_\lambda) \times \bdd \cdot \nabla q_i^N \, dx$:}

Again, combining \eqref{Estim d_2^epsi sol p petit}, the Sobolev embedding $\bW^{2,p}(\Omega) \hookrightarrow \bW^{1,p^*}(\Omega)$ and the Hölder inequality, we have: 
\begin{eqnarray*}
\begin{aligned}
\abs{\int_\Omega (\curl \bb_\lambda) \times \bdd_2^\epsilon \cdot \nabla q_i^N \, dx} &\leqslant \norm{\curl \bb_\lambda}_{\bL^{p^*}(\Omega)} \norm{\bdd_2^\epsilon}_{\bL^\frac{3}{2}(\Omega)} \norm{\nabla q_i^N}_{\bL^{p'}(\Omega)} \\
&\leqslant C_3 \epsilon \norm{\nabla q_i^N}_{\bL^{p'}(\Omega)} \norm{\bb_\lambda}_{\bW^{2,p}(\Omega)}
\end{aligned}
\end{eqnarray*}
where $C_3$ is defined in \eqref{Estimate curl bb_lambda times d_2^epsilon}.

Now, we look at the part with $\bdd_1^\epsilon$, with the Hölder inequality: 
\begin{eqnarray*}
\begin{aligned}
\abs{\int_\Omega (\curl \bb_\lambda) \times \bdd_1^\epsilon \cdot \nabla q_i^N \, dx} &\leqslant \norm{\curl \bb_\lambda}_{\bL^q(\Omega)} \norm{\bdd_1^\epsilon}_{\bL^m(\Omega)} \norm{\nabla q_i^N}_{\bL^\alpha(\Omega)} \\
&\leqslant \norm{\bb_\lambda}_{\bW^{1,q}(\Omega)} \norm{\bdd}_{\bL^3(\Omega)} \norm{\rho_{\epsilon/2}}_{\bL^t(\Omega)} \norm{\nabla q_i^N}_{\bL^\alpha(\Omega)}
\end{aligned}
\end{eqnarray*}
with $q, m, \alpha \geqslant p$ such that $\frac{1}{q}+\frac{1}{m}+\frac{1}{\alpha} = 1$ and $t > 1$ such that $1 + \frac{1}{m} = \frac{1}{3} + \frac{1}{t}$. In order to recover \eqref{Estimation interpolation}, note that, taking $\alpha$ such that $1 - \frac{1}{\alpha} = \frac{1}{p}$, we obtain the same assumptions on $q$ and $m$. Therefore, since we can take the same $q$ as in \eqref{Estimation interpolation}, we apply it with \eqref{Mollifier bound} to obtain:
\begin{eqnarray*}
\begin{aligned}
\abs{\int_\Omega (\curl \bb_\lambda) \times \bdd_1^\epsilon \cdot \nabla q_i^N \, dx} \leqslant C_\epsilon C_2 \norm{\bdd}_{\bW^{1,\frac{3}{2}}(\Omega)} \norm{\nabla q_i^N}_{\bL^\alpha(\Omega)} (\eta \norm{\bb_\lambda}_{\bW^{2,p}(\Omega)} + C_\eta \norm{\bb_\lambda}_{\bL^{p^*}(\Omega)})
\end{aligned}
\end{eqnarray*}
\end{itemize}
Finally, noting that there exists $C_q > 0$ such that, for all $m \geqslant 1$, $\norm{\nabla q_i^N}_{\bL^m(\Omega)} \leqslant C_q$, we obtain from the previous calculus that: 
\begin{eqnarray*}
\begin{aligned}
\sum_{i=1}^I |c_i^\lambda| &\leqslant I C_q \Big( \norm{\ff_\lambda}_{\bL^p(\Omega)} + \norm{P_0^\lambda}_{W^{1-\frac{1}{p},p}(\Gamma)} + \epsilon C_1 \norm{\bu_\lambda}_{\bW^{2,p}(\Omega)} + \epsilon C_3 \norm{\bb_\lambda}_{\bW^{2,p}(\Omega)} \\
&+ \eta C_\epsilon C_2 \norm{\bdd}_{\bW^{1,\frac{3}{2}}(\Omega)} \norm{\bb_\lambda}_{\bW^{2,p}(\Omega)} + C_\epsilon \norm{\curl \bw}_{\bL^\frac{3}{2}(\Omega)} \norm{\bu_\lambda}_{\bL^{p^*}(\Omega)} \\
&+ C_\eta C_\epsilon C_2 \norm{\bdd}_{\bW^{1,\frac{3}{2}}(\Omega)} \norm{\bb_\lambda}_{\bL^{p^*}(\Omega)} \Big).
\end{aligned}
\end{eqnarray*}
Then, injecting $(i)-(iv)$ in \eqref{Estimation forte Stokes Ellipt preuve}, and taking $\epsilon$ small enough such that: 
\begin{eqnarray*}
\epsilon \Big( C_2 C_3 + \max (2 C C_1, I C_q C_3) \Big) < \frac{1}{4}
\end{eqnarray*}
and $\eta$ such that
\begin{eqnarray*}
\eta \norm{\bdd}_{\bW^{1,\frac{3}{2}}(\Omega)} C C_2 C_\epsilon  < \frac{1}{4}
\end{eqnarray*}
where $C > 0$ denotes the constant $C = (1 + I C_q)$, we obtain 
\begin{eqnarray} \label{estim presque finale preuve sol forte}
\begin{aligned}
&\norm{\bu_\lambda}_{\bW^{2,p}(\Omega)} + \norm{\bb_\lambda}_{\bW^{2,p}(\Omega)} + \norm{P_\lambda}_{W^{1,p}(\Omega)} \\
&\leqslant C_{SE} \Big[ C (\norm{\ff_\lambda}_{\bL^p(\Omega)} + \norm{\bg_\lambda}_{\bL^p(\Omega)} + \norm{P_0^\lambda}_{W^{1-\frac{1}{p},p}(\Gamma)}) \\
&+ \Big( C_\epsilon (1 + 2 C C_2 C_\eta ) \norm{\bdd}_{\bW^{1,\frac{3}{2}}(\Omega)} + C_\epsilon C \norm{\curl \bw}_{\bL^\frac{3}{2}(\Omega)} \Big) (\norm{\bu_\lambda}_{\bL^{p^*}(\Omega)} + \norm{\bb_\lambda}_{\bL^{p^*}(\Omega)}) \Big] 
\end{aligned}
\end{eqnarray}
Applying the estimate \eqref{estim u,b W^1,p for p<2} in \eqref{estim presque finale preuve sol forte}, we finally obtain the estimate: 
\begin{eqnarray} \label{estim finale preuve sol forte}
\begin{aligned}
&\norm{\bu_\lambda}_{\bW^{2,p}(\Omega)} + \norm{\bb_\lambda}_{\bW^{2,p}(\Omega)} + \norm{P_\lambda}_{W^{1,p}(\Omega)} \\
&\leqslant C_{SE} \Big( \norm{\ff_\lambda}_{\bL^p(\Omega)} + \norm{\bg_\lambda}_{\bL^p(\Omega)} + \norm{P_0^\lambda}_{W^{1-\frac{1}{p},p}(\Gamma)} \Big) \Big[ C + C_3 C_f \Big( C_\epsilon (1 + 2 C C_2 C_\eta ) \norm{\bdd}_{\bW^{1,\frac{3}{2}}(\Omega)} \\
&+ C C_\epsilon \norm{\curl \bw}_{\bL^\frac{3}{2}(\Omega)} \Big) \Big( 1 + \norm{\curl \bw}_{\bL^\frac{3}{2}(\Omega)} + \norm{\bdd}_{\bW^{1,\frac{3}{2}}(\Omega)} \Big) \Big] \\
&\leqslant C_{SE} \Big( \norm{\ff_\lambda}_{\bL^p(\Omega)} + \norm{\bg_\lambda}_{\bL^p(\Omega)} + \norm{P_0^\lambda}_{W^{1-\frac{1}{p},p}(\Gamma)} \Big) \max \Big( C, C_3 C_f C_\epsilon (1 + 2 C C_2 C_\epsilon), C C_\epsilon \Big) \\
&\times \Big( 1 + \norm{\curl \bw}_{\bL^\frac{3}{2}(\Omega)} + \norm{\bdd}_{\bW^{1,\frac{3}{2}}(\Omega)} \Big)^2
\end{aligned}
\end{eqnarray}
where $C_f$ is the constant of the estimate \eqref{estim u,b W^1,p for p<2} and $C_3$ is defined in \eqref{Estimate curl bb_lambda times d_2^epsilon}.

To conclude, from the estimate \eqref{estim finale preuve sol forte} we can extract subsequences of $\bu_\lambda$, $\bb_\lambda$ and $P_\lambda$, which are still denoted $\bu_\lambda$, $\bb_\lambda$ and $P_\lambda$, such that: 

\begin{eqnarray*}
\begin{aligned}
\bu_\lambda \rightharpoonup \bu \, \, \text{and} \, \, \bb_\lambda \rightharpoonup \bb \, \, \text{in} \, \, \bW^{2,p}(\Omega), \, \, P_\lambda \rightharpoonup P \, \, \text{in} \, \, W^{1,p}(\Omega)
\end{aligned}
\end{eqnarray*}

where $(\bu, \bb, P) \in \bW^{2,p}(\Omega) \times \bW^{2,p}(\Omega) \times W^{1,p}(\Omega)$ is solution of \eqref{linearized MHD-pressure} and satisfies the estimate  \eqref{Estimation theorem strong solution p small}. \hspace{12cm}\qquad \qquad \qquad $\square$\\
%%%%%%%%%%%%%%%%%%%%%%%%%%%%%%%%%%%%%%%%%%%%%%%%%%%%%%%%%%%%%%%%%%%%%%%%%%%%%
%%%%%%%%%%%%%%%%%%%%%%%%%%%%%%%%%%%%%%%%%%%%%%%%%%%%%%%%%%%%%%%%%%%%%%%%%%%%
Next, we give here the proof of Corollary \ref{corollary strong sol p<6/5 with h} which is the extension of the previous result to the case of non-zero divergence condition for $\bu$.\\
\textbf{Proof of Corollary \ref{corollary strong sol p<6/5 with h}:} We proceed with the same reasoning as in the Corollary \ref{cor: weak W1,p for p<2}. We recover the solution of a linearized problem with a vanishing divergence by considering the Dirichlet problem:

\begin{eqnarray*}
\Delta \theta = h \, \, \text{in} \, \, \Omega \, \, \text{and} \, \, \theta = 0 \, \, \text{on} \, \, \Gamma 
\end{eqnarray*}
and setting $\bz = \bu - \nabla \theta$, thus $(\bz, \bb, P, \bc)$ is the solution of \eqref{linearized MHD-pressure} in the Theorem \ref{thm: strong W2,p for 1<p<6/5} with $\ff$ and $\bg$ replaced by $\tilde{\ff} = \ff + \nabla h - (\curl \bw) \times \nabla \theta$ and $\tilde{\bg} = \bg + \curl (\nabla \theta \times \bdd)$. 

Indeed, the assumptions of the Theorem \ref{thm: strong W2,p for 1<p<6/5} are satisfied: since $\nabla \theta \in \bW^{2,p}(\Omega) \hookrightarrow \bL^{p^{**}}(\Omega)$ and $\curl \bw \in \bL^\frac{3}{2}(\Omega)$, then $(\curl \bw) \times \nabla \theta \in \bL^p(\Omega)$. Moreover, we have by hypothesis $\nabla h \in \bL^p(\Omega)$ so $\tilde{\ff} \in \bL^p(\Omega)$. In the same way, we have $\tilde{\bg} \in \bL^p(\Omega)$ and since we only add a curl to $\bg$, then $\tilde{\bg}$ satisfies the conditions \eqref{Condition compatibilite K_N Lp}-\eqref{condition div g=0 sol W^1,p linear MHD}. Hence, we recover the estimate: 
\begin{eqnarray} \label{Estimation 1 preuve cor div non nulle sol forte}
\begin{aligned}
&\norm{\bz}_{\bW^{2,p}(\Omega)} + \norm{\bb}_{\bW^{2,p}(\Omega)} + \norm{P}_{W^{1,p}(\Omega)} \\
&\leqslant C_F \Big( 1 + \norm{\curl \bw}_{\bL^\frac{3}{2}(\Omega)} + \norm{\bdd}_{\bW^{1,\frac{3}{2}}(\Omega)} \Big)^2 \Big( \Vert\tilde{\ff}\Vert_{\bL^p(\Omega)} + \norm{\tilde{\bg}}_{\bL^p(\Omega)} + \norm{P_0}_{W^{1-\frac{1}{p},p}(\Gamma)} \Big)
\end{aligned}
\end{eqnarray}
We must bound $\Vert\tilde{\ff}\Vert_{\bL^p(\Omega)}$ and $\norm{\tilde{\bg}}_{\bL^p(\Omega)}$ term by term: 
\begin{itemize}
\item Applying the Hölder inequality and the Sobolev embedding $\bW^{2,p}(\Omega) \hookrightarrow \bL^{p^{**}}(\Omega)$, we have: 
\begin{eqnarray} \label{Estimation 2 preuve cor div non nulle sol forte}
\begin{aligned}
\Vert\tilde{\ff}\Vert_{\bL^p(\Omega)} &\leqslant \norm{\ff}_{\bL^p(\Omega)} + \norm{\nabla h}_{\bL^p(\Omega)} + \norm{\curl \bw}_{\bL^\frac{3}{2}(\Omega)} \norm{\nabla \theta}_{\bL^{p^{**}}(\Omega)} \\
&\leqslant \norm{\ff}_{\bL^p(\Omega)} + \norm{h}_{W^{1,p}(\Omega)} + C_1 \norm{\curl \bw}_{\bL^\frac{3}{2}(\Omega)} \norm{h}_{W^{1,p}(\Omega)} 
\end{aligned}
\end{eqnarray}
where $C_1$ denotes the constant of the Sobolev embedding $\bW^{2,p}(\Omega) \hookrightarrow \bL^{p^{**}}(\Omega)$. 

\item Note that $\vdiv \nabla \theta = \Delta \theta = h$, thus we rewrite $\curl (\nabla \theta \times \bdd) = - h \bdd + (\bdd \cdot \nabla) \nabla \theta - (\nabla \theta \cdot \nabla) \bdd$. Thus, we have: 
{\footnotesize
\begin{eqnarray} \label{Estimation 3 preuve cor div non nulle sol forte}
\begin{aligned}
\norm{\tilde{\bg}}_{\bL^p(\Omega)} &\leqslant \norm{\bg}_{\bL^p(\Omega)} + \norm{h \bdd}_{\bL^p(\Omega)} + \norm{(\bdd \cdot \nabla) \nabla \theta}_{\bL^p(\Omega)} + \norm{(\nabla \theta \cdot \nabla) \bdd}_{\bL^p(\Omega)} \\
&\leqslant \norm{\bg}_{\bL^p(\Omega)} + \norm{h}_{L^{p^{**}}(\Omega)} \norm{\bdd}_{\bL^3(\Omega)} + \norm{\bdd}_{\bL^3(\Omega)} \norm{\theta}_{W^{2,p^*}(\Omega)} + \norm{\nabla \theta}_{\bL^{p^{**}}(\Omega)} \norm{\nabla \bdd}_{\bL^\frac{3}{2}(\Omega)} \\
&\leqslant \norm{\bg}_{\bL^p(\Omega)} + 2 C_2 C_3 \norm{h}_{W^{1,p}(\Omega)} \norm{\bdd}_{\bW^{1,\frac{3}{2}}(\Omega)} + C_1 \norm{h}_{W^{1,p}(\Omega)} \norm{\bdd}_{\bW^{1,\frac{3}{2}}(\Omega)} 
\end{aligned}
\end{eqnarray}}
where $C_2$ and $C_3$ are the constants respectively related to the Sobolev embeddings $\bW^{1,\frac{3}{2}}(\Omega) \hookrightarrow \bL^3(\Omega)$ and $\bW^{1,p}(\Omega) \hookrightarrow \bL^{p^*}(\Omega)$. 
\end{itemize}
Hence, combining \eqref{Estimation 2 preuve cor div non nulle sol forte} and \eqref{Estimation 3 preuve cor div non nulle sol forte} with \eqref{Estimation 1 preuve cor div non nulle sol forte}, it follows that: 
\begin{eqnarray*}
\begin{aligned}
&\norm{\bz}_{\bW^{2,p}(\Omega)} + \norm{\bb}_{\bW^{2,p}(\Omega)} + \norm{P}_{W^{1,p}(\Omega)} \\
&\leqslant C_F \Big( 1 + \norm{\curl \bw}_{\bL^\frac{3}{2}(\Omega)} + \norm{\bdd}_{\bW^{1,\frac{3}{2}}(\Omega)} \Big)^2 \Big( \norm{\ff}_{\bL^p(\Omega)} + \norm{\bg}_{\bL^p(\Omega)} + \norm{P_0}_{W^{1-\frac{1}{p},p}(\Gamma)} \\
&+ \norm{h}_{W^{1,p}(\Omega)} \Big( 1 + C_1 \norm{\curl \bw}_{\bL^\frac{3}{2}(\Omega)} + (2C_2 C_3 + C_1) \norm{\bdd}_{\bW^{1,\frac{3}{2}}(\Omega)} \Big) \\
&\leqslant C_F \Big( 1 + \norm{\curl \bw}_{\bL^\frac{3}{2}(\Omega)} + \norm{\bdd}_{\bW^{1,\frac{3}{2}}(\Omega)} \Big)^2 \Big( \norm{\ff}_{\bL^p(\Omega)} + \norm{\bg}_{\bL^p(\Omega)} + \norm{P_0}_{W^{1-\frac{1}{p},p}(\Gamma)} \\
&+ \max \Big( 1, 2C_2 C_3 + C_1 \Big) \norm{h}_{W^{1,p}(\Omega)} (1 + \norm{\curl \bw}_{\bL^\frac{3}{2}(\Omega)} + \norm{\bdd}_{\bW^{1,\frac{3}{2}}(\Omega)}) \Big)  
\end{aligned}
\end{eqnarray*}
Finally, we use triangle inequality and we get the estimate \eqref{estim cor sol fortes div non nulle p petit}.  \hspace{5cm} $\square$

% 	
% \begin{thebibliography}{99}
%   \bibitem{Am-Saliha} {\sc C. Amrouche and S. Boukassa}, {\it 
% Existence and regularity of solution for a model in
% magnetohydrodynamics}. Nonlinear Analysis 190 (2020) 111602.
\bibliographystyle{plain}

\end{document}